\documentclass[a4paper,10pt]{amsart}
\usepackage{longtable,geometry}
\usepackage{graphicx}
\usepackage[english]{babel}
\usepackage[utf8]{inputenc}
\usepackage[T1]{fontenc}
\usepackage{amsmath,amsfonts,amssymb}
\usepackage{bbm} 
\usepackage{pict2e}
\usepackage[dvipsnames]{xcolor}
\usepackage{url}
\usepackage[normalem]{ulem}
\usepackage{array}
\usepackage{appendix}
\usepackage{chngcntr}
\usepackage{calrsfs}

\usepackage[foot]{amsaddr}

\usepackage{hyperref}

\newenvironment{aligne}[1]{ \everymath={\displaystyle}\begin{array}{#1}}{\end{array}}

\numberwithin{equation}{section} 

\newtheorem{theorem}{Theorem}
\newtheorem{Claim}{Claim}
\newtheorem{prop}{Proposition}[section]
\newtheorem{corollary}[prop]{Corollary}
\newtheorem{lemma}[prop]{Lemma}
\newtheorem{remark}{Remark}[section]\theoremstyle{theorem}
\newtheorem{definition}{Definition}[section]\theoremstyle{theorem}
\newtheorem{assumption}[prop]{Assumption}

\renewcommand{\paragraph}[1]{\noindent\textbf{#1}}

\newcolumntype{C}[1]{>{\centering}m{#1}}

\newcommand{\gr}[1]{{\mathbf{#1}}}
\newcommand{\ds}[1]{{\mathsf{#1}}}
\newcommand{\dr}[1]{{\rm #1}}

\newcommand{\ind}{\mathbbm{1}}
\newcommand{\ui}{{\underline{i}}}

\newcommand{\N}{\mathcal{N}}
\newcommand{\so}{\overset{\circ}{\sim}}
\newcommand{\e}{\varepsilon}

\newcommand{\Id}{{\rm Id}}

\newcommand{\R}{\mathbb{R}}
\newcommand{\ud}{\,\mathrm{d}}
\newcommand{\KK}{{\mathcal K}}

\newcommand{\DDelta}{\Delta\hspace{-6pt}\Delta}
\newcommand{\CCirc}{O\hspace{-5pt}O}

\renewcommand{\epsilon}{{\e_{\rm eff}}}

\DeclareMathOperator{\tr}{\dr{tr}}

\makeindex

\newcommand{\red}[1]{
}
\newcommand{\blue}[1]{#1
}

\title[Long time validity of the linearized Landau and uncut-off  Boltzmann equations]{Long time validity of the linearized Boltzmann uncut-off and the linearized Landau equations from the Newton Law}

\begin{document}

	\author{Corentin Le Bihan}
	\email{corentin.le.bihan@ulb.be}
	
	\begin{abstract}
		We provide a rigorous justification of the linearized Boltzmann and Landau equations for interacting particle systems with long-range interaction. The result shows that for a system of $N$ Hamiltonian particles governed by truncated power law potentials  of the form $\mathcal{U}_\epsilon(r)\sim |r/\epsilon|^{-s}$ near $r \approx 0$ (\blue{with $\epsilon$ the effective radius of the particles}),  \blue{the covariance of the equilibrium fluctuations } converges to solutions of kinetic equations in appropriate scaling limits $\e_{\rm eff} \rightarrow 0$ and $N\rightarrow \infty$, \blue{corresponding to a low density regime}. We prove that \blue{in dimension $3$}, for $s\in \red{[0,1)} \blue{(1,\infty)}$, the limiting system approaches the uncutoff linearized Boltzmann equation \red{ or the linearized Landau equation, depending on the scaling limit} \blue{for the scaling $N\epsilon^2=1$}. The Coulomb singularity $s=1$ appears as a threshold value. Kinetic scaling limits with $s\in (0,1]$ universally converge to the linearized Landau equation, and we prove the onset of the Coulomb logarithm for $s=1$. \red{To the best of our knowledge, this is the first result on the derivation of kinetic equations from interacting particle systems with long-range power-law interaction.}
	\end{abstract}

	\maketitle
	\tableofcontents
	
	\section{Introduction}
		In kinetic theory, a gas of particles can be modeled by a large system of $N$ classical particles interacting {\it via} a potential $\mathcal{U}(\cdot/\epsilon)$\blue{, with $\epsilon$ the \emph{effective radius} of a particle}. \blue{The dynamics is Hamiltonian and obeys the Newton laws with assigned potential $\mathcal{U}$.} In dimension $3$, the power laws $\mathcal{U}_s(r):=r^{-s}$, $s\geq 1$ play a fundamental role, in particular the Coulomb case $s= 1$. One of the goals of kinetic theory is the description of such a gas in the limit $N\to\infty$, $\epsilon\to 0$. Of course, \blue{the asymptotic behavior} depends on the relation between $\epsilon$ and $N$. Collisional kinetic theory concerns low density scalings, where the occupied volume $N\epsilon^3$ goes to $0$. \blue{The good scaling limit looks to be the Boltzmann--Grad limit $N\epsilon^2=1$, at least for $s>1$ (see \cite{Grad}). It corresponds to a regime such that on a time interval of length $1$, any particle passes nearby another particle (at distance less than $\epsilon$, see Remark \ref{rem:disclamer} for a comment on this scaling).}
	
	In the case $s>1$, if at time 0 the particles are "sufficiently independent", the distribution of a typical particle $f(t,x,v)$ is a solution of the Boltzmann equation (introduced by Maxwell in 1867 \cite{Maxwell} and Boltzmann in 1871 \cite{Boltzmann})
	\begin{equation}\label{eq:Boltzmann}\begin{aligne}{c}
	\partial_t f + v\cdot\nabla_x f = Q_s(f,f)\\
	Q_s(f,h)(v):=\int_{\mathbb{R}^3\times\mathbb{S}^2}\left(f(v')h(v_*')-f(v)h(v_*)\right)b_s(v-v_*,\eta)\ud{v_*}\ud{\eta},\\
	v' = \frac{v+v_*}{2}+\frac{|v-v_*|}{2}\eta,~v_*' = \frac{v+v_*}{2}-\frac{|v-v_*|}{2}\eta,
	\end{aligne}\end{equation}
	where the kernel $b_s$ depends on the potential $\mathcal{U}_s(\cdot)$. The collision operator $Q$ can be interpreted  as a jump operator for the velocities. \blue{The rigorous proof of such a theorem for long range interactions is now an open problem, and this paper can be interpreted as an intermediate result.}
	
	For a power law $\mathcal{U}_s$ with $s\geq1$, the kernel $b_s$ is equal to 
	\begin{equation}\label{eq:noyau pour loi de puissance}
		b_s(z,\eta) = |z|^{\frac{s-4}{s}}q_s(z\cdot \eta), ~{\rm with}~ q_s(\cos \theta)\underset{\theta \sim 0}\sim K \theta^{-\frac{2+s}{s}}
	\end{equation}
	for some constant $K$. Hence, the collision kernel is not integrable near the singularity $\eta\cdot\tfrac{v-v_*}{|v-v_*|}\simeq 1$ (when the collisions are grazing). We say that the Boltzmann kernel has no cutoff. However, the Boltzmann operator $Q_s$ can be defined \blue{(everywhere) for $s>1$} if the functions $f$ and $h$ are differentiable (see \cite{Goudon,Villani1}). 
	
	\blue{In the Coulombian case $s=1$, the singularity near $\theta \sim 0$ is too large to define the collision operator, even for smooth functions. In 1936, Landau proposed in \cite{Landau} an operator that describes the collision between Coulombian particles. His analysis was based on cutting the long range interaction, responsible for the "collisions" with small deviation angle $\theta$. His argument was that in a neutral Coulomb gas, the interactions between ions (big positively charged particles) are screened by electrons (small negatively charged particles). Hence, the particles interact {\it via} the effective potential $\tfrac{\chi(|x|/R)}{|x|}$ where $\chi \simeq \ind_{[0,1]}$ is a cutting function, and $R$ is called the Debye radius\footnote{In the Physics literature the constant $R$ is typically interpreted as a normalisation constant, to be estimated}. In the present paper we study the limit $R\to \infty$. ($R\gg 1$).}
	\blue{
		Landau finally obtained the collision operator $c(\log R) Q_L(f,f)$ where $c$ is a diffusion coefficient, $\log R$ is the \emph{Coulomb logarithm} coming from the singularity, and $Q_L$ is defined by
		\begin{equation}
			Q_L(f,h)(v):=2\pi\nabla_v\cdot\int_{\mathbb{R}^3}\frac{|v-v_*|^2{\rm Id}-(v-v_*)^{\otimes  2}}{|v-v_*|^3}\big(\nabla f(v)h(v_*)- f(v)\nabla h(v_*)\big)\ud{v_*}.
	\end{equation}}
	The Landau equation is 
	\begin{equation}
		\partial_t f + v\cdot\nabla_x f = cQ_L(f,f).
		\end{equation}
		
	\blue{\begin{remark}\label{rem:disclamer}
		It is not clear if, in the case $s\leq 2$, the Boltzmann equation is really the good description in the Boltzmann-Grad limit $N\epsilon^2= 1$ of the particle system with interaction through the long range potential $\mathcal{U}_{s}(\tfrac x\epsilon)=\tfrac{\epsilon^s}{|x|^s}$.
		
		Indeed, it can be compared with the mean-field regime $N\epsilon^2 = \epsilon^{2-s}\ll1$ (the interaction potential is $\tfrac{1}{N}\mathcal{U}_{s}(x)$), which is much more dilute. Then, the mean-field effects should dominate the collisional effects. One can  expect that, because of some "screening effect" (similar to that in the Coulomb case), particles interact via an effective potential with faster decay at infinity.
		
		For hard potentials $s>2$, the Boltzmann-Grad scaling is dominant, and the Boltzmann equation should be the true asymptotic description of the system. A similar discussion has been performed in \cite{SW} for the definition of the equilibrium measure for particles interacting through the potential $\mathcal{U}_s(\tfrac{\cdot}{\epsilon})$.
	\end{remark}
		
	In this paper we will work with microscopically truncated interactions and therefore avoid the discussion of screening properties. Our purpose, however, is to obtain uncutoff equations by removing the truncation. }
	\medskip
	
	\blue{The first rigorous derivation of the  Boltzmann equation was performed by Lanford for hard-spheres \cite{Lanford} and by King \cite{King} for regular potentials with compact support.}
	
	Our strategy for deriving the Boltzmann (or Landau) equation associated with the potential $\mathcal{U}_s$ is to split the problem into two steps. First, we consider \blue{a system of particles interacting {\it via}} the \blue{screened} potential \blue{\[\mathcal{U}_{s,R}(\tfrac{x}{\epsilon}):=\chi(\tfrac{|x|}{R\epsilon})\tfrac{\epsilon^s}{|x|^s},\] where $R$ is a cutoff radius (which will go to infinity),} $\chi(r):\mathbb{R}^+\to[0,1]$ is a smooth, decreasing cutoff function:
	\begin{equation*}
		\chi(0)=1,~\chi([1,\infty[)=\{0\},~\chi'\leq 0.
	\end{equation*}
		
	\blue{Taking the Boltzmann-Grad limit $N\to\infty$, $N\epsilon^2=1$ ($N\epsilon^2=(\log R)^{-1}$ in the Coulomb case $s=1$), we recover the cut-off Boltzmann equation. Defining a "collision" between two particles as the moment when they get closer than the effective radius $\epsilon$, the Boltzmann-Grad scaling is chosen such that a particle has on average one collision per unite of time.
	
	In a second time, we take the grazing collision limit $R\to \infty$ to pass from the cut-off Boltzmann equation to the Boltzmann equation associated with $\mathcal{U}_s$ if $s>1$ (respectively, the Landau equation if $s=1$).
	
	If we obtain some quantitative estimates for the two regimes, one can hope to take the limits $R\to\infty,N\to\infty$ simultaneously, assuming that $R$ grows much slower than $N$ (we will need $R$ of order $O((\log\log N)^{1/6})$).}
	
	The big difficulty is that we need the validity of the cut-off Boltzmann equation on a large time interval of order $O(1)$ (to be compared to the validity time $O(R^{-2})$ obtained by King \cite{King} for interaction potentials supported in a ball of radius $R$, see the next section). In order to get a long-time result, we look at a linearized version of the system near the thermodynamic equilibrium (or Gibbs state). This equilibrium can be defined as the probability law with density
	\begin{equation}\label{eq:measure de Gibbs canonique}
		M^N_{\epsilon,R}(X_N,V_N):=\frac{1}{\mathcal{Z}^N_{\epsilon,R}}\exp\left(-\sum_{i= 1}^N\frac{|v_i|^2}{2}-\sum_{1\leq i<j\leq N}\mathcal{U}_{s,R}\left(\frac{x_i-x_j}{\epsilon}\right)\right)
	\end{equation}
	at positions $X_N=(x_1,\cdots,x_N)$ and velocities $V_N=(v_1,\cdots,v_N)$. The term $\mathcal{Z}^N_{\epsilon,R}$ is a normalization constant such that
	\[\int M_{\epsilon,R}^N\ud X_N\ud V_N =1.\]
	
	We want to understand the fluctuation field $\zeta_\epsilon^t$ around the equilibrium: for a test function $g$, we define
	\begin{equation}\zeta_\epsilon^t(g):= \sqrt{N}\left(\frac{1}{N}\sum_{i=1}^N g(\gr{x}^\epsilon_i(t),\gr{v}^\epsilon_i(t))-\mathbb{E}_\epsilon\left[\frac{1}{N}\sum_{i=1}^N g(\gr{x}^\epsilon_i(t),\gr{v}^\epsilon_i(t))\right]\right).\end{equation}
	In the previous equality, $(\gr{x}^\epsilon_i(t),\gr{v}^\epsilon_i(t))$ denotes the position and velocity of the $i$-th particle at time $t$, and the expectation is taken with respect to the Gibbs measure $M_{\epsilon,R}^N\ud{X_N}\ud{V_N}$.	\blue{The fluctuation field $\zeta_\epsilon^t$ has been studied by Bodineau {\it et al} in the hard sphere system. They have shown $\zeta_\epsilon^t$ verifies a central limit theorem, and converges to a Gaussian field $\zeta_0^t$ described by the linearized Boltzmann equation (see \cite{BGS,BGSS1,BGSS2,LeBihan2}). Our study is focused on the description of the covariance for more general interaction potential.}
	
	One can now write a first informal version of the theorem proved in the present paper (a rigorous version is written in Theorem \ref{thm:derivation Bolzmann sans cut-off}\footnote{\blue{The vague Claims \ref{thm:proto} and \ref{thm:proto bis} are written for simplicity of the presentation in the \emph{canonical} setting, {\it id est} the number of particles is a constant $N$ that goes to infinity. In fact, we will rather work in the \emph{grand canonical} setting, where the number of particles is a random variable $\N$, with $\mathbb{E}\N$ going to infinity. The system becomes less rigid, and it avoids the apparition of additional error terms. However, one can expect that the canonical and grand canonical systems behave asymptotically in the same way (although the proof is missing).}}).
	\begin{Claim}\label{thm:proto}
		Consider a system of $N$ particles evolving with respect to Newton's laws, interacting through the pairwise potential $\mathcal{U}_{s,R}(\cdot/\epsilon)$, \blue{with $s\geq 1$}. At time zero, the particles are distributed according to the Gibbs equilibrium measure.
		Parameters $N,\epsilon,R$ tuned as
		\[N\to \infty,~R\to \infty,~ R=o\left((\log\log N)^{1/6}\right),~{\rm and}~N\epsilon^2
		=\left\{\begin{array}{cl}
		(\log R)^{-1}&~{\rm if}~s=1,
		\\ 1&~\blue{ {\rm  if~} s>1}.
		\end{array}\right.\]
		
		Fix $g$ and $h$ two test functions. Then denoting $M(v):=\frac{e^{-\frac{|v|^2}{2}}}{(2\pi)^{3/2}}$, we have that 
		\begin{equation}
		\mathbb{E}_{\blue{\epsilon}}\left[\zeta_\epsilon^t(g)\zeta_\epsilon^0(h)\right]\underset{\e\to0}\longrightarrow\int \gr{g}(t,x,v) h(x,v) M(v) \ud{x}\ud{v}
		\end{equation}
		with $\gr{g}(t,x,v)$ the solution of the linearized equation
		\begin{equation}\left\{\begin{split}
		\partial_t \gr{g} + v\cdot\nabla_v\gr{g} &= \mathcal{L}_\infty \gr{g},\\
		\gr{g}(t=0,x,v)&=g(x,v)		
		\end{split}\right. ~{\rm where}~
		\mathcal{L}_\infty \gr{g} := \left\{\begin{split} \frac{1}{M}\left(Q_L(Mg,M)+Q_L(M,Mg)\right)&~{\rm if}~s=1,\\ \frac{1}{M}\left(Q_s(Mg,M)+Q_s(M,Mg)\right)&~\blue{{\rm if}~s>1}.\end{split}\right.\end{equation}
	\end{Claim}
	
		\subsection{Modification of the scaling parameters \blue{and the sub-Coulomb case ($s\in[0,1)$)}}
	For a fix $s\geq1$ and a cut-off function $\chi:\mathbb{R}^+\to[0,1]$, we define the \blue{\emph{interaction radius} $\e:= {R}\,{\epsilon}$,  the \emph{coupling constant} $\alpha=R^{-s}$ } and the interaction potential
	\[\mathcal{V}(x):=\frac{\chi(|x|)}{|x|^s}. \]
	
	Hence, we have the equality $\mathcal{U}_{s,R}(x/\epsilon)=\alpha\mathcal{V}(x/\e)$.
	
	\blue{The scaling parameters $(\e,\alpha)$ and $(\epsilon,R)$ are two different parameterizations of the same system, and taking 
		\[N\to \infty,~R\to \infty,~ R=o\left((\log\log N)^{1/6}\right),~{\rm and}~N\epsilon^2
		=\left\{\begin{array}{cl}
			(\log R)^{-1}&~{\rm if}~s=1
			\\ 1&~\blue{ {\rm  if~} s>1}
		\end{array}\right.\]
		is equivalent to taking
		\[N\to \infty,~\alpha\to 0,~ \frac{1}{\alpha}=o\left((\log\log N)^{s/12}\right),~{\rm and}~N\e^2
		=\left\{\begin{array}{cl}
			\alpha^{-2}|\log \alpha|^{-1}&~{\rm if}~s=1
			\\ \alpha^{-2/s}&~\blue{ {\rm  if~} s>1}
		\end{array}\right.\]
		
		In the core of the proof (from Section \ref{sec: def of the system} to the end), we prefer using  $(\e,\alpha)$ as it will simplify the notation. However, the couple $(\epsilon,R)$ may be more natural to describe a gas interacting {\it via} a power law $\mathcal{U}_s$, $s\geq 1$.}
	\medskip
	
	\blue{Another advantage of taking the parameters $(\e,\alpha)$ is that it allows a natural generalization of the result to potentials in the sub-Coulomb range $s\in[0,1)$ which is the microscopic interpretation of the so-called \emph{grazing collision limit.}
		
	We look at a system of $N$ particles interacting pairwise {\it via} the potential $\alpha\mathcal{V}(\cdot/\e)$. At time $t=0$, they are distributed with respect to the canonical Gibbs measure
	\[M^N_{\e,\alpha}(X_N,V_N):=\frac{1}{\mathcal{Z}_{N,\e,\alpha}}\exp\left(-\sum_{i= 1}^N\frac{|v_i|^2}{2}-\alpha\sum_{1\leq i<j\leq N}\mathcal{V}\left(\frac{x_i-x_j}{\e}\right)\right),\]
	where $\mathcal{Z}_{N,\e,\alpha}$ is a normalization constant, and $\zeta_\e$ the fluctuation field (defined in the same way as $\zeta_\epsilon$). Then we can write an informal theorem (with the rigorous version written in Theorem \ref{thm:derivation Bolzmann sans cut-off}):	}	
	\blue{\begin{Claim}\label{thm:proto bis}
		Consider a system of $N$ particles evolving according to Newton's laws, interacting through the pairwise potential $\alpha\mathcal{\mathcal{V}}(\cdot/\e)$, where $s\in [0,1)$. At time zero, the particles are distributed with respect to the equilibrium measure.
		Parameters $N,\e,\alpha$ are tuned as
		\[N\to \infty,~\alpha\to 0,~ \tfrac1\alpha=o((\log\log N)^{1/12}),~{\rm and}~N\e^2=\alpha^{-2}.\] 
		
		Fix $g$ and $h$ two test functions. Then, denoting $M(v):=\frac{e^{-\frac{|v|^2}{2}}}{(2\pi)^{3/2}}$, 
		\begin{equation}
			\mathbb{E}_\e\left[\zeta_\e^t(g)\zeta_\e^0(h)\right]\underset{\e\to0}\longrightarrow\int \gr{g}(t,x,v) h(x,v) M(v) \ud{x}\ud{v}
		\end{equation}
		with $\gr{g}(t,x,v)$ the solution of the linearized equation
		\[\left\{\begin{array}{c}
			\partial_t \gr{g} + v\cdot\nabla_v\gr{g} =\frac{c_{\mathcal{V}}}{M}\Big(Q_L(Mg,M)+Q_L(M,Mg)\Big),\\
			\gr{g}(t=0,x,v)=g(x,v),
		\end{array}\right.\]
		where $c_{\mathcal{V}}$ is a diffusion constant defined by 
		\begin{equation}
			\displaystyle c_{\mathcal{V}} = \frac1{16\pi^2}\int_{\mathbb{R}^3} \delta(k\cdot \vec{e}_1) |k|^2 |\hat{{\mathcal{V}}}(k)|^2 \ud{k}
		\end{equation}
		where $\vec{e}_1\in\mathbb{S}^2$ an unitary vector and $\hat{{\mathcal{V}}}(k)$ is the Fourier transform of ${\mathcal{V}}$ using the convention
		\[\hat{{\mathcal{V}}}(k)= \int \mathcal{V}(x)e^{-ik\cdot x} \ud x.\]
	\end{Claim}}
		
	\blue{
	\begin{remark}\label{rem: lien notre scaling et le scaling des collisions rasantes}
		The Landau equation is usually derived from interacting particle systems in the \emph{weak coupling limit}: one fixes
		\[\alpha = \e^{1/2},~~N\e^3 = 1.\]
		For example, Bobylev, Pulvirenti, and Saffirio provided in \cite{BPS} a consistency result (a result at  time $0$) for smooth interaction in this scaling.
		
		In the present article, we are only able to treat the cases 
		\[\alpha^{-1} = o(\log|\log\e|)^{1/12}\ll \e^{-1/2},~~N\e^2\alpha^2 = 1,\]
		which are far from the weak coupling limit. The same Landau equation is, however, expected to hold in all intermediate regimes connecting Boltzmann-Grad and weak-coupling scalings (see on that subject \cite{NSV,NVW1,NVW,PS3}).
	\end{remark}
	}

	\subsection{State of the art}\label{subsec:State of the art}
	
	Now we recall some results about the derivation of the Boltzmann and Landau equation \blue{from a particle system}.
	
	In the nonlinear setting, the only results hold for potentials $\mathcal{U}(\cdot)$ supported in a ball $\{x\in\mathbb{R}^3,\,|x|\leq R\}$. In the Boltzmann-Grad scaling $N\epsilon^2=1$, the distribution of a typical particle follows the Boltzmann equation up to a time $O(1/R^2)$. The first derivation was performed by Lanford \cite{Lanford} for hard spheres ({\it i.e.} $\exp(-\mathcal{U}_{\rm hs}(r))=\ind_{r>1}$) and King \cite{King} for more general compactly supported potentials (see also \cite{GST,PSS,Delinger,BGSS4}). The previous results have two defects. They are valid only up to a small time (for the atmosphere at temperature $20^\circ\!{\rm C}$ and pressure $10^5{\rm Pa}$, the validity time is $10^{-9}s$), and the results apply only to a compactly supported interaction potential. A first long-time result out of equilibrium is \cite{IP}, in a setting where the dispersive effects are dominant. More recently, Deng, Hani and Ma provide a long time result out of equilibrium \cite{DHM} for more general initial datum $f_0$. Their result holds up to the minimum between the existing time of $f(t)$, the solution of the Boltzmann equation with initial data $f_0$, and a time $O(\log | \log N |^\alpha)$ for some $\alpha\in (0,1)$ given by a sampling time strategy.
	
	For the Landau equation, the unique results hold only at time $0$ (see \cite{BPS,Winter,Duerinckx}): the authors obtain the equality
	\[(\partial_t f)_{|t=0} = -v\cdot \nabla_x f_0 +Q_L(f_0,f_0).\]
	\blue{Note that in \cite{BPS,Winter}, the authors do not look at the real particle system but at a simplified version (they truncate the BBGKY hierarchy). In \cite{Duerinckx}, Duerinckx proves the consistency of the Lenard--Balescue equation, which can be understood as a modification of the Landau equation. However, the scaling is far from the collisional scaling that we are treating here. }
	
	\blue{This is not the first attempt to derive a linear version of the Boltzmann or Landau equations.
		
	A linear equation can be obtained from the study of the \emph{Lorentz gas}: one fixes a background of obstacles distributed with respect to the Poisson measure of parameter $N$. One follows a unique particle that interacts pairwise with the obstacles {\it via} the potential $\mathcal{U}(\cdot/\epsilon)$ with $N\e^2=1$. In the limit $N\to \infty$, the density of the tagged particle $f(t,x,v)$ follows the linear Boltzmann equation
	\begin{equation}
		\partial_t f(t,x,v) + v\cdot\nabla_x f(t,x,v) = \int_{\mathbb{S}^2} \left(f(t,x,|v|\eta)-f(t,x,v)\right) b(\eta\cdot\tfrac{v}{|v|})\ud \eta 
	\end{equation}
	where $b$ is a collision kernel. It does not depend on the norm of $v$ as the obstacles are fixed. This system was first described by Gallavotti \cite{Gallavotti} in the case of the hard spheres, and later adapted for the derivation of the linear Boltzmann equation without cutoff \cite{DP} and linear Landau equation \cite{DR}. Note that in \cite{DP}, the authors used, in the same way as this paper, interaction through a screened potential $\tilde{\mathcal{U}}(x) = \ind_{|x|\leq R}\left(|x|^{-s}-R^{-s}\right)$ for $s>1$.}
	
	\blue{A second possibility is the treatment of the \emph{linear particle setting}. We look at a system of $N$ interacting particles, initially distributed with respect to the Gibbs measure \eqref{eq:measure de Gibbs canonique}. Then, one wants to follow a tagged particle of the system, {id est} compute the covariance (for $g$, $h$ two test functions)
	\begin{equation}
		\lim_{N\to \infty}\mathbb{E}_\epsilon\left[h(\gr{x}_1^\epsilon(t),\gr{v}_1^\epsilon(t))g(\gr{x}_1^\epsilon(0),\gr{v}_1^\epsilon(0))\right] = \int h(x,v) \gr{g}(t,x,v) M(v) \ud v \ud x
	\end{equation}
	where $\gr{g}(t,x,v)$ is solution of the linear Boltzmann equation
	\begin{equation}
			\partial_t g(t,x,v) + v\cdot\nabla_x g(t,x,v) = \frac{1}{M(v)}Q_{\mathcal{U}}(Mg,M),
	\end{equation}
	where $Q_{\mathcal{U}}$ is the Boltzmann operator associated with the potential $\mathcal{U}$. The first long-time result on such a system was obtained in \cite{vBHLS}. Later, Bodineau {\it et al.} \cite{BGS} provided a quantitative proof in the hard sphere setting. This proof has been adapted to other potentials in order to derive the linear Landau equation \cite{Catapano1} or the linear Boltzmann equation without cutoff \cite{Ayi}. In her paper, Ayi does not consider interactions through the cut-off potential $\mathcal{U}_{s,R}$, but directly a long-range potential $\mathcal{U}$ with fast decay at infinity (she needs $\mathcal{U}(r)\leq O(\exp(-\exp\exp|x|^4))$). Up to our knowledge, it is the unique result where the particles interact through genuine infinite potential.}
	
	\blue{While the linear setting is a $O(1)$ perturbation of equilibrium, the linearized setting (which is treated in the present paper) is a $O(N)$ perturbation of equilibrium. In \cite{Spohn,Spohn2}, Spohn showed that the fluctuation field $\zeta_\epsilon$ formally verifies a Central Limit Theorem in the limit $N\to\infty$. The first step to prove such a theorem is to compute the covariance of the fluctuation field. This has been performed in the hard sphere setting by Bodineau {\it et al.} in \cite{BGS} in dimension 2 and \cite{BGSS1} in any dimension bigger than $3$ (note also \cite{LeBihan2}). Then, they proved that the limit is Gaussian in \cite{BGSS2}.}

\red{	\subsection{Main step of the proof}
	
	We present now the main step of the proof. As explained before, the proof can be split into two pieces: first, we derive the linearized cut-off Boltzmann equation from the particle system, and second, we pass from the cut-off linearized Boltzmann equations to the linearized Landau equation (or the uncut-off Boltzmann). The second step has already been treated by Raphael Winter and the author \cite{LW}. The main contribution of the present paper is the treatment of the first step, {\it i.e.} the derivation of the cut-off linearized Boltzmann equation.
	
	This problem was already solved by Bodineau {\it et al} in the hard sphere setting \cite{BGSS1}. However, they need a refined result of Billard theory \cite{BFK} to control the dynamical memory effect (called recollision). It is an explicit bound on the number of collisions that can occur between a fixed number of particles. Such a result cannot be easily generalised to other interaction potentials, first because the "number of collision" is not well defined (particles can overlap). In \cite{LeBihan2}, the author provided a proof avoiding the result of \cite{BFK}. It is based on subtle conditioning of the initial data, allowing to control locally the number of recollisions.
	
	In the present work, we have simplified the strategy by using a  and adapted it to particles interacting through a general compactly, supported interaction potential (see Assumption \ref{ass: potentiels pour la dérivation de Boltzmann}).}
	
	\section{Definition of the system and strategy of the proof}\label{sec: def of the system}
	\subsection{The Hamiltonian dynamics}
	Let ${\mathbb{T}}:=\mathbb{R}^d/\mathbb{Z}^d$ (with $d\geq2$) be the domain. We denote $\mathbb{D}={\mathbb{T}}\times\mathbb{R}^d$ its tangent bundle and $\mathbb{D}^n$ the $n$-particle canonical phase space. In the following, we use the notation 
	\[X_n= (x_1,\cdots,x_n)\in\mathbb{T}^n,~V_n=(v_1,\cdots,v_n)\in\mathbb{R}^{nd},~\text{and}~z_i=(x_i,v_i)\in\mathbb{{D}}.\]
	
	On each $\mathbb{D}^n$, we construct the Hamiltonian dynamics associated with the Energy
	\blue{\begin{gather}
			\mathcal{H}_n(Z_n):= \frac{1}{2}|V_n|^2+\mathcal{V}_n(X_n),~~\mathcal{V}_n(X_n):=\sum_{1\leq i<j\leq n} \alpha\mathcal{V}\left(\frac{|x_i-x_j|}{\e}\right),\\[-10pt]
			\forall i\in[1,n],\Bigg\{{\begin{split}\\[3pt]&\tfrac{\ud}{\ud t}x_i=\nabla_{v_i}\mathcal{H}_n(Z_n(t))=v_i,\\[-3pt]
					&\tfrac{\ud}{\ud{t}}v_i=-\nabla_{x_i}\mathcal{H}_n(Z_n(t))= \frac{\alpha}{\e}\sum_{\substack{j=1\\j\neq i}}^n\nabla\mathcal{V}\left(\frac{x_i-x_j}{\e}\right).\end{split}}
	\end{gather}}
	
	We impose the following condition on the interaction potential
	\begin{assumption}\label{ass: potentiels pour la dérivation de Boltzmann}
		There exists a constant $s\in [0,\infty)$ and a decreasing cut-off function $\chi\in\mathcal{C}([0,\infty))\cap\mathcal{C}^2([0,1))$ such that 
		\begin{equation}
		\mathcal{V}(x):= \frac{\chi(|x|)}{|x|^s},~\chi(0)= 1,~\chi([1,\infty))=\{0\}.
		\end{equation}
	\end{assumption}
	This dynamics is well defined for all times, almost everywhere in $\mathbb{D}^n$ with respect to the Lebesgue measure.
	
	\subsection{Grand-canonical ensemble and stationary measure}
	
	\blue{In the following, we choose not to fix the number of particles $\N$, and suppose that $\N$ is a random variable with $\mathbb{E}_\e[\N]$ going to infinity when $\e$ goes to $0$ (we say that we consider the \emph{grand canonical ensemble}). It will simplify a lot of computations\footnote{The treatment of the canonical setting remains an open problem.}.}
	
	We denote $\mathcal{D}:=\bigsqcup_{n\geq 0} \mathbb{D}^n$ the grand canonical phase space. We can then extend the Hamiltonian dynamics to $\mathcal{D}$ and denote $\gr{Z}_{\N}(t)$ the realization (defined almost surely) of the Hamiltonian flow on $\mathcal{D}$ with random initial data $\gr{Z}_{\mathcal{N}}(0)$: for $\N = n$, $\gr{Z}_{\N}(t)$ follows the Hamiltonian dynamics on $\mathbb{D}^n$.	
	
	The initial data are sampled according to the stationary measure introduced now. The \emph{grand canonical Gibbs measure} $\mathbb{P}_\e$ (and its expectation $\mathbb{E}_\e$) are defined on $\mathcal{D}$ as follows: an application $H:\mathcal{D}\rightarrow \mathbb{R}$ is a test function if there exists a sequence $(h_n)_{n\geq 0}$ with $h_n\in L^\infty(\mathbb{D}^n)$ and
	\[\text{if}~\mathcal{N}=n,~ \gr{Z}_\N=(\gr{z}_1,\cdots,\gr{z}_n),~ H(\gr{Z}_{\N}):=h_n(\gr{z}_1,\cdots,\gr{z}_n).\]
	Fixing $\mu>0$ the \emph{chemical potential}\footnote{It takes the role of the number of particle $N$ has in the  canonical setting.}, we define $\mathbb{E}_\e$ as
	\begin{equation}
	\mathbb{E}_\e[H(\gr{Z}_{\N})]:=\frac{1}{\mathcal{Z}}\sum_{n\geq0} \frac{\mu^n}{n!}\int_{\mathbb{D}^n}h_n(Z_n)\,\frac{e^{-\mathcal{H}_n(Z_n)}}{(2\pi)^{nd/2}}\ud{Z_n},
	\end{equation}
	where $\mathcal{Z}$ is a normalisation constant such that $\mathbb{E}_\e[1]=1$. \blue{The \emph{mean free path} $\mathfrak{d}$ is defined by
	 \begin{equation}
	\mathfrak{d}= \frac{1}{\mu\,\e^{d-1}}.
	\end{equation} 
	It can be interpreted as the typical distance crossed by a particle between two collisions.}
	
	\begin{figure}[h]\label{fig: libre parcourt moyen}
		\centering
		\includegraphics[scale = 0.35]{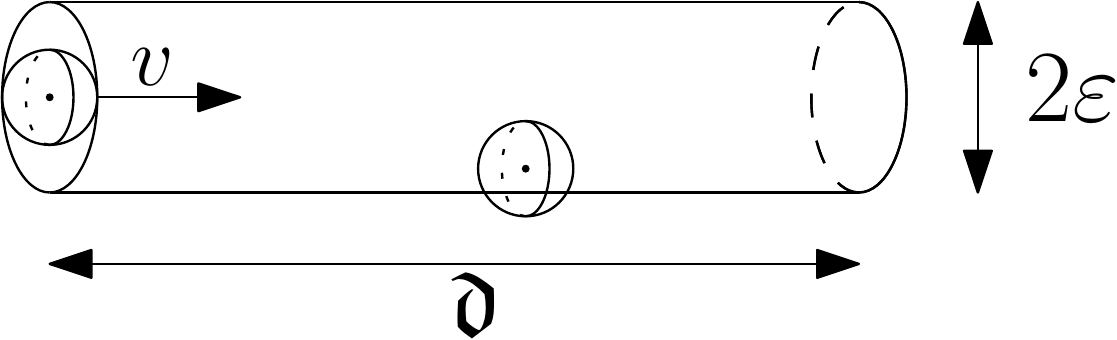}
		\caption{The first particle will meet the second one. Here $v$ is of order $1$.}
	\end{figure}
	
	\blue{We will consider only $\e$ small and $\mathfrak{d}\underset{\e\to0}{=}O(1)$, for which the series defining the Gibbs measure converges absolutely.}
	
	The empirical distribution at time $t$ is defined as the average configuration of particles at time $t$: for any $g$ test function on $\mathbb{D}$,
	\begin{equation}\label{eq:emperical measure}
	\pi_t^\e(g):=\frac{1}{\mu}\sum_{i= 1}^{\N} g(\gr{z}_i(t)).
	\end{equation}
	
	At equilibrium, we have the following law of large numbers. Denote
	\begin{equation}
	M(v) := \frac{e^{-\frac{|v|^2}{2}}}{(2\pi)^{\frac{d}{2}}}.
	\end{equation}
	
	\begin{prop}
		For any continuous and bounded test function $g:{\mathbb{T}}\times\mathbb{R}^d\to \mathbb{R}$, for all $t\in \mathbb{R}$ and for any $\delta>0$, 
		\begin{equation}
		\lim_{\e\to 0}\mathbb{P}_\e\left[\left|\pi_\e^t(g)-\int g(z) M(v) \ud{z}\right|\geq \delta\right]=0.
		\end{equation}
	\end{prop}
	\begin{remark}
		\blue{The previous result is a simple corollary of the Lanford theorem and of the stationarity in time of the measure (see \cite{Lanford,King}). From Proposition \ref{theoreme de quasi orthogonalite}, we can deduce the $L^2$ counterpart of this law of large numbers.}
	\end{remark}

	The aim of this article is to investigate the next order, namely the \emph{fluctuation field}\footnote{As the Gibbs measure is stationary in time, for all $t\in\mathbb{R}$, $\mathbb{E}_\e[\pi_0^\e(g)] = \mathbb{E}_\e[\pi^t_\e(g)]$.}
	\begin{equation}
	\zeta_\e^t(g):=\mu^{\frac{1}{2}}\Big(\pi_\e^t(g)-\mathbb{E}_\e[\pi_0^\e(g)]\Big).
	\end{equation}	
	
	\subsection{Binary collision, scattering,   and definition of the linearized Boltzmann operator}
	Interactions involving more than two particles become negligible in the Boltzmann-Grad limit.
		
	The present section is dedicated to describing the map between pre-collisional and post-collisional velocities. It is called the \emph{scattering map} (see Chapter 8 of \cite{GST} for a more detailed discussion).
	
	Consider two interacting particles $1$ and $2$ following the Hamiltonian dynamic associated with $\mathcal{H}_2$. At time $0$, particles have coordinates $(X_2(0),V_2(0))$ with
		\[x_1(0)=\e\nu, ~x_2(0)=0 ,~v_1 = v ~{\rm and}~ v_2(0)=v_*\]
	where $\nu\in\mathbb{S}^{d-1}$ and $(v-v_*)\cdot\nu>0$.
		
	The particles will interact on a finite interval $[0,[\tau]]$ with $[\tau]$ the infimum of $\{\tau>0,\,|x_2(\tau)-x_1(\tau)|>\e\}$. The time $[\tau]$ is finite and bounded by \blue{$  \frac{\e|v_1-v_2|}{|\nu\times (v_1-v_2)|^2}$}, with $\times$ the cross product (see Lemma \ref{lemma:borne temps collision}). We define $(\nu',v',v_*')$ as 
	\[\nu' := \frac{x_2(\tau)-x_1(\tau)}{\e} {\rm ~and~}(v',v_*'):=\big(v_1([\tau]),v_2([\tau])\big).\]
		
		In addition, the scattering conserves both momentum, kinetic energy, and angular momentum:
		\begin{equation}v+v_* = v'+v_*',~|v|^2+|v_*|^2=|v'|^2+|v'_*|^2{\rm ~and~}(v-v_*)\times\nu=(v'-v'_*)\times\nu'
		\end{equation}
		We deduce that 
		\begin{equation}
		|(v-v_*)\cdot\nu|=|(v'-v'_*)\cdot\nu'|.
		\end{equation}
		
	\begin{definition}\label{def: le scatering, sa vie son oeuvre}
		The \emph{scattering application} defined by 
		\begin{equation}\label{eq:def de l'application scattering}
		\xi_\alpha:(\nu,v,v_*)\mapsto(\nu',v',v'_*)
		\end{equation}
		is a local diffeomorphism which sends the measure $\ud{v}\!\ud{v_*}\!\ud\nu$ to $\ud{v'}\!\ud{v'_*}\!\ud\nu'$. \blue{In addition, it does not depend on the particle radius $\e$.}
	\end{definition}
	
	\begin{figure}[h!]
		\centering		\includegraphics[width=5cm]{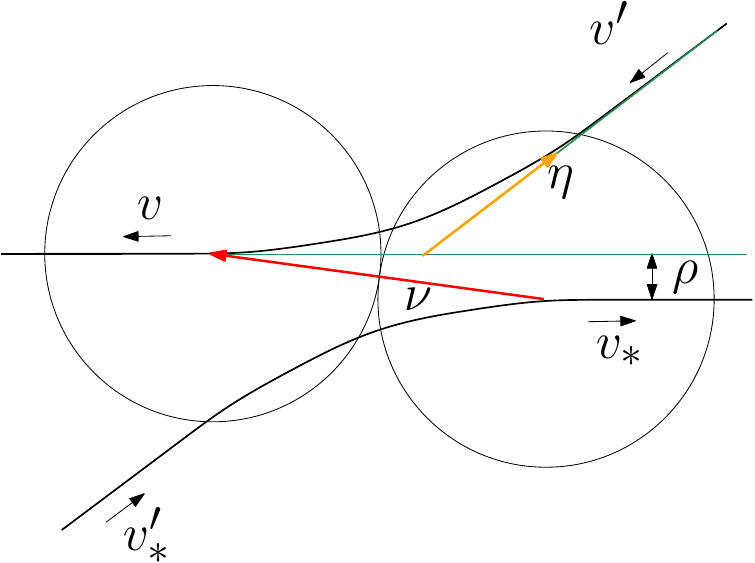}
		\caption{The scattering between two particles.}
	\end{figure}
	
	We define the linearized Bolzmann operator in the King's form:
	\begin{equation}\label{linearized Boltzmann}
	\mathcal{L}_{\mathcal{U}} g(v) := \int_{\mathbb{S}\times\mathbb{R}^d} \big(g(v')+g(v_*')-g(v)-g(v_*)\big)((v-v_*)\cdot\nu)_+M(v_*)\ud\nu\,\ud{v_*},
	\end{equation}
	where we apply the scattering with interaction potential $\mathcal{U}(\cdot)$, and $\mathcal{L}_\alpha := \mathcal{L}_{\alpha\mathcal{V}}$.
	
	This operator describes the variation of mass \blue{in a gas} due to changes of velocity of colliding particles. \blue{It is well known} that the operator $\mathcal{L}_\alpha$ is a self-adjoint non-positive operator on $L^2(M(v)\ud{z})$.
	
	\begin{remark}\label{rmk:pourquoi sans cut-off}
		We say that the Boltzmann operator $\mathcal{L}_{\mathcal{U}}$ has a cutoff because we truncate the long range interaction. 
		
		There is another interpretation of this property. For parameters $(v,v_*,\nu)$, we can define the vector $\eta$ such that 
		\[v' = \frac{v+v_*}{2}+\frac{|v-v_*|}{2}\eta,~v_*' = \frac{v+v_*}{2}-\frac{|v-v_*|}{2}\eta,\]
		and $b_\alpha(v-v_*,\eta)$ (called the \underline{collision kernel}) the Jacobian of the application $\nu\mapsto\eta$:
		\[((v-v_*)\cdot\nu)_+\ud\nu\to b_\alpha(v-v_*,\eta)\blue{\ud}\eta.\]
		
		We say that the Boltzmann operator has a cutoff because for any $v-v_*$, the following bound holds
		\[\int b_\alpha(v-v_*,\eta)\ud\eta<\infty.\]
		
		\blue{See Chapter 8 of \cite{GST}, Appendix of \cite{PSS} and Proposition 2.3.2 of \cite{LeBihan3} for explicit estimations of  the collision kernel $b_\alpha$ for certain class of interaction potential. }
	\end{remark}
	
	\subsection{Convergence to the linearized Boltzmann equation with a cut-off}	
	
	We recall that we have divided the proof of Claim \ref{thm:proto} into two steps. The first step is the Boltzmann-Grad limit $\mu\to \infty$. As we want to take the \emph{grazing collision} limit $\alpha\to0$ in a second step, we need a quantitative rate of convergence. 
	
	We define the norm
	\begin{equation}
	\|g\|_0:= \sup_{(x,v)\in\mathbb{D}} \left| M^{-1}(v) g(x,v)\right|~{\rm and}~\|g\|_k:=\sum_{|\alpha|\leq k}\|\nabla^\alpha g\|_0.
	\end{equation}
	
	\begin{theorem}\label{thm:Borne quantitativeè}
		Let $g$ and $h$ be two test functions $\mathcal{C}^1(\mathbb{D})$, with $\|g\|_1,\|h\|_1<\infty$. Then there exist three constants $C>1$, $C'>1$ and $\mathfrak{a}\in(0,1)$ depending only on the dimension such that for any $\e$ small enough, $T>1$, $\theta<\tfrac{1}{C'T^2}$, $\alpha,\mathfrak{d}\in(\log|\log\e|^{-1},1)$,
		\begin{equation}\label{Borne principale}
		\sup_{t\in[0,T]}\bigg|\mathbb{E}_\e\Big[\zeta^t_\e(h)\zeta^0_\e(g)\Big] - \int h(z)\gr{g}_\alpha(t,z)M(z)\ud z\bigg|
		\blue{\leq C \left( \frac{\theta T^2}{\mathfrak{d}^{3}}+\e^{\mathfrak{a}} \left(\frac{CT}{\mathfrak{d}}\right)^{2^{\frac{T}{\theta}+1}}\right)\|g\|_1\|h\|_1,}
		\end{equation}
		where $\gr{g}_\alpha(t,z)$ is the solution \blue{in $L^\infty(\ud t,L^2(M(v)\ud z))$} of the linearized Boltzmann equation
		\begin{equation}\label{eq:Linearized Boltzmann equation}\begin{split}
		\partial_t\gr{g}_\alpha(t)+v\cdot\nabla_x\gr{g}_\alpha(t) &= \frac{1}{\mathfrak{d}}\mathcal{L}_\alpha \gr{g}_\alpha(t),\\
		\gr{g}_\alpha(t=0) &= g
		\end{split}\end{equation}
	\end{theorem}
	
	\blue{\begin{remark}
		 It is classical that there is a unique solution to the linearized Boltzmann equation, which is bounded globally in time in $L^2(M(v)\ud z)$ (see e.g., Section 7 in \cite{CIP}).
	\end{remark}}
	
	The theorem is valid in any dimension $d\geq 2$. Its proof is the main purpose of the present article. We conclude the proof of Theorem \ref{thm:Borne quantitativeè} by Estimation \eqref{eq:Borne principale bis}, and we outlined the main step of the proof in Section \ref{subsec:Strategy of the proof}.
		
	\subsection{Derivation of the linearized Landau equation and Boltzmann equation without cut-off}\label{subsec:Derivation of the linearized Landau equation and Boltzmann equation without cut-off}
	
	\blue{In this section, we discuss the second step, namely the grazing collision limit. We fix the dimension at $d=3$ as it is the physical case.} We only state the main results, as the proof can be found in the joint article \cite{LW}.
	
	\paragraph{The case where the singularity $\tfrac{1}{r^s}$ of the potential is stronger than the Coulomb singularity ($s>1$).}
	In the limit $\alpha\to 0$, we will only see the effects of the singularity at the origin. We define the power law potential $\mathcal{U}_s(r):=1/r^s$. It is natural to guess that one has convergence of the Boltzmann operators
	\begin{equation}
	\alpha^{-\frac{2}{s}}\mathcal{L}_\alpha\to \mathcal{L}_{\mathcal{U}_s}\end{equation}
	which is a linearized Boltzmann operator without cutoff (see Appendix \ref{sec:The Linearized Boltzmann operator without cut-off} for a rigorous definition of $\mathcal{L}_{\mathcal{U}_s}$ and a justification of the scaling $\mathfrak{d}=\alpha^{2/s}$).
	
	\begin{remark}
		We say that the Boltzmann operator $\mathcal{L}_{\mathcal{U}_s}$ has no cutoff because particles can interact at long range, and the collision kernel $b_s(v-v_*,\eta)$ associated to the potential $1/r^s$ (defined in Remark \ref{rmk:pourquoi sans cut-off}) is not integrable in the $\eta$ variable (see \eqref{eq:noyau pour loi de puissance}).		
	\end{remark}
	
	\paragraph{The Coulomb case $s=1$.} It is not possible to define the Boltzmann operator for the Coulomb potential. However, we can prove (see \cite{LW}) that for $g$ a test function smooth enough, 
	\begin{equation}
	\frac{1}{\alpha^2|\log{\alpha}|}\mathcal{L}_\alpha g \underset{\alpha\to0}{\longrightarrow} c_{\mathcal{V}}\mathcal{K}g
	\end{equation}
	where $c_{\mathcal{V}}=1$ is a diffusion constant and $\mathcal{K}$ is the linearized Landau operator
	\begin{equation}\label{eq:linearized landau equation}
	\KK g (v) = \frac{2\pi}{M(v)} \nabla_v\cdot  \left(\int_{\R^3}  \frac{P^\perp_{v-v_*}}{|v-v_*|} (\nabla g(v)-\nabla g(v_*)) M(v)M(v_*) \ud{v_*}\right).
	\end{equation}
	
	\paragraph{Treat now the weak singularity $s\in[0,1)$.} For these potentials, the scaling and the diffusion constant change:
	\begin{align}\label{equiv:cPhi}
	\mathfrak{d} = \alpha^2,~~2\pi c_{\mathcal{V}} = \frac1{8\pi}\int_{\mathbb{R}^3} \delta(k\cdot \vec{e}_1) |k|^2 |\hat{{\mathcal{V}}}(k)|^2 \ud{k},
	\end{align}
	where $\vec{e}_1$ is a unit vector, and we use the convention $\hat{{\mathcal{V}}}(k)= \int_{\mathbb{R}^3} e^{-ik\cdot x} {\mathcal{V}}(x)\ud{x}$ for the Fourier transform of ${\mathcal{V}}$. Then
	\begin{equation}
	\frac{1}{\alpha^2}\mathcal{L}_\alpha g \underset{\alpha\to0}{\longrightarrow} c_{\mathcal{V}}\mathcal{K}g
	\end{equation}
	
	The previous discussion can be summarized by the following theorem:
	\begin{theorem}[\blue{L.B.-Winter, \cite{LW}}]\label{thm:LW}
		For $g:\mathbb{D}\to\mathbb{R}$ smooth and $\mathcal{V}$ respecting Assumption \ref{ass: potentiels pour la dérivation de Boltzmann}, there exists a positive constant $C$ such that
		\begin{equation}
		\left\|\mathfrak{d}_{s,\alpha}^{-1}\mathcal{L}_\alpha g -\mathcal{L}_\infty g\right\|_{L^2(M(v)\ud{z})}\leq \frac{C}{|\log\alpha|}\|g\|_3,
		\end{equation}
		where $\mathcal{L}_\infty$ and $\mathfrak{d}_{s,\alpha}$ are given by 		
		\begin{center}			
		\begin{tabular}{|C{2cm}|C{3cm}|C{6cm}|}
			\hline Singularity&Mean free-path&Limiting operator\tabularnewline
			\hline\hline $s>1$&$\mathfrak{d}_{s,\alpha} := \alpha^{2/s}$ &$\mathcal{L}_\infty=\mathcal{L}_{\mathcal{U}_s}$\tabularnewline
			\hline $s=1$&$\mathfrak{d}_{s,\alpha} := \alpha^{2}|\log\alpha|$ &$\mathcal{L}_\infty=\mathcal{K}$\tabularnewline
			\hline $0\leq s<1$&$\mathfrak{d}_{s,\alpha} := \alpha^{2}$ &$\mathcal{L}_\infty=c_{\mathcal{V}}\mathcal{K}$,\\
			$\displaystyle c_{\mathcal{V}} = \frac1{16\pi^2}\int_{\mathbb{R}^3} \delta(k\cdot \vec{e}_1) |k|^2 |\hat{{\mathcal{V}}}(k)|^2 \ud{k}$\\ \tabularnewline\hline
		\end{tabular}
		\end{center}
		\smallskip 		
		In addition, defining $\gr{g}_\infty(t)$ the solution of 
		\begin{equation}\label{eq:Linearized Boltzmann sans cut-off equation}\begin{split}
		\partial_t\gr{g}_\infty(t)+v\cdot\nabla_x\gr{g}_\infty(t) &=\mathcal{L}_\infty \gr{g}_\infty(t),\\
		\gr{g}_\infty(t=0) &= g
		\end{split}\end{equation}
		and $\gr{g}_\alpha$ the solution of \eqref{eq:Linearized Boltzmann equation} with $\mathfrak{d}:=\mathfrak{d}_{s,\alpha}$, the following convergence holds 
		\begin{equation}
		\gr{g}_\alpha \underset{\alpha\to0}{\overset{*}{\rightharpoonup}}\gr{g}_\infty ~{\rm in}~L^{\infty}_t(\mathbb{R}^+(L^2(M(v)\ud{z})).
		\end{equation}
	\end{theorem}

	Combining it with Theorem \ref{thm:Borne quantitativeè} we obtain the main theorem:
	
	\begin{theorem}\label{thm:derivation Bolzmann sans cut-off}
		Let $f,g\in L^2(M(v)\ud{z})$ be two test functions.
		
		Consider a potential $\mathcal{V}$ such that the Assumptions \ref{ass: potentiels pour la dérivation de Boltzmann} are verified and $\mathcal{V}(r)\underset{r\to 0^+}\sim\frac{1}{r^s}$, \blue{ $s\geq 0$}.
		
		Fix the scaling $\mu\e^2\mathfrak{d}_{s,\alpha}=1$.
		Then we have the following convergence result: for all $t\geq 0$,
		\[\mathbb{E}_\e\left[\zeta_\e^t(h)\zeta_\e^0(g)\right]\underset{\substack{\e\rightarrow0\\\alpha\to0\\\blue{\alpha\log|\log \e|^{\frac{\tilde{s}}6}\to\infty}}}{\longrightarrow} \int h(z)\gr{g}_\infty(t,z)M(z)\ud z \]
		where $\tilde{s} = \max(s,1)$ and $\gr{g}_\infty(t)$ is the solution of the equation \eqref{eq:Linearized Boltzmann sans cut-off equation}.
	\end{theorem}
	
	\blue{\begin{remark}
			We recall that in the sup-Coulomb case ($s>1$), taking the limit 
			\[\e\to0,~\alpha\to0, ~\mu\,\e^2 = \mathfrak{d}^{-1}_{s,\alpha}=\alpha^{-2/s}\text{ and }\alpha\log|\log \e|^{\frac s6}\to\infty\]
			is equivalent to taking the limit
			\[\epsilon\to0,~R\to\infty, ~\mu\,\epsilon^2 = 1 \text{ and }\frac{\log|\log \e|^{\frac16}}{R}\to\infty.\]
	\end{remark}}

	\begin{proof}[Proof of Theorem \ref{thm:derivation Bolzmann sans cut-off}]
		First, the space  $E:=\{g:\mathbb{D}\to\mathbb{R},~\left\|g\right\|_1<\infty\}$ is dense in $L^2(M(v\ud{z}))$.
		
		Since the two bilinear operators
		\begin{align*}
		(h,g)\mapsto \mathbb{E}_\e\left[\zeta_\e^t(h)\zeta_\e^0(g)\right],~(h,g)\mapsto \int h(z)\gr{g}_\infty(t,z)M(z)\ud z
		\end{align*}
		are  continuous on $L^2(M(v)\ud{z})$ (see \cite{BGSS1}), it is sufficient to take $g,h\in E$. 
		
		Set $T:= \max(1,t)$. Fixing $\theta :=\frac{1}{\beta\log|\log\e|}$ for $\beta\in(0,1)$ small enough, 
		\[C \left( C \frac{T^2\theta}{\mathfrak{d}_{s,\alpha}^3} + \frac{T}{\theta}2^{T^2/\theta^2}\left(\frac{CT}{\mathfrak{d}_{s,\alpha}}\right)^{2^{T/\theta}}\e^{\mathfrak{a}/2}\right)= \blue{o(1)}\]
		
		Hence Theorem \ref{thm:Borne quantitativeè} provides 
		\[\mathbb{E}_\e\left[\zeta_\e^t(g)\zeta_\e^0(h)\right] = \int h(z)\gr{g}_\alpha(t,z)M(v)\ud{z}+o(\|g\|_1\|h\|_1).\]
		
		Theorem \ref{thm:LW} provides the convergence
		\[\int h(z)\gr{g}_\alpha(t,z)M(v)\ud{z}\to\int h(z)\gr{g}_\infty(t,z)M(v)\ud{z}.\]
		
		This concludes the proof.
	\end{proof}

\red{	 \subsection{Central Limit Theorem for the fluctuation field}
	
	One can ask if it is possible to go further in the description of the fluctuation field. In the hard spheres setting, it is possible to prove a Central Limit Theorem for $\zeta_\e^t$.
	
	At time $0$, one can show that $\zeta_\e^0$ converges in law to the Gaussian field $\zeta^0$ define by 
	\begin{definition}
		Let $\zeta^0$ the Gaussian field on $\mathbb{D}$ of covariance
		\begin{equation}
		\forall f,g \in L(M(v)dz),~\left\{\begin{split}
		\mathbb{E}\left[\zeta^0(h)\zeta^0(g)\right]&= \int h(z)g(z) M(v)dz,\\
		\mathbb{E}\left[\zeta^0(g)\right]&=0.
		\end{split}\right.
		\end{equation}
	\end{definition}
	
	It is possible to generalize this result to the time dependent process (see \cite{Spohn,Spohn2,BGSS} for short time result and \cite{BGSS2} for long time result).
	\begin{theorem}[Bodineau, Gallagher, Saint-Raymond, Simonella, \cite{BGSS2}]
		Consider the hard spheres system in $d$-dimensional torus $\mathbb{T}^d$ ($d$ bigger than $3$) and fix the Boltzmann-Grad scaling $\mu\e^{d-1}=1$. The fluctuation field $(\zeta_\e)_{t\geq0}$ converges for al time to $\zeta^t$, the Gaussian field solution of the fluctuating Boltzmann equation 
		\begin{equation}
		\left\{\begin{split}
		&d\zeta^t=\mathcal{L}_B\zeta^t dt+d\xi^t\\
		&\zeta^{t=0}=\zeta^0
		\end{split}\right..
		\end{equation}
		The field $\xi^t$ is the mean free Gaussian field of covariance
		\begin{equation}
		\begin{split}
		\mathbb{E}\Bigg[\int_0^T h(z_1)\xi^{\tau_1}(dz_1)d\tau_1&\int_0^T h(z_2)\xi^{\tau_2}(dz_2)d\tau_2\Bigg]\\
		&:=\frac{1}{2}\int_0^Td\tau\int d\mu(z_1,z_2,\eta)M(v_1)M(v_2)\gr{\Delta} h \,\gr{\Delta} g,
		\end{split}
		\end{equation}
		where
		\[d\mu(z_1,z_2,\eta) := \delta_{x_1=x_2}b((v_1-v_2),\eta)dz_1dz_2d\eta,\]
		\[\gr{\Delta} h := h(v_1')+h(v_2')-h(v_1)-h(v_2),\]
		and $b((v_1-v_2),\eta)$ the hard sphere collision kernel.
	\end{theorem}
	
	As the particle dynamics has a memory (it is a purely deterministic process), it does not preserve the initial Gaussian structure. In \cite{BGSS2}, the Gaussian property is proved by checking asymptotically the Wick's law for the limiting field.
	
	One can asks if such a theorem still applied in a system of particle interacting through a more general potential $\mathcal{V}$. For the moment the question is still open.c
	\bigskip }
	
	\noindent\textbf{Notations.} For $\omega\subset\mathbb{N}$ a finite subset and $r\leq |\omega|$, we denote $\mathcal{P}^r_\omega$ the set of the unordered partitions $(\rho_1,\cdots,\rho_r)$ of the set $\omega$.
	
	For $m<n$ two integers, we denote $[m,n] := \{m,m+1,\cdots,n\}$ and $[n]:=[1,n]$.
	
	For $Z_n\in\mathbb{D}^n$, and $\omega\subset[n]$, we denote 
	\[Z_{\omega}:=(z_{\omega(1)},\cdots, z_{\omega(|\omega|)})\]
	where $\omega(i)$ is the $i$-th element of $\omega$ counted in increasing order.
	
	Given a family of particles indices $\{i_1,\cdots,i_n\}$, the notation $(i_1,\cdots,i_n)$ indicates the ordered sequence in which $\forall k\neq l$, $i_k\neq i_l$. In addition,
	\begin{itemize}
		\item $\ui_n := (i_1,\cdots,i_n)$,
		\item for $m\leq n$, $\ui_m=(i_1,\cdots,i_m)$, and more generally for $\omega \subset [1,n]$, $\ui_\omega := (i_{\min \omega}, \cdots,i_{\max \omega})$,
		\item for $0\leq m< n$ and $(i_1,\cdots,i_m)$, $\underset{(i_{m+1},\cdots,i_n)}{\sum}$ denotes the sum over every family $(i_{m+1},\cdots,i_n)$ such that for $1\leq k<l\leq n$, $i_k\neq i_l$, and \[\underset{\ui_n}{\sum}=\underset{(i_{1},\cdots,i_n)}{\sum},\]
		\item $\gr{Z}_{\ui_n}:= (\gr{z}_{i_1},\cdots,\gr{z}_{i_n})$, as an ordered sequence.
	\end{itemize}
	
	We also precise the Landau\footnote{from Edmund Landau and not Lev Landau.} notation: $A=B+O(D)$ means that there exists a constant $C$ depending only on the dimension such that $|A-B|< C \,D$. We denote $A \lesssim B$ if $A=O(B)$.
	
	When we perform estimations, $C$ is a positive constant depending only on the dimension (which can change from one line to another), and the final time $t$ is supposed to be bigger than $1$ (in general, we prefer to denote $\tau$ any intermediate time).
	
	Finally, let $h_n$ be a function on $\mathbb{D}^n$. We denote (by a slight abuse of notation)
	\[\mathbb{E}_\e\big[h_n\big]:=\mathbb{E}_\e\Bigg[\frac{1}{\mu^n}\sum_{(i_1,\cdots,i_n)}h_n\big(\gr{Z}_{\ui_n}\big)\Bigg]\]
	and the associated centered function defined on $\mathcal{D}_\e$
	\[\hat{h}_n(\gr{Z}_{\N}):=\frac{1}{\mu^n}\sum_{(i_1,\cdots,i_n)}h_n\big(\gr{Z}_{\ui_n}\big)-\mathbb{E}_\e\big[h_n\big].\]

	\subsection{Strategy of the proof of Theorem \ref{thm:Borne quantitativeè}}\label{subsec:Strategy of the proof}
	\blue{The general strategy of proof of Theorem \ref{thm:Borne quantitativeè} is inspired by \cite{BGSS1,LeBihan2,LeBihan3}, in which the hard sphere system is studied. It is known, however, that the treatment of smooth potentials in the low-density limit leads to a number of delicate complications and methodology. This is already true for the law of large numbers, see eg \cite{GST,PSS} where the proof of \cite{King} was completed and extended. The difficulties are due to the fact that interactions are not instantaneous. On the other hand, the extension to a long time of \eqref{Borne principale} for hard spheres, achieved in \cite{LeBihan2} (the short-time version was known long ago in \cite{Spohn2}) involves a different class of difficulty and technology. The main task of this paper is to push much further all the above-mentioned techniques in such a way as to allow the grazing collision limit $\alpha\to0$. This combination is nontrivial, and we outline the main features and novelties involved to achieve the result.}
	
	As $\zeta^0_\e(g)$ is a mean-free random variable on $\mathcal{D}$, we can write 
	\begin{equation}
	\mathbb{E}_\e\left[\zeta^t_\e(h)\zeta^0_\e(g)\right] = \frac{1}{\sqrt{\mu}}\mathbb{E}_\e\left[\sum_{i=1}^{\N}h(\gr{z}_i(t))\zeta_\e^0(g)\right].
	\end{equation}
	
	We see that the function $h$ is evaluated at time $t$, whereas the function $g$ is evaluated at time $0$. The first step of the proof is the construction of a family of functionals $(\Phi_{1,n})_{1\leq n}$, $\Phi_{1,n}^t:L^\infty(\mathbb{D})\to L^\infty(\mathbb{D}^n)$ such that for any initial configuration $\gr{Z}_\N\in\mathcal{D}$,
	\begin{equation}\label{eq: decomposition sur les pseudotrajectoires}
	\sum_{i= 1}^\N h(\gr{z}_i(t)) = \sum_{n\geq 1}\sum_{\ui_n}\Phi_{1,n}^t[h](\gr{Z}_{\ui_n}(0)).
	\end{equation}
	
	The first part of Section \ref{sec:Development along pseudotrajectories and time sampling} is dedicated to giving an explicit expression to the $\Phi_{1,n}^t$\blue{, {\it via} a \emph{dynamical cluster development} (introduced by Sinai \cite{Sinai} in a different setting and later by Bodineau {\it et al.} for hard spheres in Boltzmann-Grad scaling (see \cite{BGSS5})). If the particle $1$ does not meet any other particles, it moves along a straight line. We deduce that 
	\[\Phi_{1,1}^t[h](x_1,v_1) = h(x_1+t v_1,v_1).\]}
	\blue{If the particle $1$ meets only the particle $2$, then the trajectory of $1$ (denoted $\ds{z}_1(t,\text{$1$ and $2$ interact})$) is governed by the Hamiltonian equation associated to the energy 
	\[\frac{|\ds{v}_1|^2}{2}+\frac{|\ds{v}_2|^2}{2}+\frac{\alpha}{2}\mathcal{V}\left(\frac{\ds{x}_1-\ds{x}_2}{\e}\right),\]
	with initial conditions $(\ds{z}_1(0),\ds{z}_2(0))=(z_1,z_2)$. In order to write $\Phi_{1,2}^t[h](z_1,z_2)$, one also needs to compensate for the term $\Phi_{1,}^t[h](z_1)$ which appears now to be an error. Note that in this error term, the particles $1$ and $2$ do not interact. One writes
	\[(\ds{z}_1,\ds{z}_2)(t,\text{$1$ and $2$ do not interact}) := (x_1+t v_1,v_1, x_2+tv_2,v_2). \]
	We deduce that 
	\[\Phi_{1,2}^t[h](z_1,z_2) = \left[h(\ds{z}_1(t,\text{$1$ and $2$ interact})-  h(\ds{z}_1(t,\text{$1$ and $2$ do not interact}\right])\ind_{\text{$1$ and $2$ meet}}.\]
	In order to generalize this construction, we define the pseudotrajectory $\ds{Z}_n(t,\text{history})$ (for the moment in a vague sense, see Definition \ref{def:pseutrajectory1} for a proper definition) as a trajectory of $n$ particles where some couple of particles interacts and some couple of particles ignores each other. The "histories" are combinatorial parameters describing which couple of particle interacts (see Remark \ref{rem: racontage des histoires}). We denote $\mathfrak{H}_n$ the set of histories (it will be precisely defined in Remark \ref{rem: racontage des histoires}). We finally obtain 
	\begin{equation} 
		\Phi_{1,n}^t[h](Z_n):=\frac{1}{(n-1)!}\sum_{{\rm history}\in\mathfrak{H}_n} h(\ds{z}_1(t,Z_n,{\rm history}))\ind_{\rm history}\sigma({\rm history})
	\end{equation}
	In the preceding formula, $\ind_{\rm history}$ requires that the pseudotrajectory is possible (if we have all the needed collisions), and $\sigma({\rm history}) = \pm1$ is a sign.}
	\begin{figure}[h!]
		\includegraphics[width=10cm]{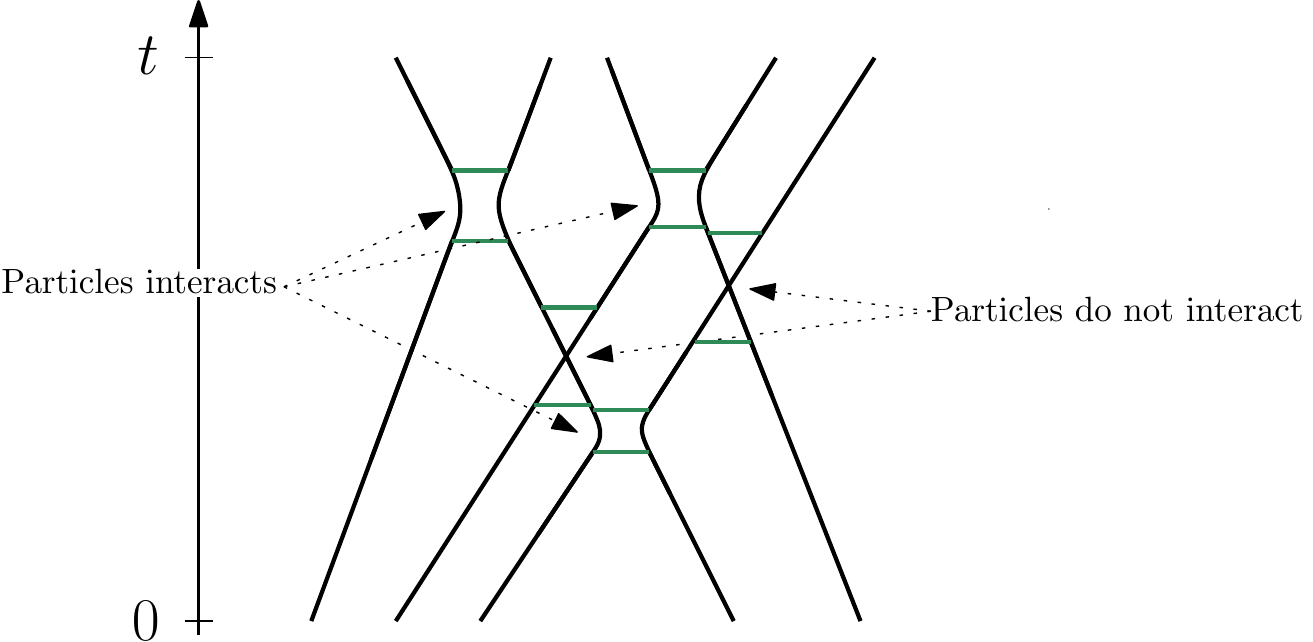}
		\caption{Exemple of pseudotrajectory for four particles.}
	\end{figure}
	
	\blue{\begin{remark}
			 This representation differs from the one used in \cite{King,GST, PSS}, which is based on the Grad's representation of the BBGKY hierarchy \cite{Grad}. It leads to a drastic improvement to the argument in \cite{LeBihan2} (see below and Remark \ref{rem:Comparaison avec le vieux papier}).
	\end{remark}}

	Following the classical derivation of the Boltzmann equation (and here of the linearized Boltzmann equation), there will be two main steps. First, we need to prove that each term of the expansion in \eqref{eq: decomposition sur les pseudotrajectoires} 
	\begin{equation}
	\mathbb{E}_\e\left[\mu^{-\frac{1}{2}}\sum_{\ui_n}\Phi_{1,n}^t[h](\gr{Z}_{\ui_n}(0))\zeta^0_\e(g)\right]
	\end{equation}
	converges to its formal limit. In the limit, the dynamics loses its Hamiltonian character, and particles become punctual (see Section \ref{sec:Treatment of the main part}).
	
	The main obstacles to this convergence are \emph{multiple \blue{encounters} }(interactions between more than three particles) and \emph{recollisions}. A recollision can be defined \blue{(a more proper definition will be given in Definition \ref{def:pathology})} as a meeting (with interaction or with overlap) between two particles $q$ and $\bar{q}$, beginning at time $\tau$ and such that we can find \blue{a sequence of couples of particles $(q=q_1,\bar{q}_1),(q_2,\bar{q}_2)\cdots (q_r,\bar{q}_r=\bar{q})$, such that $q_{i+1}\in\{q_i,\bar{q}_i\}$ and that $q_i$ and $\bar{q}_i$ meet before time $\tau$.} Recollisions and multiple \blue{encounters} become rare in the limit $\e\to0$ (quantitative estimations are performed in Section \ref{subsec:estimation reco}).
	
	The second step is an {\it a priori} bound of the terms of the series \eqref{eq: decomposition sur les pseudotrajectoires}. To improve estimates to a longer time interval than the one obtained in \cite{GST,PSS}, it is convenient to consider $L^2$ estimates (see \cite{BGS,BGSS1,BGSS2,LeBihan2}). Indeed, because $\zeta^0_\e(g)$ is a mean-free random variable, for any intermediate time $t_s\in[0,t]$,
	\begin{equation*}
	\begin{split}
	\Bigg|\mathbb{E}_\e\Bigg[\mu^{-\frac{1}{2}}\sum_{\ui_n}\Phi_{1,n}^{t-t_s}[h](\gr{Z}_{\ui_n}(t_s))\zeta^0_\e(g)\Bigg]\Bigg| &= \Big|\mathbb{E}_\e\Big[\mu^{n-\frac{1}{2}}\hat{\Phi}_{1,n}^{t-t_s}[h](\gr{Z}_{\N}(t_s))\zeta^0_\e(g)\Big]\Big|\\[-7pt]
	&\leq \mathbb{E}_\e\Big[\mu^{2n-1}\left(\hat{\Phi}_{1,n}^{t-t_s}[h](\gr{Z}_{\N}(t_s))\right)^2\Big]^{\frac{1}{2}}\mathbb{E}_\e\Big[\left(\zeta_\e^0(g)\right)^2\Big]^{\frac{1}{2}}\\
	&\leq \mathbb{E}_\e\Big[\mu^{2n-1}\left(\hat{\Phi}_{1,n}^{t-t_s}[h](\gr{Z}_{\N}(0))\right)^2\Big]^{\frac{1}{2}}\mathbb{E}_\e\Big[\left(\zeta_\e^0(g)\right)^2\Big]^{\frac{1}{2}},
	\end{split}
	\end{equation*}
	using a Cauchy-Schwartz inequality and the invariance of the Gibbs measure. Hence, it is possible to start a development along pseudotrajectories and stop at time $t_s$ when they become "pathological". \blue{We say that pseudotrajectories become pathological when their number explodes (this notion will be made precise after \eqref{eq: flemme de donner un nom} and is linked to recollisions and multiple encounters).} Then we can ignore what happens in the time interval $[0,t_s]$.
	
	\blue{We need to bound the $\mathbb{E}_\e\big[\big(\hat{\Phi}_{1,n}^{t}[h]\big)^2\big]$, which are based on the estimations of terms of type  
	\begin{equation}\label{eq: machin a borne}
		\sup_{l\leq n}~ \frac{n^l}{\mu^{l-1}}\int \left|{\Phi}_{1,n}^{t-t_s}[h](Z_n){\Phi}_{1,n}^{t-t_s}[h](Z_{[n-l,2n-l]})\right|e^{-\mathcal{H}_{2n-l}(Z_{2n-m})}\ud{Z_{2n-m}},
	\end{equation}
	see Proposition \ref{theoreme de quasi orthogonalite}.}
	
	\blue{\begin{remark}
		 An $L^\infty\to L^1$ estimation is used in the classical derivation of the Boltzmann equation (see \cite{Lanford,King,GST,PSS}). It is valid only for short times. The linear version of the problem (one tagged particle followed in a background initially at equilibrium) is only a $O(1)$ perturbation of equilibrium in $L^\infty$. Thus, the $L^1$ bounds are valid for all time (see \cite{vBHLS,BGS2,Ayi,Catapano1}). The linearized setting is a $O(\mu)$ perturbation of the equilibrium, and $L^1$ bounds are no longer sufficient to reach long time out of equilibrium (Spohn used them to describe the fluctuations on short time in \cite{Spohn}).
	\end{remark}}
	
	\blue{Unfortunately, we do not know how to efficiently take into account the sign $\sigma({\rm history})$ in the bound of $\Phi_{1,n}^{t-t_s}$. Thus, we resort to using the naive bound
	\begin{equation}
		\left|\Phi_{1,n}^{t}[h]\right| \lesssim \frac{\|h\|}{(n-1)!}\sum_{{\rm history}\in\mathfrak{H}'_n} \ind_{\rm history}.
	\end{equation}
	where $\mathfrak{H}_n'\subset \mathfrak{H}_n$ is a set of parameters that is a little bit smaller than $\mathfrak{H}_n$ (it will be made precise in Remark \ref{rem: racontage des histoires}). Two particles with initial coordinates $(x_1,v_1)$ and $(x_2,v_2)$ can collide on the time interval $[0,t]$ if and only if $x_2$ is in a cylinder of radius $\e$ and length $t|v_1-v_2|$. We deduce that 
	\begin{equation}
		\int \ind_{\text{$1$ and $2$ collide on $[0,t]$}} e^{-\mathcal{H}_2}\ud Z_2 \lesssim t\e^{d-1} \lesssim \frac{t}{\mu\mathfrak{d}}.
	\end{equation}}

	\blue{For a pseudotrajectory involving $n$ particles, there are at least $n-1$ collisions (all particles are linked by a chain of collisions). Using a generalization of the preceding argument, we deduce the following $L^1$ bound
	\begin{equation}\label{eq:truc a borne}
		\int   \sum_{{\rm history}\in\mathfrak{H}_n'}\ind_{{\rm history}\in\mathfrak{H}_n'} e^{-\mathcal{H}_n}\ud Z_n \lesssim n^n \left(\frac{t}{\mu\mathfrak{d}}\right)^{n-1} 
	\end{equation}
	for some $C$ (a rigorous proof is provided in Proposition II.6 of \cite{BGSS5}\footnote{The estimation is provided in \cite{BGSS5} in the hard sphere setting but should be generalized to more general interaction. We will not need this precise estimation, and we will limit ourselves to proving for a given history \[\int  \ind_{\rm history} e^{-\mathcal{H}_n}\ud Z_n \lesssim \left(\tfrac{Cnt}{\mu\mathfrak{d}}\right)^{n-1}.\] }). As we perform $L^2$ estimates, we will encounter terms of the form $\ind_{{\rm history}_1}(Z_n)\ind_{{\rm history}_2}(Z_n)$ where $({\rm history}_1,{\rm history}_2)\in(\mathfrak{H}_n')^2$ are two different histories. The problem is that the geometry of the set $\{Z_n|\ind_{{\rm history}_1}(Z_n)=1\}$ is quite complicated and we do not know a better estimate than the naive bound
	\[\ind_{{\rm history}_1}(Z_n)\ind_{{\rm history}_2}(Z_n)\leq \ind_{{\rm history}_1}(Z_n).\]
	From this, we may obtain the bound on \eqref{eq: machin a borne} bigger than
	\begin{equation}\label{eq: flemme de donner un nom}
		\frac{n^n}{\mu^{n-1}}\int \left|\Phi_{1,n}^{t}[h]\right|^2e^{-\mathcal{H}_n} \ud Z_n \lesssim {\|h\|}\left(\frac{Ct}{\mu^2\mathfrak{d}}\right)^{n-1} \left|\mathfrak{H}_n'\right|.
	\end{equation}}
	
	\blue{Thus, we are led to counting the number of "histories", {\it id est}, the number of pseudotrajectories involving $n$ particles. The set $|\mathfrak{H}_n'|$ is too large (it is of order $O((Cn)^n)$ for some constant $C>0$, see Remark \ref{rem: racontage des histoires}) to allow the series $\sum_n \mathbb{E}\left[(\hat{\Phi}_{1,n}^t)^2\right]^{\frac{1}{2}}$ to converge. We would need that $|\mathfrak{H}_n'|$ is of order $O(C^n)$ for some constant $C>0$. To overcome this problem, we will perform the pseudotrajectory decomposition only until an intermediate time $t_s$, such that the set of collision parameters needed $\mathfrak{H}_n''$ remains controllable. We observe that there are two reasons for the set of parameters to explode: the \emph{pathology} (recollisions or multiple meetings) and an uncontrolled number of particles $n$.}
	
	We now introduce two samplings, one to control regular collisions and one to control recollisions and multiple \blue{encounters}.
	
	The first sampling has a relatively large step $\theta:=\tfrac{1}{\beta\log|\log\e|}$ (for some constant $\beta$ large enough). We stop the pseudotrajectory development at time $t-k\theta$ if there are more than $2^k$ particles involved in the pseudotrajectory. Hence, the number of particles at time $0$ remains controlled. \blue{This sampling follows \cite{BGS2,BGS}\footnote{Note that in \cite{Fougeres}, the author improved the sampling strategy in the linear setting in order to obtain a better convergence rate than \eqref{Borne principale}.}.}
	
	The second sampling has a shorter step, $\delta:=\e^{1/12}$. We stop the expansion at time $t_s:=t-k\delta$ if the pseudotrajectory has at least one recollision on $[t_s,t]$ (but no recollision on $[t_s+\delta,t]$). Imposing recollisions creates an additional geometric condition, and thus, an extra-smallness gain.
	
	Unfortunately, we still have too many possible histories. In order to reduce their number, we \blue{follow the idea of \cite{LeBihan2} and} separate the pseudotrajectories into two categories. In \emph{non-pathological} pseudotrajectories, \blue{there is no recollision nor multiple \blue{encounters} on the time interval $[t_s,t_s+\delta]$.} We are in a setting close to the case without recollision, and we only need $C^n$ parameters ($C$ a fixed constant) to describe the histories.
	
	We explain now how to treat the pathological recollisions part. The initial data $\gr{{Z}}_\N(0)$ is conditioned such that on each interval $[k\delta,(k+1)\delta]$ a particle can \emph{encounter}\footnote{The meaning of \emph{encounter} will be precise in Definition \ref{def:possible cluster}.} with only a finite number of particles $\gamma$ (we will take $\gamma := 12d$). Hence, for a pseudotrajectory $\dr{z}_1(t,Z_n,{\rm history})$, the history has to describe first a partition of $[n]$ into small clusters of particles that interact together on $[0,\delta]$ and how they really interact. As the size of each cluster is uniformly bounded, the number of histories is at most of order $C^n$ for some $C>1$.
	
	\begin{figure}[h!]
		\includegraphics[width=13cm]{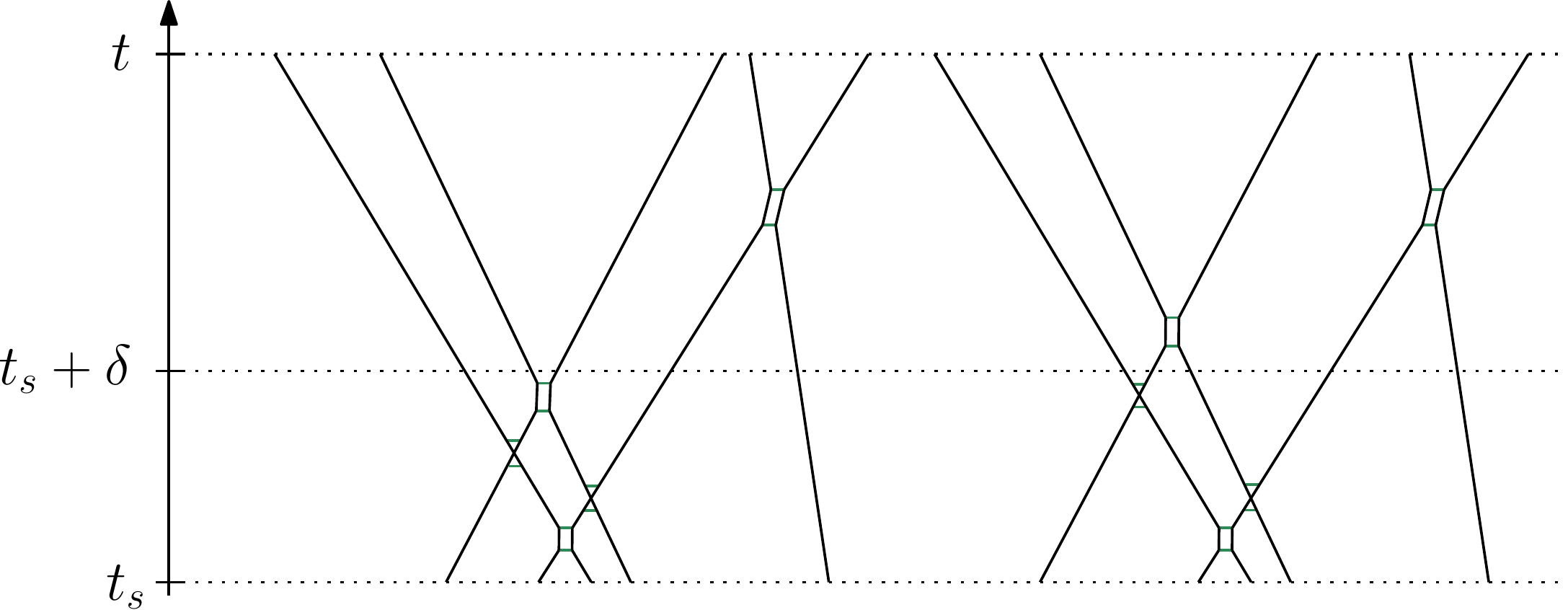}
		\caption{An example of one pathological pseudotrajectory (on the left) and a non-pathological one (on the right)}
	\end{figure}
	
	\blue{The paper is organised as follows:
	\begin{itemize}
		\item In Section \ref{sec:Development along pseudotrajectories and time sampling}, we give a proper definition of histories, and we use it to construct the functionals $\Phi_{1,n}^t$. Then we implement the two samplings. This allows us to decompose $\mathbb{E}_\e\left[\zeta_\e^t(h)\zeta_\e^0(g)\right]$ into a main term, plus error terms.
		
		\item In Section \ref{Quasi-orthogonality estimates}, we introduce  standard $L^2(\mathbb{P}_\e)$ estimates that are necessary to bound terms of the form $\mathbb{E}_\e\big|\hat{\Psi}\big|^2$ (for some symmetric function $\Psi_n:\mathbb{D}^n\to\mathbb{R}$) by the integrals
		\[\forall m \leq n,~ \int \left|{\Psi}_n(Z_n){\Psi}_{n}(Z_{[n-m,2n-m]})\right|e^{-\mathcal{H}_{2n-m}(Z_{2n-m})}\ud{Z_{2n-m}}.\]
		The proof is an adaptation of Section 3 of \cite{BGSS1}, and is based on static cumulant decompositions.
		
		\item Section \ref{Clustering estimations} is dedicated to the bound of the pseudotrajectory development $\Phi_{1,n}^0$ in the case where recollisions and multiple meetings are forbidden. This allows us to treat the pseudotrajectory with many particles.
		
		\item Section \ref{sec:Treatment of the main part} is dedicated to the convergence of the main term. 
	
		\item Section \ref{Estimation of the long range recollisions.} is dedicated to the treatment of the non-pathological recollisions.
		
		\item  Section \ref{$L^2$ estimation of the local recollision part} is dedicated to the treatment of pathological recollisions and of multiple meetings. 
		
		\item  Annex \ref{sec: Geometric estimates} is dedicated to the analyses of trajectories leading to recollisions of multiple \blue{encounters}. We use a strategy similar to the one of \cite{PSS}. It has the advantage of giving estimations independent of the form of the potential $\mathcal{V}$ (supposing it respects Assumption \ref{ass: potentiels pour la dérivation de Boltzmann}).
	\end{itemize}}

	\section{Development along pseudotrajectories and time sampling}\label{sec:Development along pseudotrajectories and time sampling}
	\subsection{Dynamical cluster development}
	For any test functions  $h$ and $g:\mathbb{D}\to\mathbb{R}$, we want to compute
	\[\mathbb{E}_\e\left[\zeta^t_\e(h)\zeta_\e^0(g)\right] = \frac{1}{\mu}\mathbb{E}_\e\left[\sum_{i=1}^{\N}h(\gr{z}_i(t))\sum_{j=1}^{\N}g(\gr{z}_j(0))\right].\]
	
	We have a sum evaluated at time $t$ and a sum evaluated at time $0$. In order to compute it, we have to pull back the second sum to time $0$: we want to construct a family of applications $\Phi_{1,n}^t:L^{\infty}{(\mathbb{D})}\to L^\infty(\mathbb{D}^n)$ such that for almost all initial data $\gr{Z}_{\N}(0)\in\mathcal{D}$
	\[h(\gr{z}_{i_1}(t))= \sum_{n\geq1}\sum_{(i_2,\cdots,i_n)}\Phi_{1,n}^t[h](\gr{Z}_{\ui_n}(0)).\]
	More generally, we will construct a family of functional $\Phi_{m,n}^t:L^{\infty}(\mathbb{D}^m)\to L^\infty(\mathbb{D}^n)$ (with $m<n$) such that for any test functions $h_m\in L^\infty(\mathbb{D}^m)$,
	\[h_m(\gr{Z}_{\ui_m}(t))= \sum_{n\geq1}\sum_{(i_{m+1},\cdots,i_n)}\Phi_{m,n}^t[h_m](\gr{Z}_{\ui_n}(0)).\]

	\begin{remark}[Comparison with the hard sphere setting]
		In the hard spheres setting, a tree pseudotrajectories development is used as it comes directly from the BBGKY hierarchy (see, for example, \cite{Lanford, PS, BGSS1, LeBihan2}). We begin at time $0$ with $n$ particles, and at each collision, we can remove or not remove  one particle to end at time $t$ with $m$ particles. However, in the case of physical potential, writing the BBGKY hierarchy is difficult as particles can overlap, and there can be interaction between more than three particles (see \cite{Grad,King,GST,PSS} for a description of the BBGKY hierarchy). Hence, we will use a different kind of pseudotrajectory development called ``dynamical cluster development'' (see \cite{Sinai,PSW, BGSS5}, from which we take inspiration).
		\begin{figure}[h!]
			\centering
			\includegraphics[width=10cm]{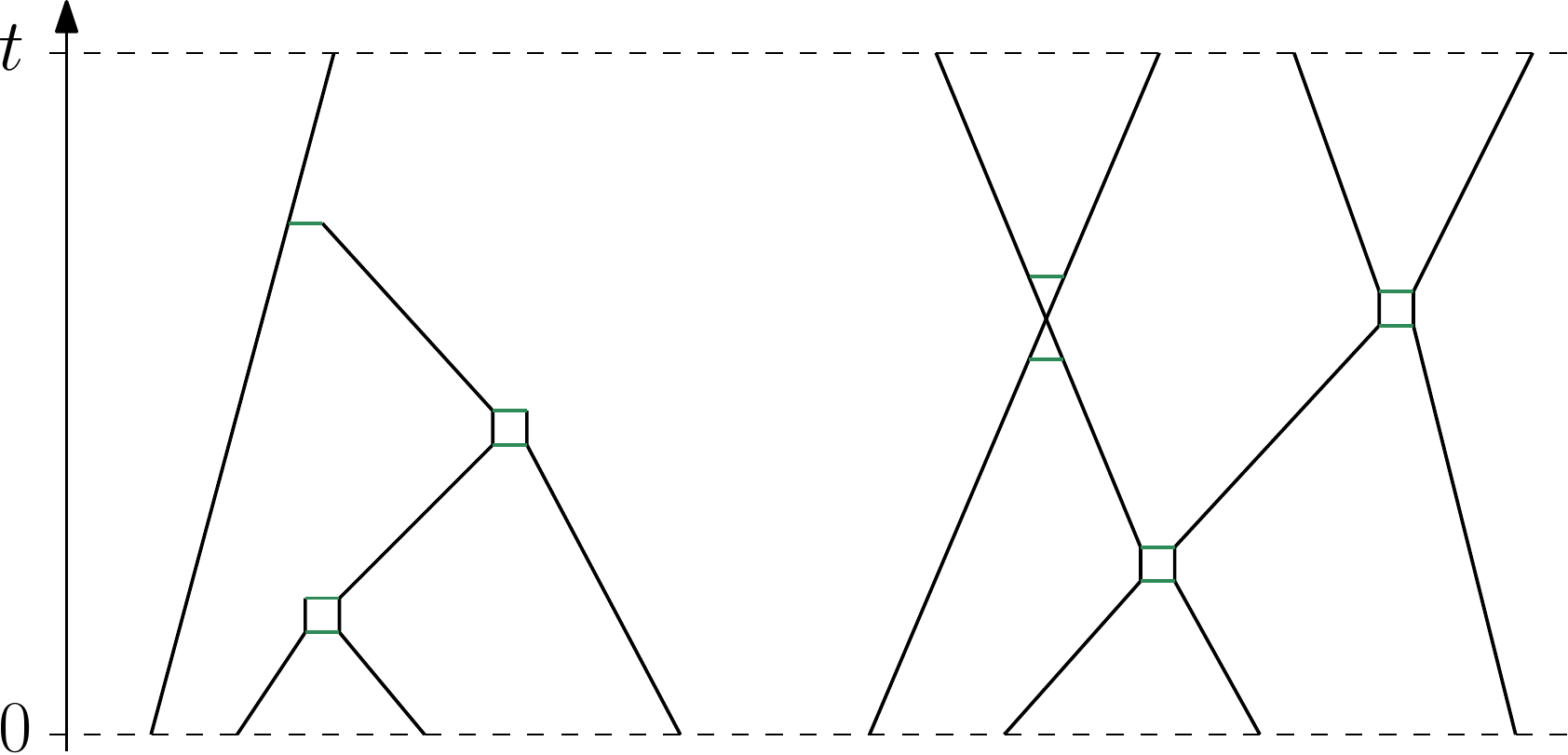}
			\caption{On the left a tree pseudotrajectory, on the right a graph pseudotrajectory.\label{fig: tree vs graph}}
		\end{figure}
	\end{remark}
	
	Fix $\lambda\subset \mathbb{N}$ a finite set of particles. We denote by $\ds{Z}^\lambda(\tau) = (\mathsf{X}^\lambda(\tau),\mathsf{V}^\lambda(\tau))$ the Hamiltonian trajectory, linked to the energy
	\[\mathcal{H}_\lambda(Z_\lambda) :=\sum_{q\in\lambda}\frac{|v_q|^2}{2}+\frac{\alpha}{2}\sum_{\substack{q,\bar{q}\in\lambda\\q\neq \bar{q}}} \mathcal{V}\left(\frac{x_q-x_{\bar{q}}}{\e}\right)\]
	of the particles $\lambda$ (isolated of the other particles) with initial data $\ds{Z}^\lambda(0)=Z_\lambda$. For any subset $\lambda'\subset\lambda$, we denote $\ds{Z}_{\lambda'}^\lambda(\tau)$ the trajectory of particles $\lambda'$ in $\ds{Z}^\lambda(\tau)$.
	
	\begin{definition}\label{def:DDelta}
		Given $Z_\lambda\in \mathbb{D}^{|\lambda|}$, we construct the graph $G$ with vertex $\lambda$ and $(q,\bar{q})\in\lambda^2$ is an edge if and only if $q<\bar{q}$ and if there exists a time $\tau \in[0,t]$ such that 
		\[\exists \tau\in[0,t],~\left|\ds{x}^\lambda_q(\tau)-\ds{x}^\lambda_{\bar{q}}(\tau)\right|\leq \e.\]
		
		We say that $\ds{Z}^\lambda(\tau)$ forms a \emph{dynamical cluster} if the graph $G$ is connected. We denote $\DDelta_{|\lambda|}(Z_\lambda)$ the indicator function that the trajectory $\ds{Z}^\lambda(\tau)$ forms a dynamical cluster.
		
		In the same way, for $\omega\subset\lambda$, we say that $\ds{Z}^\lambda(\tau)$ forms a \emph{$\omega$-cluster} if, in the collision of $\ds{Z}^\lambda(\tau)$, all the particles are in the same connected components of $G$ that one of the particles of $\omega$. The function $\DDelta_{|\lambda|}^{\omega}(Z_\lambda)$ is equal to $1$ if $\ds{Z}^\lambda(\tau)$ is a $\omega$-cluster, $0$ else.
	\end{definition}

	\begin{remark}
		In the following, we consider that all the graphs are unoriented.
	\end{remark}

	\begin{definition}\label{def:CCirc}
		We say that trajectories $\ds{Z}^{\lambda}(\tau)$ and $\ds{Z}^{\lambda'}(\tau)$ (with $\lambda\cap\lambda'=\emptyset$) have a \emph{dynamical overlap} if there exists a couple of particles $(q,{q}')\in\lambda\times\lambda'$ and some time $\tau\in[0,t]$, such that $|\mathsf{x}_q^\lambda(\tau)-\mathsf{x}_{{q}'}^{\lambda'}(\tau)|\leq \e$. Then we denote $\lambda\so\lambda'$.

	For $(Z_{\lambda_1},\cdots,Z_{\lambda_{\gr{l}}})\in \prod_{i=1}^{\gr{l}}\mathbb{D}^{|\lambda_i|}$ initial data, we look at the indicator function that for any $i\neq j$, $\ds{Z}^{\lambda_i}(\tau)$ and $\ds{Z}^{\lambda_j}(\tau)$ have no dynamical overlap. We can expand it as	
	\begin{equation}\begin{split}
	\prod_{1\leq i<j\leq \gr{l}}\big(1-\ind_{\lambda_i\so \lambda_j}\big)=\sum_{\substack{\omega\subset[1,l]\\1\in\omega}}&\underset{:=\CCirc_{|\omega|}(Z_{\lambda_1},Z_{\lambda_{\omega(2)}},\cdots,Z_{\lambda_{\omega(|\omega|)}})}{\underbrace{\sum_{C\in\mathcal{C}(\omega)}\prod_{(i,j)\in E(C)}-\ind_{\lambda_i\so \lambda_j}}}\prod_{\substack{(i,j)\in([\gr{l}]\setminus \omega)^2\\i\neq j}}\big(1-\ind_{\lambda_i\so \lambda_j}\big).\\
	&
	\end{split}\end{equation}
	We have defined $(\CCirc_l)_l$ as the cumulants of the dynamical overlap indicator functions.
	\end{definition}

	We make a partition of $\mathcal{D}$ depending on the way particles interact during the time interval $[0,t]$: fixing $\N\in\mathbb{N}$ and $\ui_m$, (we recall that $\mathcal{P}_\omega^r$ is the set of partitions into  $r$ subsets of $\omega$)
	\begin{multline*}
	h_m(\gr{Z}_{\ui_m}(t))=\sum_{\gr{l}=1}^\N \sum_{\substack{\ui_m\subset\lambda_1\\ (\lambda_2,\cdots,\lambda_{\gr{l}})\in\mathcal{P}_{[\N]\setminus\lambda_1}^{\gr{l}-1}}}h_m(\gr{Z}_{\ui_m}(t))\DDelta^{\ui_m}_{\lambda_1}(\gr{Z}_{\lambda_1})\prod_{i=2}^{\gr{l}}\DDelta_{|\lambda_i|}(\gr{Z}_{\lambda_i})\prod_{1\leq i<j\leq \gr{l}}\big(1-\ind_{\lambda_i\so \lambda_j}\big)\\
	=\sum_{\gr{l}=1}^\N \sum_{\substack{\ui_m\subset\lambda_1\\ (\lambda_2,\cdots,\lambda_{\gr{l}})\in\mathcal{P}_{[\N]\setminus\lambda_1}^{\gr{l}-1}}}h_m(\gr{Z}_{\ui_m}(t))\DDelta^{\ui_m}_{\lambda_1}(\gr{Z}_{\lambda_1})\prod_{i=2}^{\gr{l}}\DDelta_{|\lambda_i|}(\gr{Z}_{\lambda_i})\sum_{\substack{\omega\subset[\gr{l}]\\1\in\omega}}\CCirc_{|\omega|}(\gr{Z}_{\underline{\lambda}_\omega})\\[-5pt]
	\times\prod_{\substack{(i,j)\in(\omega\setminus[\gr{l}])^2\\i\neq j}}\big(1-\ind_{\lambda_i\so \lambda_j}\big).
	\end{multline*}
	
	We make the change of variables
	\[\left(\gr{l},\left(\lambda_1,\cdots,\lambda_{\gr{l}}\right),\omega\right)\mapsto\left(\rho,\gr{l}_1,\left(\bar{\lambda}_1,\cdots,\bar{\lambda}_{\gr{l}_1}\right),\gr{l}_2,\left(\tilde{\lambda}_1,\cdots,\tilde{\lambda}_{\gr{l}_2}\right)\right)\]
	where
	\begin{gather*}
		\rho:=\bigcup_{i\in\omega}\lambda_i,\,\gr{l}_1:=|\omega|,\,\gr{l}_2 := \gr{l}-|\omega|,\,\bar{\lambda}_1 := \lambda_1,\\
		\left(\bar{\lambda}_2,\cdots,\bar{\lambda}_{\gr{l}_1}\right):=(\lambda_j)_{j\in \omega\setminus\{1\}} {\rm\,and\,}\left(\tilde{\lambda}_1,\cdots,\tilde{\lambda}_{\gr{l}_2}\right):=(\lambda_j)_{j\in[\gr l]\setminus \omega}.\end{gather*}
	
	The set $\rho$ is the set of particles linked to $\ui_m$ \textit{via} a chain of interactions or overlaps. We get that $	h_m(\gr{Z}_{\ui_m}(t))$ is equal to
	\begin{multline*}
			\sum_{\rho\supset\ui_m}\sum_{\gr{l}_1=1}^{|\rho|} \sum_{\substack{\ui_m\subset\bar{\lambda}_1\subset \rho\\ (\bar{\lambda}_2,\cdots,\bar{\lambda}_{\gr{l}_1})\in\mathcal{P}_{{\rho\setminus\bar{\lambda}_1}}^{\gr{l}_1-1}}}h_m\left(\ds{Z}^{\bar{\lambda}_1}_{\ui_m}(t)\right)\DDelta^{\ui_m}_{\bar{\lambda}_1}\left(\gr{Z}_{\bar{\lambda}_1}\right)\prod_{i=2}^{\gr{l}_1}\DDelta_{|\bar{\lambda}_i|}\left(\gr{Z}_{\bar{\lambda}_i}\right)\CCirc_{\gr{l}_1}\!\left(\gr{Z}_{\bar{\lambda}_1},\cdots,\gr{Z}_{\bar{\lambda}_{\gr{l}_1}}\right)\\[-10pt]
			\times\sum_{\gr{l}_2=1}^{\N-|\rho|}\sum_{(\tilde{\lambda}_1,\cdots,\tilde{\lambda}_{\gr{l}_2})\in\mathcal{P}^{\gr{l}_2}_{[\N]\setminus\rho}}\prod_{i=1}^{\gr{l}_2}\DDelta_{|\tilde{\lambda}_i|}(\gr{Z}_{\tilde{\lambda}_i})\!\!\prod_{\substack{(i,j)\in([\N]\setminus\omega)^2\\i\neq j}}\!\!\big(1-\ind_{\tilde{\lambda}_i\so \tilde{\lambda}_j}\big).
	\end{multline*}
	The second line is the sum over all possible partitions $(\tilde{\lambda}_1,\cdots,\tilde{\lambda}_{\gr{l}_2})$ of $[\N]\setminus\rho$ of the indicator function that they are effectively the dynamical cluster of the initial data. Hence, it is equal to one. We identify the $n$-th \emph{dynamical cumulant} as
	\begin{multline}
	\Phi^t_{m, n}[h_m](Z_n):=\frac{1}{(n-m)!}\sum_{{\gr{l}}=1}^n \sum_{\substack{ [m] \subset \lambda_1\subset [n]}}\sum_{\substack{(\lambda_2,\cdots,\lambda_\gr{l})\\ \in\mathcal{P}^{\gr{l}-1}_{[n]\setminus\lambda_1}}} h_m(\ds{Z}^{\lambda_1}_{[m]}(t))\DDelta_{|\lambda_1|}^{[m]}(Z_{\lambda_1})\prod_{i=2}^{\gr{l}}\DDelta_{|\lambda_i|}(Z_{\lambda_i})\\[-10pt]
	\times \CCirc_l(Z_{\lambda_1},\cdots,Z_{\lambda_{\gr{l}}}),
	\end{multline}
	and we can now write the dynamical cluster expansion:	
	\begin{prop}
		Fix a family of particle $\ui_m$. For almost all $\gr{Z}_\N\in\mathcal{D}$, we have 
		\begin{equation}
			h_m\left(\gr{Z}_{\ui_m}(t)\right) = \sum_{n\geq m} \sum_{(i_{m+1},\cdots,i_{n})} \Phi^t_{m, n}[h_m]\left(\gr{Z}_{\ui_n}(0)\right).
		\end{equation}
	\end{prop}
	
	\begin{figure}[h]
		\centering
		\label{fig:exemple-dynamical-cumulant}
		\includegraphics[width=13cm]{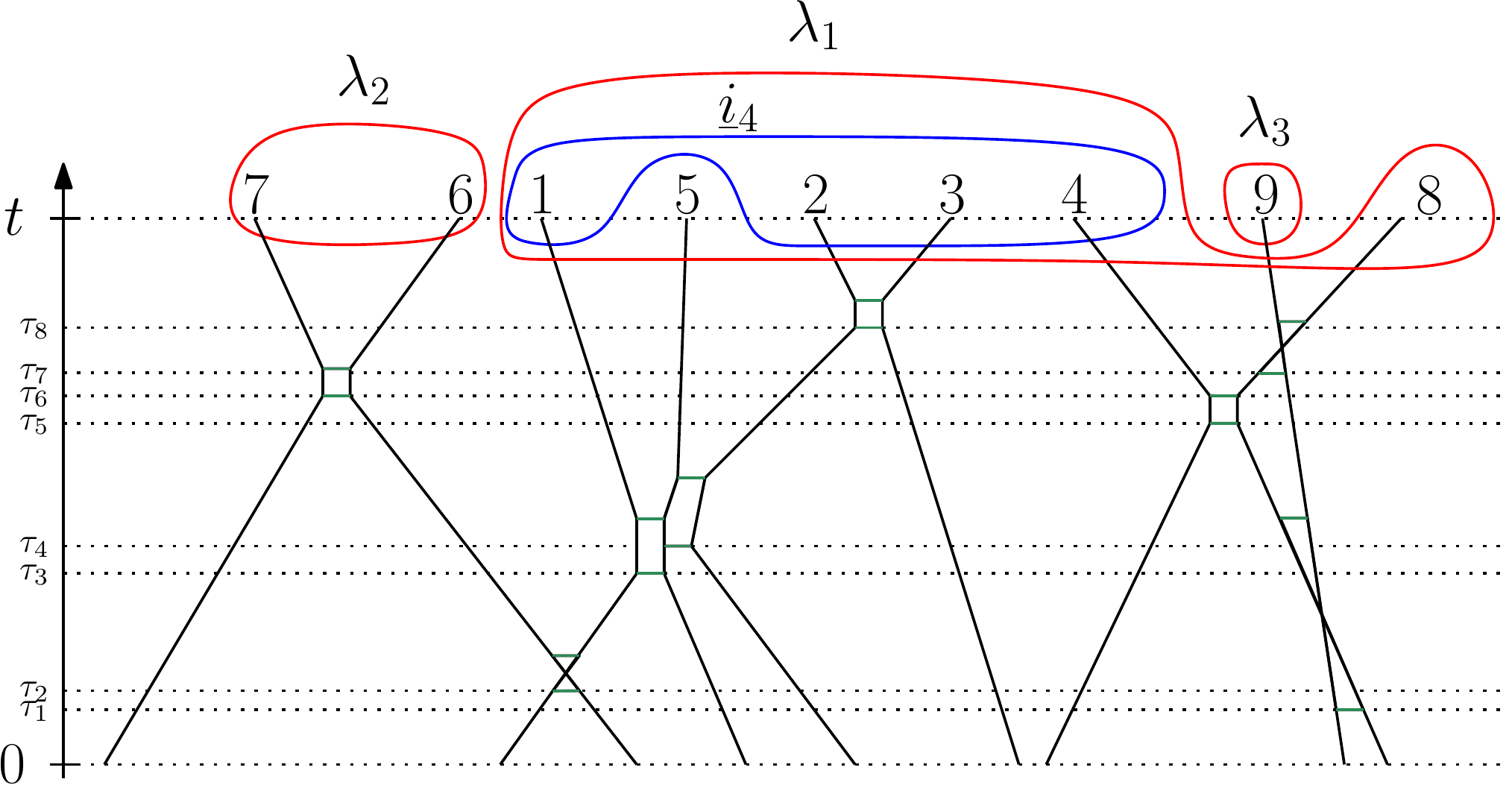}
		\caption{Example of trajectory in a dynamical cumulant. We want to follow the particles $\{1,2,3,4\}$.}
	\end{figure}

	\begin{definition}[First type of pseudotrajectory]\label{def:pseutrajectory1}
		In the following, for a given  $m\leq n$,   $\underline{\lambda}=(\lambda_1,\cdots,\lambda_{\gr{l}})$ a partition of $[n]$, we denote $\ds{Z}(t,Z_n,\underline{\lambda})$ the trajectory of the $n$ particles following the Hamiltonian dynamics linked to
	\[\mathcal{H}_{\underline{\lambda}}(Z_n) := \sum_{\ell=1}^{\gr{l}}\mathcal{H}_{\lambda_\ell}(Z_{\lambda_\ell}).\]
	\end{definition}

	We define now the notion of collision graph:	
	\begin{definition}[collision graph]
		Fix $m\leq n$, collision parameters  $\underline{\lambda}:=(\lambda_1,\cdots,\lambda_{\gr{l}})$ and an initial position $Z_n\in\mathbb{D}^n$.
		
		We construct the \emph{collision graph} with vertex $[n]$ and with labeled edges of the form $(i,j)_{\tau,s}$, $\tau\in[0,t]$, $s\in\{\pm1\}$. The edges $(i,j)_{\tau,s}$ is in the graph if 
		\begin{itemize}
			\item $\tau\in(0,t)$, $\left\vert\ds{x}_i(\tau,\underline{\lambda})-\ds{x}_j(\tau,\underline{\lambda})\right\vert=\e$, $\left(\ds{x}_i(\tau,\underline{\lambda})-\ds{x}_j(\tau,\underline{\lambda})\right)\cdot\left(\ds{v}_i(\tau,\underline{\lambda})-\ds{v}_j(\tau,\underline{\lambda})\right)>0$,
			\item or $\tau=0$, $\left\vert\ds{x}_i(0,\bar{\lambda})-\ds{x}_j(0,\bar{\lambda})\right\vert<\e$,
			\item $s= 1$ if $i$ and $j$ are in the same $\lambda_k$, $s=-1$ else.
		\end{itemize}
	\end{definition}
	
	\begin{figure}[h!]
		\centering
		\includegraphics[width=10cm]{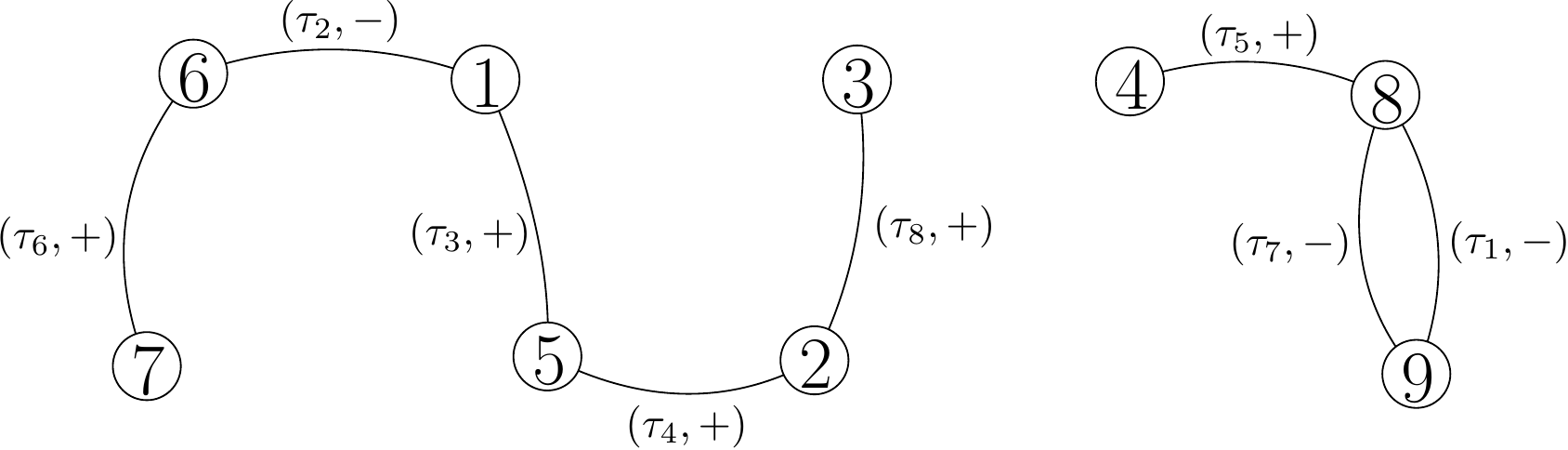}
		\caption{The collision graph associated with the pseudotrajectory of Figure \ref{fig:exemple-dynamical-cumulant}.}
	\end{figure}
	
	\begin{remark}
		Fix  $(\lambda_1,\cdots,\lambda_\ell)$ a partition of $[n]$. Using Penrose's tree inequality (see \cite{Penrose,BGSS, Jansen}), the cumulant function $\CCirc_{n}(Z_{\lambda_1},\cdots,Z_{\lambda_n})$ is bounded by 
		\begin{equation}\label{eq: Penrose's tree inequality}
		\left|\CCirc_{n}(Z_{\lambda_1},\cdots,Z_{\lambda_n})\right|\leq \sum_{T\in\mathcal{T}([\ell])} \prod_{(i,j)\in E(T)}\ind_{\lambda_i\so\lambda_j}
		\end{equation}
		where $\mathcal{T}([\ell])$ is the set of simply connected graph on $[\ell]$. The case of equality is reached, so we cannot expect a good $L^\infty$ bound of $\CCirc_n$.
	\end{remark}
	
	\blue{\begin{remark}\label{rem: racontage des histoires}
			We will describe the set $\{\text{history}\}$ that has been used in	 Section \ref{subsec:Strategy of the proof} for the decomposition of $\Phi_{1,n}^t[h]$: ${\rm history}$ is a triplet $(\gr{l},(\lambda_2,\cdots,\lambda_{\gr{l}}),C)\in\mathfrak{H}_n $, where the set of history $\mathfrak{H}_n$ is defined by
			\begin{gather*}
				\mathfrak{H}_n := \left\{(\gr{l},(\lambda_2,\cdots,\lambda_{\gr{l}}),C)\big| \gr{l}\in[1,n]\,;\, \{1\}\subset\lambda_1\subset[n]\,;\, (\lambda_2,\cdots,\lambda_{\gr{l}})\in \mathcal{P}_{\lambda_1^c}^{\gr{l}-1}\,;\,C\in\mathcal{C}([\gr{l}])\right\},\\
				\ind_{\rm history} := \prod_{i=1}^{\gr{l}}\DDelta_{|\lambda_i|}(\gr{Z}_{\lambda_i})\prod_{(i,j)\in E(C)}\ind_{\lambda_i\so \lambda_j}\\
				\sigma({\rm history}) := (-1)^{|E(C)|},
			\end{gather*}
			where $|E(C)|$ is the number of edges of the graph $C$. One has 
			\[\Phi_{1,n}^t[h](Z_n):=\frac{1}{(n-1)!}\sum_{{\rm history}\in\mathfrak{H}_n} h(\ds{z}^{\lambda_1}_1(t,Z_{\lambda_1})\ind_{\rm history}\sigma({\rm history}).\]			
			
			Using the Penrose's tree inequality \eqref{eq: Penrose's tree inequality}, one obtains
			\[\left|\Phi_{1,n}^t[h](Z_n)\right|\leq \frac{1}{(n-1)!}\sum_{{\rm history}\in\mathfrak{H}'_n} h(\ds{z}^{\lambda_1}_1(t,Z_{\lambda_1})\ind_{\rm history}\sigma({\rm history})\]
			where $\mathfrak{H}'_n\subset\mathfrak{H}_n$ is a smaller set of parameters
			\[\mathfrak{H}'_n := \left\{(\gr{l},(\lambda_2,\cdots,\lambda_{\gr{l}}),C)\big| \gr{l}\in[1,n]\,;\, \{1\}\subset\lambda_1\subset[n]\,;\, (\lambda_2,\cdots,\lambda_{\gr{l}})\in \mathcal{P}_{[n]\setminus\lambda_1}^{\gr{l}-1}\,;\,C\in\mathcal{T}([\gr{l}])\right\}.\]
			
			Using that $|\cup_{\gr l}\mathcal{P}^{\gr l}_{[n]}|\leq e^n$, $|\mathcal{T}([\gr l])|=\gr l ^{\gr l -2}$,
			\begin{align*}
				|\mathfrak{H}'_n| &=  \sum_{{\gr{l}}=1}^n \left|\left\{(\gr{l},(\lambda_2,\cdots,\lambda_{\gr{l}}))\big| \gr{l}\in[1,n]\,;\, \{1\}\subset\lambda_1\subset[n]\,;\, (\lambda_2,\cdots,\lambda_{\gr{l}})\in \mathcal{P}_{[n]\setminus\lambda_1}^{\gr{l}-1}\right\}\right|\times \left|\mathcal{T}([\gr{l}])\right|,\\
				&\leq \sum_{{\gr{l}}=1}^n C^n \times \gr{l}^{\gr l -2} \leq C^n n^n
			\end{align*}
			for some constant $C>0$, which is faster in $n$ than any geometric sequence. This shows that estimations of the form \eqref{eq: flemme de donner un nom} cannot be useful yet.
	\end{remark}}
	
	\subsection{Conditioning}\label{Conditioning}
	We describe now the conditioning used to control the pathological recollisions, that were introduced in Section \ref{subsec:Strategy of the proof}.
	
	\begin{definition}[Possible cluster]\label{def:possible cluster}
		Let $Z_r\in\mathbb{D}^r$ an initial configuration. Consider $\omega_1,\cdots,\omega_p$ a family of subsets of $[r]$ such that
		\[\bigcup_{i= 1}^p\omega_i = [r],\]
		and $(\underline{\lambda}_i)_{i\leq p} = (\lambda_i^1,\cdots,\lambda_i^{l_i})_{i\leq p}$ where each $\underline{\lambda}_i$ is a partition of the corresponding $\omega_i$. We denote $\mathcal{G}_i$ the collision graph of the pseudotrajectory $\ds{Z}(\tau,Z_{\omega_i},\underline{\lambda}_i)$ on the time interval $[0,\delta]$. The graph $\mathcal{G}$ is the merge of all the $\mathcal{G}_i$.
		
		We say that $Z_r$ forms a \emph{possible cluster} if there exists a couple $((\omega_i)_i$, $(\underline{\lambda}_i)_i)$ such that the graph $\mathcal{G}$ is connected.
	\end{definition}
	
	\begin{definition}[Definition of the set $\Upsilon_\e$]\label{def: possible cluster}
		Let $\gamma>0$ be a fixed integer, $\delta>0$ a time scale, and $\mathbb{V}>0$ a velocity bound. We construct $\Upsilon_\e\subset\mathcal{D}$ the set of particle configurations such that for any time $\tau\in\{0,\delta,2\delta,\cdots,$ $t\}$, there is no possible cluster of size bigger than $\gamma$ at time $\tau$, and inside any subset of particles $\omega\subset[1,\N]$ with less than $\gamma$ elements, $\tfrac{1}{2}\|\gr{V}_\omega(\tau)\|^2$ is bounded by $\tfrac{1}{2}\mathbb{V}^2$.
	\end{definition} We have the following bound on the measure of the complement of $\Upsilon_\e$:
	\begin{prop}
		There exists a constant $C_\gamma$ depending only on $\gamma$ and on the dimension such that
		\begin{equation}
		\mathbb{P}_\e\left(\Upsilon^c_\e\right)\leq C_\gamma\frac{t}{\delta}\left(\mu\delta^\gamma+\mu^\gamma e^{-\mathbb{V}^2/4}\right).
		\end{equation}
	\end{prop}
	\begin{proof}
		We take the notation of the definition \ref{def:possible cluster}.
		
		\blue{Suppose that $Z_r$ (with $r>\gamma$) forms a possible cluster. We want to show that there exists a subset $\varpi\subset[r]$ such that $Z_\varpi$ forms a possible cluster and $\gamma+1\leq|\varpi|\leq 2\gamma+2$. As $Z_r$ forms a possible cluster, there exists a couple $((\omega_i)_i,(\underline{\lambda}_i)_i)$ such that the graph $\mathcal{G}$ defined in Definition \ref{def: possible cluster} is connected. We define $\mathcal{G}_\tau$ as the subgraph of $\mathcal{G}$ with edges
		\[\left\{(q,\bar{q})_{(\tau',\sigma)}\in E(\mathcal{G}), \tau'<\tau \right\}.\]
		We consider $\tau$ the infimum of $\{\tau'>0,\mathcal{G}_{\tau'}\text{ is connected}\}$. Then, the graph $\mathcal{G}_\tau$ has exactly two connected components. One of them, denoted $\varpi$, verifies $\lceil\tfrac{r}{2}\rceil\leq |\varpi|\leq r-1$. Iterating the procedure, we obtain the expected result.
		
		We deduce that}
		\begin{align*}
		\mathbb{P}_\e\Big(&\Upsilon_{\e}^c\Big)\leq \sum_{k=0}^{t/\delta}\mathbb{E}_\e\left[\sum_{n=\gamma+1}^{2(\gamma+1)}\frac{1}{n!}\sum_{(i_1,\cdots,i_{n})}\ind_{\substack{\gr{Z}_{\ui_{n}}(k\delta)\text{\,form\,a}\\{\rm possible\,cluster}}}+\sum_{n= 1}^\gamma\frac{1}{n!}\sum_{(i_1,\cdots,i_{n})}\ind_{\|\gr{V}_{\ui_{n}}(k\delta)\|\geq\mathbb{V}}\right]\\
		&\leq \frac{t}{\delta}\left(\sum_{n=\gamma+1}^{2(\gamma+1)}\frac{1}{n!}\mu^{n}\int\ind_{\substack{Z_{n}\text{\,form\,a~~~}\\ {\rm possible\,cluster}}}M^{\otimes n}\ud Z_{n}+\sum_{n= 1}^\gamma\frac{1}{n!}\mu^n\int\ind_{\|V_{k'}\|\geq\mathbb{V}}\,M^{\otimes n}\ud Z_{n}\right).
		\end{align*}
		
		Using that (see Lemma \ref{lem:preuve du lemme a la con})
		\begin{equation}\label{eq:estimation a la con}
		\int\ind_{\substack{Z_{n}\text{\,form\,a~~~}\\ {\rm possible\,cluster}}}M^{\otimes n}\ud Z_{n}\leq C_\gamma \mu^{-n+1}\delta^{n-1},
		\end{equation}
		\begin{equation}
		\int\ind_{\|V_{k'}\|\geq\mathbb{V}}\,M^{\otimes n}\ud Z_{n}\leq C_n e^{-\frac{\mathbb{V}^2}{4}}
		\end{equation}
		we obtain the expected result.
		
		We used that the Gibbs measure is time invariant.
	\end{proof}
	
	Hence, if we fix $\delta:= \e^{1/12}$, $\mathbb{V}:=|\log\e|$ and  fix $\gamma=24d$, $\mathbb{P}_\e(\Upsilon_{\e}^c)$ is $O(\e^d)$.

	\subsection{The main part of the cumulant}\label{subsec:The main part of the cumulant}
	
	We define three kinds of pathology for the pseudotrajectories.
	
	\begin{definition}\label{def:pathology}
		Fix $m\leq n$, collision parameters  $(\lambda_1,\cdots,\lambda_{\gr{l}})$ and an initial position $Z_n\in\mathbb{D}^n$.
		\begin{itemize}
		\item There is an \emph{overlap} if there are two particles $q,q'$ and a time $\tau\in\delta\mathbb{Z}\cap[0,t]$ such that $|\ds{x}_q(\tau)-\ds	{x}_{q'}(\tau)|\leq\e$.
		 
		\item Fix a time $\tau$ and particles $i_i,\cdots i_k$. We define a graph $G^\tau$ with vertex $\{i_1,\cdots i_k\}$, and where $(i_a,i_b)$ is an edge if and only if 
		\[\Big|\ds{x}_{i_a}(\tau,\underline{\lambda})-\ds{x}_{i_b}(\tau,\underline{\lambda})\Big|\leq \e.\]
		
		There is a \emph{multiple \blue{encounter}} between $i_i,\cdots i_k$ at time $\tau$ if $G^\tau$ is connected.
		
		\item Fix $Z_n\in\mathbb{D}^n$ such that there is not a multiple \blue{encounters} during $[0,t]$.
		
		We say that there is a \emph{recollision} if the collision graph has a cycle.
		\end{itemize}
	\end{definition}
	
	\blue{These pathological terms will be considered as error terms, and we forbid them in the main part of our development:} we define $\Phi^{0,t}_{m, n}$ as the part of  $\Phi^{t}_{m, n}$ with only non-pathological pseudotrajectories
	\[\begin{aligne}{r@{}l}
	\Phi^{0,t}_{m, n}[h_m](Z_n)&:=\frac{1}{(n-m)!}\sum_{{\gr{l}}=1}^n \sum_{\substack{ [m] \subset \lambda_1\subset [n]}}\sum_{\substack{(\lambda_2,\cdots,\lambda_\gr{l})\\ \in\mathcal{P}^{\gr{l}-1}_{[n]\setminus\lambda_1}}} h_m(\ds{Z}_{[m]}(t,\underline{\lambda})) \CCirc_l(Z_{\lambda_1},\cdots,Z_{\lambda_{\gr{l}}})\\
	\multicolumn{2}{r}{\times\DDelta_{|\lambda_1|}^{[m]}(Z_{\lambda_1})\prod_{i=2}^{\gr{l}}\DDelta_{|\lambda_i|}(Z_{\lambda_i})\ind_{\substack{\text{no~pathology}}}}.
	\end{aligne}\]
	
	Forgetting the pathological cases allows us to consider a simpler parametrization of the pseudotrajectory. We denote $G$ as the collision graph. The graph $G'$ is constructed by removing the edges $(i,j)_{\tau,s}$ where $i$ and $j$ are in $[m]$. The edges of $G'$ can be ordered: $(i_k,j_k)_{\tau_k,s_k}$ with $\tau_1< \tau_2< \cdots< \tau_{n-m}$ (the $\tau_i$ are disjoint for almost all initial data).
	
	We can completely reconstruct the pseudotrajectory by considering only the sequence $s_1,\cdots,$ $s_{k-m}$ and the set of tagged particles $[m]$. This is performed in the following definition.
	
	\begin{definition}[Second definition of a pseudotrajectory]\label{def:pseudotrajectory2}
	Fix $m\leq n$, an initial position $Z_n$ and parameters $(s_k)_{k\leq n-m}\in\{\pm 1\}^{n-m}$ and $\omega\subset[n]$ with $|\omega|=m$. In order to construct the pseudotrajectory $\ds{Z}(\tau,Z_n,\omega,(s_k)_k)$, we need an auxiliary function $\iota:[0,t]\to\mathbb{N}$, which is increasing, \blue{piecewise constant}, and left-continuous function.
	
	At $\tau=0$, we set $\ds{Z}(0,Z_n,\omega,(s_k)_k):=Z_n$ and $\iota(0):=1$.
	
	Suppose that the pseudo trajectory $\ds{Z}(\cdot,Z_n,\omega,(s_k)_k)$ and $\iota(\cdot)$ are constructed in the time interval $[0,\tau]$. At time $\tau$ particles $i$ and $j$ meet, \textit{i.e.}
	\[\left\vert\ds{x}_i(\tau)-\ds{x}_j(\tau)\right\vert=\e,~\left(\ds{x}_i(\tau)-\ds{x}_j(\tau)\right)\left(\ds{v}_i(\tau)-\ds{v}_j(\tau)\right)>0.\]
	
	If $(i,j)\in\omega^2$, the two particles interact and we fix $\underset{\tau'\searrow \tau}\lim\iota(\tau'):=\iota(\tau)$.
	
	Otherwise, we fix $\underset{\tau'\searrow \tau}\lim\iota(\tau'):=\iota(\tau)+1$ and we look at $s_{\iota(\tau)}$. If $s_{\iota(\tau)} = 1$ the two particles interact: as long as $|\ds{x}_i-\ds{x}_j<\e$, they follow the dynamic 
	\[\left\{\begin{split}
	\dot{\ds{x}}_i = \ds{v}_i,&~\dot{\ds{v}}_i = \frac{\blue{\alpha}}{\e}\nabla\mathcal{V}\left(\frac{\ds{x}_i-\ds{x}_j}{\e}\right),~\\
	\dot{\ds{x}}_j = \ds{v}_j,&~\dot{\ds{v}}_j = \frac{-\blue{\alpha}}{\e}\nabla\mathcal{V}\left(\frac{\ds{x}_i-\ds{x}_j}{\e}\right).
	\end{split}\right.\]
	If $s_{\iota(\tau)} = -1$ the two particles ignore each other: as long as $|\ds{x}_i-\ds{x}_j|<\e$,
	\[\left\{\begin{split}
	\dot{\ds{x}}_i = \ds{v}_i,&~\dot{\ds{v}}_i = 0,\\
	\dot{\ds{x}}_j = \ds{v}_j,&~\dot{\ds{v}}_j = 0.
	\end{split}\right.\]
	In both cases, we define $\tau^+>\tau$ as the first time bigger than $\tau$ such that
	\[\left\vert\ds{x}_i(\tau^+)-\ds{x}_j(\tau^+)\right\vert=\e,~\left(\ds{x}_i(\tau^+)-\ds{x}_j(\tau^+)\right)\left(\ds{v}_i(\tau^+)-\ds{v}_j(\tau^+)\right)<0.\]
		
	We denote $\mathcal{R}^t_{\omega,(s_k)_{k}}\subset \mathbb{D}^n$ the set of initial parameters such that the pseudotrajectory has a connected collision graph and has neither multiple \blue{encounters}, nor recollision or nor overlaps. Hence, on $\mathcal{R}^t_{\omega,(s_k)_{k}}\subset \mathbb{D}^n$, the previous construction has no ambiguity.
	\end{definition}
	
	We can reconstruct the partition $(\lambda_1,\cdots,\lambda_{\gr{l}})$ for given $(s_i)_{i\leq n-m}$. We define the graph $G$ as a subgraph of the collision graph $\mathcal{G}$ by removing the edges of the form $(i,j)_{\tau,-1}$ (we keep only the interactions). The cluster $\lambda_1$ is the union of the connected components in $G$ of the particles $[m]$. The $(\lambda_2,\cdots,\lambda_{\gr{l}})$ are the other connected components.
	
	We have the following equality
	\begin{equation}\begin{aligne}{r@{}l}
	\Phi^{0,t}_{m, n}[h_m](Z_n)=\frac{1}{(n-m)!}\sum_{(s_k)_{k\leq n-m}}\prod_{i=1}^{n-m}s_i~ \ind_{\mathcal{R}^t_{[m],(s_k)_{k}}}h_m(\ds{Z}_{[m]}(\tau,Z_n,[m],(s_k)_k))
	\end{aligne}\end{equation}
	We denote
	\begin{equation}
	\Phi^{>,t}_{m,n} := \Phi^{t}_{m,n}-\Phi^{0,t}_{m,n}.
	\end{equation}

	\subsection{Iteration of the pseudotrajectory development}
	The present section is dedicated to the implementation of the short-time sampling (of step $\delta$).
	
	The construction of Section \ref{subsec:The main part of the cumulant} is efficient on a short time interval. To raise a long-time result, we need to iterate these kinds of pseudotrajectory representations and to compute sums of the form
	\[\sum_{\ui_{n_2}} \Phi^{0,\delta}_{n_1,n_2}\circ\Phi^{0,\delta}_{n_0,n_1}[h_{n_0}](\gr{Z}_{\ui_{n_2}}),\]
	where $n_0\leq n_1\leq n_2$ are three integers.
	
	\begin{remark}\label{rem: tree versu graph}
		In the usual framework, the pseudotrajectories \blue{are trees} (see for example \cite{BGSS1,BGSS2,LeBihan2}): there are more and more particles as we go backward in time. Hence, the development has naturally a semi-group structure, and it is straightforward to continue the development.
		
		In the present discussion, the pseudotrajectories have a graph structure: particles do not disappear. Hence, we need to work to iterate the process \blue{(see Figure \ref{fig: tree vs graph})}.
	\end{remark}
	
	\blue{We introduce the \emph{semi-tree} condition in order to recover some semi-group structure.
	\begin{definition}[semi-tree condition]
		Fix $\omega_1\subset\omega_2$ two finite sets, $t$ and $\delta$ two positive real numbers such that $K:=\tfrac{t}{\delta}\in \mathbb{N}^*$. Fix $G$ a collision graph such that
		\[E(G)\subset \left\{(i,j)_{(\tau,\sigma)}\Big\vert i,j\in\omega_2\,;\,\sigma\in\{\pm 1\}\,;\,\tau\in (0,t)\right\},\]
		and such that $G$ is connected and simply-connected.
		
		We define the decreasing sequence of sets $(\varpi_k)_{k\in[0,K]}$, where $\varpi_k$ is the connected components of $\omega_1$ in the graph of edges 
		\[\{(i,j)\vert \exists (\sigma,\tau)\in \{\pm 1\}\times [k\delta,t]~\text{such that}~(i,j)_{\tau,\sigma}\in G\}.\]
		
		The graph $G$ verifies the semi-tree condition with respect to $\omega_1$ if for any edges $(i,j)_{\tau,\sigma}\in G$ with $\tau$ in $[k\delta, (k+1)\delta[$, one has 
		\[(i,j)\in\varpi_{k+1}\times(\varpi_{k}\setminus\varpi_{k+1})\cup(\varpi_{k}\setminus\varpi_{k+1})\times\varpi_{k+1}.\]
		
		Fix $n\in\mathbb{N}^*$, $\omega\subset [n]$ and collision parameters $(s_k)_{k\leq n-|\omega|}\in \{\pm 1\}^{n-|\omega|}$. We define $\mathcal{R}^{0,t}_{\omega, (s_k)_k}\subset\mathcal{R}^{0,t}_{\omega, (s_k)_k}$ as the set of initial parameters $Z_n$ such that the collision graph of the pseudotrajectory $\ds{Z}(\cdot,Z_n,\omega,(s_k)_k)$ verifies the semi-tree condition with respect to $\omega$.
		
		\begin{figure}[h!]\label{fig: semi tree condition}
			\centering
			\includegraphics[width=7cm]{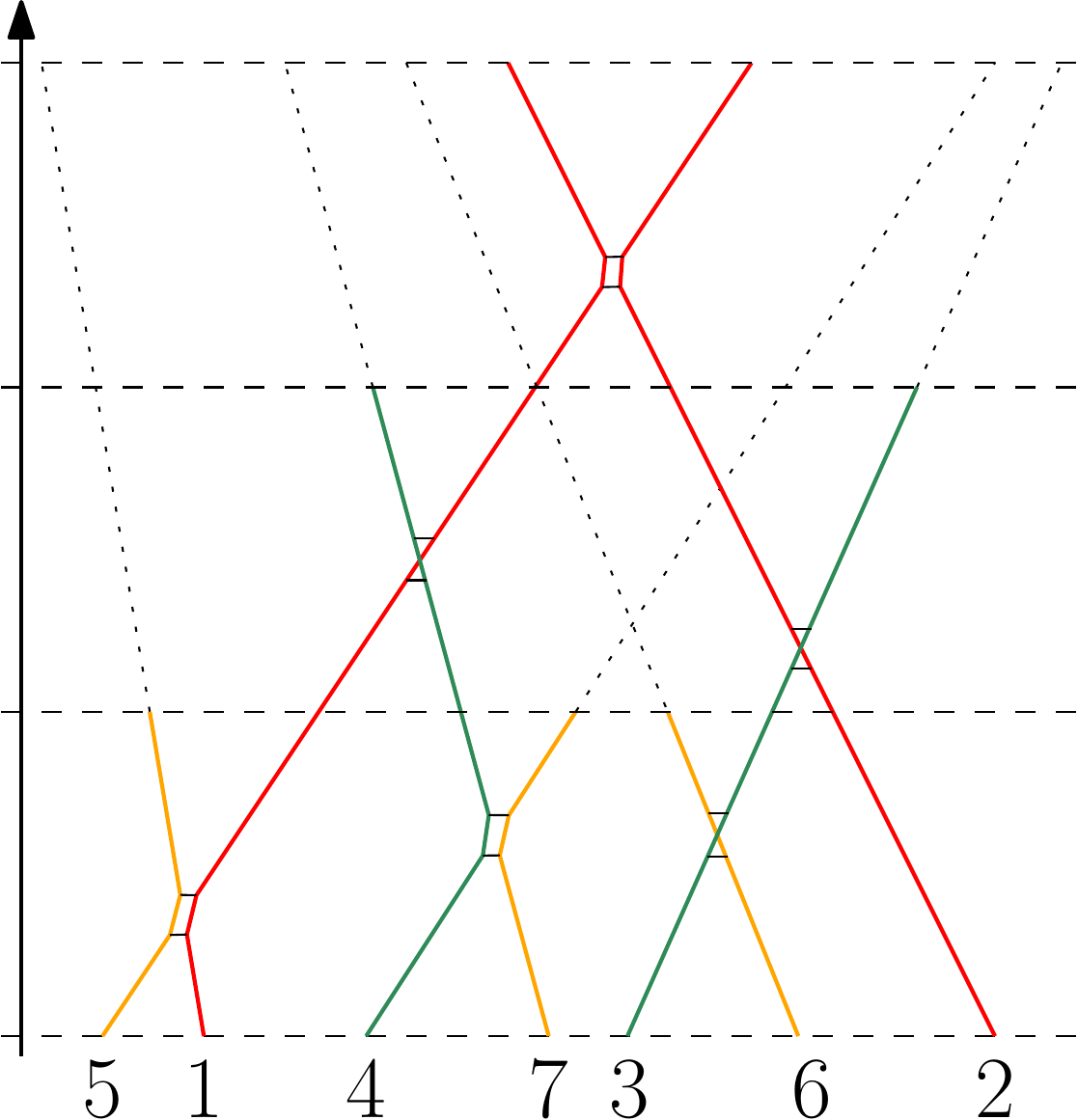}
			\caption{Here the pseudotrajectory checks the semi-tree condition, with $\omega_1 =\{1\}$, $\varpi_1=\{1,2\}$, $\varpi_2 = \{1,2,3,4\}$ and $\varpi_3 = \{1,2,3,4,5,6\}$. In the picture, when the trajectory of a particle is a dotted line, it does not overlap nor interact with any other particle \blue{(a crossing with a dotted line is an artifact of the one-dimensional drawing).}}
		\end{figure}
 	\end{definition}}
	
	We need a new definition of pseudotrajectory:
	\begin{definition}[Third definition of pseudotrajectory]\label{def:pseudotrajectory3}
		Fix $m\leq n$. For a family of parameters $(\omega_1,\omega_2,(s_k)_{k\leq n-m})$ with $\omega_1\subset\omega_2\subset[n]$, $|\omega_1|=m$ and $(s_k)_{k\leq n-m}\in\{\pm 1\}^{n-m}$, we define $\dr{Z}(\tau,Z_n,\omega_1,\omega_2,(s_k)_k)$ as 
		\begin{itemize}
			\item for $\tau\leq\delta,$
			\[\dr{Z}(\tau,Z_n,\omega_1,\omega_2,(s_k)_{k\leq n-m}) := \ds{Z}(\tau,Z_n,\omega_2,(s_k)_{k\leq n-|\omega_2|}),\]
			\item for $\tau> \delta$,
			\[\dr{Z}_{\omega_2}(\tau,Z_n,\omega_1,\omega_2,(s_k)_{k\leq n-m}):= \ds{Z}(\tau-\delta,\dr{Z}_{\omega_2}(\delta),\omega_1,(s_k)_{n-|\omega_2|<k\leq n-m}),\]
			and for all $i\in [n]\setminus\omega_2$
			\[\dr{z}_i(\tau) := (\dr{x}_i(\delta)+(\tau-\delta)\dr{v}_i(\delta),\dr{v}_i(\delta)).\]
		\end{itemize}
	\end{definition}

	\begin{remark}
		Note that the particles in $[n_2]\setminus\omega_2$ are virtual since time $\delta$: they do not interact with any other particle.
	\end{remark}
	
	\blue{\begin{definition}
		We define the collision graph $\mathcal{G}$ of the pseudotrajectory  $\dr{Z}(\cdot,Z_n,\omega_1,\omega_2,(s_k)_k)$. We define  the graph $G'$ as the subgraph of $\mathcal{G}$ with edges
		\[\big\{(i,j)_{\tau,\sigma}\in E(\mathcal{G}),~\tau\in[\delta,t],~(i,j)\in\omega_2^2\big\}.\]
		
		Fix $\omega_1\subset\omega_2\subset [n]$ and collision parameters $(s_k)_{k\leq n-|\omega_1|}$. We define the set  
		\[\mathcal{R}^t_{\omega_1,\omega_2,(s_k)_k}\subset \left\{Z_n\in \mathcal{R}^\delta_{\omega_2,(s_k)_{k\leq n-|\omega_2|}}\Big\vert \dr{Z}(\cdot,Z_n,\omega_2,(s_k)_{k\leq n-|\omega_1|})\in \mathcal{R}^\delta_{\omega_1,(s_k)_{0\leq k-|\omega_2|\leq n-|\omega_1|}} \right\}\]
		such that $G'$ has no cycle and verifies the semi-tree condition with respect to $\omega_1$.
	\end{definition}}

	\blue{Fix $m\leq n$, $\omega$ such that $[m] \subset \omega\subset [n]$ and $(s_k)_{k\leq n-m}$ some collision parameters.} We split $\mathcal{R}^t_{[m],\omega,(s_k)_k}$ into two pieces: $\mathcal{R}^{>,t}_{[m],\omega,(s_k)_k}$ where the collision graph has at least one cycle and $\mathcal{R}^{0,t}_{[m],\omega,(s_k)_k}$ where the collision graph  has no cycle.
	
	For $Z_{n}\in \mathcal{R}^{0,t}_{[m],\omega,(s_k)_k}$, there are exactly $n-m$ collisions in the collision graph, and $\omega$ is not needed to reconstruct the pseudotrajectory. In addition, its collision graph checks the semi-tree conditions. Fixing the parameters $(s_k)_k$, the sets $(\mathcal{R}^{0,t}_{[m],\omega,(s_k)_k})_{\omega}$ are disjoint, as $\omega$ is the union of the connected components of the vertices $1,\cdots,m$  in  the graph $G'$. \blue{We deduce that}
	\[\mathcal{R}^{0,t}_{[m],(s_k)_k}=\bigsqcup_{[m]\subset\omega\subset[n]}\mathcal{R}^{0,t}_{[m],\omega,(s_k)_k}.\]
	
	We introduce now the functionals
	\begin{gather}
		\Psi^{0,t}_{n_0,n_1}[h_{n_0}]:=\frac{1}{(n_1-n_0)!}\sum_{(s_k)_{k\leq n_1-n_0}}\prod_{k=1}^{n_1-n_0}\!\!s_k\, h_{n_0}\left(\ds{Z}_{[n_0]}(t,\cdot,[n_0],(s_k)_k)\right)\ind_{\mathcal{R}^{0,t}_{[n_0],(s_k)_k}},\\
		\Psi^{>,t}_{n_0,n_1}[h_{n_0}]:=\frac{1}{(n_1-n_0)!}\sum_{\substack{[n_0]\subset \omega \subset [n_1] \\(s_k)_{k\leq n_1-n_0}}}\prod_{k=1}^{n_1-n_0}\!\!s_k\, h_{n_0}\left(\dr{Z}_{[n_0]}(t,\cdot,[n_0],\omega,(s_k)_k)\right)\ind_{\mathcal{R}^{>,t}_{[n_0],\omega,(s_k)_k}}.
	\end{gather}
	
	Fix $n_0\leq n_1\leq n_2$ and a test function $h_{n_0}$. We obtain directly
	\begin{multline*}
	\Phi^{0,\delta}_{n_1,n_2}\circ\Psi^{0,t}_{n_0,n_1}[h_{n_0}](Z_{n_2}) \\
	= \frac{1}{(n_2-n_1)!(n_1-n_0)!}\sum_{(s_k)_{k\leq n_2-n_0}}\prod_{k=1}^{n_2-n_0}\!\!s_k\, h_{n_0}\left(\dr{Z}_{[n_0]}(t+\delta,Z_{n_2},[n_0],[n_1],(s_k)_k)\right)\ind_{\mathcal{R}^{t+\delta}_{[n_0],[n_1],(s_k)_k}}.
	\end{multline*}
	Then, summing on $n_1$ and $\omega$,
	\begin{multline*}
	\sum_{n_1 =n_0}^{n_2}\frac{(n_2-n_1)!(n_1-n_0)!}{(n_2-n_0)!}\sum_{\substack{[n_0]\subset \omega \subset [n_2]\\ |\omega|= n_1}}\Phi^{0,\delta}_{n_1,n_2}\circ\Psi^{0,t}_{n_0,n_1}[h_{n_0}](Z_{[n_0]},Z_{\omega\setminus[n_0]},Z_{[n_2]\setminus \omega})\\
	\begin{split}
		&=\frac{1}{(n_2-n_0)!}\sum_{\substack{[n_0]\subset \omega \subset [n_2] \\(s_k)_{k\leq n_2-n_0}}}\prod_{k=1}^{n_2-n_0}\!\!s_k\, h_{n_0}\left(\dr{Z}_{[n_0]}(t+\delta,Z_{n_2},[n_0],\omega,(s_k)_k)\right)\ind_{\mathcal{R}^{t+\delta}_{[n_0],\omega,(s_k)_k}}\\
		&=\Psi^{0,t+\delta}_{n_0,n_2}[h_{n_0}](Z_{n_2})+\Psi^{>,t+\delta}_{n_0,n_2}[h_{n_0}](Z_{n_2}).
	\end{split}
	\end{multline*}

	\blue{Using the symmetry of the summation set,
	\begin{multline}
	 \sum_{n_1 =n_0}^{n_2}\sum_{\ui_{n_2}}\Phi^{0,\delta}_{n_1,n_2}\circ\Psi^{0,t}_{n_0,n_1}[h_{n_0}](\gr{Z}_{\ui_{n_2}})\\
	\begin{split}
		&= \sum_{n_1 =n_0}^{n_2}\sum_{\ui_{n_2}}\binom{n_2-n_0}{n_1-n_0}^{-1}\sum_{\substack{[n_0]\subset \omega \subset [n_2]\\ |\omega|= n_1}}\Phi^{0,\delta}_{n_1,n_2}\circ\Psi^{0,t}_{n_0,n_1}[h_{n_0}](\gr{Z}_{\ui_{[n_0]}},\gr{Z}_{\ui_{\omega\setminus[n_0]}},\gr{Z}_{\ui_{[n_2]\setminus \omega}})\\
		&=\sum_{\ui_{n_2}} \Psi^{0,t+\delta}_{n_0,n_2}[h_{n_0}](\gr{Z}_{\ui_{n_2}}) +\sum_{\ui_{n_2}} \Psi^{>,t+\delta}_{n_0,n_2}[h_{n_0}](\gr{Z}_{\ui_{n_2}}).
	\end{split}
	\end{multline}}
	
	The functional $\Psi^{0,t}_{m,n}$ are introduced to implement the sampling: for $t>2\delta$ and $\gr{Z}_\N\in\Upsilon_{\e}$ (note that between the first and second lines, we use that by definition $\Phi_{1,n}^{0,\delta}=\Psi_{1,n}^{0,\delta}$)
	\begin{multline*}
	\begin{split}
		\sum_{i=1}^\N h(\gr{z}_i(t)) &=\sum_{n\geq 1}\sum_{\ui_n} \Phi_{1,n}^{0,\delta}[h](\gr{Z}_{\ui_n}(t-\delta))+\sum_{n\geq 1}\sum_{\ui_n} \Phi_{1,n}^{>,\delta}[h](\gr{Z}_{\ui_n}(t-\delta))\\
	&=\sum_{n'\geq n\geq 0}\sum_{\ui_{n'}} \left(\Phi^{0,\delta}_{n,n'}\Psi_{1,n}^{0,\delta}[h](\gr{Z}_{\ui_{n'}}(t-2\delta)) +\Phi^{>,\delta}_{n,n'}\Psi^{0,\delta}_{1,n}[h](\gr{Z}_{\ui_{n'}}(t-k\delta))\right)
	\end{split}\\
	+\sum_{n\geq 1}\sum_{\ui_n} \Phi_{1,n}^{>,\delta}[h](\gr{Z}_{\ui_n}(t-\delta)).
	\end{multline*}
	
	This can be rewritten as
	\begin{multline*}
	\sum_{i=1}^\N h(\gr{z}_i(t))=\sum_{n \geq 1}\sum_{\ui_n}\Psi_{1,n}^{0,t}[h](\gr{Z}_{\ui_n}(0)) +\sum_{k=1}^2 \sum_{n\geq 1} \sum_{\ui_n} \Psi^{>,k\delta}_{1,n}[h](\gr{Z}_{\ui_n}(t-k\delta)) \\
	+\sum_{k=1}^2\sum_{1\leq n\leq n'} \sum_{\ui_{n'}} \Phi^{>,\delta}_{n,n'}\circ\Psi^{0,(k-1)\delta}_{1,n}[h](\gr{Z}_{\ui_{n'}}(t-k\delta)).
	\end{multline*}
	
	The preceding computation can be iterated: for some time $t$, $\theta<t$ and $\delta$ such that $\theta/\delta = K\in\mathbb{N}$, and any initial data $\gr{Z}_\N\in \Upsilon_{\e}$
	\begin{multline}
	\sum_{i=1}^\N h(\gr{z}_i(t)) =\sum_{n \geq 1}\sum_{\ui_n}\Psi_{1,n}^{0,\theta}[h](\gr{Z}_{\ui_n}(t-\theta))+\sum_{k=1}^K \sum_{n\geq 1} \sum_{\ui_n} \Psi^{>,k\delta}_{1,n}[h](\gr{Z}_{\ui_n}(t-k\delta)) \\
	+\sum_{k=1}^K\sum_{1\leq n\leq n'} \sum_{\ui_{n'}} \Phi^{>,\delta}_{n,n'}\circ\Psi^{0,(k-1)\delta}_{1,n}[h](\gr{Z}_{\ui_{n'}}(t-k\delta)).
	\end{multline}

	\subsection{The decomposition of the covariance}
	
	The final ingredient is a second sampling on a longer time scale $\theta \simeq 1/\beta\log|\log\e|$, which controls the growth of the number of collisions.
	
	\begin{definition}[Number of particles at time $\tau$]\label{def: remaining particles}
		Fix $t$ and $\delta$ such that $t/\delta= K\in\mathbb{N}^*$, parameters $(\{1\},\omega_2,(s_k)_{k\leq n-m})$ and admissible initial data $Z_{n_2}\in \mathcal{R}^{0,t}_{\{1\},\omega_2,(s_k)_k}$. We denote $\mathcal{G}$ the collision graph. For $\tau = k\delta$, the \emph{number of particles at time $\tau$}, $\mathfrak{n}(\tau)$, is defined as the size of the connected component of $\{1\}$ in the graph with edges
		\[\{(i,j)\vert \exists (\sigma,\tau)\in \{\pm 1\}\times [k\delta,t]~\text{such that}~(i,j)_{\tau,\sigma}\in \mathcal{G}\}.\]
	\end{definition}
	We want to control the number of particles $(\mathfrak{n}(t-k\theta))_k$ such that it grows slower than the geometric sequence $2^k$.
	
	\blue{Fix $1\leq n_1\leq\cdots \leq n_{l}$. We denote $\underline{n}_l:=(n_i)_{i\leq l}$. For $t \in((l-1)\theta,l\theta]$
	\begin{equation}
	\Psi_{\underline{n}_l}^{0,t}[h] := \frac{1}{(n_l-1)!}\sum_{(s_k)_{k\leq n_l-1}}\prod_{k=1}^{n_l-1}s_k h(\dr{z}_{1}(t,\cdot,\{1\},(s_k)_k))\ind_{\mathcal{R}^{0,t}_{\{1\},(s_k)_k}}\prod_{i=1}^{\lfloor t/\theta\rfloor}\ind_{\mathfrak{n}(t-i\theta)=n_i},\\
	\end{equation}
	and for $t\in [(l-2)\theta,(l-1)\theta]$
	\begin{equation}
	\Psi_{\underline{n}_{l}}^{>,t}[h] := \frac{1}{(n_{l}-1)!}\sum_{\substack{1\in\omega\subset[n_{l}]\\(s_k)_{k\leq n_{l}-1}}}\!\!\prod_{k=1}^{n_{l}-1}s_k\, h(\dr{z}_{1}(t,\cdot,\{1\},\omega,(s_k)_k))\ind_{\mathcal{R}^{>,t}_{\{1\},\omega,(s_k)_k}}\prod_{i=1}^{l-2}\ind_{\mathfrak{n}(t-i\theta)=n_i}.
	\end{equation}}
	
	We can iterate the preceding decomposition of $\sum_{i= 1}^\N h(\gr{z}_i(t))$. The decomposition is performed until reaching the time $0$: denoting $K:=t/\theta\in\mathbb{N}$, $K':=\theta/\delta\in\mathbb{N}$, $\theta \simeq 1/\beta\log|\log\e|$ (for $\beta$ small enough), and $\delta\simeq \e^{1/12}$, for almost any initial data $\gr{Z}_\N(0)\in\mathcal{D}$,	
	\begin{align}
	\sum_{i=1}^\N h(\gr{z}_i(t)) =&\sum_{\substack{(n_j)_{j\leq K}\\0\leq n_j-n_{j-1}\leq 2^j}} \sum_{\ui_{n_k}}\Psi^{0,t}_{\underline{n}_k}[h]\left(\gr{Z}_{\ui_{n_k}}(t)\right)\\
	\label{eq:partie avec trop de coll}&+\sum_{1\leq k\leq K}\sum_{\substack{(n_j)_{j\leq k-1}\\0\leq n_j-n_{j-1}\leq 2^j}}\sum_{n_k\geq 2^k+n_{k-1}}\sum_{\ui_{n_k}}\Psi^{0,k\theta}_{\underline{n}_k}[h]\left(\gr{Z}_{\ui_{n_k}}(t-k\theta)\right)\\
	&\label{eq:partie des reco pas patho}+\sum_{\substack{0\leq k\leq K-1\\1\leq k'\leq K'}}\sum_{\substack{(n_j)_{j\leq k}\\0\leq n_j-n_{j-1}\leq 2^j}}\sum_{n_{k+2}\geq n_{k+1}\geq n_k}\sum_{\ui_{n_k}}\Psi^{>,t-t_s}_{\underline{n}_{k+1}}[h]\left(\gr{Z}_{\ui_{n_k}}(t_s)\right)\\
	\label{eq:partie des reco patho}&+\sum_{\substack{0\leq k\leq K-1\\1\leq k'\leq K'}}\sum_{\substack{(n_j)_{j\leq k}\\0\leq n_j-n_{j-1}\leq 2^j}}\sum_{\substack{n_{k+1}\geq n_k\\n_{k+2}\geq n_{k+1}}}\sum_{\ui_{n_{k+2}}}\Phi^{>,\delta}_{n_{k+1},n_{k+2}}\Psi^{0,t_s-\delta}_{\underline{n}_{k+1}}[h]\left(\gr{Z}_{\ui_{n_{k+2}}}(t_s)\right)
	\end{align}
	where $t_s:=t-k\theta-k'\delta$.
	
	Finally, the covariance is split  into five parts
	\begin{equation}
	\mathbb{E}_\e\left[\zeta^t_\e(h)\zeta^0(g)\right] = G^{\text{main}}_\e(t)+G^{\text{clus}}_\e(t)+G^{\text{exp}}_\e(t)+G^{\text{rec},1}_\e(t)+G^{\text{rec},2}_\e(t),
	\end{equation}
	where we have separated 
	\begin{itemize}
	\item the main part,
	\begin{equation}
	G^{\text{main}}_\e(t):=\sum_{\substack{(n_j)_{j\leq K}\\0\leq n_j-n_{j-1}\leq 2^j}}\mathbb{E}_\e\left[\frac{1}{\sqrt{\mu}}\sum_{\ui_{n_K}}\Psi^{0,t}_{\underline{n}_K}[h]\Big(\gr{Z}_{\ui_{n_K}}(0)\Big)\zeta_\e^0(g)\right],
	\end{equation}
	
	\item the first error due to the symmetric conditioning and the suppression of the overlaps,	
	\begin{equation}
	G^{\text{clus}}_\e(t):=\mathbb{E}_\e\Big[\zeta^t_\e(h)\zeta_\e^0(g)\ind_{\Upsilon^c_\e}\Big]-\sum_{\substack{(n_j)_{j\leq K}\\0\leq n_j-n_{j-1}\leq 2^j}}\mathbb{E}_\e\left[\frac{1}{\sqrt{\mu}}\sum_{\ui_{n_K}}\Psi^{0,t}_{\underline{n}_K}[h]\left(\gr{Z}_{\ui_{n_K}}(0)\right)\zeta_\e^0(g)\ind_{\Upsilon^c_\e}\right],
	\end{equation}
	\item  the part controlling the growth of the number of particles,
	\begin{equation}
	G^{\text{exp}}_\e(t):=\mathbb{E}_\e\left[\eqref{eq:partie avec trop de coll}\times\frac{1}{\sqrt{\mu}}\zeta_\e^0(g)\ind_{\Upsilon_\e}\right],
	\end{equation}
	\item the part corresponding to non-local recollision,
	\begin{equation}
	G^{\text{rec},1}_\e(t):=\mathbb{E}_\e\left[\eqref{eq:partie des reco pas patho}\times\frac{1}{\sqrt{\mu}}\zeta_\e^0(g)\ind_{\Upsilon_\e}\right],
	\end{equation}
	\item and the part corresponding to local recollision,
	\begin{equation}
	G^{\text{rec},2}_\e(t):=\mathbb{E}_\e\left[\eqref{eq:partie des reco patho}\times\frac{1}{\sqrt{\mu}}\zeta_\e^0(g)\ind_{\Upsilon_\e}\right].
	\end{equation}
	\end{itemize}
	
	The parts $G^{\text{clus}}_\e(t)$ and $G^{\text{exp}}_\e(t)$ are estimated by \eqref{Estimation morceau 1}:
	\[\left|G^{\text{exp}}_\e(t)+G^{\text{rec}}_\e(t)\right| \leq C\|g\|_0\|h\|_0\left(\e^{1/3}(C\tfrac{t}{\mathfrak{d}})^{2^{K}}+~\tfrac{\theta t^2}{\mathfrak{d}^{3}}\right),\]
	the part $G^{\text{rec},1}_\e(t)$ is estimated by \eqref{Estimation morceau 3}: for some $\mathfrak{a}>0$ depending only on the dimension,
	\[\left|G_\e^{\text{rec},1}(t)\right|\leq  \|g\|_0\|h\|_0  (Ct)^{2^{K}+d+9}\e^{{\mathfrak{a}}/2},\]
	the part $G^{\text{rec},2}_\e(t)$ is bounded at \eqref{Estimation morceau 4}:
	\[\Big|G_\e^{\text{rec},2}(t)\Big|\leq C\|h\|_0 \|g\|_0  (C\tfrac{t}{\mathfrak{d}})^{2^{K+1}}\e^{\frac{{\mathfrak{a}}}{2}},\]
	and the convergence of $G^{\text{main}}_\e(t)$ is given by \eqref{Estimation morceau 6}:
	\[G^{\rm main}_\e(t)= \int_{\mathbb{D}}h(z)\gr{g}_\alpha(t,z)M(z)dz+O\left( \left(C\tfrac{\theta t}{\mathfrak{d}^2}+\e^{\mathfrak{a}} K2^{K^2} (\tfrac{Ct}{\mathfrak{d}})^{2^{K+1}}\right)\|h\|_1\|g\|_1\right)\]
	where $\gr{g}_\alpha(t,z)$ is the solution of the linearized Boltzmann equation \eqref{eq:Linearized Boltzmann equation}. Combining these four estimations and that $K\leq T/\theta$, we obtain the expected bound \eqref{Borne principale}
	\begin{equation}\label{eq:Borne principale bis}
	\mathbb{E}_\e\left[\zeta_\e^t(g)\zeta_\e^0(h)\right]= \int_{\mathbb{D}}h(z)\gr{g}_\alpha(t,z)M(z)dz+O\left( \left(C\tfrac{\theta T^2}{\mathfrak{d}^{3}}+\e^{\mathfrak{a}} (\tfrac{CT}{\mathfrak{d}})^{2^{\frac{T}{\theta}+1}}\right)\|h\|_1\|g\|_1\right)
	\end{equation}
	
	\begin{remark}
		In this section, we have defined three different pseudotrajectories:
		\begin{itemize}
			\item in Definition \ref{def:pseutrajectory1} we have defined the general definition of pseudotrajectory, which is used in the estimation of pathological recollision $G_\e^{\text{rec},1}(t)$,
			\item the pseudotrajectories of Definition \ref{def:pseudotrajectory2} have no recollision and are used to treat $G_\e^{\text{main}}(t)$, $G_\e^{\text{clust}}(t)$ and $G_\e^{\text{exp}}(t)$,
			\item Definition \ref{def:pseudotrajectory3} describes pseudotrajectories with non-pathological recollision. They are used to bound $G_\e^{\text{rec},2}(t)$.
		\end{itemize}
	\end{remark}
	
	\blue{\begin{remark}[Comparison with the strategy of \cite{LeBihan2,LeBihan3}]\label{rem:Comparaison avec le vieux papier}
		As explained in Remark \ref{rem: tree versu graph}, we used in \cite{LeBihan2} a \emph{tree pseudotrajectory} representation, while we use in the present article the graph pseudotrajectories. There are three advantages to this construction:
		\begin{itemize}
			\item The tree pseudotrajectories have already been used in the description of gas interacting {\it via} compactly supported potential \cite{Grad,King,GST,PSS,Ayi,Catapano1}. One of the encountered difficulties  is the treatment of multiple interactions, which have to be treated separately from binary interactions even before imposing some conditioning. Our presentation avoids this difficulty.
			
			\item In \cite{LeBihan2}, we needed two conditionings (see Section 2.3 of \cite{LeBihan2}). The first one is symmetric, and looks like $\Upsilon_\e$ (we denote it $\tilde{\Upsilon}_{\e}$) : we impose the same condition on all the particles. We also needed a second conditioning, which is asymmetric: we construct an indicator function $\mathcal{X}_{\ui_n}:\mathcal{D}\to\{0,1\}$ to forbid a recollision on a small slice $[k\delta,(k+1)\delta]$ that involves one of the particles of $\ui_n$. Then we made the decomposition
			\[\sum_{\ui_n} \Psi_n(\gr{Z}_{\ui_n}) = \sum_{\ui_n} \Psi_n(\gr{Z}_{\ui_n})\mathcal{X}_{\ui_n}(\gr{Z}_\N) +\sum_{\ui_n} \Psi_n(\gr{Z}_{\ui_n})(1-\mathcal{X}_{\ui_n}(\gr{Z}_\N)).\]
			
			The use of the graph pseudotrajectories allows us to avoid this second conditioning.
			
			\item In \cite{LeBihan2}, we used the following conditioning:
				let $Z_n\in\mathbb{D}^n$ be a particle configuration. We consider the  graph $G$ of vertices $\{1,\cdots,n\}$ and of edges 
				\[\{(i,j)\in[1,m]^2,~d(x_i,x_j)<2\delta\mathbb{D}\}.\]
				The coordinates $Z_n$ form a \emph{distance cluster} if the graph $G$ is connected.
				
				We construct the conditioning $\tilde{\Upsilon}_\e$ as we have constructed $\Upsilon_\e$, by replacing the notion of \emph{possible cluster} by the distance clusters (we suppose that there is not a distance cluster larger than some integer $\gamma$).
				
				The conditioning $\tilde{\Upsilon}_\e$ is more constraining than ${\Upsilon}_\e$: we obtained in Proposition 2.3 of \cite{LeBihan2}
				\[\mathbb{P}_\e\left({\tilde{\Upsilon}_\e}^c\right)\leq C_\gamma\frac{t}{\delta}\left(\mu_\e\delta^{-1}\left(\mu_\e\delta^d\mathbb{V}^d\right)^\gamma+\mu_\e^\gamma e^{-\mathbb{V}^2/4}\right),\]
				which is much larger than our estimation on $\mathbb{P}_\e\left(\tilde{\Upsilon}^c_\e\right)$.
				
				The conditioning $\tilde{\Upsilon}_\e$ is only useful if we take $\delta$ small enough (of order $\e^{1-\frac{1}{2d}}$ for example). The problem is that the gain of smallness due to recollisions has to be smaller than $\delta$. If we restrict ourselves to the hard sphere setting, there is no problem, as one can obtain a gain of order $\e|\log\e|$ (see the Appendix of \cite{BGSS1}). In Assumption 2.3.1 of \cite{LeBihan3}, we proposed a family of interaction potentials for which the same estimation holds.
				
				The cluster decomposition of this paper allows us to take a much larger $\delta = \e^{\frac{1}{12}}$ and then use the strategy of \cite{PSS} to bound the recollisions. This improvement is possible because we used the graph pseudo-trajectory. 
		\end{itemize} 
	\end{remark}}
	\section{Quasi-orthogonality estimates}\label{Quasi-orthogonality estimates}
	
	The different error terms obtained in the previous section are of the form
	\begin{equation*}
	\mathbb{E}_\e\left[\sum_{\ui_n}\Phi_{n}[h](\gr{Z}_{\ui_n}(t_{s}))\zeta_\e^0(g)\ind_{\Upsilon_\e}\right]
	\end{equation*}
	where the $\Phi_n: L^\infty(\mathbb{D})\to L^\infty(\mathbb{D}^n)$ are continuous functionals. In order to bound the errors, we need an $L^2(\mathbb{P}_\e)$ bound of \[\hat{\Phi}_n=\frac{1}{\mu^n}\sum_{\ui_n}\Phi_{n}[h](\gr{Z}_{\ui_n})-\mathbb{E}[\Phi_n].\] 
	
	The following section is dedicated to the derivation of such an estimate, using detailed estimations on the functionals $\Phi_n[h]$. We will use, in particular, that we can bound the $\Phi_n[h](Z_n)$ by looking only at the relative positions of particles inside $Z_n$.
	
	\begin{definition}
	We denote for $y\in{\mathbb{T}}$ the translation operator
	\begin{equation}
	\tr_y:\left\{\begin{array}{r@{~}c@{~}l} \mathbb{D}^n~~&\rightarrow&~~\mathbb{D}^n\\
	(X_n,V_n)&\mapsto& (x_1+y,\cdots,x_n+y,V_n)
	\end{array}\right..
	\end{equation} 
	
	Fix $n,m$ two integers, $g_n:\mathbb{D}^n\to \mathbb{R}$, $h_m:\mathbb{D}^m\to \mathbb{R}$ two functions, and $l\in[0,\min(n,m)]$. We define the multiplication on $l$ variable $\circledast_l$ as
	\begin{multline}
	g_n\circledast_lh_{m}(Z_{n+m-l})\\
	:=\frac{1}{(n+m-l)!n!m!}\sum_{\substack{\sigma,\sigma',\sigma''\in\mathfrak{S}([n+m-l])\\ \sigma'_{\vert [1,n]^c}=\Id\\ \sigma''_{\vert [n-l,2n-l]^c}=\Id}} g_n(Z_{\sigma\sigma'([1
	,n])})h_{m}(Z_{\sigma\sigma''([n+1-l,n+m-l])}).
	\end{multline}
	where $\mathfrak{S}(\omega)$ is the group of permutation of $\omega$.
	\end{definition}

	\begin{prop}\label{theoreme de quasi orthogonalite}
		Fix $m<n$ two positive integers, and $g_n:\mathbb{D}^n\to \mathbb{R}$, $h_{m}\mathbb{D}^m\to \mathbb{R}$ two functions such that there exists a finite sequence $(c_0,c_0',c_1,\cdots,c_n)\in\mathbb{R}^{n+2}_+$  bounding $g_n,\,h_{m}$ in the following way:
		\begin{equation}\label{borne sur g_n}
		\int\limits_{x_1=0}\sup_{y\in{\mathbb{T}}}\big|g_n\big(\tr_yZ_n\big)\big|\frac{e^{-\mathcal{H}_{n}(Z_{n})}}{(2\pi)^{\frac{nd}{2}}}\ud X_{2,n}\ud V_n\leq c_0,
		\end{equation}
		\begin{equation}\label{borne sur h_m}
		\int\limits_{x_1=0}\sup_{y\in{\mathbb{T}}}\big|h_{m}\big(\tr_yZ_{m}\big)\big|\frac{e^{-\mathcal{H}_{m}(Z_{m})}}{(2\pi)^{\frac{md}{2}}}\ud X_{2,m}\ud V_{m}\leq c'_0
		\end{equation}
		and for all $l \in [1,m]$
		\begin{equation}
		\begin{split}
		\int\limits_{x_1=0}\sup_{y\in{\mathbb{T}}} \big|g_n\circledast_lh_{m}\big(\tr_yZ_{n+m-l}\big)\big|\frac{e^{-\mathcal{H}_{n+m-l}(Z_{n+m-l})}}{(2\pi)^{\frac{(n+m-l)d}{2}}}\ud X_{2,n+m-l}\ud V_{n+m-l}
		\leq \frac{\mu^{l-1}}{n^l}c_l.
		\end{split}
		\end{equation}
		There exists a constant $C>0$ depending only on the dimension such that
		\blue{\begin{equation}\label{borne espérence}
				\mathbb{E}_\e\big[g_n\big|= \int g_n(Z_n)\frac{e^{-\mathcal{H_n}(Z_n)}}{(2\pi)^{nd/2}}\ud{Z_n}+ O(C^nc_0\tfrac{\e}{\mathfrak{d}})
		\end{equation}}
		and 
		\blue{\begin{multline}\label{quasi covariance}
		\mathbb{E}_\e\Big[\mu\hat{g}_n\hat{h}_{m}\Big]=\sum_{l=1}^{m}\binom{n}{l}\binom{m}{l}\frac{l!}{\mu^{l-1}}\int g_n\circledast_lh_{m}(Z_{n+m-l})\frac{e^{-\mathcal{H}_{n+m-l}(Z_{n+m-l})}}{(2\pi)^{\frac{(n+m-l)d}{2}}}\ud  Z_{n+m-l}	\\
		+\,O\big(\big( c_0c'_0 +\max_{1\leq l\leq m}c_l\big)\,C^{n+m}\tfrac{\e}{\mathfrak{d}}\,\big).\end{multline}}
		
		In particular
		\begin{equation}\label{borne variance}
		\big|\mathbb{E}_\e\big[g_n\big]\big|=  O(C^nc_0),~~\left|\mathbb{E}_\e\big[\mu\hat{g}_n\hat{h}_{m}\big]\right|\leq C^{n+m} \left(\max_{1\leq l\leq m}c_l+c_0c_0'\,\frac{\e}{\mathfrak{d}}\right).
		\end{equation}
	\end{prop}
	\blue{\begin{remark}
		This theorem is the counterpart of Proposition 3.1 of \cite{BGSS1}.
	\end{remark}}
	
	\begin{proof}[Proof of Proposition \ref{theoreme de quasi orthogonalite}]~
		
		$\bullet~~$We begin with the proof of \eqref{borne espérence}.
		
		Using invariance of $\mathcal{H}_p$ under permutation, 
		\[\begin{split}
		\mathbb{E}_\e[g_n] &=\frac{1}{\mathcal{Z}\mu^n}\sum_{p\geq n}\frac{\mu^p}{p!}\int \sum_{\substack{(i_1,\cdots i_n)\\\forall k, i_k\leq p}}g_n(Z_{\ui_n})e^{-\mathcal{H}_p(Z_p)}\frac{\ud Z_p}{(2\pi)^{dp/2}}\\
		&= \frac{1}{\mathcal{Z}\mu^n}\sum_{p\geq n}\frac{\mu^p}{p!}\frac{p!}{(n-p!)}\int g_n(Z_n)e^{-\mathcal{H}_p(Z_p)}\frac{\ud Z_p}{(2\pi)^{dp/2}}\\
		&= \frac{1}{\mathcal{Z}}\sum_{p\geq 0}\frac{\mu^p}{p!}\int g_n(Z_n)e^{-\mathcal{V}_{n+p}(X_n,\underline{X}_p)}M^{\otimes n}\ud Z_n \ud \underline{X}_p.
		\end{split}\]
		
		We recall the notation
		\[\mathcal{V}_n(X_n):=\alpha\sum_{1\leq i<j\leq n}\mathcal{V}\left(\frac{x_i-x_j}{\e}\right)\]
		and we denote in the following $\Omega:=\{X_n,\underline{x}_1,\cdots,\underline{x}_p\}$ and for $X,Y\in\Omega$, 
		\begin{equation}\label{eq:def des varphi}\varphi(\underline{x}_i,\underline{x}_j):= \exp\left(-\alpha\mathcal{V}\left(\frac{\underline{x}_i-\underline{x}_j}{\e}\right)\right)-1,~\varphi(X_N,\underline{x}_j):= \exp\left(-\alpha\sum_{i=1}^N\mathcal{V}\left(\frac{{x}_i-\underline{x}_j}{\e}\right)\right)-1.\end{equation}
		Defining $d((x_1,\cdots,x_n),(y_1,\cdots,y_m))$ as the minimum of the $|x_i-y_j|$, we can bound $\varphi$ by 
		\[ -\ind_{d(X,Y)<\e} \leq \varphi(X,Y) \leq0.\]
		
		We decompose $\exp\left({-\mathcal{V}_{n+p}(X_{n+1},\underline{X}_p)}\right)$
		\[\begin{split}
		e^{-\mathcal{V}_{n+p}(X_{n+1},\underline{X}_p)}
		&= e^{-\mathcal{V}_{n}^\e(X_{n})}\prod_{\substack{(X,Y)\in\Omega^2\\X\neq Y}}\left(1+\varphi(X,Y)\right)=e^{-\mathcal{V}_{n}(X_{n})}\sum_{G\in \mathcal{G}(\Omega)}\prod_{(X,Y)\in E(G)}\varphi(X,Y)
		\end{split}\]
		where $\mathcal{G}(\Omega)$ is the set of non-oriented graphs on $\Omega$ and $E(G)$ the set of edges of $G$. 
		Denoting by $\mathcal{C}(\omega)$ the set of connected graphs on $\omega$,
		\begin{equation}\label{decomposition e^V_n+p}\begin{split}
		&\exp\left(-\mathcal{V}_{n+p}(X_{n},\underline{X}_p)\right)\\
		&=\sum_{\substack{\omega\subset[1,p]}}\Bigg(e^{-\mathcal{V}_{n}(X_{n})}\sum_{\substack{G\in \mathcal{C}(\omega\cup \{X_n\})}}\prod_{(X,Y)\in E(G)}\varphi(X,Y)\sum_{G\in \mathcal{G}([p]\setminus\omega)}\prod_{(X,Y)\in E(G)}\varphi(X,Y)\Bigg)\\
		&=e^{-\mathcal{V}_{n}(X_{n})}\sum_{\substack{\omega\subset[1,p]}}\Bigg(e^{-\mathcal{V}_{p-|\omega|}(\underline{X}_{[p]\setminus\omega})}\sum_{\substack{G\in \mathcal{C}(\omega\cup\{X_n\})}}\prod_{(X,Y)\in E(G)}\varphi(X,Y)\Bigg)\\
		&=:e^{-\mathcal{V}_{n}(X_{n})}\sum_{\substack{\omega\subset[1,p]}}e^{-\mathcal{V}_{p-|\omega|}(\underline{X}_{[p]\setminus\omega})}\psi_{|\omega|}^n(X_n,\underline{X}_\omega)\;.
		\end{split}\end{equation}
		
		Thus, using exchangeability, $\mathbb{E}_\e[g_n]$ is equal to
		\begin{equation}\label{dévelopement de l'esperence}
		\begin{split}
		&\frac{1}{\mathcal{Z}}\sum_{p\geq 0}\sum_{p_1+p_2=p}\frac{\mu^p}{p!}\frac{p!}{p_1!p_2!}\int g_n(Z_n)\psi^n_{p_1}(X_n,\underline{X}_{p_1})e^{-\mathcal{V}_{p_2}(\underline{X}'_{p_2})}\frac{e^{-\mathcal{H}_{n}(Z_{n})}}{(2\pi)^{\frac{nd}{2}}}\!\ud{Z_n}\!\ud{\underline{X}_{p_1}}\!\ud{\underline{X}_{p_2}'}\\
		&=\left(\frac{1}{\mathcal{Z}}\sum_{p\geq0}\frac{\mu^p}{p!}\int e^{-\mathcal{V}_p(\underline{X}_p)}\ud{\underline{X}_p}\right)\left(\sum_{p\geq 0}\frac{\mu^p}{p!}\int g_n(Z_n)\psi^n_{p}(X_n,\underline{X}_{p})\frac{e^{-\mathcal{H}_{n}(Z_{n})}}{(2\pi)^{\frac{nd}{2}}}\ud{Z_n} \ud{\underline{X}_{p}}\right)\\
		&=\sum_{p\geq 0}\frac{\mu^p}{p!}\int g_n(Z_n)\psi^n_{p}(X_n,\underline{X}_{p})\frac{e^{-\mathcal{H}_{n}(Z_{n})}}{(2\pi)^{\frac{nd}{2}}}\ud{Z_n}\ud{\underline{X}_{p}}.
		\end{split}
		\end{equation}
		
		We recall Penrose's tree inequality (see \cite{Penrose,BGSS6, Jansen}), for function $\varphi$ defined in \eqref{eq:def des varphi},
		\begin{equation}
		\left|\sum_{C\in\mathcal{C}(\Omega)}\prod_{(X,Y)\in E(C)}\varphi(X,Y)\right|\leq \sum_{T\in\mathcal{T}(\Omega)}\prod_{(X,Y)\in E(T)}|\varphi(X,Y)|
		\leq \sum_{T\in\mathcal{T}(\Omega)}\prod_{(X,Y)\in E(T)}\ind_{d(X,Y)<\e}
		\end{equation}
		with $\mathcal{T}(\Omega)$ the set of trees (minimally connected graphs) on $\Omega$. Fix  $\tr_{-x_1}{X_n}$ (the relative position between particles $1,\cdots,n$). Integrating a constraint $\varphi(\underline{x}_i,\underline{x}_j)$ provides a factor $\gr{c}_d\e^d$, $\varphi(X_n,\underline{x}_j)$ a factor $n\gr{c}_d\e^d$ (where $\gr{c}_d$ is the volume of a sphere of diameter $1$). As there are (see for example the Section 2 of \cite{BGSS6} or \cite{Jansen})
		\[\frac{(p-1)!}{(d_0-1)!(d_1-1)!\cdots(d_{p}-1)!}\]
		trees with specified vertex degrees $d_0,\cdots,d_{p}$ associated to vertices $X_n,\,\underline{x}_1,\cdots,\,\underline{x}_p$ , we get
		\begin{equation}\label{Borne de l'integrale de psi np}\begin{split}
		\bigg|\int \psi^n_p(&X_n\underline{X}_p)\ud{\underline{X}_p }\bigg|\leq \sum_{\substack{d_1,\cdots,d_{p}\geq 1\\d_0+\cdots+d_{p}=2p}}\frac{(p-1)!}{(d_0-1)!(d_1-1)!\cdots(d_{p}-1)!} n^{d_0}(\gr{c}_{\gr{d}}\e^d)^p\\
		&\leq (p-1)!(\gr{c}_{\gr{d}}\e^d)^p\left(\sum_{d_0\geq 1}\frac{n^{d_0}}{(d_0-1)!}\right)\left(\sum_{d_1\geq 1}\frac{1}{(d_1-1)!}\right)\cdots\left(\sum_{d_{p}\geq 1}\frac{1}{(d_p-1)!}\right)\\
		&\leq (p-1)!\,ne^{n}\big(e\gr{c}_{\gr{d}}\e^d\big)^p.
		\end{split}\end{equation}
		We can integrate on the rest of the parameters using \eqref{borne sur g_n}. Hence, as $\psi_0^n(Z_n)= \frac{e^{-\mathcal{H}_n(Z_n)}}{(2\pi)^{nd/2}}$
		\blue{\begin{align*}
			\left|\mathbb{E}_\e[g_n]|-\int g_n(Z_n)\frac{e^{-\mathcal{H_n}(Z_n)}}{(2\pi)^{nd/2}}\ud{Z_n}\right|& \leq\sum_{p\geq 1}\frac{(p-1)!\,ne^{n}\big(e\gr{c}_{\gr{d}}\mu\e^d\big)^p}{p!}\int |g_n(Z_n)|\frac{e^{-\mathcal{H}_{n}(Z_{n})}}{(2\pi)^{\frac{nd}{2}}}\ud{Z_n}\\
			&\leq \frac{c_0 C^{n}\e}{\mathfrak{d}}\sum_{p\geq0} (C\e/\mathfrak{d})^p.
		\end{align*}}
		For some constant $C$ depending only on the dimension. The series converges for $\e$ small enough as $\mathfrak{d}\gg\e$. This concludes the proof of \eqref{borne espérence}.~~\\
		
		$\bullet~~$We treat now \eqref{quasi covariance}. Recall first that
		\[\mathbb{E}_\e\Big[\mu\hat{g}_n\hat{h}_{m}\Big] = \frac{1}{\mu^{n+m-1}}\mathbb{E}_\e\left[\sum_{\ui_n}g_n(\gr{Z}_{\ui_n})\sum_{\underline{j}_{m}}h_{m}(\gr{Z}_{\underline{j}_{m}})\right]-\mu\mathbb{E}_\e\left[g_n\right]\mathbb{E}_\e\left[h_{m}\right].\]
		
		Let us count the number of ways such that $\ui_n$ and $\underline{j}_{m}$ can intersect on a set of length $l$. We have to choose two sets $A\subset[n]$ and $A'\subset[m]$ of length $l$, and a bijection $\sigma:A\to A'$ such that for all indices $k\in A$, $i_k=j_{\sigma{k}}$ and  that $\ui_{A^c}$ does not intersect $\underline{j}_{(A^c)'}$. Thus, using the symmetry, 
		\[\begin{split}\mathbb{E}_\e\Big[\mu\,\hat{g}_n\hat{h}_{m}\Big] = &\sum_{l=1}^{m}\binom{n}{l}\binom{m}{l}\frac{l!}{\mu^{l-1}}\mathbb{E}_\e\big[g_n\circledast_lh_{m}\big]\\
		&+\mu\left(\mathbb{E}_\e\left[\frac{1}{\mu^{n+m}}\sum_{\underline{i}_{n+m}}g_n(\gr{Z}_{\ui_{n}})h_{m}(\gr{Z}_{\ui_{n+1,n+m}})\right]-\mathbb{E}_\e\left[g_n\right]\mathbb{E}_\e\left[h_m\right]\right).
		\end{split}\]
		
		To estimate the error term in \eqref{quasi covariance}, we need to compute  
		\[\begin{split}
		\mathbb{E}_\e\Bigg[&\frac{1}{\mu^{n+m}}\sum_{\underline{i}_{n+m}}g_n(\gr{Z}_{\ui_{n}})h_{m}(\gr{Z}_{\ui_{n+1,n+m}})\Bigg]\\
		&=\frac{1}{\mathcal{Z}}\sum_{p\geq 0} \frac{\mu^p}{p!}\int g_n(Z_n)h_{m}(Z'_{m})\exp\left(-\mathcal{V}_{n+m+p}(X_{n},X'_{m},\underline{X}_p)\right)M^{\otimes n}\ud{Z_{n}}M^{\otimes m}\ud{Z'_{m}}\ud{\underline{X}_p}.
		\end{split}\]
		
		We denote in the following 
		$\Omega:=\{X_n,X_m',\underline{x}_1,\cdots,\underline{x}_p\}$,
		and we decompose 
		\[\begin{split}
		\exp\left(-\mathcal{V}_{n+m+p}(X_{n},X'_{m},\underline{X}_p)\right)
		&= e^{-\mathcal{V}_{n}(X_{n})}e^{-\mathcal{V}_{m}(X'_{m})}\prod_{\substack{(X,Y)\in\Omega^2\\X\neq Y}}\left(1+\varphi(X,Y)\right)\\
		&=e^{-\mathcal{V}_{n}(X_{n})}e^{-\mathcal{V}_{m}(X'_{m})}\sum_{G\in \mathcal{G}(\Omega)}\prod_{(X,Y)\in E(G)}\varphi(X,Y)
		\end{split}\]
		where
		\[\varphi(X_n,X'_m):=\exp\left(-\alpha\sum_{i= 1}^n\sum_{j=1}^m\mathcal{V}\left(\frac{x_i-x'_j}{\e}\right)\right).\]
		We make a partition depending on the connected components of $X_n$ and $X'_{m}$ in $G$,
		\begin{multline*}
		\exp\left(-\mathcal{V}_{n+m+p}(X_{n},X'_{m},\underline{X}_p)+\mathcal{V}_n(X_n)+\mathcal{V}_m(X'_m))\right) =\sum_{\substack{\omega\subset[1,p]}}\psi^{n,m}_{|\omega|}(X_{n},X_{m}',\underline{X}_\omega)e^{-\mathcal{V}_{p-|\omega|}(\underline{X}_{[p]\setminus\omega})}\\
		+\sum_{\substack{\omega_1,\omega_1\subset[1,p]\\\omega_1\cap\omega_2= \emptyset}}\psi^n_{|\omega_1|}(X_n,\underline{X}_{\omega_1})\psi^{m}_{|\omega_2|}(X'_{m},\underline{X}_{\omega_2})e^{-\mathcal{V}_{p-|\omega_1\cup\omega_2|}(\underline{X}_{[p]\setminus(\omega_1\cup\omega_2)})}.
		\end{multline*}
		where the first line corresponds to $X_n$ and $X'_m$ in the same connected components, and second corresponds to $X_n$ and $X'_m$ in disjoint connected components. In the preceding equation, we denote 
		\[\psi^{n,m}_{|\omega|}(X_{n},X_{m}',\underline{X}_\omega)=\sum_{\substack{G\in \mathcal{C}(\omega\cup\\ \{X_n,X_{m}'\})}}\prod_{(X,Y)\in E(G)}\varphi(X,Y).\]
		
		Permutating the indices and using \eqref{dévelopement de l'esperence}, we obtain the following equality:
		\[\everymath={\displaystyle}\begin{array}{r@{}l@{}r}
		\frac{1}{\mathcal{Z}}\sum_{p\geq 0} &\multicolumn{2}{l}{\frac{\mu^p}{p!}\int g_n(Z_n)h_{m}(Z'_{m})\sum_{\substack{\omega_1,\omega_1\subset[1,p]\\\omega_1\cap\omega_2= \emptyset}}\psi^n_{|\omega_1|}(X_n,\underline{X}_{\omega_1})\psi^{m}_{|\omega_2|}(X'_{m},\underline{X}_{\omega_2})e^{-\mathcal{V}_{|(\omega_1\cup\omega_2)^c|}(\underline{X}_{(\omega_1\cup\omega_2)^c})}}\\[-7pt]
		&&\times \frac{e^{-\mathcal{H}_{n}(Z_{n})}}{(2\pi)^{\frac{nd}{2}}}\ud{Z_n}\frac{e^{-\mathcal{H}_{m}(Z'_{m})}}{(2\pi)^{\frac{md}{2}}}\ud Z'_{m}\ud \underline{X}_p\\[5pt]
		&\multicolumn{2}{l}{=\frac{1}{\mathcal{Z}}\sum_{p\geq 0}\sum_{p_1+p_2+p_3=p} \frac{\mu^p}{p!}\frac{p!}{p_1!p_2!p_3!}\int g_n(Z_n)h_{n'}(Z'_{n'})\psi^{n}_{p_1}(X_{n},\underline{X}_{p_1})\psi^{1}_{p_2}(x_{n+1},\underline{X}'_{p_2})}\\
		&&\times\Big(\frac{e^{-\mathcal{H}_{n}(Z_{n})}}{(2\pi)^{\frac{nd}{2}}}\ud Z_{n}d\underline{X}_{p_1}\Big)\Big(\frac{e^{-\mathcal{H}_{m}(Z'_{m})}}{(2\pi)^{\frac{md}{2}}}\ud Z'_{n'}\ud \underline{X}'_{p_2}\Big)\Big(e^{-\mathcal{V}_{p_3}(\underline{X}''_{p_3})}\ud \underline{X}''_{p_3}\Big)\\
		&\multicolumn{2}{l}{=\mathbb{E}_\e[g_n]\mathbb{E}_\e[h_{n'}],}
		\end{array}\]
		and in the same way
		\[\begin{aligne}{r@{}l@{}r}
		\frac{1}{\mathcal{Z}}&\multicolumn{2}{l}{\sum_{p\geq 0} \frac{\mu^p}{p!}\int g_n(Z_n)h_{m}(Z'_{m})\sum_{\substack{\omega\subset[1,p]}}\psi^{n,m}_{|\omega|}(X_{n},X_{m}',\underline{X}_\omega)e^{-\mathcal{V}_{||\omega^c|}(\underline{X}_{\omega^c})}}\\[-5pt]
		&&\times\frac{e^{-\mathcal{H}_{n}(Z_{n})-\mathcal{H}_{m}(Z'_{m})}}{(2\pi)^{\frac{(n+m)d}{2}}}\ud{Z_n}\ud Z'_{m}\ud \underline{X}_p\\[10pt]
		&\multicolumn{2}{l}{=\frac{1}{\mathcal{Z}}\sum_{p\geq 0}\sum_{p_1+p_2=p} \frac{\mu^p}{p!}\frac{p!}{p_1!p_2!}\int g_n(Z_n)h_{m}(Z'_{m})\psi^{n,m}_{|\omega|}(X_{n},X_{m}',\underline{X}_{p_1})e^{-\mathcal{V}_{p_2}(\underline{X}'_{p_2})}}\\
		&&\times\frac{e^{-\mathcal{H}_{n}(Z_{n})-\mathcal{H}_{m}(Z'_{m})}}{(2\pi)^{\frac{(n+m)d}{2}}}\ud{Z_n}\ud Z'_{m}\ud\underline{X}_{p_1}\ud\underline{X}'_{p_2}\\[13pt]
		&\multicolumn{2}{l}{=\sum_{p_1\geq 0} \frac{\mu^p}{p_1!}\int g_n(Z_n)h_{m}(Z'_{m})\psi^{n,m}_{|\omega|}(X_{n},X_{m}',\underline{X}_{p_1})\frac{e^{-\mathcal{H}_{n}(Z_{n})-\mathcal{H}_{m}(Z'_{m})}}{(2\pi)^{\frac{(n+m)d}{2}}}\ud{Z_n}\ud Z'_{m}\ud\underline{X}_{p_1}\ud\underline{X}'_{p_2}.}
		\end{aligne}\]
		
		We will use again the Penrose tree inequality,
		\begin{equation*}
		\left|\psi^{n,m}_{|\omega|}(X_{n},X_{m}',\underline{X}_{p_1})\right|\leq \sum_{T\in\mathcal{T}(\Omega)}\prod_{(X,Y)\in E(T)}|\varphi(X,Y)|\leq \sum_{T\in\mathcal{T}(\Omega)}\prod_{(X,Y)\in E(T)}\ind_{d(X,Y)<\e}.
		\end{equation*}
		
		First, we fix $\tr_{-x_1}{X_n}$ and $\tr_{-x'_1}{X'_m}$. Integrating a constraint $\varphi(\underline{x}_i,\underline{x}_j)$ provides a factor $\gr{c}_d\e^d$, $\varphi(X_n,\underline{x}_j)$ a factor $n\gr{c}_d\e^d$, $\varphi(X'_m,\underline{x}_j)$ a factor $m\gr{c}_d\e^d$, and $\varphi(X_n,X'_m)$ a factor $nm\gr{c}_d\e^d$. Denoting $d_0,d'_0,d_1\cdots,$ $d_{p}$ the degrees of $X_n,\,X_m',\,\underline{x}_1,\cdots,\,\underline{x}_m$ and $\hat{x}_1:= x_1-x'_1$,
		\begin{equation}\begin{split}
		\bigg|\int \psi^{n,m}_{|\omega|}&(X_{n},X_{m}',\underline{X}_{p_1})d\underline{X}_pd\hat{x}_1\bigg|\\
		&\leq \sum_{\substack{d'_0,d_0,\cdots,d_{p}\geq 1\\d'_0+d_0+\cdots+d_{p}=2p}}\frac{p!}{(d'_0-1)(d_0-1)!\cdots(d_{p}-1)!} n^{d_0}m^{d'_0}(\gr{c}_{\gr{d}}\e^d)^{+1}\\
		&\leq p!\big(\gr{c}_{\gr{d}}\e^d\big)^{p+1}nm\,e^{n+m+p}.
		\end{split}\end{equation}
		We can integrate on the rest of the parameters using \eqref{borne sur g_n} and \eqref{borne sur h_m}, and finally	
		\[\begin{split}
		&\mu\left(\mathbb{E}_\e\left[\frac{1}{\mu^{n+m}}\sum_{\underline{i}_{n+m}}g_n(\gr{Z}_{\ui_{n}})h_{m}(\gr{Z}_{\ui_{n+1,n+m}})\right]-\mathbb{E}_\e\left[g_n\right]\mathbb{E}_\e\left[g\right]\right).\\
		&\leq c_0 c_0' \mu\sum_{p\geq 0} \frac{\mu^p}{p!}p!\big(\gr{c}_{\gr{d}}\e^d\big)^{p+1}nm\,e^{n+m+p} \\
		&\leq \mu\e^dnm(\gr{c}_{\gr{d}}e)^{n+m}c_0 c_0' \sum_{p\geq 0} (e\gr{c}_{\gr{d}}\e)^{p}\\
		&\leq (\e/\mathfrak{d}) C^{n+m+1}\sum_{p\geq 0} (e\gr{c}_{\gr{d}}\e/\mathfrak{d})^{p}
		\end{split}\]
		which converges for $\e$ small enough.
		
		\blue{$\bullet~~$To conclude the proof, we apply the estimation \eqref{borne espérence} to \eqref{quasi covariance}:
		\[\begin{split}\mathbb{E}_\e\Big[\mu\,\hat{g}_n\hat{h}_{m}\Big] = &\sum_{l=1}^{m}\binom{n}{l}\binom{m}{l}\frac{l!}{\mu^{l-1}}\mathbb{E}_\e\big[g_n\circledast_lh_{m}\big]+O(C^{n+m}c_0c'_0\tfrac{\e}{\mathfrak{d}})\\
		=&\sum_{l=1}^{m}\binom{n}{l}\binom{m}{l}\frac{l!}{\mu^{l-1}}\int g_n\circledast_lh_{m}(Z_{n+m-l})\frac{e^{-\mathcal{H}_{n+m-l}(Z_{n+m-l})}}{(2\pi)^{\frac{(n+m-l)d}{2}}}\ud  Z_{n+m-l} \\
		&+ O\left(\frac{\e}{\mathfrak{d}}\left(C^{n+m}c_0c'_0 +\sum_{l=1}^{m}\binom{n}{l}\binom{m}{l}\frac{l!}{\mu^{l-1}}\frac{\mu^{l-1}C^l}{n^l}c_l\right)\right).
		\end{split}\]
		Then,
		\begin{equation*}
		\sum_{l=1}^{m}\binom{n}{l}\binom{m}{l}\frac{l!}{\mu^{l-1}}\frac{\mu^{l-1}C^l}{n^l}c_l
		\leq \sum_{l=1}^{m}\binom{m}{l}\frac{n!}{(n-l)!n^l}C^lc_l\
		\leq (1+C)^m\max_{1\leq l \leq m}c_l
		\end{equation*}
		which concludes the proof.
		
		$\bullet~~$ The inequalities \eqref{borne variance} can be obtained in the same way.}
	\end{proof}

	Note also the following bound in $L^p$ norms of the fluctuation field.
	\begin{prop}
		For any $p\in[2,\infty)$, there exists a constant $C_p>0$ such that
		
		\begin{equation}\label{$L^p$ bound of the fluctuation field}
		\left(\mathbb{E}_\e\left[\zeta^0_\e(g)^p\right]\right)^{1/p}\leq C_p \|g\|_{L^p(M(v)dz)}.
		\end{equation}
	\end{prop}

	The proof can be found in Appendix A of \cite{BGSS1}.
	
	From these estimations, one can deduce the following corollary:
	\begin{corollary}\label{Corollaire utilisant la quasi orthogonalite}
		Let $h_n$ be a test function satisfying the conditions of Proposition \ref{theoreme de quasi orthogonalite}. Then there exists a constant $C>0$ such that
		\begin{equation}
		\begin{split}
		\Bigg|\mathbb{E}_\e\Bigg[\frac{1}{\sqrt{\mu}}\sum_{\ui_n}&h_n(\gr{Z}_{\ui_n}(t_s))\zeta^0_\e(g)\ind_{\Upsilon_{\e}}\Bigg]\Bigg|\leq C^n\mu^{n-1} \mathbb{E}_\e \big[\zeta_\e^0(g)^2\big]^{1/2}\left(c_0+\left(\sup_{1\leq l\leq n}c_l\right)^{1/2}\right).
		\end{split}
		\end{equation}
		
	\end{corollary}
	\begin{proof}
		\blue{We can decompose the left hand side as}
		\[\begin{split}\mathbb{E}_\e\Bigg[\frac{1}{\sqrt{\mu}}&\sum_{\ui_n}h_n(\gr{Z}_{\ui_n}(t_s))\zeta^0_\e(g)\ind_{\Upsilon_{\e}}\Bigg]=\mu^{n-1}\mathbb{E}_\e\Bigg[\mu^{\frac{1}{2}-n}\sum_{\ui_{n}}h_n(\gr{Z}_{\ui_n}(t_s))\zeta^0_\e(g)\ind_{\Upsilon_{\e}}\Bigg]\\
		&~~~=\mu^{n-1}\bigg(\mathbb{E}_\e\Big[\mu^{\frac{1}{2}}\widehat{h_n}(\gr{Z}_\N(t_s))\,\zeta^0_\e(g)\ind_{\Upsilon_{\e}}\Big]+\mathbb{E}_\e\left[h_n\right]\mathbb{E}_\e\Big[\mu^{\frac12}\zeta^0_\e(g)\ind_{\Upsilon_{\e}}\Big]\bigg)\\
		&~~~=\mu^{n-1}\bigg(\mathbb{E}_\e\Big[\mu^{\frac{1}{2}}\widehat{h_n}(\gr{Z}_\N(t_s))\,\zeta^0_\e(g)\ind_{\Upsilon_{\e}}\Big]+\mathbb{E}_\e\left[h_n\right]\mathbb{E}_\e\Big[\zeta^0_\e(g)\mu^{\frac{1}{2}} \left(-\ind_{\Upsilon^c_{\e}})\right)\Big]\bigg)\;.
		\end{split}\]
		Using that $\mathbb{E}_\e[\zeta_\e^0(g)]=0$. Applying the Cauchy-Schwarz inequality, we obtain
		\[\begin{split}
		\Bigg|\mathbb{E}_\e\Bigg[&\mu^{-\frac{1}{2}}\sum_{\ui_n}h_n(\gr{Z}_{\ui_n}(t_s))\zeta^0_\e(g)\ind_{\Upsilon_{\e}}\Bigg]\Bigg|\\
		&\leq\mu^{n-1}\left(\mathbb{E}_\e\left[\mu\left[\widehat{h_n}\right]^2\right]^{\frac{1}{2}}\mathbb{E}_\e\big[\zeta^0_\e(g)^2\big]^{\frac{1}{2}}+\mathbb{E}_\e\left[h_n\right]\mathbb{E}_\e\big[\zeta^0_\e(g)^2\big]^{\frac{1}{2}}\big(\mu\mathbb{P}_\e\big[\Upsilon_{\e}^c\big]\big)^{\frac1 2}\right)\;.
		\end{split}\]
		
		We apply now Proposition \ref{theoreme de quasi orthogonalite}. The bound on $\mathbb{P}_\e\left[\Upsilon_{\e}^c\right]$ given in Section \ref{Conditioning} and the bound on the $L^p$ norm of $\zeta^0_\e(g)$ \eqref{$L^p$ bound of the fluctuation field} lead  to the stated corollary.
	\end{proof}
	
	\section{Clustering estimations without recollision}\label{Clustering estimations}
	The objective of this section is to bound $G^{\text{clust}}_\e(t)$ and $G^{\text{exp}}_\e(t)$, defined by
	\begin{equation*}
	\begin{split}
	G_\e^{\text{clust}}(t) := \mathbb{E}_\e\Big[\zeta^t_\e(h)&\zeta^0(g)\ind_{\Upsilon^c_\e}\Big]-\sum_{\substack{(n_j)_{j\leq K}\\0\leq n_j-n_{j-1}\leq 2^j}}\mathbb{E}_\e\left[\frac{1}{\sqrt{\mu}}\sum_{\ui_{n_K}}\Psi^{0,t}_{\underline{n}_K}[h]\left(\gr{Z}_{\ui_{n_K}}(0)\right)\zeta_\e^0(g)\ind_{\Upsilon^c_\e}\right],
	\end{split}
	\end{equation*}
	\begin{equation*}
	G_\e^{\text{exp}}(t) := \sum_{1\leq k\leq K}\sum_{\substack{(n_j)_{j\leq k-1}\\0\leq n_j-n_{j-1}\leq 2^j}}\sum_{n_k\geq 2^k+n_{k-1}}\mathbb{E}_\e\left[\frac{1}{\sqrt{\mu}}\sum_{\ui_{n_k}}\Psi^{0,k\theta}_{\underline{n}_k}[h]\left(\gr{Z}_{\ui_{n_k}}(t-k\theta)\right)\zeta_\e^0(g)\ind_{\Upsilon_\e}\right].
	\end{equation*}
	
	\begin{prop}\label{prop: Estimation morceau 1 et 2}
		For $\e>0$ small enough,
		
		\begin{equation}\label{Estimation morceau 1}
		\left|G^{\text{exp}}_\e(t)+G^{\text{rec}}_\e(t)\right| \leq C\|g\|_0\|h\|_0\left(\e^{1/3}(Ct)^{2^{t/\theta}}+~\blue{\tfrac{\theta t^2}{\mathfrak{d}^3}}\right)
		\end{equation}
	\end{prop}

	To obtain the stated result, we need first the following bounds on the pseudotrajectory developments without recollisions of type $\Psi_{\underline{n}_k}^{0,k\theta}[h]$:
	\begin{prop}\label{prop:estimation sans reco}
		Fix $k\in\mathbb{N}$, $\underline{n}:=(n_1,\cdots,n_k)\in\mathbb{N}^k$ with $n_1\leq n_2\leq\cdots\leq n_k$. Then, fixing $x_0 = 0$,
		\begin{equation}\label{Estimation sans reco 1}
		\int\sup_{y\in{\mathbb{T}}}\big|{\Psi}_{\underline{n}_k}^{0,k\theta}[h](\tr_yZ_{n_k})\big|\frac{e^{-\mathcal{H}_{n_k}(Z_{n_k})}}{(2\pi)^{\frac{n_kd}{2}}}\ud V_{n_k}\ud X_{2,n_k}\leq \frac{\|h\|_0}{(\mu\mathfrak{d})^{n_k-1}}C^{n_k}\theta^{n_k-n_{k-1}}(k\theta)^{n_{k-1}-1},
		\end{equation}
		and, for $m\in[1,n_k]$,
		\begin{multline}\label{Estimation sans reco 2}
		\begin{split}
		\int\sup_{y\in{\mathbb{T}}}\big|{\Psi}_{\underline{n}_k}^{0,k\theta}[h]\circledast_m{\Psi}_{\underline{n}_k}^{0,k\theta}[h](\tr_yZ_{2n_k-m})\big|\frac{e^{-\mathcal{H}_{2n_k-m}(Z_{2n_k-m})}}{(2\pi)^{\frac{(2n_k-m)d}{2}}}\ud V_{2n_K-m}\ud X_{2,2n_K-m}\\
		\leq \frac{\mu^{m-1}}{n_k^m}\left(\frac{\|h\|_0}{(\mu\mathfrak{d})^{n_k-1}}C^{n_k}\right)^2\theta^{n_k-n_{k-1}}(k\theta)^{n_{k-1}+n_{k}-1}.
		\end{split}
		\end{multline}
	\end{prop}
	
	Using Corollary \ref{Corollaire utilisant la quasi orthogonalite} and the previous estimations,
	\[\begin{split}
	\Bigg|\mathbb{E}_\e\Bigg[\mu^{-\frac12}\sum_{\ui_{n_k}}\Psi_{\underline{n}_k}^{0,k\theta}[h]\left(\gr{Z}_{\underline{i}_{n_k}}(t-k\theta)\right)\zeta^0_\e(g)\ind_{\Upsilon_\e}\Bigg]\Bigg|	&\leq \|g\|_0\|h\|_0C^{n_k} \left((\tfrac{\theta}{\mathfrak{d}})^{n_k-n_{k-1}}(\tfrac{k\theta}{\mathfrak{d}})^{n_{k-1}-1}\right.\\
	&\quad\qquad\qquad\qquad+\left.(\tfrac{\theta}{\mathfrak{d}})^{\frac{n_k-n_{k-1}}{2}}(\tfrac{k\theta}{\mathfrak{d}})^{\frac{n_k+n_{k-1}-1}2}\right)\\
	&\leq \|g\|_0\|h\|_0C^{n_k} (\tfrac{\theta}{\mathfrak{d}})^{\frac{n_k-n_{k-1}}{2}}(\tfrac{t}{\mathfrak{d}})^{n_{k}},
	\end{split}\]
	and in the same way,
	\[\begin{split}
	\mathbb{E}_\e\Bigg[\mu^{-\frac12}\sum_{\ui_{n_K}}&  \Psi_{\underline{n}_k}^{0,k\theta}[h]\left(\gr{Z}_{\underline{i}_{n_K}}(0)\right)\zeta^0_\e(g)\ind_{\Upsilon^c_\e}\Bigg]= O\left(\e^{\frac12}\|g\|_0\|h\|_0C^{n_k} (t/\mathfrak{d})^{n_{k}}\right).
	\end{split}\]
	
	Summing over all possible $(n_1,\cdots,n_k)$, we obtain
	\begin{equation*}\begin{split}
	\left|G_\e^{\text{exp}}(t)\right| &\leq \sum_{k=1}^K \sum_{\substack{n_1\leq\cdots\leq n_{k-1}\\n_j-n_{j-1}\leq 2^j}}\sum_{n_k>2^k+n_{k-1}} \|g\|_0\|h\|_0C^{n_k} (\theta/\mathfrak{d})^{\frac{n_k-n_{k-1}}{2}}(t/\mathfrak{d})^{n_{k}}\\
	&\leq \|g\|_0\|h\|_0\sum_{k=1}^K \sum_{\substack{n_1\leq\cdots\leq n_{k-1}\\n_j-n_{j-1}\leq 2^j}}\sum_{n_k>2^k+n_{k-1}} \left(C\tfrac{\theta t^2}{\mathfrak{d}^3}\right)^{\frac{n_k-n_{k-1}}{2}}\\
	&\leq C \|g\|_0\|h\|_0\sum_{k=1}^K 2^{k^2}\left(C\tfrac{\theta t^2}{\mathfrak{d}^3}\right)^{2^{k-1}}\leq C\|g\|_0\|h\|_0\tfrac{\theta t^2}{\mathfrak{d}^3}.
	\end{split}\end{equation*}
	as the series converges for $\theta$ small enough. In the same way 
	\begin{equation*}\begin{split}
	\left|G_\e^{\text{clust}}(t)\right| \leq& \mathbb{P}_{\e}(\Upsilon^c_{\e})^{\frac{1}{4}}\mathbb{E}_\e[\zeta^0_\e(g)^4]^{\frac14}\mathbb{E}_\e[\zeta^0_\e(h)^2]^{\frac12}\\
	&\quad+ \!\!\sum_{\substack{n_1\leq\cdots\leq n_K\\n_j-n_{j-1}\leq 2^j}}\!\!\mathbb{P}_{\e}(\Upsilon^c_{\e})^{\frac{1}{4}}\|g\|_0\|h\|_0C^{n_k} t^{n_{k}}+\|g\|_0\|h\|_0(\tfrac{C\theta}{\mathfrak{d}})^{n_{K}}(\tfrac{\e}{\mathfrak{d}})^\frac{1}{2}\\
	\leq& C\|g\|_0\|h\|_0\e^{\frac{1}{3}}2^{K^2}(Ct)^{2^K}.
	\end{split}\end{equation*}
	This concludes the proof of \eqref{Estimation morceau 1}.

	\begin{proof}[Proof of \eqref{Estimation sans reco 1}]
		We recall that for $t = k\theta$ and that
		\[{\Psi}_{\underline{n}_k}^{0,k\theta}[h] := \frac{1}{(n_k-1)!}\sum_{(s_l)_{l\leq n_k-1}}\prod_{l=1}^{n_k-1}s_l h(\ds{z}_{1}(t,\cdot,\{1\},(s_l)_l))\ind_{\mathcal{R}^{0,t}_{\{q\},(s_l)_l}}\prod_{i=0}^{k-1}\ind_{\mathfrak{n}(i\theta)=n_{n_{k-i}}}.\]
		This gives directly the following bound on $\bar{\Psi}_{\underline{n}_k}^{0,t}[h]$
		\begin{equation}\label{|Phi^0_n|}
		\left|\bar{\Psi}_{\underline{n}_k}^{0,t}[h] \right| \leq \frac{\|h\|_0}{(n_k-1)!}\sum_{(s_l)_{l\leq n_k-1}}\sum_{(s_l)_{l\leq n_k-1}} \ind_{\mathcal{R}^{0,t}_{\{1\},(s_l)_l}}\ind_{\mathfrak{n}(\theta)=n_{k-1}}.
		\end{equation}
		
		As the right-hand side of \eqref{|Phi^0_n|} is invariant under translations, it is sufficient to fix $x_1= 0$ and integrate with respect to $(X_{2,n_k},V_{n_k})$.
		
		We define the \emph{clustering tree} $T^{>}$ as the sequence $(q_i,\bar{q}_i)_{1\leq i \leq n_k-1}$ where the $i$-th collision involves particles $q_i$ and $\bar{q}_i$ (and $q_i<\bar{q}_i$). 
		
		Since in the present section, pseudotrajectories have no recollision, \blue{the collision graph is a tree (a simply connected graph). Hence, we can construct the clustering by forgetting the scattering times associated with each edge but keeping the order of the collisions.} It can be used to parametrize a partition of $\mathcal{R}^{0,t}_{\{q\},(s_l)_l}$.
		
		Let us fix a clustering tree. We perform the following change of variables
		\[X_{2,n_k}\mapsto(\hat{x}_1,\cdots,\hat{x}_{n_k-1}),~\forall i \in [1,n_k-1],~\hat{x}_i:=x_{q_i}-x_{\bar{q}_i}\]
		
		Fix then $\tau_{i+1}$ the time of the $(i+1)$-th collision, as well as the relative positions $\hat{x}_1,\cdots,\hat{x}_{i-1}$. We denote $T_i = \theta$ if $i\leq n_k-n_{k-1}$, $t$ else (at least $n_k-n_{k-1}$ clustering collisions happen before time $\theta$). The $i$-th collision set is defined by
		\[B_{T^>,i}:=\left\{\hat{x}_i\Big|\exists \tau\in(0,T_i\wedge \tau_{i+1}),~|\ds{x}_{q_{i}}(\tau)-\ds{x}_{\bar{q}_i}(\tau)|\leq \e\right\}.\]
		
		Because particles $\ds{x}_{q_{i}}(\tau)$ and $\ds{x}_{\bar{q}_i}(\tau)$ are independent until their first meeting, we can perform the change of variable $\hat{x}_i\mapsto(\tau_i,\eta_i)$ where $\tau_i$ is the first meeting time and 
		\[\eta_i:=\frac{\ds{x}_{q_{i}}(\tau_i)-\ds{x}_{\bar{q}_i}(\tau_i)}{\left|\ds{x}_{q_{i}}(\tau_i)-\ds{x}_{\bar{q}_i}(\tau_i)\right|}.\]
		
		It sends the Lebesgue measure $\!\ud\hat{x}_i$ to the measure $\e^{d-1}((\ds{v}_{q_{i}}(\tau_i)-\ds{v}_{\bar{q}_i}(\tau_i))\cdot\eta_i)_+\ud\eta_i\ud \tau_i$ and 
		\[\int \ind_{B_{T^>,i}}\ud\hat{x}_i \leq {C}{\e^{d-1}}\int _0^{T_i\wedge \tau_{i+1}}\left|\ds{v}_{q_{i}}(\tau_i)-\ds{v}_{\bar{q}_i}(\tau_i)\right|\ud \tau_i.\]
		
		We want now to sum on every possible edge $(q_{i},\bar{q}_i)$. Hence, we need to control
		\[\sum_{(q_{i},\bar{q}_i)}\left|\ds{v}_{q_{i}}(\tau_i)-\ds{v}_{\bar{q}_i}(\tau_i)\right|\leq 2n_k\sum_{k}\left|\ds{v}_{k}(\tau_i)\right|\leq 2n_k\left(n_k\sum_{k}\left|\ds{v}_{k}(\tau_i)\right|^2\right)^{1/2}\leq n_k\left(n_k+|\ds{V}_{n_k}(\tau_i)|^2\right)\]
		
		\begin{lemma}
			Consider a time $\tau\in[0,t]$, collision parameters $(\omega_1,\omega_2,(s_i)_i)$ and an initial position $Z_{n}\in\mathbb{D}^{n}$. Then
			\[\tfrac{1}{2}|\ds{V}(\tau,Z_n,\omega_1,\omega_2,(s_i)_i)|^2\leq \mathcal{H}_n(Z_n),\]
			as there is no overlap between particles.
		\end{lemma}
		
		\begin{proof}
			We begin by defining the notion of \emph{clusters on a time segment}:
			\noindent\begin{definition}
				Consider two times $0\leq \tau_a<\tau_b\leq t$. We denote $\mathcal{G}$ the collision graph of the pseudotrajectory $\ds{Z}_{n_k}(\cdot,(\omega,(s_i)_i),Z_{n_k})$ on the time interval $[\tau_a,\tau_b]$ and $G$ a graph with edges
				\[\Big\{(q,\bar{q})\in[n]^2,~\exists\tau'\in[\tau_a,\tau_b],~(q,\bar{q})_{1,\tau}\in\mathcal{G}\}.\]
				We take only into account the collisions with interaction. We define $\underline{\kappa}:=(\kappa_1,\cdots,\kappa_{k})$ the \emph{clusters on the segment $[\tau_1,\tau_2]$} the connected components of $G$ (defined in the following of Definition \ref{def:pseudotrajectory2}).
				
				Note that if $\tau_a$ lies between the beginning of the collision implying $s_j$ and the beginning of the collision implying $s_{j+1}$, then $\underline{\kappa}$ only depends on the $(s_i)_{i\leq j}$.
			\end{definition}
			
			We distinguished the cases $\tau>\delta$ and $\tau\leq \delta.$
			
			\begin{itemize}
				\item First, if $\tau\leq \delta$. We consider $(\kappa_1,\cdots,\kappa_{k})$ the cluster on the segment $[0,\delta]$ constructed in the following Definition \ref{def:pseudotrajectory2}. The pseudotrajectory is the Hamiltonian trajectory associated with the energy
				\[\mathcal{H}_{\underline{\kappa}}(Z_n):=\sum_{i=1}^k \Bigg(\sum_{q\in\kappa_i}\frac{|v_q|^2}{2}+\sum_{\substack{q,\bar{q}\in\kappa_i\\q\neq\bar{q}}}\frac{\alpha}{2}\mathcal{V}\left(\frac{x_q-x_{\bar{q}}}{\e}\right)\Bigg).\]
				Hence
				\[\tfrac{1}{2}|\ds{V}(\tau,Z_n,\omega_1,\omega_2,(s_i)_i)|^2\leq \mathcal{H}^{\underline{\kappa}}(\ds{Z}_n(\tau)\leq \mathcal{H}^{\underline{\kappa}}(Z_{n})\leq \mathcal{H}_n(Z_n).\]
				\item If $\tau>\delta$, consider $\underline{\kappa}$ and $\underline{\kappa}'$ the clusters on $[\delta,\tau]$ and on $[0,\delta]$. After time $\delta$, the particles outside $\omega_2$ stop interacting, and before time $\delta$, the couple of particles in $\omega_2$ cannot overlap. Hence, $\underline{\kappa}'$ is a finer partition of $[n]$ than $\underline{\kappa}$ and $\mathcal{H}_{\underline{\kappa}}\leq \mathcal{H}_{\underline{\kappa}'}$. Thus
				\begin{equation}\label{eq:borne energie cinetique}\frac{1}{2}|\ds{V}_{n_k}(\tau)|^2\leq \mathcal{H}_{\underline{\kappa}}(\ds{Z}_{n_k}(\tau))=\mathcal{H}_{\underline{\kappa}}(\ds{Z}_{n_k}(\delta))\leq \mathcal{H}_{\underline{\kappa}'}(\ds{Z}_{n_k}(\delta))=\mathcal{H}_{\underline{\kappa}'}(\ds{Z}_{n_k}(0))\leq\mathcal{H}_n(Z_n).\end{equation}
			\end{itemize}
		\end{proof}
		
		As we suppose that there is no overlap at time $0$, we have $\mathcal{H}_n(Z_n)=\tfrac{|V_n|^2}{2}$. Hence, using the Boltzmann-Grad scaling $\mu\e^{d-1}\mathfrak{d}= 1$,
		\[\begin{split} \sum_{(q_i,\bar{q}_i)_{i}}&\int\ud\hat{x}_{1}\ind_{B_{T^>,1}}\cdots\int\ud\hat{x}_{n_k-1}\ind_{B_{T^>,n_k-1}} \ind_{\mathfrak{n}(\theta)=n_{k-1}}e^{-\mathcal{H}_{n_k}(Z_{n_k})} \\
		\leq& \left(\frac{Cn_k}{\mu\mathfrak{d}}\right)^{n_k-1}\left(n_k + \mathcal{H}_{n_k}(Z_{n_k})\right)^{n_k-1}e^{-\mathcal{H}_{n_k}(Z_{n_k})}\int_0^{T_{n_k-1}}\ud \tau_{n_k-1}\cdots\int_0^{T_{1}\wedge \tau_{2}}\ud \tau_{1}\\
		\leq& \left(\frac{Cn_k}{\mu\mathfrak{d}}\right)^{n_k-1}n_k^{n_k-1}e^{-\frac{\mathcal{H}_{n_k}(Z_{n_k})}{2}}\frac{t^{n_{k-1}-1}}{(n_{k-1}-1)!}\frac{\theta^{n_k-n_{k-1}}}{(n_k-n_{k-1})!}\\
		\leq& \left(\frac{\tilde{C}}{\mu\mathfrak{d}}\right)^{n_k-1}n_k^{n_k-1}e^{-\frac{|V_{n_k}|^2}{4}}t^{n_{k-1}-1}\theta^{n_k-n_{k-1}},\end{split}\]
		We used these two classical inequalities
		\begin{align*}
		\frac{(a+b)^{a+b}}{a!b!}\leq e^{a+b}&\frac{(a+b)!}{a!b!}\leq (2e)^{a+b}\\ 
		\forall A,B>0,~,x\in\mathbb{R}^+,~~\left(A+x\right)^B e^{-\frac{x}{2}} =& B^B\left(\frac{A+x}{B}e^{-\frac{A+x}{2B}}\right)^B e^{\frac{A}{4}}\leq\left(\tfrac{4B}{e}\right)^B e^{\frac{A}{4}}.
		\end{align*}
		
		Finally, we sum on $V_{n_k}$, on the $2^{n_k-1}$ possible $(s_i)_i$ and on the $q\in[1,n_k]$, and we divide by the remaining $(n_k!)$. This gives the expected estimation.
	\end{proof}
	
	\begin{proof}[Proof of \eqref{Estimation sans reco 2}]
		We begin as in the previous paragraph. Using that 
		\[\left|\left\{(\omega,\omega',q, q')\big|~\omega\cup \omega' = [2n_k-m],~|\omega|=|\omega'|=n_k,~q\in \omega,~q'\in \omega'\right\}\right| = \frac{n_k^2(2n_k-m)!}{(n_k-m)!^2m!},\]
		and that the right hand side of \eqref{|Phi^0_n|} is symmetric, one has 
		\begin{multline}
		\Big|{\Psi}_{\underline{n}_k}^{0,t}[h]\circledast_m\bar{\Psi}_{\underline{n}_k}^{0,t}[h](Z_{2n_k-m})\Big|\\
		\leq \frac{\|h\|_0^2}{(n_k!)^2} \frac{(n_k-m)!^2m!}{n_k^2(2n_k-m)!}\sum_{(\omega,\omega',q,q')}
		\sum_{\substack{(s_l)_{l\leq n_k-1}\\(s'_l)_{l\leq n_k-1}}} 	\ind_{\mathcal{R}^{0,t}_{\{q\},(s_l)_l}}(Z_{\omega}) \ind_{\mathfrak{n}(\theta)={n_{k-1}}}(Z_{\omega})\ind_{\mathcal{R}^{0,t}_{\{q'\},(s'_l)_l}}(Z_{\omega'}).
		\end{multline}
		where $\mathfrak{n}(\theta)$ is the number of particles at time $\theta$ in the pseudotrajectory $\ds{Z}(t,\cdot,\{1\},(s_l)_l)$. The right-hand side is invariant under translation. Hence, without loss of generality, we can suppose that $1\notin \omega'$ and fix $x_1= 0$.
		
		We have to consider two pseudotrajectories \[\ds{Z}(\tau):=\ds{Z}(\tau,Z_{\omega},\{q\},(s_l)_l) {\rm~and~} \dr{Z}'(\tau):=\ds{Z}(\tau,Z_{\omega'},\{q'\},(s'_l)_l).\] 
		
		We want to estimate 
		\[\int \ind_{\mathcal{R}^{0,t}_{\{q'\},(s'_l)_l}}(Z_{\omega'})e^{-\frac{1}{2}\mathcal{H}_{2n_k-m}(Z_{2n_k-m})}\ud{Z_{\omega'\setminus\omega}}.\]
		
		Fix $Z_\omega$ and denote $T_a$ the clustering tree of the pseudotrajectory $\ds{Z}(t)$, constructed as in the proof of \eqref{Estimation sans reco 1}. Next, we construct the clustering tree associated with the second pseudotrajectory: let $(q_i,\bar{q}_i)_{i\leq \ell}$ be the edges of the collision graph of $\ds{Z}'(\tau)$, taking temporal order. Set $\bar{T}_0= \emptyset$. Suppose that $\bar{T}_i$ is constructed. Then $\bar{T}_{i+1} := \bar{T}_i\cup\{(q_i,\bar{q}_i)\}$ if the graph $T_a\cup \bar{T}_i\cup\{(q_i,\bar{q}_i)\}$ has no cycle. Else $\bar{T}_{i+1} := \bar{T}_i$. At the $\ell$-step we have constructed an ordered graph $T_b:=\bar{T}_\ell$ with $n_k-m$ edges.
		
		The $T_b$ define a partition of $\{Z_{\omega'\setminus\omega}\in\mathbb{D}^{n_k-m}|Z_{\omega'}\in \mathcal{R}^{0,t}_{\{q'\},(s'_l)_l}\}$ where the coordinates $Z_\omega$ are fixed.		
		
		The rest of the proof is almost identical to the proof of \eqref{Estimation sans reco 1}. Fix the clustering tree $T_b= (q_i,$ $\bar{q_i})_{n_k\leq i\leq 2n_k-m}$, and perform the following change of variables
		\[X_{\omega'\setminus\omega}\mapsto(\hat{x}_{n_k},\cdots,\hat{x}_{2n_k-m-1}),~\forall i \in [n_k,2n_k-m-1],~\hat{x}_i:=x_{q_i}-x_{\bar{q}_i}.\]
		
		Fix $\tau_{i+1}$, the time of the $(i+1)$-th collision, and relative positions $\hat{x}_{n_k},\cdots,\hat{x}_{i-1}$. We define the $i$-th collision set as
		\[B_{T^>,i}:=\left\{\hat{x}_i\Big|\exists \tau\in(0,t\wedge \tau_{i+1}),~|\ds{x}'_{q_{i}}(\tau)-\ds{x}'_{\bar{q}_i}(\tau)|\leq \e\right\}.\]
		
		As in the preceding lemma, we can perform the change of variable $\hat{x}_i\mapsto(\tau_i,\eta_i)$ where $\tau_i$ is the first meeting time and 
		\[\eta_i:=\frac{\ds{x'}_{q_{i}}(\tau_i)-\ds{x'}_{\bar{q}_i}(\tau_i)}{\e}.\]
		We have
		\[\begin{split}
		\sum_{(q_{i},\bar{q}_i)}\int &\ind_{B_{T^>,i}}\ud\hat{x}_i \leq {C}{\e^{d-1}}\int _0^{\tau_{i+1}}\sum_{ (q_{i},\bar{q}_i)\in\omega'^2}\left|\ds{v}'_{q_{i}}(\tau_i)-\ds{v}'_{\bar{q}_i}(\tau_i)\right|\ud \tau_i.
		\end{split}\]
		
		Using the same method as in the proof of \eqref{eq:borne energie cinetique}, we have 
		\begin{align*}
		\sum_{ (q_{i},\bar{q}_i)\in\omega'^2}\Big|\ds{v}'_{q_{i}}(\tau_i)-\ds{v}'_{\bar{q}_i}(\tau_i)\Big|\leq n_k+\left|\ds{V}'_{\omega'}(\tau_i)\right|^2&\leq 2n_k+2\mathcal{H}_{n_k}(Z_{\omega'}).
		\end{align*}

		Using that
		\[\mathcal{H}_{2n_k-m}(Z_{2n_k-m})\leq\tfrac{1}{2}\left( \mathcal{H}_{n_k}(Z_\omega)+\mathcal{H}_{n_k}(Z_{\omega'})\right),\]
		we can apply the same computation as above, 
		\[\begin{split}
		\int\ind_{\mathcal{R}^{0,t}_{\{q'\},(s'_l)_l}}(Z_{\omega'}) e^{-\mathcal{H}_{2n_k-m}}\ud Z_{\omega'\setminus\omega}
		&\leq \sum_{T_b}\int e^{-\frac{\mathcal{H}_{n_k}(Z_\omega)}{2}} \prod_{i=1}^{n_k-m}\ind_{B_{T^>,i}}\ud \hat{x}_i~ e^{-\mathcal{H}_{n_k}(Z_{\omega'})}dV_{\omega'}\\
		&\leq C\left(\frac{C}{\mu\mathfrak{d}}\right)^{n_k-m} t^{n_{k}-m}(2n_k-m)^{n_k-m}.
		\end{split}\]
		
		We can estimate
		\[\int\ind_{\mathcal{R}^{0,t}_{\{q\},(s_l)_l}}(Z_{\omega}) e^{-\frac12\mathcal{H}_{2n_k-m}}\ud X_{\omega\setminus\{1\}}\ud V_{\omega}\]
		as in the proof of \eqref{Estimation sans reco 1}. We get the expected result by summing on  all the possible parameters $(s_i)_i$, $(s'_i)_i$, $q$, $q'$, $\omega$ and $\omega'$.
	\end{proof}
	
	\section{Treatment of the main part}\label{sec:Treatment of the main part}
	The aim of this section \blue{is} the proof of
	\[G^{\rm main}_\e(t)= \int_{\mathbb{D}}h(z)\gr{g}_\alpha(t,z)M(z)dz+O\left( \left(C\tfrac{\theta t}{\mathfrak{d}^2}+\e^{\mathfrak{a}} (\tfrac{Ct}{\mathfrak{d}})^{2^{K+1}}\right)\|h\|_1\|g\|_1\right),\]
	where $\gr{g}_\alpha(t,z)$ is the solution of the linearized Boltzmann equation \eqref{eq:Linearized Boltzmann equation} \blue{and $\mathfrak{a}\in(0,1)$ is some fix constant depending only on the dimensio}n.
	 
	\subsection{Duality formula}
	We recall that
	\[\begin{split}
	G_\e^{\text{main}}(t) ~~~&=\!\!\!\sum_{\substack{(n_j)_{j\leq K}\\0\leq n_j-n_{j-1}\leq 2^j}}\mathbb{E}_\e\left[\frac{1}{\sqrt{\mu}}\sum_{\ui_{n_K}}\Psi^{0,t}_{\underline{n}_K}[h]\left(\gr{Z}_{\ui_{n_K}}(0)\right)\zeta_\e^0(g)\right]\\[-0pt]
	&= \sum_{\substack{n_1\leq\cdots\leq n_K\\n_j-n_{j-1}\leq 2^j}} \mathbb{E}_\e\left[\mu^{n_K}\,\hat{\Psi}^{0,t}_{\underline{n}_K}[h]\,\hat{g}\right]
	\end{split}\]
	where $\Psi^{0,t}_{\underline{n}_K}[h]$ is the development of $h(z_i(t))$ along pseudotrajectories with $n_k$ remaining particles at time $t-k\theta$, and neither recollision nor overlap nor multiple \blue{encounters}.
	
	We denote 
	\blue{\begin{equation}
		g_{n_K}(Z_{n_K}):=\sum_{i=1}^{n_K}g(z_i)
	\end{equation}}
	
	Then, using the equality \eqref{quasi covariance} and $L^1$ estimations on $\Psi^{0,t}_{\underline{n}_K}[h]$ of Section \ref{Clustering estimations}, we have for $h$ and $g$ in $L^\infty$
	\begin{equation*}
	\begin{split}
	&G_\e^{\text{main}}(t) \\
	&= \sum_{\substack{n_1\leq\cdots\leq n_K\\n_j-n_{j-1}\leq 2^j}} \int\mu^{n_K-1}\Psi^{0,t}_{\underline{n}_K}[h]\left(Z_{n_K}\right)g_{n_K}(Z_{n_K})\,\frac{e^{-\mathcal{H}_{n_K}(Z_{n_K})}\ud Z_{n_K}}{(2\pi)^\frac{n_Kd}{2}}+O\left(\tfrac{\e}{\mathfrak{d}} \sum_{\substack{\underline{n}}} (\tfrac{Ct}{\mathfrak{d}})^{n_k}\|h\|_0\|g\|_0\right)\\
	&=\sum_{\substack{\underline{n}_K}} \int\mu^{n_K-1}\Psi^{0,t}_{\underline{n}_K}[h]\left(Z_{n_K}\right)g_{n_K}(Z_{n_K})\,\frac{e^{-\mathcal{H}_{n_K}(Z_{n_K})}\ud Z_{n_K}}{(2\pi)^\frac{n_Kd}{2}}+O\left(\tfrac{\e}{\mathfrak{d}} \left(K2^{K^2} (\tfrac{Ct}{\mathfrak{d}})^{2^{K+1}}\|h\|_0\|g\|_0\right)\right).
	\end{split}
	\end{equation*}
	
	We want to compute the asymptotics of each term in the sum. As we suppose that there is no overlap
	\[\begin{split}&\int\mu_\e^{n_K-1}\Psi^{0,t}_{\underline{n}_K}[h]\left(Z_{n_K}\right)g_{n_K}(Z_{n_K})\frac{e^{-\mathcal{H}_{n_K}(Z_{n_K})}\ud Z_{n_K}}{(2\pi)^\frac{n_Kd}{2}}\\
	&=\frac{\mu^{n_K-1}}{(n_K-1)!}\sum_{\substack{(s_k)_{k}}}\prod_{k=1}^{n_l-1}s_k\int_{\mathcal{R}^{0,t}_{\{1\},(s_k)_k}}\!\!\!\!h(\ds{z}^\e_{1}(t,Z_{n_K},\{1\},(s_k)_k))g_{n_K}(Z_{n_K})\prod_{i=1}^{K}\ind_{\mathfrak{n}(t-i\theta)=n_i}M^{\otimes n_K}\ud Z_{n_K}.\end{split}\]
	where $\mathcal{R}^{0,t}_{\{1\},(s_k)_k}$ is the set of initial parameters such that the pseudotrajectory has no recollision and $\mathfrak{n}(\tau)$ is the number of remaining particles at time $\tau$ (see definition \ref{def: remaining particles}). We had an exponent $\e$ on $\ds{z}^\e_1$ to mark the $\e$-dependence of the pseudotrajectory.
	\medskip
	
	We want to construct the limiting process of the pseudotrajectory $\ds{Z}_n^\e(\tau)$. 
	
	We denote ${T}$ the \emph{clustering tree} as the sequence $(q_i,\bar{q}_i,\bar{s}_i)_{i\leq n_K-1}$ such that the $i$-th collision happens between particles $q_i$ and $\bar{q}_i$ (with $q_i<\bar{q}_i$) and $\bar{s}_i$ is equal to $1$ (respectively $-1$) if the particles interact (respectively do not interact). Fixing the initial velocities $V_{n_K}$, we perform the change of variable
	\begin{equation*}
	X_{n_K}\mapsto(x_1,(\nu_i,\tau_i)_{i\leq n_K-1}),.
	\end{equation*}
	where $\tau_i$ is the first time when particles $q_i$ and $\bar{q}_i$ meet, an
	\[\nu_i=\frac{\ds{x}^\e_{q_i}(\tau_i)-\ds{x}^\e_{\bar{q}_i}(\tau_i)}{\e}.\]
	The  Jacobian of this application is 
	\begin{equation*}
	\ud X_{n_K} \to \prod_{i=1}^{n_{k}-1}\e^{d-1}\left(\left(\ds{v}^\e_{q_i}(\tau_i)-\ds{v}^\e_{\bar{q}_i}(\tau_i)\right)\cdot\nu_i\right)_+\ud\nu_i\ud\tau_i	=:\frac{\Lambda_T(V_{n_K},\nu_{[n_k-1]})}{(\mu\mathfrak{d})^{n_k-1}} \ud\nu_{[n_k-1]}\ud\tau_{[n_K-1]}~\ud x_1
	\end{equation*}
	where we have denoted
	\[\nu_{[n_k-1]} = (\nu_1,\cdots,\nu_{n_k-1}),~~\tau_{[n_k-1]} = (\tau_1,\cdots,\tau_{n_k-1}).\]
	
	The kernel $\Lambda(V_{n_K},\nu_{[n_k-1]})$ only depends on the successive velocities $(\ds{v}^\e_{q_i}(\tau_i),\ds{v}^\e_{\bar{q}_i}(\tau_i))$ which can be deduced from the collision graph, forgetting the exact values of the $\ud\tau_{[n_k-1]}$ (since we have forbidden the pathological pseudotrajectories).
	
	We defined the signature of the collision tree $\sigma(T):=\bar{s}_1\bar{s}_2\cdots\bar{s}_{n_K}$, the set of collision times
	\[\mathfrak{T}_{\underline{n}_K} := \left\{(\tau_i)_{i\leq n_K-1},~\tau_i\leq\tau_{i+1},~\forall k\leq K,~j\in[n_K-n_{K-k},n_K-n_{K-k-1}],~k\theta\leq \tau_j\leq (k+1)\theta\right\}\]
	and for a given family $\tau_{[n_K-1]}$, we define $\mathfrak{G}^\e_{T}(\tau_{[n_K-1]})$ the set of coordinates $(x_1,(\nu_i)_{i\leq n_K-1},V_{n_K})$ such that the pseudotrajectory has no recollision, and for all $j$, $\left(\ds{v}^\e_{q_i}(\tau_i)-\ds{v}^\e_{\bar{q}_i}(\tau_i)\right)\cdot\nu_i$ is positive.
	
	The map 
	\[\begin{aligne}{c}
	\bigsqcup_{(s_k)_k}\{s_k\}\times\left(\mathcal{R}^{0,t}_{\{1\},(s_k)_k}\cap\{\text{no~overlap}\}\cap\bigcap_{j\leq K-1}\{\mathfrak{n}(j\theta) = n_{K-j}\big\}\right)\to\bigsqcup_{T} \{T\}\times\mathfrak{T}_{\underline{n}_K}\times\mathfrak{G}^\e_{T}\\
	(X_{n_K},V_{n_K})\mapsto (x_1,(\nu_i,\tau_i)_{i\leq n_K-1},V_{n_K})
	\end{aligne} \]
	is a diffeomosphism and 
	\begin{multline}
	\int\mu^{n_K-1}\Psi^{0,t}_{\underline{n}_K}[h]\left(Z_{n_K}\right)g_{n_K}(Z_{n_K})  M^{\otimes n_K}\ud Z_{n_K}\\
	=\frac{\mathfrak{d}^{-n_K+1}}{(n_K-1)!}\sum_{T}\sigma(T)\int_{\mathfrak{T}_{\underline{n}_K}\times\mathfrak{G}^\e_{T}}h(\ds{z}^\e_{1}(t,T)g_{n_K}(\ds{Z}^\e_{n_K}(0,T))\\
	\times M^{\otimes n_K}\Lambda_T(V_{n_K},\nu_{[n_k-1]})\ud\nu_{[n_k-1]} \ud\tau_{[n_K-1]}\ud x_1\ud V_{n_K}.
	\end{multline}
	
	\begin{definition}[Pseudotrajectories for punctual particles]
		Fix a collision tree $T:=(q_i,\bar{q}_i,\bar{s}_i)$ and collision parameters $(V_{n_k},\tau_{[n_k-1]},\nu_{[n_k-1]})$. We now define the pseudotrajectories for punctual particles. The velocities $\ds{V}^0_{n_K}(\tau,T)$ follow a jump process: at time $0$, $\ds{V}^0_{n_K}(\tau=0,T)=V_{n_K}$. At time $\tau_i$, if $\bar{s}_i= 1$ the velocities of particles $q_i$, $\bar{q}_i$ jump to $v_{q_i}(\tau_i^+)$, $v_{\bar{q}_i}(\tau_i^+)$ given by $(v_{q_i}(\tau_i^+),v_{\bar{q}_i}(\tau_i^+),\tilde{\nu}_i):=\xi_{\alpha}(v_{q_i}(\tau_i^-),v_{\bar{q}_i}(\tau_i^-),{\nu}_i)$ ($\xi_\alpha$ the scattering map defined in Definition \ref{def: le scatering, sa vie son oeuvre}).
	
		We defined $\mathfrak{G}^0_{T}$ the set of the $(x_1,(\nu_i)_{i\leq n_K-1},V_{n_K})$ such that for all $j$, $\left(\ds{v}^0_{q_i}(\tau_i)-\ds{v}^0_{\bar{q}_i}(\tau_i)\right)\cdot\nu_i$ is positive. Note that $\mathfrak{G}^\e_{T}\subset\mathfrak{G}^0_{T}$.
	\end{definition}

	We have formally the convergence
	\begin{multline*}
	\int\mu^{n_K-1}\Psi^{0,t}_{\underline{n}_K}[h]\left(Z_{n_K}\right)g_{n_K}(Z_{n_K})M^{\otimes n_K}dZ_{n_K}\\
	\underset{\e\to0}{\longrightarrow}\frac{\mathfrak{d}^{-n_K+1}}{(n_K-1)!}\sum_{T}\sigma(T)\int_{\mathfrak{T}_{\underline{n}_K}\times\mathfrak{G}^0_{T}}h(\ds{z}^0_{1}(t,T)g_{n_K}(\ds{Z}^0_{n_K}(0,T))\\
	\times\Lambda_T(V_{n_K},\nu_{[n_k-1]})\ud\nu_{[n_k-1]} \ud\tau_{[n_K-1]}\ud x_1M^{\otimes n_K}\ud V_{n_K}.
	\end{multline*}
	
	In order to have explicit rates of convergence, we decompose the error into two parts:
	\begin{multline}
	\int\mu^{n_K-1}\,\Psi^{0,t}_{\underline{n}_K}[h]g_{n_K}^\e M^{\otimes n_K}\ud Z_{n_K}\\
	=\eqref{eq:R1}+\eqref{eq:R2}\red{+\eqref{eq:R3}}+\frac{\mathfrak{d}^{-n_K+1}}{(n_K-1)!}\sum_{T}\sigma(T)\int_{\mathfrak{T}_{\underline{n}_K}\times\mathfrak{G}^0_{T}}h(\ds{z}^0_{1}(t,T)g_{n_K}(\ds{Z}^0_{n_K}(0,T))\\
	\times\Lambda_T(V_{n_K},\nu_{[n_k-1]})\ud\nu_{[n_k-1]} \ud\tau_{[n_K-1]}\ud x_1M^{\otimes n_K}\ud V_{n_K},
	\end{multline}
	where we define
	\begin{multline}
	\label{eq:R1}
	=\frac{\mathfrak{d}^{-n_K+1}}{(n_K-1)!}\sum_{T}\sigma(T)\int_{\mathfrak{T}_{\underline{n}_K}\times\mathfrak{G}^\e_{T}}\Big(h(\ds{z}^\e_{1}(t,T)g_{n_K}(\ds{Z}^\e_{n_K}(0,T))-h(\ds{z}^0_{1}(t,T)g_{n_K}(\ds{Z}^0_{n_K}(0,T))\Big)
	\\
	\times\Lambda_T(V_{n_K},\nu_{[n_k-1]})\ud\nu_{[n_k-1]} \ud\tau_{[n_K-1]}\ud x_1M^{\otimes n_K}\ud V_{n_K},
	\end{multline}
	\begin{multline}
	\label{eq:R2}
	=-\frac{\mathfrak{d}^{-n_K+1}}{(n_K-1)!}\sum_{T}\sigma(T)\int_{\mathfrak{G}_T^0}h(\ds{z}^\e_{1}(t,T)g_{n_K}(\ds{Z}^\e_{n_K}(0,T))\left(1-\ind_{\mathfrak{G}_T^\e}\right)
	\\
	\times\Lambda_T(V_{n_K},\nu_{[n_k-1]})\ud\nu_{[n_k-1]} \ud\tau_{[n_K-1]}\ud x_1M^{\otimes n_K}\ud V_{n_K},
	\end{multline}
	\red{\begin{multline}
	\label{eq:R3}
	=\frac{\mathfrak{d}^{-n_K+1}}{(n_K-1)!}\sum_{T}\sigma(T)\int_{\mathfrak{T}_{\underline{n}_K}\times\mathfrak{G}^\e_{T}}h(\ds{z}^\e_{1}(t,T)\Big(g^\e_{n_K}(\ds{Z}^\e_{n_K}(0,T))-g_{n_K}(\ds{Z}^\e_{n_K}(0,T))\Big)
	\\
	\times\Lambda_T(V_{n_K},\nu_{[n_k-1]})\ud\nu_{[n_k-1]} \ud\tau_{[n_K-1]}\ud x_1M^{\otimes n_K}\ud V_{n_K}.
	\end{multline}}
	
	The error parts are estimated using the following standard results:
	\begin{lemma}\label{Borne sur la taille des parrametres d'arbre}
		Fix $\bar{n}:=(n_1,\cdots,n_k)$ and denote for $p\in[1,2]$
		\[\Lambda_T^p(V_{n_K},\nu_{[n_k-1]}):=\prod_{i=1}^{n_{k}-1}\left|\ds{v}^0_{q_i}(\tau^-_i)-\ds{v}^0_{\bar{q}_i}(\tau^-_i)\right|^p\]
		For any $\e>0$ sufficiently small, we have
		\begin{multline}
		\frac{\mathfrak{d}^{-n_K+1}}{(n_K-1)!}\sum_{T}\sigma(T)\int_{\mathfrak{T}_{\underline{n}_K}\times\mathfrak{G}^0_{T}}\Lambda_T^p(V_{n_K},\nu_{[n_K-1]}) \ud\tau_{[n_K-1]}\ud\nu_{[n_K-1]}\ud x_1M^{\otimes n_K}\ud V_{n_K}\\
		\leq C^{n_K}t^{n_{K-1}-1}\theta^{n_K-n_{K-1}}.
		\end{multline}
	\end{lemma}
	\begin{proof}
		Fix first the collision tree $T:=(q_i,\bar{q}_i,\bar{s}_i)_i$. We sum on each $\nu_i$ in the decreasing order:
		\begin{multline}
		\sum_{(q_i,\bar{q}_i,\bar{s}_i)}\int \left|\ds{v}^0_{q_i}(\tau^-_i)-\ds{v}^0_{\bar{q}_i}(\tau^-_i)\right|^p\ud\nu_i \leq C \sum_{(\bar{q}_i,\bar{s}_i)}\left|\ds{v}_{q_i}(\tau_i)-\ds{v}_{q_i}(\tau_i)\right|^p\leq C n_K^{2-\frac{p}{2}}\left|\ds{V}_{n_K}(\tau_i)\right|^{\frac{p}{2}}\\
		\leq C n_K^{2-\frac{p}{2}}\left|V_{n_K}\right|^{\frac{p}{2}}
		\end{multline}
		using the conservation of energy.
		
		Hence, 
		\begin{multline*}
		\sum_{T}\int \Lambda^p_T(V_{n_K},\nu_{[n_k-1]})\ud\nu_{[n_k-1]}M^{\otimes n_K}dV_{n_K}\leq C^{n_K}n_K^{2-\frac{p}{2}}\int\left|V_{n_K}\right|^{\frac{n_Kp}{2}}M^{\otimes n_K}\ud V_{n_K}\\
		\leq {C}_1^{n_K}n_K^{2n_K}\int e^{-\frac{1}{4}|V_{n_K}|^2}\ud V_{n_K}\leq {C}_2^{n_K}n_K^{2n_K}.
		\end{multline*}
		
		Integrating the collision times
		\begin{multline*}\int_{\mathfrak{T}_{\underline{n}_K}}\ud\tau_{[n_K-1]}\leq \prod_{k= 0}^{K-1} \frac{\theta^{n_k-n_{k+1}}}{(n_k-n_{k+1})!}\leq \frac{((K-1)\theta)^{n_{K-1}-1}}{(n_{K-1}-1)!}\frac{\theta^{n_K-n_{K+1}}}{(n_K-n_{K+1})!}\\
		\leq \frac{2^{n_K-1}t^{n_{K-1}-1}\theta^{n_K-n_{K+1}}}{(n_K-1)!}.\end{multline*}
		
		Finally, we multiply the two previous inequalities and $\tfrac{1}{(n_K-1)!}$. Using the Stirling formula, we obtain the expected estimation.
	\end{proof}
	
	\begin{lemma}\label{Borne sur les chevauchements}
		Fix $\bar{n}:=(n_1,\cdots,n_K)$. \blue{There exists a constant $\mathfrak{a}\in(0,1))$, depending only on the dimension, such that} for any $\e>0$ sufficiently small, we have
		\begin{multline}\label{eq:Borne sur les chevauchements}
			\frac{\mathfrak{d}^{-n_K+1}}{(n_K-1)!}\sum_{T}\int_{\mathfrak{T}_{\underline{n}_K}\times\mathfrak{G}^0_{T}}\left|1-\ind_{\mathfrak{G}_T^\e}\right|\Lambda_T(V_{n_K},\nu_{[n_K-1]})\ud\nu_{[n_K-1]}\ud\tau_{[n_K-1]}~\ud x_1M^{\otimes n_K}\ud V_{n_K}\\
		\leq C^{n_K}t^{n_K+10}\e^{\mathfrak{a}}.
		\end{multline}
	\end{lemma}
	
	This is an estimation of the set of parameters leading to a pathology (a recollision, a triple interaction, or an overlap). It is proven in Annex \ref{subsec:estimation reco}. From Lemma \ref{Borne sur les chevauchements} we deduce
	\[|\eqref{eq:R2}|\leq C(Ct)^{n_K}\e^{\mathfrak{a}}\|g\|\,\|h\|.\]

	\begin{lemma}\label{Borne entre la trajectoire limite et la trajectoire d'enskog}
		Fix $\bar{n}:=(n_1,\cdots,n_k)$, $T$, $\e>0$ and $(x_1,(\tau_i,\nu_i)_{i},V_{n_K})\in\mathfrak{G}^{\e}_{T}$. We have
		\blue{\begin{equation}\forall\tau\in[0,t],~
		\big|\ds{X}^\e_{n_K}(\tau)-\ds{X}^0_{n_K}(\tau)\big|\leq 2n_K\mathbb{V}\sum_{i=1}^{n_K-1}\frac{\e\left|\dr{v}_{q_{i}}(\tau^-_{i})-\dr{v}_{q_{i}}(\tau^-_{i})\right|}{\left|(\dr{v}_{q_{i}}(\tau^-_{i})-\dr{v}_{q_{i}}(\tau^-_{i}))\times \nu_i\right|^2}\;.
		\end{equation}}
	\end{lemma}
	\begin{proof}
		Thanks to the estimation of the interaction time \eqref{eq:borne temps collision}, the $i$-th collision lasts at most a time \blue{$\tfrac{\e\left|\dr{v}_{q_{i}}(\tau^-_{i})-\dr{v}_{q_{i}}(\tau^-_{i})\right|}{\left|(\dr{v}_{q_{i}}(\tau^-_{i})-\dr{v}_{q_{i}}(\tau^-_{i}))\times \nu_i\right|^2}$}. Hence, the two trajectories $\ds{X}^\e_{n_K}(\tau)$ and $\ds{X}^0_{n_K}(\tau)$ have coincident velocities for $\tau$ outside the union of the interval \[\blue{\bigcup_{i=1}^{n_K-1}\left[\tau_i,\tau_i+\tfrac{\e\left|\dr{v}_{q_{i}}(\tau^-_{i})-\dr{v}_{q_{i}}(\tau^-_{i})\right|}{\left|(\dr{v}_{q_{i}}(\tau^-_{i})-\dr{v}_{q_{i}}(\tau^-_{i}))\times \nu_i\right|^2}\right]}.\]
		
		During a collision, a particle can cross a distance smaller than $\blue{\tfrac{\e\mathbb{V}\left|\dr{v}_{q_{i}}(\tau^-_{i})-\dr{v}_{q_{i}}(\tau^-_{i})\right|}{\left|(\dr{v}_{q_{i}}(\tau^-_{i})-\dr{v}_{q_{i}}(\tau^-_{i}))\times \nu_i\right|^2}}$ which bounds the error that a collision creates. Hence, after $n_K$ collisions, summing over all the possible particles, we obtain the expected bound.
		
	\end{proof}
	
	\begin{lemma}
		Fix $\bar{n}:=(n_1,\cdots,n_K)$. For any $\e>0$ sufficiently small , we have
		\begin{align}\label{eq: borne R1}
		|\eqref{eq:R1}| \leq C(Ct)^{n_K-1}\mathbb{V}\e^{1/2} \|g\|_1\|h\|_1.
		\end{align}
	\end{lemma}
	\begin{proof}
		We have forbid any recollision, multiple \red{interaction}\blue{encounter}, and overlap. Hence, the velocities of pseudotrajectories of particles of sizes $\e$ and $0$ coincide. Using the inequality $|f(z)-f(z')|\leq( 1\wedge |z|)  \|f\|_1$,
		\blue{\begin{gather*}
		\begin{split}
		\eqref{eq:R1}\leq\frac{\mathfrak{d}^{-n_K+1}}{(n_K-1)!}\sum_{T}\int_{\mathfrak{T}_{\underline{n}_K}\times\mathfrak{G}^\e_{T}}2n_K\mathbb{V}\|h\|_1\|g\|_1\sum_{i=1}^{n_K-1}1\wedge&\tfrac{\e\left|\dr{v}_{q_{i}}(\tau^-_{i})-\dr{v}_{q_{i}}(\tau^-_{i})\right|}{\left|\left(\dr{v}_{q_{i}}(\tau^-_{i})-\dr{v}_{q_{i}}(\tau^-_{i})\right)\times \nu_i\right|^2}\\
		\times\Lambda&(V_{n_K},\nu_{[n_K-1]})\ud\nu_{[n_K-1]} \ud\tau_{[n_K-1]}\ud x_1M^{\otimes n_K}\ud V_{n_K}
		\end{split}\\
		\leq\frac{\left(\tfrac{Ct}{\mathfrak{d}}\right)^{-n_K+1}\mathbb{V}\|h\|_1\|g\|_1}{(n_K!)^2}\sum_{\substack{1\leq i\leq n_K-1\\[3pt]T}}\int_{\mathfrak{G}^\e_{T}}1\wedge\tfrac{\e\left|\dr{v}_{q_{i}}(\tau^-_{i})-\dr{v}_{q_{i}}(\tau^-_{i})\right|}{\left|\left(\dr{v}_{q_{i}}(\tau^-_{i})-\dr{v}_{q_{i}}(\tau^-_{i})\right)\times \nu_i\right|^2} \Lambda_T(V_{n_K},\nu_{[n_K-1]})\ud\nu_{[n_K-1]} M^{\otimes n_K}\ud V_{n_K}.
		\end{gather*}}
		
		We need to bound 
		\blue{\begin{equation}\label{eq:R4}
		\sum_{T}\int_{\mathfrak{G}^\e_{T}}1\wedge\tfrac{\e\left|\dr{v}_{q_{i}}(\tau^-_{i})-\dr{v}_{q_{i}}(\tau^-_{i})\right|}{\left|\left(\dr{v}_{q_{i}}(\tau^-_{i})-\dr{v}_{q_{i}}(\tau^-_{i})\right)\times \nu_i\right|^2} \Lambda_T(V_{n_K},\nu_{[n_K-1]})\ud\nu_{[n_K-1]} M^{\otimes n_K}\ud V_{n_K}
		\end{equation}}
		Note that $\dr{v}^0_q(\tau_i^+)$ does not depend on the $\tau_{[n_K-1]}$, but only on the order of the collisions.
		
		Fix a collision tree $T=(q_i,q'_i,s_i)_i$. We define for $i\in[1,n_K-1]$ the applications $\left(\Xi_T^i\right)_{1\leq i\leq n_K}$ as $\Xi_T^i = \text{id}$ if $i=1$, and
		\begin{equation}\label{eq:def des Xi}
		\Xi_T^i: (V_{n_K},\nu_{[n_K-1]})\mapsto \left\{\begin{split}
		&\Big(v_1,\cdots,\overset{q_{i-1}}{v'_{q_{i-1}}},\cdots,\overset{q'_{i-1}}{v'_{q'_{i-1}}},\cdots,v_{n_K},\nu_1,\cdots,\overset{i-1}{\nu'_{i-1}},\cdots,\nu_{n_K-1}\Big)~{\rm if}~s_i = 1\\
		&\Big(V_{n_K},\nu_1,\cdots,\overset{i-1}{-\nu_{i-1}},\cdots,\nu_{n_K-1}\Big)~{\rm if}~s_i = 1
		\end{split}\right.\end{equation}
		with the new velocities given by the scattering $(v'_{q_{i-1}},v'_{q'_{i-1}},\nu'_{i-1}):=\xi_\alpha((v'_{q_{i-1}},v'_{q'_{i-1}},\nu'_{i-1}))$.
		We have that 
		\[(\dr{V}_{n_K}(\tau_i^-),\nu^i_{[n_K-1]}):=\Xi_T^i\Xi_T^{i-1}\cdots\Xi_T^1 ({V}_{n_K},\nu_{[n_K-1]}).\]
		Using that the Jacobian of the scattering $\xi_\alpha$  is $1$ and the conservation by the scattering of the energy and angular momentum, the Jacobian of the transformation $\Xi_T^i\Xi_T^{i-1}\cdots\Xi_T^1$ is 
		\[\Lambda_T(V_{n_K},\nu_{[n_K-1]}) \ud\nu_{[n_K-1]}\ud x_1M^{\otimes n_K}\ud V_{n_K}\to\Lambda_T^{(i)}(V_{n_K},\nu_{[n_K-1]})\ud\nu_{[n_K-1]} \ud x_1M^{\otimes n_K}\ud V_{n_K}\]
		where we start now the velocity process at time $\tau_i^-$ with $\dr{V}_{n_K}(\tau_i^-) :=V_{n_K}$ and
		\[\Lambda_T^{(i)}(V_{n_K},\nu_{[n_K-1]}):=\prod_{j=1}^{i-1}\left(\left(\ds{v}_{q_j}(\tau_j^+)-\ds{v}_{\bar{q}_j}(\tau_j^+)\right)\cdot\nu_j\right)_+ \prod_{j=i}^{n_{k}-1}\left(\left(\ds{v}_{q_j}(\tau_j^-)-\ds{v}_{\bar{q}_j}(\tau_j^-)\right)\cdot\nu_j\right)_-.\]
		
		\blue{Hence, 
		\[\begin{split}
		\eqref{eq:R4} \leq \sum_{T}\int_{\mathfrak{G}^\e_{T}}1\wedge\frac{\e|v_{q_{i}}-{v}_{q_{i}}|}{\left|\left(v_{q_{i}}-{v}_{q_{i}}\right)\times\nu_i\right|^2}\Lambda_T^{(i)}(V_{n_K},\nu_{[n_K-1]})\ud\nu_{[n_K-1]} M^{\otimes n_K}\ud V_{n_K}.
		\end{split}\]
		Using the usual bound on $\sum_T \Lambda(V_{n_K},\ud\nu_{[n_K-1]})$ that can be adapted to $\Lambda_i$, 	and that for $\vec{e}_1\in\mathbb{S}^{d-1}$,  
		\[\int_{\mathbb{S}^{d-1}} 1\wedge \tfrac{\delta}{|\vec{e}_1\times\sigma|^2}d\sigma \lesssim \int_0^\pi\left( 1\wedge \tfrac{\delta}{|\sin \theta|^2}\right) \sin^{d-2}\theta d\theta \lesssim \delta^{1/2}, \]
		one has
		\begin{multline*}
		\eqref{eq:R4} \leq \sum_{(q,q')}\int(Cn_K|V_{n_K}|^2+1)^{n_K}1\wedge\tfrac{\e|v_{q}-{v}_{q'}|}{\left|\left(v_{q}-{v}_{q'}\right)\times\nu\right|^2} M^{\otimes n_K}\ud V_{n_K}\ud\nu\\
		\leq n_K^2(C'n^2_K)^{n_K}\int\frac{\e^{1/2}e^{-\frac{|V_{n_K}|^2}{4}}\ud{V_{n_K}}}{|v_1-v_2|^{1/2}}\leq\e^{1/2} (C''n^2_K)^{n_K}
		\end{multline*}
		as $\tfrac{1}{|v_1-v_2|^{1/2}}$ is an integrable singularity. This concludes the proof.}
	\end{proof}
	
	Finally, we get for any $h$ and $g$ Lipschitz
	\begin{multline*}
	\int\mu_\e^{n_K-1}\,\Psi^{0,t}_{\underline{n}_K}[h]\,g_{n_K}^\e M^{\otimes n_K}\ud Z_{n_K}	=\frac{\mathfrak{d}^{-n_K+1}}{(n_K-1)!}\sum_{T}\sigma(T)\int_{\mathfrak{T}_{\underline{n}_K}\times\mathfrak{G}^0_{T}}h(\ds{z}^0_{1}(t,T)g_{n_K}(\ds{Z}^0_{n_K}(0,T))\\
	\times\Lambda_T(V_{n_K},\nu_{[n_K-1]})\ud\nu_{[n_K-1]} \ud\tau_{[n_K-1]}~\ud x_1M^{\otimes n_K}\ud V_{n_K}\\
	+O\bigg(\e^{\mathfrak{a}} (\tfrac{Ct}{\mathfrak{d}})^{n_K}\|h\|_1\|g\|_1\bigg). \end{multline*}
	and therefore
	\begin{multline}\label{Estimation morceau 5}
	G_\e^{\text{main}}(t) =\sum_{\substack{n_1\leq\cdots\leq n_K\\n_j-n_{j-1}\leq 2^j}}\frac{\mathfrak{d}^{-n_K+1}}{(n_K-1)!}\sum_{T}\sigma(T)\int_{\mathfrak{T}_{\underline{n}_K}\times\mathfrak{G}^0_{T}}h(\ds{z}^0_{1}(t,T)g_{n_K}(\ds{Z}^0_{n_K}(0,T))\\[-0pt]
	\times\Lambda_T(V_{n_K},\nu_{[n_k-1]})\ud\nu_{[n_k-1]} \ud\tau_{[n_K-1]}~\ud x_1M^{\otimes n_K}\ud V_{n_K}
	+\,O\left(\e^{\mathfrak{a}}  (\tfrac{Ct}{\mathfrak{d}})^{2^{K+1}}\|h\|_1\|g\|_1\right).
	\end{multline}
	
	\subsection{Linearized Boltzmann equation}We identify now the main part of \eqref{Estimation morceau 5}.
	
	Let $\gr{g}_\alpha$ be the solution of the linearized Boltzmann equation 
	\[\begin{split}
	\partial_t\gr{g}_\alpha(t)+v\cdot\nabla_x\gr{g}_\alpha(t) &= \frac{1}{\mathfrak{d}}\mathcal{L}_\alpha \gr{g}_\alpha(t),\\
	\gr{g}_\alpha(t=0) &= g
	\end{split}\]
	where $\mathcal{L}_\alpha $ is the linearized Boltzmann operator associated to the potential $\alpha\mathcal{V}(\cdot)$
	\[\mathcal{L}_\alpha g(v):=\int_{\mathbb{R}^d\times\mathbb{S}^{d-1}} \big(g(v')+g(v_*')-g(v)-g(v_*)\big)M(v_*)\big((v-v_*)\cdot\nu\big)_+\ud\nu\,\ud v_*.\]
	
	This equation can be rewritten in the Duhamel form:
	\[\gr{g}_\alpha(t) = S(t)g+\frac{1}{\mathfrak{d}}\int_0^t S(t-\tau_1)\mathcal{L}_\alpha\gr{g}_\alpha(\tau_1)\ud\tau_1\]
	where $S(\tau)$ is the free transport 
	\[S(\tau)g(x,v)=g(x-tv,v).\]
	
	We iterate this formula, but we still want to cut the cases with too many collisions in a short time interval (as for the particle system). Let's define
	\[Q_{m,n}(\tau)[g] = \frac{1}{\mathfrak{d}^{m-n}}\int_0^\tau \ud\tau_{n}\int_0^{\tau_{n}}\cdots\int_0^{t_{m+2}}\ud\tau_{m+1} S(t-\tau_{n})\mathcal{L}_\alpha S(\tau_{n}-\tau_{n-1})\cdots\mathcal{L}_\alpha S(\tau_{m+1})g,\]
	and for $\underline{n}_k:=(n_1,\cdots,n_k)$ with $1\leq n_1\leq \cdots\leq n_k$,
	\[Q_{\underline{n}_k}(\tau)g=Q_{1,n_1}(\tfrac{\tau}{k})Q_{n_1,n_2}(\tfrac{\tau}{k})\cdots Q_{n_{k-1},n_k}(\tfrac{\tau}{k})[g].\]
	We have
	\begin{equation}\label{decompo g(t)}
	\begin{split}
	\gr{g}(t)&=\sum_{\substack{n_1\leq\cdots\leq n_K\\n_j-n_{j-1}\leq 2^j}} Q_{\underline{n}_k}(t) [g]+\sum_{k=1}^K \sum_{\substack{n_1\leq\cdots\leq n_{k-1}\\n_j-n_{j-1}\leq 2^j}}\sum_{n_k>+n_{k-1}+2^k}Q_{\underline{n}_k}(k\theta) [\gr{g}_\alpha(t-k\theta)].
	\end{split}
	\end{equation}
	
	In a first time, we bound the term of the sum: we have the classical estimates
	\begin{prop}
		There exists a constant $C$ such that for any $g\in L^2(M(v)dz)$, and $\underline{n}:=(n_1,\cdots,n_k)$,
		\begin{equation}
		\big\|Q_{\underline{n}}(k\theta)g \big\|_{L^2(M^2(v)dz)}\leq \big(\tfrac{C(k-1)\theta}{\mathfrak{d}}\big)^{\frac{n_{k-1}}{2}}\big(\tfrac{C\theta}{\mathfrak{d}}\big)^{\frac{n_k-n_{k-1}}{2}}\,\|g\|_{L^2(M(v)dz)}.
		\end{equation}
	\end{prop}
	The proof is the same as the one of Proposition 7.5 of \cite{LeBihan2}.
	
	Because $\gr{g}_\alpha(t)$ is bounded in $L^\infty_tL^2(M(v)dz)$ by $\|g\|_{L^2(M(v)dz)}\leq C\|g\|_0$, we can bound the rest term of \eqref{decompo g(t)} by
	\begin{equation}\label{eq: estimation du dernier reste}
	\begin{split}
		\Bigg| \sum_{k=1}^K \sum_{\substack{n_1\leq\cdots\leq n_{k-1}\\n_j-n_{j-1}\leq 2^j}}&\sum_{n_k>+n_{k-1}+2^k}\int h(z)Q_{\underline{n}_k}(k\theta) [\gr{g}_\alpha(t-k\theta)](z)M(v)\ud z\Bigg|\\
		&\leq \sum_{k=1}^K \sum_{\substack{n_1\leq\cdots\leq n_{k-1}\\n_j-n_{j-1}\leq 2^j}}\sum_{n_k>+n_{k-1}+2^k}\big(\tfrac{C(k-1)\theta}{\mathfrak{d}}\big)^{\frac{n_{k-1}}{2}}\big(\tfrac{C\theta}{\mathfrak{d}}\big)^{\frac{n_k-n_{k-1}}{2}}\,\|g\|_0\|h\|_0\\
		&\leq \sum_{k=1}^K \sum_{\substack{n_1\leq\cdots\leq n_{k-1}\\n_j-n_{j-1}\leq 2^j}}\sum_{n_k>2^k+n_{k-1}} \left(C\tfrac{\theta t}{\mathfrak{d}^2}\right)^{\frac{n_k-n_{k-1}}{2}}\|g\|_0\|h\|_0\\
		&\leq C \sum_{k=1}^K 2^{k^2}\left(C\tfrac{\theta t}{\mathfrak{d}^2}\right)^{2^{k-1}}\|g\|\|h\|\leq C\tfrac{\theta t}{\mathfrak{d}^2}\|g\|_0\|h\|_0.
	\end{split}
	\end{equation}
	The series converges since $\tfrac{\theta t}{\mathfrak{d}^2}<1$.

	The final step is the identification of the main part in \eqref{decompo g(t)}:
	\begin{prop}
		Fix $\underline{n}_K:=(n_1,\cdots,n_K)$ an increasing sequence of integer. Then 
		\begin{multline}
		\int_{\mathbb{D}} h(z)Q_{\underline{n}_k}(t)[g](z) M(v)dz = \frac{\mathfrak{d}^{-n_k+1}}{(n_k-1)!}\sum_{T}\sigma(T)\int_{\mathfrak{T}_{\underline{n}_K}\times\mathfrak{G}^0_{T}}h(\ds{z}^0_{1}(t,T)g_{n_k}(\ds{Z}^0_{n_k}(0,T))		\\
		\times\Lambda_T(V_{n_K},\nu_{[n_k-1]})\ud\nu_{[n_k-1]} \ud\tau_{[n_k-1]}\ud x_1M^{\otimes n_k}\ud V_{n_k}.
		\end{multline}
	\end{prop}
	\blue{\begin{remark}
			It is the counterpart of the Step 2 of the proof of Proposition IV.1 of \cite{BGSS5}. 
		\end{remark}
		\begin{proof}
		We fix for the moment the collision times $(\tau_i)_i$.
		
		\begin{definition}
			Fix a collision tree $T:=(q_i,\bar{q}_i,\bar{s}_i)_{i\leq n_k-1)}$ and a  final particle $q_f$. We say that a sequence $(i_1,\cdots,i_\ell)$ is \emph{causal} if
			\[i_1<\cdots< i_\ell,~\forall j<\ell,~\{q_{i_{j}},\bar{q}_{i_{j}}\}\cap\{q_{i_{j+1}},\bar{q}_{i_{j+1}}\}\neq\emptyset.\]
			A particle $\bar{q}$ \emph{influences} the particle $1$ (respectively $q_f$) if there exists a causal sequence $(i_1,\cdots,i_\ell)$ such that $\bar{q}\in\{q_{i_1},\bar{q}_{i_1}\}$ and $1\in\{q_{i_\ell},\bar{q}_{i_\ell}\}$ (respectively $q_f\in\{q_{i_1},\bar{q}_{i_1}\}$ and $\bar{q}\in\{q_{i_\ell},\bar{q}_{i_\ell}\}$).
		\end{definition}

		We begin by developing $g_{n_K}$:
		\begin{multline}\label{eq: somme sur tout les arbre}
		\sum_{T}\sigma(T)\int_{\mathbb{G}_T^0}h(\ds{z}^0_{1}(t,T))g_{n_k}(\ds{Z}^0_{n_k}(0,T))\Lambda_T(V_{n_K},\nu_{[n_k-1]})\ud\nu_{[n_k-1]} \ud\tau_{[n_k-1]}\ud x_1M^{\otimes n_k}\ud V_{n_k}\\
		=\sum_{q_f=1}^{n_k}\sum_{T}\sigma(T)\int_{\mathbb{G}_T^0}h(\ds{z}^0_{1}(t,T))g(\ds{z}^0_{q_f}(0,T))\Lambda_T(V_{n_K},\nu_{[n_k-1]})\ud\nu_{[n_k-1]} \ud\tau_{[n_k-1]}\ud x_1M^{\otimes n_k}\ud V_{n_k}.
		\end{multline}
		
		Suppose that there exists a particle $\bar{q}$ that does not influence $1$. If $\bar{q}$ has more than two neighbor particles, one of them is farther from the particle $1$ in the graph $T$ (the graph $T$ is a tree). Then, this new particle does not influence $1$. We deduce that without loss of generality, we can suppose that $q$ has only one neighbor. Then, it has only one collision $\bar{\iota}$.
		
		We now use the application $\Xi_T^i$ defined in \eqref{eq:def des Xi}. We recall that 
		\[\Xi_T^{\bar{\iota}}\Xi_T^{\bar{\iota}-1}\cdots\Xi_T^{1}(V_{n_K},\nu_{[n_K-1]}) = (\tilde{V}_{n_K}=\ds{V}_{n_K}(\tau_{\bar{\iota}}^-),\tilde{\nu}_{[n_K-1]}).\]
		
		In a second time, for a fixed $(V_{n_K},\nu_{[n_K-1]})$, we perform the translation $x_1\mapsto \tilde{x}_{1}:=\ds{x}_1(\tau_{\bar{\iota}})$.
		The Jacobian of $\tau_{\ds{x}_i(\tau_i)}\Xi_T^{\bar{\iota}}\Xi_T^{\bar{\iota}-1}\cdots\Xi_T^{1}$ is 
		\[\Lambda_T(V_{n_K},\nu_{[n_K-1]}) \ud\nu_{[n_K-1]}\ud x_1M^{\otimes n_K}\ud {V}_{n_K}\to\Lambda_T^{(i)}(\tilde{V}_{n_K},\nu_{[n_K-1]}) \ud\nu_{[n_K-1]}\ud \tilde{x}_1M^{\otimes n_K}\ud \tilde{V}_{n_K}.\]
		We start now the velocity process at time $\tau_i^-$ with $\dr{V}_{n_K}(\tau_i^-) :=\tilde{V}_{n_K}$ and
		\[\Lambda_T^{(i)}(\tilde{V}_{n_K},\nu_{[n_K-1]}):=\prod_{j=1}^{i-1}\left(\left(\ds{v}_{q_j}(\tau_j^+)-\ds{v}_{\bar{q}_j}(\tau_j^+)\right)\cdot\nu_j\right)_+ \prod_{j=i}^{n_{k}-1}\left(\left(\ds{v}_{q_j}(\tau_j^-)-\ds{v}_{\bar{q}_j}(\tau_j^-)\right)\cdot\nu_j\right)_-.\]
		
		We pair $T$ with the tree $\tilde{T}$ as 
		\[\tilde{T}:=\left\{\begin{aligne}{cl}
		(q_j,\bar{q}_j,\bar{s}_j) &~\text{for}~j\neq i\\
		(q_j,\bar{q}_j,-\bar{s}_j) &~\text{for}~j= i.
		\end{aligne}\right.\]
		Then $\sigma(T) = -\sigma(\tilde{T})$, and for same $(\tilde{V}_{n_K},\tau_{[n_K-1]},\nu_{[n_K-1]})$, we have $\ds{z}^0_{1}(t,T)=\ds{z}^0_{1}(t,\tilde{T})$ and $\ds{z}^0_{q}(0,T)=\ds{z}^0_{q}(0,\tilde{T})$. We have $\Lambda_T^{(i)}(V_{n_K},\nu_{[n_K-1]})= \Lambda_{\tilde{T}}^{(i)}(V_{n_K},\nu_{[n_K-1]})$. Thus 
		\begin{multline*}
		\sigma(T)\int_{\mathbb{G}_T^0}h(\ds{z}^0_{1}(t,T))g(\ds{z}^0_{q}(0,T))\Lambda(V_{n_K},\nu_{[n_K-1]}) \ud\nu_{[n_K-1]}\ud x_1M^{\otimes n_K}\ud V_{n_K}\\
		=-\sigma(\tilde{T})\int_{\mathbb{G}_{\tilde{T}}^0}h(\ds{z}^0_{1}(t,\tilde{T}))g(\ds{z}^0_{q}(0,\tilde{T}))\Lambda(V_{n_K},\nu_{[n_K-1]})\ud\nu_{[n_K-1]}\ud x_1M^{\otimes n_K}\ud V_{n_K}.
		\end{multline*}
		
		The same strategy can be reproduced if there exists some particle that does not influence the final particle $q_f$.	Hence, it remains only in \eqref{eq: somme sur tout les arbre} the trees such that every particle influences both $1$ and $q_f$. The other terms are exactly compensated. 
		
		For a remaining tree $T=(q_i,\bar{q}_i, \nu_i)_i$ we can prove the following  lemma:
		\begin{lemma}
			For all $i$, the set $\{q_i,\bar{q}_i\}\cap\{q_{i+1},\bar{q}_{i+1}\}$ has exactly one element.
		\end{lemma}
		\begin{proof}
			First, there exists a causal sequence $(q_{i_1},\bar{q}_{i_1}), \cdots, (q_{i_k},\bar{q}_{i_k})$ such that $q_f\in\{q_{i_1},\bar{q}_{i_1}\}$ and $1\in\{q_{i_k},\bar{q}_{i_k}\}$.
			
			Consider $q\notin \bigcup_l \{(q_{i_k},\bar{q}_{i_k})\}$. There exist two causal paths $(q_{j_1},\bar{q}_{j_1}), \cdots, (q_{j_\ell},\bar{q}_{j_\ell})$ from $q_f$ to $q$ and $(q_{j'_1},\bar{q}_{j'_1}), \cdots, (q_{j'_{\ell'}},\bar{q}_{j'_{\ell'}})$ from $q$ to $1$. We define $\ell_0$ such that $\forall l\leq \ell_0$, $(q_{j_{\ell-l+1}},\bar{q}_{j_{\ell-l+1}})=(q_{j'_l},\bar{q}_{j'_l})$. The sequence  $(q_{j'_1},\bar{q}_{j'_1}), \cdots, (q_{j'_{\ell_0}},\bar{q}_{j'_{\ell_0}})$ is both increasing and decreasing, thus $\ell_0 \leq 1$. The sequence  $(q_{j'_1},\bar{q}_{j'_1}), \cdots, (q_{j_{\ell-\ell_0}},\bar{q}_{j_{\ell-\ell_0}}),(q_{j'_{\ell_0+1}},\bar{q}_{j'_{\ell_0+1}}), \cdots, (q_{j'_{\ell'}},\bar{q}_{j'_{\ell'}})$ is a causal path from $q_f$ to $1$. We deduce that $\ell_0 = 1$. 
			
			Finally, for any $l\leq k-1$, $j \in ]i_l,i_{l+1}[$, one $\{q_j,\bar{q}_j\}\cap\{q_{i_l},\bar{q}_{i_l}\}\cap\{q_{i_{l+1}},\bar{q}_{i_{l+1}}\}$. The conclusion follows
		\end{proof}
		Using the lemma, we can construct the sequences $(\tilde{q}_i)_{0\leq i\leq n_K-1}$,  $(\tilde{q}'_i)_{1\leq i\leq n_K-1}$ and  $(\tilde{s}_i)_{1\leq i\leq n_K-1}$ by 
		\[\tilde{q}_0 := q_f,~\tilde{q}_{n_K-1}:=1,~\{\tilde{q}_i\}:=\{q_i,\bar{q}_i\}\cap\{q_{i+1},\bar{q}_{i+1}\},~\{\tilde{q}'_i\}:=\{q_i,\bar{q}_i\}\setminus\{\tilde{q}_i\}\]
		and 
		\[\tilde{s}_i:=\left\{\begin{split}
		&\,1~\text{if}~q_i=q_{i-1}\\
		&-1~\text{else}.
		\end{split}\right.\]
		
		The sequence $(\tilde{q}'_i)_i$ encodes the order in which particles collide. In addition, we can reconstruct $T$ for a given sequence $(\bar{s}_i,\tilde{s}_i,\tilde{q}_i')_i$.
		
		We reorder the particles such that $\tilde{q}'_i =n_K-i$ (there are $(n_K-1)!$ possibilities).
				
		Finally, we have to identify the four possible $(\bar{s}_i,\tilde{s}_i)_i$ with the four parts of $\mathcal{L}_\alpha$: $(1,1)$ with $g(v')$ (we follow  the same particle that is deviated by the collision), $(1,-1)$ with $-g(v)$ (we follow  the same particle that is not deviated by the collision), $(-1,1)$ with $g(v_*')$ and $(-1,-1)$ with $-g(v_*)$. There are $(n_K-1)!$ possible sequences $(\tilde{q}_i')_i$.
		
		We conclude that
		\begin{equation*}\begin{split}
		\frac{1}{(n_K-1)!}\sum_{T}\sigma(T)\int_{\mathbb{G}_T^0}h(&\ds{z}^0_{1}(t,T))g_{n_K}(\ds{Z}^0_{n_K}(0,T))\Lambda(V_{n_K},\nu_{[n_K-1]})\ud\nu_{[n_K-1]}\ud x_1M^{\otimes n_K}\ud V_{n_K}\\
		&=\int_{\mathbb{D}} h(z)S(t-\tau_{n_K-1})\mathcal{L}_\alpha S(\tau_{n_K-1}-\tau_{n_K-2})\cdots\mathcal{L}_\alpha S(\tau_{1})g(z)~M(v)\ud z.
		\end{split}\end{equation*}
		We obtain the expected result by integrating with respect to $(\tau_1,\cdots,\tau_{n_K-1})$.
	\end{proof}}
	
	Combining the preceding proposition and the estimations \eqref{eq: estimation du dernier reste} and \eqref{Estimation morceau 5}, we obtain:
	\begin{equation}\label{Estimation morceau 6}
	G^{\rm main}_\e(t)= \int_{\mathbb{D}}h(z)\gr{g}_\alpha(t,z)M(z)\ud z+O\left(\left(\tfrac{\theta t}{\mathfrak{d}^2}+\e^{\mathfrak{a}} K2^{K^2} (\tfrac{Ct}{\mathfrak{d}})^{2^{K+1}}\right)\|h\|_1\|g\|_1\right).
	\end{equation}

	\section{Estimation of non-pathological recollisions}\label{Estimation of the long range recollisions.}
	
	In the last two sections, we estimate the error terms where the pseudotrajectory can have a recollision. We begin with the case of non-pathological recollision.
	\begin{equation*}
	G^{\text{rec},1}_\e(t)=\sum_{\substack{0\leq k\leq K-1\\1\leq k'\leq K'}}\sum_{\substack{(n_j)_{j\leq k}\\0\leq n_j-n_{j-1}\leq 2^j}}\sum_{n_{k+2}\geq n_{k+2}\geq n_k}\mathbb{E}_\e\left[\frac{1}{\sqrt{\mu}}\sum_{\ui_{n_k}}\Psi^{>,t-t_s}_{\underline{n}_{k+1}}[h]\left(\gr{Z}_{\ui_{n_k}}(t_s)\right)\zeta_\e^0(g)\ind_{\Upsilon_\e}\right]
	\end{equation*}
	where $t-t_s = k\theta+k'\theta'$.
	
	\begin{prop}
		\blue{There exists a constant $\mathfrak{a}\in(0,1))$, depending only on the dimension, such that} for $\e$ small enough, 
		\begin{equation}\label{Estimation morceau 3}\begin{split}
		\left|G_\e^{\text{rec},1}(t)\right|	&\leq  \|g\|\|h\| \e^{{\mathfrak{a}}/2} (C't)^{2^{t/\theta}+d+9}.
		\end{split}\end{equation}
	\end{prop}

	It is sufficient to prove the two following estimations:
	
	\begin{prop}\label{prop:estimation avec reco}
		Fix $k\in\mathbb{N}$, $\underline{n}:=(n_1,\cdots,n_{k+2})\in\mathbb{N}^k$. Then fixing $x_1=0$ we have
		\begin{multline}
		\label{Estimation avec reco 1}
		\int\sup_{y\in{\mathbb{T}}}\big|{\Psi}^{>,t-t_s}_{\underline{n}_{k+2}}[h](\tr_y Z_{n_{k+2}})\big|\frac{e^{-\mathcal{H}_{n_{k+2}}}}{(2\pi)^{\frac{n_kd}{d}}}\ud V_{n_{k+2}}\ud X_{2,n_{k+2}}\\
		\leq \e^{\mathfrak{a}}\frac{\|h\|_0}{(\mu\mathfrak{d})^{n_{k+2}-1}}C^{n_{k+2}}\delta^2\theta^{(n_{k+2}-n_k-3)_+}t^{n_{k}+9+d} \e^{\mathfrak{a}},
		\end{multline}
		
		\noindent and, for $m\in[1,n_{k+2}]$,
		
		\begin{multline}\label{Estimation avec reco 2}	
		\int\sup_{y\in{\mathbb{T}}}\big|{\Psi}^{>,t-t_s}_{\underline{n}_{k+2}}[h]\circledast_l{\Psi}^{>,t-t_s}_{\underline{n}_{k+2}}[h](\tr_yZ_{2n_{k+2}-m})\big| M^{\otimes (2n_{k+2}-m)}\ud V_{2n_{n+2}-m}\ud X_{2,2n_{k+2}-m}\\
		\leq\frac{\mu^{m-1}}{n_{k+2}^m} \e^{\mathfrak{a}}\left(\frac{\|h\|_0}{(\mu\mathfrak{d})^{n_{k+2}-1}}C^{n_{k+2}}\right)^2\delta^2\theta^{(n_{k+2}-n_k-3)_+}t^{n_{k}+n_{k+2}+9+d}.
		\end{multline}
	\end{prop}
	
	Using these estimations and Corollary \ref{Corollaire utilisant la quasi orthogonalite}, 
	\[\begin{split}
	\Bigg|\mathbb{E}_\e\Bigg[&\mu^{-\frac12}\sum_{\ui_{n_{k+2}}}{\Psi}^{>,t-t_s}_{\underline{n}_{k+2}}[h]\left(\gr{Z}_{\underline{i}_{n_{k+2}}}(t_s)\right)\zeta^0_\e(g)\ind_{\Upsilon_\e}\Bigg]\Bigg|\\
	&\leq \|h\|_0 \|g\|_0\frac{C^{n_{k+2}}}{\mathfrak{d}^{n_{k+2}-1}}\left(\e^{\frac{1}{2}+{\mathfrak{a}}}\theta^{(n_{k+2}-n_{k}-3)_+}\delta^2t^{n_k+d+9}+\left(e^{{\mathfrak{a}}}\theta^{(n_{k+2}-n_{k}-3)_+}\delta^2t^{n_k+d+9+m}\right)^{\frac{1}{2}}\right)\\
	&\leq \|g\|_0\|h\|_0 \delta\e^{\frac{{\mathfrak{a}}}{2}}C^{n_{k+2}} (\tfrac{\theta}{\mathfrak{d}})^{\frac{(n_{k+2}-n_{k}-3)_+}{2}}(\tfrac{t}{\mathfrak{d}})^{\frac{n_{k+2}+n_k}{2}+d+9}\\[3pt]
	&\leq \|g\|_0\|h\|_0 \delta\e^{\frac{{\mathfrak{a}}}{2}}(\tfrac{Ct}{\mathfrak{d}})^{n_k+d+9} (\tfrac{Ct\theta}{\mathfrak{d}^2})^{\frac{(n_{k+2}-n_{k}-3)_+}{2}}.
	\end{split}\]
	
	Using that $\tfrac{Ct\theta}{\mathfrak{d}^2} < 1$ and $K'\delta= \theta\leq 1$, we can sum on $k$, $k'$ and $\underline{n}_{k+2}$
	\begin{equation*}\begin{split}
	\left|G_\e^{\text{rec},1}(t)\right|&\leq\sum_{\substack{1\leq k\leq K-1\\1\leq k'\leq K'}} \sum_{\substack{n_1\leq\cdots\leq n_k\\n_j-n_{j-1}\leq 2^j}} \sum_{\substack{n_{k+2}\geq n_{k+1}\geq n_k}}\|g\|_0\|h\|_0 \delta\e^{\frac{{\mathfrak{a}}}{2}}(\tfrac{Ct}{\mathfrak{d}})^{n_k+d+9} (\tfrac{Ct\theta}{\mathfrak{d}^2})^{\frac{(n_{k+2}-n_{k}-2)_+}{2}} \\
	&\leq \|g\|_0\|h\|_0 K'\delta\e^{\frac{{\mathfrak{a}}}{2}} K^{K^2}(\tfrac{Ct}{\mathfrak{d}})^{2^K+d+9}\\
	&\leq  \|g\|_0\|h\|_0 \e^{\frac{\mathfrak{a}}{2}} (\tfrac{C't}{2\mathfrak{d}})^{2^K+d+9}
	\end{split}\end{equation*}
	This concludes the proof of \eqref{Estimation morceau 3}.
	
	\begin{proof}[Proof of \eqref{Estimation avec reco 1}]
		We recall that the pseudotrajectory development takes the form
		\begin{equation*}
		{\Psi}_{\underline{n}_{k+2}}^{>,t-t_s}[h] := \frac{1}{(n_{k+2}-1)!}\sum_{\substack{ 1\in\omega\subset[n_{k+2}]\\|\omega|=n_{k+1}\\(s_i)_{i\leq n_{i+2}-1}}}\!\!\prod_{k=1}^{n_{l}-1}s_k\, h(\ds{z}_{q}(t,\cdot,\{q\},\omega,(s_i)_i))\ind_{\mathcal{R}^{>,t-t_s}_{\{q\},\omega,(s_i)_i}}\prod_{i=1}^{k}\ind_{\mathfrak{n}(t-i\theta)=n_i}.
		\end{equation*}
		
		Here $\mathcal{R}^{>,t-t_s}_{\{q\},\omega,(s_i)_i}$ is the set of initial configurations $Z_{n_{k+2}}$ such that the pseudotrajectory has
		\begin{itemize}
			\item $1$ the final particle at time $t-t_s$,
			\item $\omega$ the set of particles at time $\delta$,
			\item at least one recollision,
			\item no pathological recollision (thanks to the conditioning on $\mathcal{R}^{>,t-t_s}_{\{q\},\omega,(s_i)_i}$).
		\end{itemize}
		
		\begin{lemma}\label{Estimation données initial avec reco}
			There exists a constant ${\mathfrak{a}}\in(0,1)$ such that for any $\underline{n}$, $k'$ and $(s_i)_i$,
			\begin{equation}\label{eq: etimation pseudo avec reco}
			\int \ind_{\mathcal{R}^{>,t-t_s}_{\{q\},\omega,(s_i)_i}} M^{\otimes n_{k+2}}\ud X_{2,n_{k+2}}\ud V_{n_{k+2}}
			\leq C'\left (\frac{C'n_{k+2}}{\mu\mathfrak{d}}\right)^{n_{k+2}-1}\delta^{2}\theta^{(n_{k+2}-n_k-2)_+}t^{n_k+2d+4}\e^{\mathfrak{a}}.
			\end{equation}
		\end{lemma}
		
		\begin{proof}
			We may define the clustering tree $T^>$ as before, 
			by looking at collisions in temporal order and keeping only the clustering collisions.
			However, this will not be sufficient to characterize the initial data.
						
			Let $(\bar q,\bar{q}')$ (with $\bar q<\bar{q}'$) be the first two particles having a non-clustering collision, $\tau_{\rm cycle}$ the time of this collision, and $c\in [1,n_{k+2}-1]$ such that $\tau_{\rm cycle}$ lies between  the times of the $c$-th and the $(c+1)$-th clustering collision. The parameters $\bar{T}:=(T^>,(\bar q,\bar{q}',c))$ provide a partition of the set of initial data.
			

			We denote 
			\[\mathfrak{T}_{\underline{n}_{k+2}} := \left\{\begin{array}{c}
			(\tau_i)_{i\leq n_{k+2}-1},~\tau_i\leq\tau_{i+1},\\
			\forall j \leq  n_{k+2}-n_{k+1},~\tau_j\leq \delta\\
			\forall j \leq n_{k+2}-n_{k}[,~\tau_j\leq k'\delta\\
			\forall \ell\leq k,~j \leq n_{k+2}-n_{k+2-\ell-1},~\tau_j\leq k'\delta+(\ell+1)\theta
			\end{array}\right\}\]
			
			For a given initial data $Z_n$ and $\bar{T}:=(T^>,(\bar q,\bar{q}',c))$, we define $\tau_i$ as the time of the $i$-th clustering collision and $\nu_i:=(\dr{x}_{q_i}(\tau_i)-\dr{x}_{q_i}(\tau_i'))/\e$. We denote $\mathfrak{T}_{\underline{n}_{k+2}}\times\mathfrak{G}^{>,t-t_s,\bar{T}}_{\{q\},\omega,(s_i)_i}$ the image of the set of initial datum  
			\[\mathcal{R}^{>,t-t_s}_{\{q\},\omega,(s_i)_i}\cap\{T^> \text{ is the collision tree, first collision implies }(\bar{q},\bar{q}') \text{ during  }(\tau_c,\tau_{c+1})\}\]
			by the application $(X_{2,n_{k+2}},V_{n_{k+2}})\to(\tau_{[n_{k+2}-1]},\nu_{[n_{k+2}-1]},V_{n_{k+2}})$.	
			\begin{multline*}
			\int \ind_{\mathcal{R}^{>,t-t_s}_{\{q\},\omega,(s_i)_i}} \frac{e^{-\mathcal{H}_{n_{k+2}}}}{(2\pi)^{\frac{n_{k+2}d}{2}}}\ud X_{2,n_{k+2}}dV_{n_{k+2}}\\
			=\frac{1}{(\mu\mathfrak{d})^{n_{k+2}-1}}\sum_{\bar{T}}\int_{\mathfrak{T}_{\underline{n}_{k+2}}\times\mathfrak{G}^{>,t-t_s,\bar{T}}_{\{q\},\omega,(s_i)_i}}\prod_{i=1}^{n_{k+2}-1} |(\dr{v}^\e_{q_i}(\tau_i)-\dr{v}^\e_{q'_i}(\tau_i))\cdot\nu_i| \ud{\nu_i}\ud{\tau_i} M^{\otimes n_{k+2}}dV_{n_{k+2}}.
			\end{multline*}			
			
		\red{	If the first recollision involves particles $q$ and $q'$ at time $\tau_{\rm rec}\in]\tau,{c},\tau_{c+1}[$, we consider $\omega\subset[n_{k+2}]$ the connected components of $\{q,q'\}$ in the collision graph on time interval $[0,\tau_{\rm rec})$ (it only depends on $c$). As before the first recollision, the pseudotrajectory $\dr{Z}_\omega^\e(\tau)$ and its formal limit $\dr{Z}_\omega(\tau)$ are closed up to a translation (thanks to Lemma \ref{Borne entre la trajectoire limite et la trajectoire d'enskog}): there exists a $y_0\in\mathbb{T}$ such that 
			\[\forall \tau\in[0,\tau_{\rm rec}],~|\dr{X}^0_\omega(\tau)-\tr_{y_0}\dr{X}^\e_\omega(\tau)|\leq \sum_{i=1}^{n_{k+2}-1}\frac{2n_{k+2}\mathbb{V}\e}{\left|\dr{v}_{q_{i}}(\tau^-_{i})-\dr{v}_{q_{i}}(\tau^-_{i})\right|}.\]
			Hence, if there is a recollision, 
			\begin{equation}\label{eq:recollision def alternative1}|\dr{x}^0_q(\tau_{\rm rec})-\dr{x}^0_{q'}(\tau_{\rm rec})|\leq \e+\sum_{i=1}^{n_{k+2}-1}\frac{2n_{k+2}\mathbb{V}\e}{\left|\dr{v}_{q_{i}}(\tau^-_{i})-\dr{v}_{q_{i}}(\tau^-_{i})\right|}.\end{equation}
			We can only study the limiting flow and define a recollision as "there exists a time $\tau_{\rm rec}$ such that \eqref{eq:recollision def alternative1} is verified: we have 
			\[\mathfrak{T}_{\underline{n}_{k+2}}\times\mathfrak{G}^{>,t-t_s,T}_{\{q\},\omega,(s_i)_i}\subset \mathfrak{T}_{\underline{n}_{k+2}}\times\mathfrak{G}^{0}_{T}\setminus\mathfrak{T}_{\underline{n}_{k+2}}\times\mathfrak{G}^{\e}_{T}.\]}
			
			Using Lemma \ref{lem:estimation des reco}, we get
			\begin{equation}\label{eq:estimation possibilite de reco}\begin{split}
			\int \ind_{\mathcal{R}^{>,t-t_s}_{\{q\},\omega,(s_i)_i}}\frac{e^{-\mathcal{H}_{n_{k+2}}}}{(2\pi)^{\frac{n_{k+2}d}{2}}}\ud X_{2,n_{k+2}}dV_{n_{k+2}}	\leq&\frac{(Cn_{k+2})^{n_{k+2}}}{(\mu\mathfrak{d})^{n_{k+2}-1}}t^{n_k-1}\theta^{(n_{k+2}-n_k-1)_+}\e^{1/4}\\
			\leq&\frac{(Cn_{k+2})^{n_{k+2}}}{(\mu\mathfrak{d})^{n_{k+2}-1}}t^{n_k-1}\theta^{(n_{k+2}-n_k-1)_+}\delta^2\e^{1/12}
			\end{split}\end{equation}
			using that $\delta = \e^{1/12}$.
			
			\begin{figure}[h!]
				\centering
				\includegraphics[width=10cm]{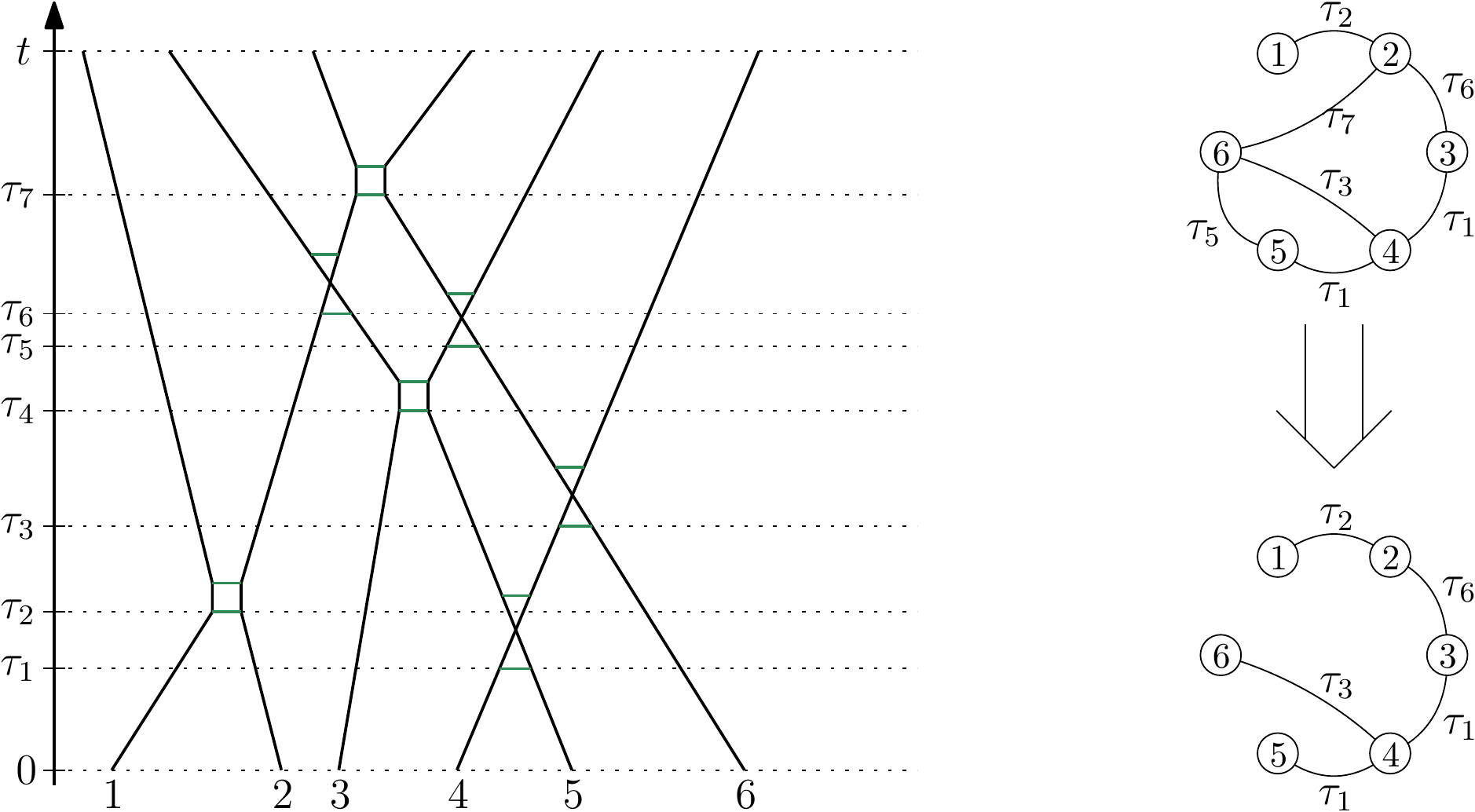}
				\caption{Example of construction of a clustering tree.}
			\end{figure}
		\end{proof}
		
		We obtain the expected result by summing on
		\[(s_i)_{i\leq n_{k+2}-1}\in \{\pm 1\}^{n_{k+2}-1},~\omega\subset[n_{k+2}],~q\in\omega\]
		and dividing by $n_{k+2}!$\,.
	\end{proof}
	
	\begin{proof}[Proof of \eqref{Estimation avec reco 2}]
		We use first the same bound as in the proof of \eqref{Estimation avec reco 1} and of \eqref{Estimation sans reco 2},
		\begin{multline}
		\left|{\Psi}_{\underline{n}_{k+2}}^{>,t-t_s}[h]\circledast_m{\Psi}_{\underline{n}_{k+2}}^{>,t-t_s}[h](Z_{2n_{k+2}-m})\right|\\
		\leq \frac{\|h\|^2}{(n_{k+2}!)^2}\frac{(n_{k+2}-m)!^2m!}{n_{k+2}^2(2n_{k+2}-m)!}\sum_{\substack{\bar{\omega}\cup\bar{\omega}' =[2n_{k+2}-m]\\|\bar{\omega}|=|\bar{\omega}'|=n_{k+2}}} \sum_{\substack{q\in\omega\subset\bar{\omega}\\ |\omega|=n_{k+1}\\(s_i)_{i\leq n_{i+2}-1}}}\sum_{\substack{q'\in\omega'\subset\bar{\omega}'\\ |\omega'|=n_{k+1}\\(s'_i)_{i\leq n_{i+2}-1}}} \!\!\ind_{\mathcal{R}^{>,t-t_s}_{\{q\},\omega,(s_i)_i}}\!\!\!(Z_{\bar{\omega}})~~\\
		\times\ind_{\mathfrak{n}(k'\delta)=n_k}\ind_{\mathcal{R}^{>,t-t_s}_{\{q'\},\omega'			,(s'_i)_i}}\!\!\!(Z_{\bar{\omega}'}).
		\end{multline}
		where $\mathfrak{n}(\theta)$ is the number of particles at time $\theta$ in the pseudotrajectory $\dr{Z}(\tau)$.
		Note that the formula is invariant under translation. We can then set $x_1=0$ and integrate with respect to the other variables. 
		
		Using the same strategy as in the proof of \eqref{Estimation sans reco 2}, we have
		\begin{equation*}
		\int\ind_{\mathcal{R}^{>,t-t_s}_{\{q'\},\omega'			,(s'_i)_i}}\!\!\!(Z_{\bar{\omega}'}) e^{-\frac12\mathcal{H}_{2n_{k+2}-m}}\ud Z_{\bar{\omega}'\setminus\bar{\omega}}\leq C\left(\frac{C}{\mu\mathfrak{d}}\right)^{n_k-m} t^{n_{k+2}-m}(2n_{k+2}-m)^{n_{k+2}-m}.
		\end{equation*}
		
		The sum over the remaining particles is estimated using \eqref{eq:estimation possibilite de reco}
		\begin{multline*}
		\int \ind_{\mathcal{R}^{>,t-t_s}_{\{q'\},\omega'			,(s'_i)_i}}\!\!\!(Z_{\bar{\omega}'})\ind_{\mathcal{R}^{>,t-t_s}_{\{q\},\omega,(s_i)_i}}(Z_\omega) \frac{e^{-\mathcal{H}_{2n_{k+2}-m}}}{(2\pi)^{\frac{(2n_{k+2}-m)d}{2}}}\ud X_{1,2n_{k+2}-m}dV_{2n_{k+2}-m}	\\
		\leq\frac{(C2n_{k+2})^{2n_{k+2}-m}}{(\mu\mathfrak{d})^{2n_{k+2}-m-1}}\delta^2\theta^{(n_{k+2}-n_k-3)_+}t^{n_{k}+n_{k+2}+9+d}\e^{\mathfrak{a}}.
		\end{multline*}
		
		We obtain the expected result by combining the two estimations, summing on the possible parameters $((s_i)_i,\bar{\omega},\omega,q)$ and $((s'_i)_i,\bar{\omega}',\omega',q')$ and then dividing by $(n_{k+2})!^2$.
	\end{proof}
	
	\section{Estimation of the local recollisions}\label{$L^2$ estimation of the local recollision part}
		
	In the present section, we discuss $G_\e^{\text{rec},2}(t)$ defined by
	\begin{multline*}
	G_\e^{\text{rec},2}(t)=\sum_{\substack{0\leq k\leq K-1\\1\leq k'\leq K'}}\sum_{\substack{(n_j)_{j\leq k}\\0\leq n_j-n_{j-1}\leq 2^j}}\sum_{n_{k+2}\geq n_{k+1}\geq n_k}\mathbb{E}_\e\Bigg[\frac{1}{\sqrt{\mu}}\zeta_\e^0(g)\ind_{\Upsilon_\e}\\[-10pt]
	\times\sum_{\ui_{n_{k+2}}}\Phi^{>,\delta}_{n_{k+1},n_{k+2}}\Psi^{,t-t_s-\delta}_{\underline{n}_{k+1}}[h]\left(\gr{Z}_{\ui_{n_{k+2}}}(t_s)\right)\Bigg]
	\end{multline*} 
	for $t_s:=t-k\theta-k'\delta$.

	We will prove the following bound:
	\begin{prop}
	\blue{There exists a constant $\mathfrak{a}\in(0,1))$, depending only on the dimension, such that} for $\e>0$ small enough, we have
	\begin{equation}\label{Estimation morceau 4}
	\Big|G_\e^{\text{rec},2}(t)\Big|\leq C\|h\|_0 \|g\|_0  (C\tfrac{t}{\mathfrak{d}})^{2^{K+1}}\e^{\frac{{\mathfrak{a}}}{2}}.
	\end{equation}
	\end{prop}

	In the following, we denote
	\[\begin{split}
		\bar{\Phi}_{\underline{n}_{k+2}}^{k'}(Z_{n_{k+2}}):=\frac{1}{(n_{k+2})!}\sum_{\sigma\in\mathfrak{S}_{n_{k+2}}}\Phi^{>,\delta}_{n_{k+1},n_{k+2}}\Psi^{,t-t_s-\delta}_{\underline{n}_{k+1}}[h]\left(Z_{\sigma[n_{k+2}]}\right)
	\end{split}\]
	
	The aim of this part is to prove the following bound on $\Phi_{\underline{n}_{k+2},p}^{k'}$:
	\begin{prop}
		Fix $n_1\leq\cdots\leq n_{k+2}\leq p$. For $m\in \{1,\cdots p\}$ we have for $x_1=0$ 
		\begin{multline}\label{Estimation reco loc 1}
		\int \sup_{y\in{\mathbb{T}}}\big|\bar{\Phi}_{\underline{n}_{k+2}}^{k'}(\tr_yZ_{n_{k+2}})\big|\frac{e^{-\mathcal{H}_{n_{k+2}}(Z_{n_{k+2}})}}{(2\pi)^{\frac{dn_{k+2}}{2}}}\ud Z_{2,n_{k+2}}\ud v_1\\
		\leq \frac{\|h\|_0}{(\mu\mathfrak{d})^{n_{k+2}-1}} C^{n_{k+2}}\delta^2\e^{\mathfrak{a}} \theta^{(n_{k+2}-n_k-2)_+}t^{n_k-1},
		\end{multline}
		\begin{multline}\label{Estimation reco loc 2}
		\int\sup_{y\in{\mathbb{T}}} \big|\bar{\Phi}_{\underline{n}_{k+2}}^{k'}\circledast_m\bar{\Phi}_{\underline{n}_{k+2}}^{k'}(\tr_yZ_{2n_{k+2}-m})\frac{e^{-\mathcal{H}_{2n_{k+2}-m}}}{(2\pi)^{\frac{d(2n_{k+2}-m)}{2}}}\ud Z_{2,2n_{k+2}-m}\ud v_{1}	\\
		\leq \frac{\mu^{m-1}}{n_{k+2}^m}\left(\frac{\|h\|_0}{(\mu\mathfrak{d})^{n_{k+2}-1}} C^{n_{k+2}}\right)^2\delta^2\e^{\mathfrak{a}} \theta^{(n_{k+2}-n_k-2)_+}t^{n_k-1+n_{k+2}}.
		\end{multline}
	\end{prop}
	
		Using the estimations \eqref{Estimation reco loc 1} and \eqref{Estimation reco loc 2}, one obtains
	\begin{equation*}\everymath={\displaystyle}\begin{array}{r@{}r@{}r}
	\Bigg|\mathbb{E}_\e\Bigg[\mu^{-\frac{1}{2}}\sum_{\ui_{n_{k+2}}}\Phi_{\underline{n}_{k+2}}^{k'}(\gr{Z}_{\ui_{n_{k+2}}}(t_s))\,\zeta^0_\e(g)\ind_{\Upsilon_\e}\Bigg]\Bigg)\Bigg|&\multicolumn{2}{l}{\leq \|g\|_0 \|h\|_0 C^{n_{k+2}}\Big((\tfrac{\delta}{\mathfrak{d}})^2 (\tfrac{\theta}{\mathfrak{d}})^{(n_{k+2}-n_k-2)_+}(\tfrac{t}{\mathfrak{d}})^{n_k-1}\e^{\frac{2{\mathfrak{a}}+1}2}}\\[-8pt]
	&&+\Big((\tfrac{\delta}{\mathfrak{d}})^2 (\tfrac{\theta}{\mathfrak{d}})^{(n_{k+2}-n_k-2)_+}(\tfrac{t}{\mathfrak{d}})^{n_{k+2}+n_k-1}\e^{{\mathfrak{a}}}\Big)^{\frac12}\Big)\\[8pt]
	&\multicolumn{2}{l}{\leq\delta\e^{\frac{{\mathfrak{a}}}{2}}\|h\|_0\|g\|_0 C^{n_{k+2}} (\tfrac{t^2\theta}{\mathfrak{d}^3})^{\frac{(n_{k+2}-n_k-2)_+}{2}}(\tfrac{t}{\mathfrak{d}})^{n_k},}
	\end{array}\end{equation*}
	as $\e^{\frac{1}{2}}/\mathfrak{d}\to0$.
	
	Using that $\tfrac{Ct\theta}{\mathfrak{d}^2}\leq 1$ and $K'\delta= \theta\leq 1$, we can sum on $k$, $k'$ and $\underline{n}_{k+2}$
	\begin{equation*}\begin{split}
	\left|G_\e^{\text{rec},2}(t)\right|&\leq\sum_{\substack{1\leq k\leq K-1\\1\leq k'\leq K'}} \sum_{\substack{n_1\leq\cdots\leq n_k\\n_j-n_{j-1}\leq 2^j}} \sum_{\substack{n_{k+2}\geq n_{k+1}\geq n_k}}\|g\|_0\|h\|_0 \delta\e^{\frac{{\mathfrak{a}}}{2}}(\tfrac{Ct}{\mathfrak{d}})^{n_k} (\tfrac{Ct\theta}{\mathfrak{d}^2})^{\frac{(n_{k+2}-n_{k}-2)_+}{2}} \\
	&\leq \|g\|_0\|h\|_0 K'\delta\e^{\frac{{\mathfrak{a}}}{2}} K^{K^2}(\tfrac{Ct}{\mathfrak{d}})^{2^K}\\
	&\leq  \|g\|_0\|h\|_0 \e^{\frac{\mathfrak{a}}{2}} (\tfrac{C't}{2\mathfrak{d}})^{2^K}
	\end{split}\end{equation*}
	This concludes the proof of \eqref{Estimation morceau 4}.
	
	\begin{proof}[Proof of \eqref{Estimation reco loc 1}]
	We recall that
	\[\begin{split}
	\bar{\Phi}_{\underline{n}_{k+2}}^{k'}(Z_{n_{k+2}}):=\frac{1}{(n_{k+2})!}\sum_{\sigma\in\mathfrak{S}_{n_{k+2}}}\Phi^{>,\delta}_{n_{k+1},n_{k+2}}\Psi^{,t-t_s-\delta}_{\underline{n}_{k+1}}[h]\left(Z_{\sigma[n_{k+2}]}\right)
	\end{split}\]
	
	In $\Psi^{0,t-t_s-\delta}_{\underline{n}_{k+1}}\Phi^{0,\delta}_{n_{k+1},n_{k+2}}[h]\left(Z_{[n_{k+2}]}\right)$ we see three sets of indices:
	\begin{itemize}
		\item $1$ the last particle,
		\item $[1,n_{k+1}]$ the set of particles in "final" tree pseudotrajectory development,
		\item $[n_{k+1}+1,n_{k+2}]$ the particles added in the first time interval.
	\end{itemize}
	Any permutation $\sigma$ that sends $[1,n_{k+1}]$ and $[n_{k+1}+1,n_{k+2}]$ onto themselves stabilizes the function $\Psi^{0,t-t_s-\delta}_{\underline{n}_{k+1}}\Phi^{0,\delta}_{n_{k+1},n_{k+2}}[h]$. Hence, $\bar{\Phi}_{\underline{n}_{k+2},p}^{k'}(Z_p)$ is equal to 
	\[\begin{split}
	\frac{(n_{k+1}-1)!(n_{k+2}-n_{k+1})!}{(n_{k+2})!} \sum_{\substack{\omega_1\sqcup\omega_2=[n_{k+2}]\\|\omega_1|=n_{k+1}\\q_1\in\omega_1}}\Phi^{>,\delta}_{n_{k+1},n_{k+2}}\Psi^{,t-t_s-\delta}_{\underline{n}_{k+1}}[h]\left(z_{q_1},Z_{\omega_1\setminus\{q_1\}},Z_{\omega_2}\right).
	\end{split}\]
	Let us develop $\Psi^{0,t-t_s-\delta}_{\underline{n}_{k+1}}\Phi^{0,\delta}_{n_{k+1},n_{k+2}}[h]$.
	For $(s_i)_i\in\{\pm1\}^{n_{k+1}-1}$, $(\omega_1,\omega_2)$ a partition of $[n_{k+2}]$ and $(\lambda_1,\cdots,\lambda_{\gr{l}})$ a partition of $[n_{k+2}]$ with $\omega_1\subset\lambda_1$, we define the pseudotrajectory $\bar{\ds{Z}}(\tau,Z_{n_{k+2}}, q_1,\omega_1,\omega_2,$ $(s_i)_{i},(\lambda_j)_j)$ by
	\begin{itemize}
		\item for $\tau\leq \delta$, 
		\[\bar{\ds{Z}}(\tau) := \ds{Z}(\tau,Z_{\omega_1},(\lambda_{j})_j)\]
		\item for $\tau>\delta$, the particle of $\omega_3$ are removed and 
		\[\bar{\ds{Z}}_{\omega_1}(\tau) := \ds{Z}(\tau-\delta,\bar{\ds{Z}}_{\omega_1}(\delta),\{q_1\},(s_i)_i).\]
	\end{itemize}
	Then $\bar{\Phi}_{\underline{n}_{k+2}}^{k'}(Z_{n_{k+2}})$ is equal to 
	\begin{multline*} \frac{1}{(n_{k+2})!}~\sum_{\substack{\omega_1\sqcup\omega_2=[n_{k+2}]\\|\omega_1|=n_{k+1}\\q_1\in\omega_1}}\sum_{(s_i)_i}\Bigg(\sum_{{\gr{l}}=1}^n \sum_{\substack{\ \lambda_1\supset\omega_1}}\sum_{\substack{(\lambda_2,\cdots,\lambda_\gr{l})\\ \in\mathcal{P}^{\gr{l}-1}_{ \omega_2\setminus\lambda_1}}}\Bigg)h\left(\bar{\ds{z}}_{\omega_1}(k\theta+k'\delta)\right)\\
	\times\ind_{\substack{\bar{\ds{Z}}(\cdot)\,{\rm has\,a}\\{\rm pathology}\\{\rm on\,}[0,\delta]}}\Bigg(\prod_{i=1}^{n_{k+2}-1}\!\!\!{s}_i~\ind_{\mathcal{R}^{0,k\theta+(k-1)\delta}_{\omega_1,(s_i)_i}}(\bar{\ds{Z}}_{\omega_2}(\delta)) \prod_{i=1}^{k}\ind_{\mathfrak{n}(t-i\theta)=n_i}\Bigg)\\
	\times\Bigg(\CCirc_l(Z_{\lambda_1},\cdots,Z_{\lambda_{\gr{l}}}) \DDelta_{|\lambda_1|}^{[m]}(Z_{\lambda_1})\prod_{i=2}^{\gr{l}}\DDelta_{|\lambda_i|}(Z_{\lambda_i}) \Bigg).
	\end{multline*}
	The functions $\CCirc$ and $\DDelta$ are defined in Definitions \ref{def:DDelta} and \ref{def:CCirc} (in the definition, they are defined on a time interval $[0,t]$; here, they are defined on $[0,\delta]$).
	
	The function $\CCirc_l(Z_{\lambda_1},\cdots,Z_{\lambda_{\gr{l}}})$ can be bounded by the Penrose's tree inequality (see for example \cite{BGSS,Jansen}),
	\[\left|\CCirc_l(Z_{\lambda_1},\cdots,Z_{\lambda_{\gr{l}}})\right|= \left|\sum_{C\in\mathcal{C}([\gr{l}])}\prod_{(i,j)\in E(C)} -\ind_{{\lambda}_i\so {\lambda}_j}\right| \leq \sum_{T\in\mathcal{T}([\gr{l}])}\prod_{(i,j)\in E(T)} \ind_{{\lambda}_i\so {\lambda}_j}.\]
	The set $\mathcal{T}([\gr{l}])$ is the set of minimally connected graphs with vertices $[\gr{l}]$.
	
	We explain now how to take advantage of the pathology of $\bar{\ds{Z}}(\cdot)$.
	
	\begin{definition}\label{def: O_pi}
		For $r\geq 3$, we define  the set $\mathcal{O}_r$ as 
		\begin{multline}\mathcal{O}_r :=\Big\{Z_r\in \mathbb{D}^r,~\exists (\varpi_1,\cdots,\varpi_{\gr{l}}),~{\rm the~collision~graph~of~}\ds{Z}_r(\cdot,Z_z,(\varpi_1,\cdots,\varpi_{\gr{l}}))\text{~on~}[0,\delta]~{\rm is~}\\
		{\rm connected,~and~the~pseudotrajectory~has~a~pathology}\Big\}.\end{multline}
		We recall that a pathology can be an overlap, a multiple \blue{encounter}, or a recollision (see Definition \ref{def:pathology}).
		
		For $r=2$, we define 
		\begin{equation}
		\mathcal{O}_2:=\{|x_1-x_2|\leq \e\}\cup\{|(x_1-x_2)+\delta (v_1-v_2)|\leq\e\}.
		\end{equation}
		
		Finally, for $\varpi\subset[n_{k+2}]$, the set $\mathcal{O}_{\varpi}$ is defined as
		\begin{equation}
		\mathcal{O}_{\varpi}:=\left\{Z_{n_{k+2}}\in\mathbb{D}^{n_{k+2}},~Z_\varpi \in \mathcal{O}_{|\varpi|}\right\}.
		\end{equation}
	\end{definition}
	The $\mathcal{O}_\varpi$ allows to control the recollision condition 
	\[\ind_{\substack{\bar{\ds{Z}}(\cdot)\,{\rm has\,a}\\{\rm pathology}\\{\rm on\,}[0,\delta]}}\leq \sum_{\varpi\subset[n_{k+2}]}\ind_{\mathcal{O}_{\varpi}}. \]

	This leads to the following bound on $\left|\bar{\Phi}_{\underline{n}_{k+2}}^{k'}(Z_{n_{k+2}})\right|$ :
	\begin{multline}\label{borne 1 sur phi l} \frac{\|h\|_0}{(n_{k+2})!}\sum_{\substack{\omega_1\sqcup\omega_2=[n_{k+2}]\\|\omega_1|=n_{k+1}\\q_1\in\omega_1}}\sum_{\substack{\varpi\subset[n_{k+2}]\\ |\varpi|\geq 2}}\sum_{(s_i)_i}\Bigg(\sum_{{\gr{l}}=1}^n \sum_{\substack{\ \lambda_1\supset\omega_1}}\sum_{\substack{(\lambda_2,\cdots,\lambda_\gr{l})\\ \in\mathcal{P}^{\gr{l}-1}_{ \omega_2\setminus\lambda_1}}}\ind_{\mathcal{R}^{0,k\theta+(k-1)\delta}_{\{q_1\},(s_i)_i}}(\bar{\ds{Z}}_{\omega_1}(\delta)) \ind_{\mathfrak{n}(k'\delta)=n_{k}}\\[-10pt]
	\times\ind_{\mathcal{O}_{\varpi}}\sum_{T\in\mathcal{T}([\gr{l}])}\prod_{(i,j)\in E(T)} \ind_{{\lambda}_i\so {\lambda}_j} \DDelta_{|\lambda_1|}^{[n_{k+1}]}(Z_{\lambda_1})\prod_{i=2}^{\gr{l}}\DDelta_{|\lambda_i|}(Z_{\lambda_i}) \Bigg).
	\end{multline}
	
	Note that the right hand-side is invariant under translation. Thus, one can fix $x_1 = 0$ and integrate with respect to the other variables.
	
	We introduce a partition to control the pseudo-trajectory in the time interval $[0,\delta]$.
	
	\begin{definition}[Possible cluster partition]
		Given $Z_{n_{k+2}}\in\mathbb{D}^{n_{k+2}}$, we construct the graph $G$ with vertices $[n_{k+2}]$. The pair $(i,j)$ is an edge of $G$ if and only if there exists $\tilde{\omega}\subset [n_{k+2}]$ and $(\tilde{\lambda}_1,\cdots,\tilde{\lambda}_\ell)$ a partition of $\tilde{\omega}$ such that the collision graph of $\dr{Z}(\cdot,Z_{\tilde{\omega}},\tilde{\lambda}_1,\cdots,\tilde{\lambda}_\ell)$ on time interval $[0,\delta]$ is connected. We introduce $\underline{\rho}:=(\rho_1,\cdots,\rho_{\gr{r}})$ the \emph{possible cluster partition} as the set of the connected components of $G$.
		
		We define $\mathcal{D}^{\underline{\rho}}_\e\subset\mathbb{D}^{n_{k+2}}$ as the set such that $\underline{\rho}$ is the possible cluster partition. The $\big(\mathcal{D}^{\underline{\rho}}_\e)_{\underline{\rho}}$ form a partition of $\mathbb{D}^{n_{k+2}}$. 
	\end{definition}
		
	By definition of the potential cluster, a particle cannot interact with a particle of an other cluster for any time in $[0,\delta]$. Thus the systems $\rho^i$ are isolated on $[0,\delta]$ and all the dynamics in $[0,\delta]$ are encoded inside the $(\rho_i)$. 
	
	The parametrization of the pseudotrajectories is changed to a more adapted one. There exists a $\rho_i$ containing $\varpi$. With a little loss of symmetry, one can suppose that it is $\rho_1$. In the same way, for any $\lambda_j$ with $j\neq 1$ there exists some $\rho_i$ containing $\lambda_j$. For any $\rho_i$
	\begin{itemize}
		\item $\underline{\omega}^i := (\omega_1^i,\omega_2^i)$ the partition of $\rho_i$ defined by $\omega^i_j:=\omega_j\cap\rho_i$, note that the set $\omega^i_1$ cannot be empty,
		\item $\underline{\lambda}^{i}:=\{\lambda^i_1:=\lambda_1\cap\rho_i\}\cup\{\lambda_j \text{~for~} j\geq 2 \text{~with~}\lambda_j\subset \rho_i\}$ a partition of $\rho_i$,
		\item  for $i\geq 1$, $\mathfrak{p}_i:=(\underline{\omega}^i,\underline{\lambda}^i)$,
		\item $\mathfrak{p}_1:=(\underline{\omega}^1,\underline{\lambda}^1,\varpi)$.
	\end{itemize}
	
	The set of possible $\mathfrak{p}_i$ is denoted $\mathfrak{P}(\rho_i)$. Because $\rho_i$ is of size at most $\gamma$, there exists a constant $C_\gamma$ depending only on $\gamma$ such that $|\mathfrak{P}(\rho_i)|\leq C_\gamma$. For a fix partition $\underline{\rho}$, the map $(\underline{\omega},{\varpi},\underline{\lambda})\mapsto (\mathfrak{p}_i)_i$ is onto. 
	
	The possible cluster partition also contains the overlap: if we denote two dynamical clusters $\lambda_j$ and $\lambda_{j'}$ with $j,j'\geq 2$, there exists a $\rho_i$ containing  both, and if $\lambda_j\subset\rho_i$ has an overlap with $\lambda_1$, then $\lambda_j$ has an overlap with $\lambda_1^i$. This last property allows us to rewrite the overlap cumulant: for any $Z_{n_{k+2}}$ in  $\mathcal{D}^{\underline{\rho}}_\e$,
	\begin{equation*}
	\Big|\psi_{\gr{l}}\big(Z_{\lambda_1},\cdots,Z_{\lambda_{\gr{l}}}\big)\Big|\leq\sum_{T\in\mathcal{T}([\gr{l}])}\prod_{(i,j)\in E(T)}\ind_{\lambda_i\so\lambda_j}
	\leq \prod_{i= 1}^{\gr{r}}\sum_{T_i\in\mathcal{T}([|\underline{\lambda^i}|])}\prod_{(j,j')\in E(T_i)}\ind_{\lambda_j^i\so\lambda^i_{j'}}
	\leq\prod_{i= 1}^{\gr{l}}\Big|\mathcal{T}([|\underline{\lambda^i}|])\Big|.
	\end{equation*}
	The right-hand side is bounded using that \[\Big|\mathcal{T}([|\underline{\lambda^i}|])\Big|\leq |\underline{\lambda^i}|^{|\underline{\lambda^i}|-2}\leq |\rho_i|^{|\rho_i|}\]
	(see section 2 of \cite{BGSS}). As the symmetric conditioning imposes that $|\rho_i|\leq \gamma$, the cumulants $\Big|\psi_{\gr{l}}\big(Z_{\lambda_1},\cdots,Z_{\lambda_{\gr{l}}}\big)\Big|$ are smaller than $(\gamma^\gamma)^{ n_{k+2}}$.
	
	We now have the following bound
	\begin{equation}\label{eq:borne utilisant les cluster possible}\big|\bar{\Phi}_{\underline{n}_{k+2},p}^{k'}(Z_p)\big|
	\leq \frac{\gamma^{n_{k+2}}\|h\|_0}{(n_{k+2})!}\sum_{q_1=1}^{n_{k+2}}\sum_{{\gr{r}}=1}^{n_{k+2}}\sum_{\underline{\rho}\in\mathcal{P}^{\gr{r}}_{n_{k+2}}}\sum_{\substack{(s_i)_i\\ \underline{\mathfrak{p}}\in\underset{i}{\prod}\mathfrak{P}(\rho_i)}}~\ind_{\mathcal{R}^{\underline{\rho},\underline{\mathfrak{p}}}_{q_1,(s_i)_i}}(Z_{n_{k+2}})\prod_{i=1}^{\gr{r}}\Delta_{\mathfrak{p}_i}(Z_{\rho_i}),
	\end{equation}
	where we denote
	\begin{gather*}
	\Delta_{\mathfrak{p}_1}(Z_{\rho_1}):=\ind_{\mathcal{O}_{\varpi}}\ind_{\!\!\!\!\tiny\begin{array}{l}Z_{\rho_1}\text{\,form\,a\,possible\,cluster}\end{array}}\!\!,\\
	\forall i\geq 2,~\Delta_{\mathfrak{p}_i}(Z_{\rho_i}):=\ind_{\!\!\!\!\tiny\begin{array}{l}Z_{\rho_i}\text{\,form\,a\,possible\,cluster}\end{array}} ~{\rm and}\\
	\mathcal{R}^{\underline{\rho},\underline{\mathfrak{p}}}_{q_1,(s_i)_i} :=  \Big\{Z_{p}\in\mathcal{D}^{\underline{\rho}}_\e,\,\bar{\ds{Z}}_{\omega_1}(\delta)\in\mathcal{R}^{0,k\theta+(k-1)\delta}_{\{q_1\},(s_i)_i}\Big\}.\end{gather*}
	
	The same method as in \cite{BGSS6} is used to control the condition $\ind_{\mathcal{R}^{\underline{\rho},\underline{\mathfrak{p}}}_{((s_i)_i}}$.
	
	For a pseudotrajectory $\bar{\ds{Z}}_{n_{k+2}}(\tau)$, consider its collision graph $\mathcal{G}_{\omega_1\cup\omega_2}^{[0,t-t_s]}$. Then, we can construct the graph $G$ by identifying in $\mathcal{G}_{\omega_1\cup\omega_2}^{[0,t-t_s]}$ the particles in the same cluster $\rho_i$. Finally we can construct the \emph{clustering trees} $T^>:=(\nu_i,\bar{\nu}_i)_{1\leq i\leq \gr{r}-1}$ where the $i$-th clustering collision in $G$ happens between cluster $\rho_{\nu_i}$ and $\rho_{\bar{\nu_i}}$.
	
	\begin{figure}[h!]
		\centering
		\includegraphics[width=13cm]{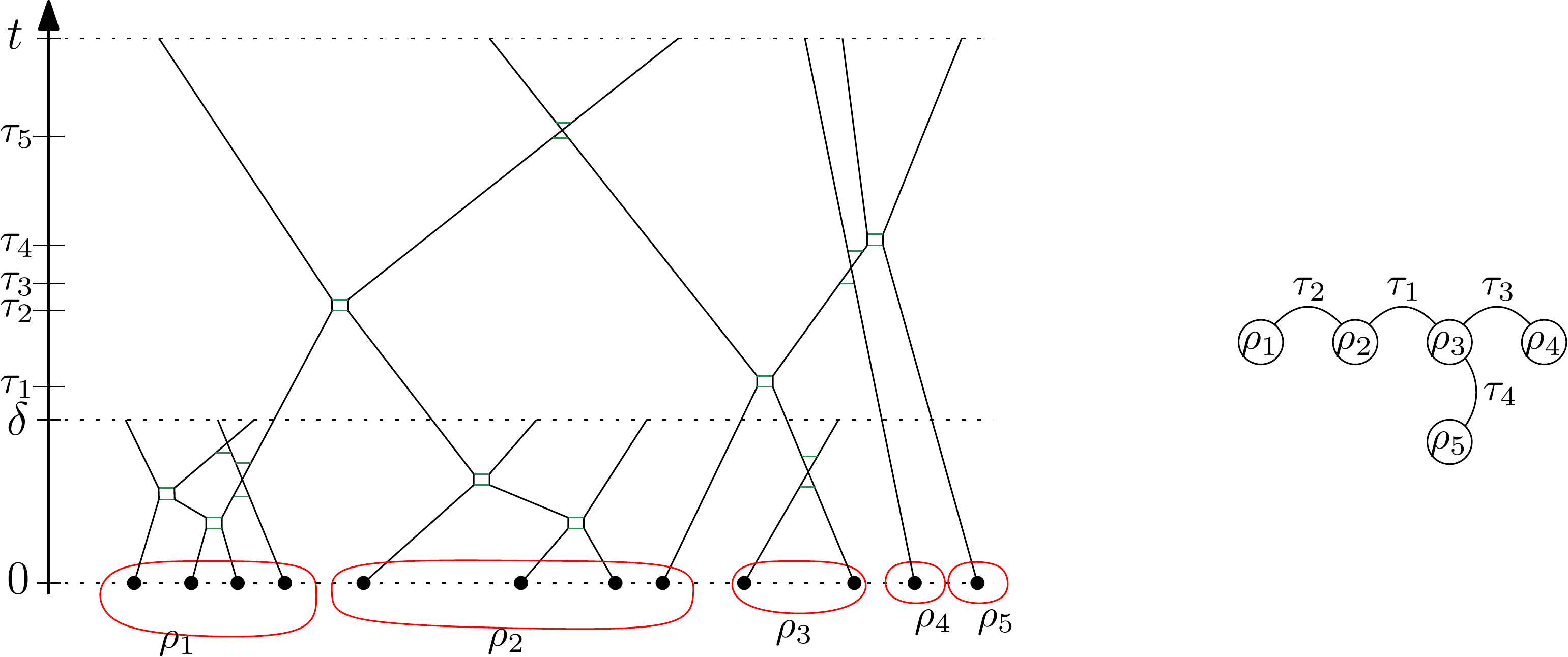}
		\caption{Example of construction of the clustering stets.}
	\end{figure}
	
	We need to count the number of clustering collisions of $T^>$ happening between time $\delta$ and time $k'\delta$. If $\gr{r}>n_k$, all the $\gr{r}-1$ collisions in $T^>$ cannot correspond to the $n_k-1$ collisions of the time interval $[k'\delta,\theta]$. Thus, at least $(\gr{r}-n_k)_+$ collisions happen during $[\delta,k'\delta]\subset[0,2\theta]$.
	
	One needs a different representation of collision graphs. Let $L_0$ be equal to $\{\{1\},\cdots,$ $\{\gr{r}\}\}$. The $L_i$ and $(\nu_{(i)},\bar{\nu}_{(i)})$ are constructed sequentially: suppose that $L_{i-1}=(c_1,\cdots,c_l)$, the $(c_j)$ forming a partition of $[1,r]$. The $i$-th collision happens between cluster $\nu_i\in c_a$ and $\bar{\nu}_i\in c_b$. Then:
	\begin{itemize}
		\item $L_i:=\big(L_{i-1}\setminus\{c_a,c_b\}\big)\cup\{c_a\cup c_b\}$,
		\item $\{\nu_{(i)},\bar{\nu}_{(i)}\}:=\{c_a,c_b\}$ with $\max \nu_{(i)} <\max \bar{\nu}_{(i)}$.
	\end{itemize}
	The $(\nu_{(i)},\bar{\nu}_{(i)})_i$ define a partition of $\mathcal{T}^>([\gr{r}])$ (the set of ordered trees on $[1,\gr{r}]$).
	
	\begin{figure}[h]
		\centering
		\includegraphics[width=13cm]{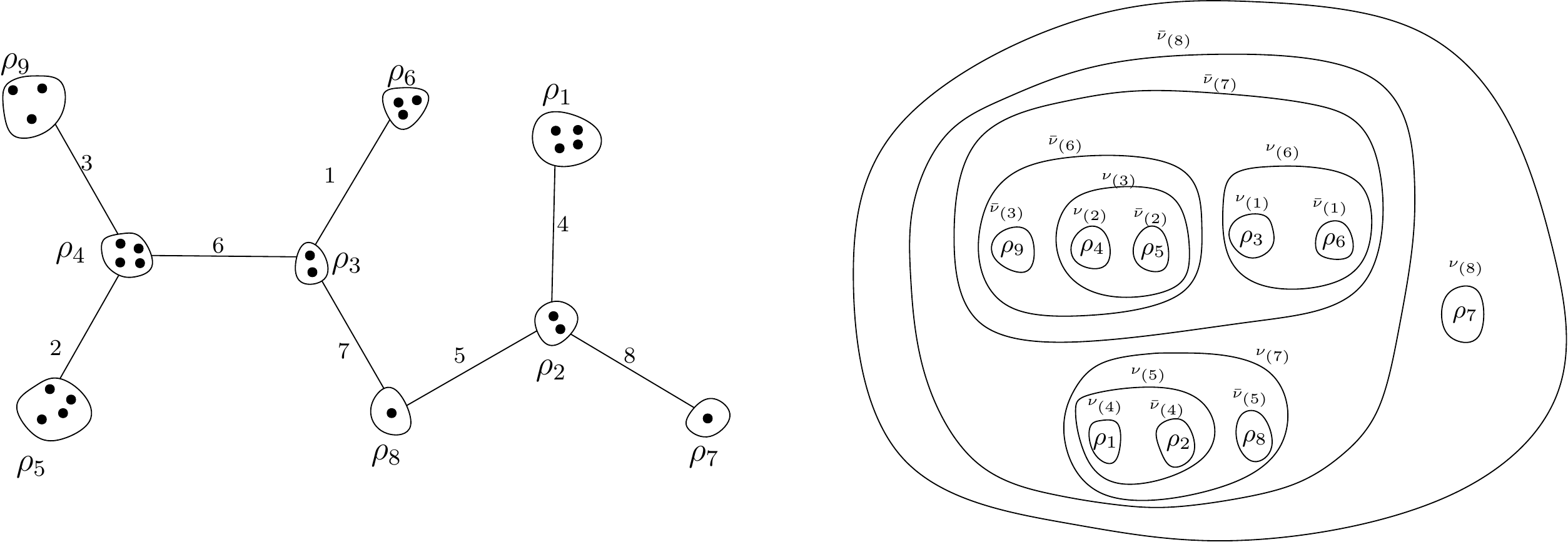}
		\caption{An example of construction of the representation $(\nu_{(i)},\bar{\nu}_{(i)})_i$ from a clustering graph. The graph }
	\end{figure}

	We performed the following change of variables: 
	\[\forall i\in\{1,\cdots,\gr{r}-1\},~\hat{x}_i:=x_{\min\nu_{(i)}}-x_{\min\bar{\nu}_{(i)}},~\tilde{X}_i:=\tr_{-x_{\min \rho_i}}X_{\rho_i},\]
	\[X_{2,n_{k+2}}\mapsto(\hat{x}_1\cdots,\hat{x}_{\gr{r}-1},\tilde{X}_1,\cdots,\tilde{X}_{\gr{r}}).\]
	
	The condition $\mathcal{R}^{\underline{\rho},\underline{\mathfrak{p}}}_{q_1,(s_i)_i} $ is integrated first with respect to $(\hat{x}_1,\cdots,\hat{x}_{\gr{r}-1})$, where relative positions inside a cluster $\tilde{X}_i$ are kept constant. The $(\Delta_{\mathfrak{p}_i})_i$ will be summed later with respect to the $(\tilde{X}_i)_i$.
	
	Fix $\tau_{i+1}$ the time of the $(i+1)$-th clustering collision, and the relative positions $\hat{x}_{i-1},\cdots,\hat{x}_{1}$. We define the $i$-th clustering set 
		\[\begin{aligne}{cc}
		B_i := \bigcup_{\substack{q\in \omega_{2,\nu_{(i)}}\\\bar{q}\in\omega_{1,\bar{\nu}_{(i)}}}} B_i^{q,\bar{q}}\\
		{\rm with}\qquad\omega_{1,\nu_{(i)}}:= \bigcup_{j\in\nu_{(i)}}\omega^j_1,\qquad\omega_{1,\bar{\nu}_{(i)}}:= \bigcup_{\bar{\jmath}\in\bar{\nu}_{(i)}}\omega_2^{\bar{\jmath}},\\
		B_i^{q,\bar{q}}:=\Big\{\hat{x}_i ~ \Big\vert~\exists \tau_i\in[0,\tau_{i+1}\wedge T_i],~|\bar{\ds{x}}_{\bar{q}}(\tau_i)-\bar{\ds{x}}_{\bar{q}}(\tau_i)|=\e\Big\}
		\end{aligne}\]
		and $T_i:=2\theta$ for the $(\gr{r}-n_k)_+$ first collisions, $t$ else. We used that $\omega_1\cup\omega_2$ is the set of particles existing after time $\delta$.
		
		Up to time $\tau_i$ the curves $\bar{\ds{x}}_q$ and $\bar{\ds{x}}_{\bar{q}}$ are independent. Hence, we can perform the change of variables $\hat{x}_i\mapsto (\tau_i,\eta_i)$, where $\tau_i$ is the minimal collision time and
		\[\eta_i := \frac{\bar{\ds{x}}_{\bar{q}}(\tau_i)-\bar{\ds{x}}_{q}(\tau_i)}{|\bar{\ds{x}}_{\bar{q}}(\tau_i)-\bar{\ds{x}}_{{q}}(\tau_i)|}.\]
		The Jacobian of this diffeomorphism is $\e^{d-1}|(\bar{\ds{v}}_{\bar{q}}(\tau_i)-\bar{\ds{v}}_{{q}}(\tau_i))\cdot\eta_i|\ud\tau_i\ud\eta_i$.
		
		As the particles in $\omega_{1,\nu_{(i)}}$ and $\omega_{1,\bar{\nu}_{(i)}}$ are isolated during $[\delta,\tau_i]$, their energies are conserved. The sum of relative velocities can be bounded by
		\[\begin{split}
		\sum_{\substack{q\in\omega_{1,{\nu}_{(i)}}\\\bar{q}\in\omega_{1,\bar{\nu}_{(i)}}}}|\bar{\ds{v}}_{\bar{q}}(\tau_i)-\bar{\ds{v}}_{{q}}(\tau_i)|&\leq |\bar{\ds{V}}_{\omega_{1,{\nu}_{(i)}}}(\tau_i)|\,|\omega_{1,{\nu}_{(i)}}|^{1/2}|\omega_{1,\bar{\nu}_{(i)}}|+|\bar{\ds{V}}_{\omega_{1,\bar{\nu}_{(i)}}}(\tau_i)|\,|\omega_{1,\bar{\nu}_{(i)}}|^{1/2}|\omega_{1,{\nu}_{(i)}}|\\[-15pt]
		& \leq\Big(|\omega_{1,{\nu}_{(i)}}|+|\bar{\ds{V}}_{\omega_{1,{\nu}_{(i)}}}(\tau_i)|^2\Big)\Big(|\omega_{1,\bar{\nu}_{(i)}}|+|\bar{\ds{V}}_{\omega_{1,\bar{\nu}_{(i)}}}(\tau_i)|^2\Big).
		\end{split}\]
		
		Using the same method as in the proof of \ref{eq:borne energie cinetique},
		\[\frac{1}{2}\left|\bar{\ds{V}}_{\omega_{1,{\nu}_{(i)}}}(\tau_i)\right|^2 \leq \mathcal{H}_{|\omega_{1,{\nu}_{(i)}}|}\left(\bar{\ds{Z}}_{\omega_{1,{\nu}_{(i)}}}(\delta)\right)\leq \mathcal{H}_{|\lambda_{1,\bar{\nu}_{(i)}}|}\left(\bar{\ds{Z}}_{\lambda_{1,\bar{\nu}_{(i)}}}(\delta)\right),\]
		where we denote
		\[\lambda_{1,\bar{\nu}_{(i)}}:= \bigcup_{\bar{\jmath}\in\bar{\nu}_{(i)}}\lambda_1^j.\]
		At time $\delta$, the particles in two different clusters cannot interact (by definition of a possible cluster), 
		\[\mathcal{H}_{|\lambda_{1,{\nu}_{(i)}}|}(\bar{\ds{Z}}_{\lambda_{1,{\nu}_{(i)}}}(\delta))= \sum_{j\in\nu_{(i)}}\mathcal{H}_{|\lambda_1^j|}(\bar{\ds{Z}}_{\lambda_1^j}(\delta))= \sum_{j\in\nu_{(i)}}\mathcal{H}_{|\lambda_1^j|}(\bar{\ds{Z}}_{\lambda_1^j}(0))\leq \sum_{j\in\nu_{(i)}}\mathcal{H}_{|\rho_j|}(Z_{\rho_j}).\]
		
		We conclude that 
		\begin{equation}\label{eq:borne energie cinetique2}
		\begin{split}\sum_{\substack{q\in\omega_{1,{\nu}_{(i)}}\\\bar{q}\in\omega_{1,\bar{\nu}_{(i)}}}}|\bar{\ds{v}}_{\bar{q}}(\tau_i)-\bar{\ds{v}}_{{q}}(\tau_i)|&\leq 4\sum_{\substack{\nu_i\in\nu_{(i)}\\\bar{\nu}_{i}\in\bar{\nu}_{(i)}}}\Big(|\rho_{{\nu}_i}|+\mathcal{H}_{|\rho_{{\nu}_i}|}(Z_{\rho_{{\nu}_i}})\Big)\Big(|\rho_{\bar{\nu}_i}|+\mathcal{H}_{|\rho_{\bar{\nu}_i}|}(Z_{\rho_{\bar{\nu}_i}})\Big).\end{split}
		\end{equation}
				
		This gives the following bound on $|B_i|$ (using the Boltzmann-Grad scaling $\mu\mathfrak{d}\e^{d-1}= 1$)
		\[\begin{split}
		|B_i|&\leq \frac{C}{\mu\mathfrak{d}}\int^{t_{i+1}\wedge T_i}_0\ud \tau_i\sum_{q,\bar{q}}|\bar{\ds{v}}_{{q}}(\tau_i)-\bar{\ds{v}}_{\bar{q}}(\tau_i)|\\
		&\leq\frac{C}{\mu\mathfrak{d}}\sum_{\substack{\nu_i\in\nu_{(i)}\\\bar{\nu}_{i}\in\bar{\nu}_{(i)}}}\Big(|\rho_{{\nu}_i}|+\mathcal{H}_{|\rho_{{\nu}_i}|}(Z_{\rho_{{\nu}_i}})\Big)\Big(|\rho_{\bar{\nu}_i}|+\mathcal{H}_{|\rho_{\bar{\nu}_i}|}(Z_{\rho_{\bar{\nu}_i}})\Big)\int^{t_{i+1}\wedge T_i}_0\ud \tau_i.
		\end{split}\]
		
		Permuting the product and the sum, 
		\begin{multline*}
		\sum_{(\nu_{(i)},\bar{\nu}_{(i)})_i}\prod_{i=1}^ {\gr{r}-1}\sum_{\substack{\nu_i\in\nu_{(i)}\\\bar{\nu}_{i}\in\bar{\nu}_{(i)}}}\Big(|\rho_{{\nu}_i}|+\mathcal{H}_{|\rho_{{\nu}_i}|}(Z_{\rho_{{\nu}_i}})\Big)\Big(|\rho_{\bar{\nu}_i}|+\mathcal{H}_{|\rho_{\bar{\nu}_i}|}(Z_{\rho_{\bar{\nu}_i}})\Big)\\[-15pt]
		=\sum_{(\nu_{i},\bar{\nu}_{i})_i}\prod_{i=1}^ {\gr{r}-1}\Big(|\rho_{{\nu}_i}|+\mathcal{H}_{|\rho_{{\nu}_i}|}(Z_{\rho_{{\nu}_i}})\Big)\Big(|\rho_{\bar{\nu}_i}|+\mathcal{H}_{|\rho_{\bar{\nu}_i}|}(Z_{\rho_{\bar{\nu}_i}})\Big).
		\end{multline*}
		
		Using that
		\[\forall a,b\in\mathbb{N},~\frac{(a+b)!}{a!b!}\leq 2^{a+b},\]
		we have 
		\[\begin{split}
		\int_0^t\ud t_{\gr{r}-1}\cdots\int_0^{t_2\wedge T_2}\ud t_1&\leq\frac{t^{n_k\wedge \gr{r}-1}}{(n_k\wedge \gr{r}-1)!}\frac{\theta^{(\gr{r}-n_k)_+}}{((\gr{r}-n_k)_+)!}\leq 2^{n_{k+2}}\frac{t^{n_k\wedge \gr{r}-1}\theta^{(\gr{r}-n_k)_+}}{(\gr{r}-1)!}.
		\end{split}\]
		
		We can now sum up on the clustering collisions:
		\[\begin{split}
		\int&\ind_{\mathcal{R}^{\underline{\rho},\underline{\mathfrak{p}}}_{(s_i)_i} }\ud\hat{x}_1\cdots\ud\hat{x}_{\gr{r}-1}\leq \sum_{(\nu_{(i)},\bar{\nu}_{(i)})}\int \ud\hat{x}'_1 \ind_{B_1} \int \ud\hat{x}'_2\ind_{B_2}\cdots\int \ud\hat{x}_{\gr{r}-1}\ind_{B_{\gr{r}-1}}\\
		&\leq\left(\frac{C}{\mu\mathfrak{d}}\right)^{\gr{r}-1}\int_0^t\ud t_{\gr{r}-1}\cdots\int_0^{t_2\wedge T_2}\ud t_1 \sum_{(\nu_{i},\bar{\nu}_{i})_i}\prod_{i=1}^ {\gr{r}-1}\Big(|\rho_{{\nu}_i}|+\mathcal{H}_{|\rho_{{\nu}_i}|}(Z_{\rho_{{\nu}_i}})\Big)\Big(|\rho_{\bar{\nu}_i}|+\mathcal{H}_{|\rho_{\bar{\nu}_i}|}(Z_{\rho_{\bar{\nu}_i}})\Big)\\
		&\leq\left(\frac{2C}{\mu\mathfrak{d}}\right)^{\gr{r}-1}\frac{t^{n_k\wedge \gr{r}-1}\theta^{(\gr{r}-n_k)_+}}{(\gr{r}-1)!}\sum_{(\nu_{i},\bar{\nu}_{i})_i}\prod_{i=1}^ {\gr{r}-1}\Big(|\rho_{{\nu}_i}|+\mathcal{H}_{|\rho_{{\nu}_i}|}(Z_{\rho_{{\nu}_i}})\Big)\Big(|\rho_{\bar{\nu}_i}|+\mathcal{H}_{|\rho_{\bar{\nu}_i}|}(Z_{\rho_{\bar{\nu}_i}})\Big).
		\end{split}\]
		
		We can forget the order of the edges of $T^> = (\nu_i,\bar{\nu}_i)_i$, which gives a factor $\gr{r}!$. Secondly, denoting $d_i(G)$ the degree of the vertex $i$ in a graph $G$ and  $\mathcal{T}([\gr{r}])$ the set of minimally (not oriented) connected graphs on $[1,\gr{r}]$, we can write the preceding inequality as
		\[\begin{split}
		\int\ind_{\mathcal{R}^{\underline{\rho},\underline{\mathfrak{p}}}_{q_1,(s_i)_i} }\ud\hat{x}_1\cdots\ud\hat{x}_{\gr{r}-1}\leq\left(\frac{2C}{\mu\mathfrak{d}}\right)^{\gr{r}-1}t^{n_k\wedge \gr{r}-1}\theta^{(\gr{r}-n_k)_+}\sum_{T\in\mathcal{T}([\gr{r}])}\prod_{i=1}^r \Big(|\rho_{i}|+\mathcal{H}_{|\rho_i|}(Z_{\rho_i})\Big)^{d_i(T)}.
		\end{split}\]
		
		For $A,B>0$, $x\in\mathbb{R}^+$, there exists a constant $C>0$ such that
		\[\left(A+x\right)^B e^{-\frac{x}{4}} \leq\left(\tfrac{4B}{e}\right)^B e^{\frac{A}{4}}.\]
		We use this inequality and that on $\mathcal{D}^{\underline{\rho}}_{\rho}$, $\mathcal{H}_n(Z_n) = \sum_i \mathcal{H}_{|\rho_i|}(Z_{\rho_i})$ to bound 
		\begin{multline*}		
		\int\ind_{\mathcal{R}^{\underline{\rho},\underline{\mathfrak{p}}}_{(s_i)_i} }e^{-\frac{1}{2}\mathcal{H}_{n_{k+2}}(Z_{n_{k+2}})}\ud\hat{x}_1\cdots\ud\hat{x}_{\gr{r}-1}\\
		\begin{split}&\leq\left(\frac{C}{\mu\mathfrak{d}}\right)^{\gr{r}-1}t^{n_k\wedge r-1}\theta^{(\gr{r}-n_k)_+}\sum_{T\in\mathcal{T}([\gr{r}])}\prod_{i=1}^{\gr{r}} \Big(|\rho_{i}|+\mathcal{H}_{|\rho_i|}(Z_{\rho_i})\Big)^{d_i(T)}e^{-\frac{1}{2}\sum_{i=1}^{\gr{r}}\mathcal{H}_{|\rho_i|}(Z_{\rho_i})}\\
		&\leq \tilde{C}^{n_{k+2}}\frac{t^{n_k\wedge \gr{r}-1}\theta^{(\gr{r}-n_k)_+}}{(\mu\mathfrak{d})^{\gr{r}-1}}\sum_{T\in\mathcal{T}([\gr{r}])}\prod_{i=1}^{\gr{r}} d_i(T)^{d_i(T)}\;.
		\end{split}
		\end{multline*}
		As the sum of the $d_i(T)$ is equal to $2\gr{r}-2$, we have by convexity of $x\mapsto x\log x$
		\[\sum_{i=1}^{\gr{r}} d_i(T)\log{d_i(T)} \leq\gr{r}\frac{\sum_{i=1}^{\gr{r}} d_i(T)}{\gr r}\log\frac{\sum_{i=1}^{\gr{r}} d_i(T)}{\gr r}\leq (2\gr{r}-2) \log 2 \]
		and $|\mathcal{T}([\gr{r}])|$ is equal to $\gr{r}^{\gr{r}-2}$,
		\begin{multline*}		
		\int\ind_{\mathcal{R}^{\underline{\rho},\underline{\mathfrak{p}}}_{(s_i)_i} }e^{-\frac{1}{2}\mathcal{H}_{n_{k+2}}(Z_{n_{k+2}})}\ud\hat{x}_1\cdots\ud\hat{x}_{\gr{r}-1}\\
		\begin{split}&\leq\left(\frac{C}{\mu\mathfrak{d}}\right)^{\gr{r}-1}t^{n_k\wedge r-1}\theta^{(\gr{r}-n_k)_+}\sum_{T\in\mathcal{T}([\gr{r}])}\prod_{i=1}^{\gr{r}} \Big(|\rho_{i}|+\mathcal{H}_{|\rho_i|}(Z_{\rho_i})\Big)^{d_i(T)}e^{-\frac{1}{2}\sum_{i=1}^{\gr{r}}\mathcal{H}_{|\rho_i|}(Z_{\rho_i})}\\
		&\leq \tilde{C}'^{n_{k+2}}\frac{t^{n_k\wedge \gr{r}-1}\theta^{(\gr{r}-n_k)_+}}{(\mu\mathfrak{d})^{\gr{r}-1}} (\gr{r}-1)!\;.
		\end{split}
		\end{multline*}
		
		We can now integrate the condition $\Delta_{\mathfrak{p}_i}(Z_{\rho_i})$. The particles in $Z_{\rho_i}$ have to form a possible cluster. Because clusters are of size at most $\gamma$,
		\begin{align}
		\label{estimation recollision}
		&\int_{{\mathbb{T}}^{|\rho_1|-1}\times(\mathbb{R}^d)^{|\rho_1|}} \Delta_{\mathfrak{p}_1}(Z_{\rho_1})\tfrac{e^{-\frac{1}{2}\mathcal{H}_{|\rho_1|}(Z_{\rho_1})}}{(2\pi)^{d|\rho_1|/2}}\ud\tilde{X}_{i}\ud V_{\rho_1}\leq\frac{C_r\delta^{\max\{2,|\rho_1|-1\}}}{(\mu\mathfrak{d})^{|\rho_1|-1}}\e^{\mathfrak{a}},\\
		&\int_{{\mathbb{T}}^{|\rho_i|-1}\times(\mathbb{R}^d)^{|\rho_i|}}\Delta_{\mathfrak{p}_i}(Z_{\rho_i}) \tfrac{e^{-\frac{1}{2}\mathcal{H}_{|\rho_i|}(Z_{\rho_i})}}{(2\pi)^{d|\rho_i|/2}}\ud\tilde{X}_{i}\ud V_{\rho_i}\leq C_\gamma \left(\frac{\delta}{\mathfrak{d}\mu}\right)^{|\rho_i|-1},
		\end{align}
		The second inequality is a clustering estimation, similar to the ones threatened in the proof of \eqref{Estimation sans reco 1}. In the first inequality, we use recollision estimates as in the proof of \eqref{Estimation avec reco 1}. The proofs are given in Appendix \ref{subsec: derniere appencicite}.
		
		Integrating the $\Delta_{\mathfrak{p}_i}$ leads to
		\begin{multline*}
		\int\ind_{\mathcal{R}^{\underline{\rho},\underline{\mathfrak{p}}}_{(s_i)_i}}(Z_{n_{k+2}})\prod_{i=1}^{\gr{r}}\Delta_{\mathfrak{p}_i}(Z_{\rho_i})\frac{e^{-\mathcal{H}_{n_{k+2}}}}{(2\pi)^{dn_{k+2}/2}}\ud X_{2,n_{k+2}}\ud V_{n_{k+2}}\\
		\leq(\gr{r}-1)! C^{n_{k+2}}\frac{t^{n_k\wedge \gr{r}-1}\theta^{(\gr{r}-n_k)_+}}{(\mu\mathfrak{d})^{\gr{r}-1}} \prod_{i =1}^{\gr{r}}\left(\frac{\delta}{\mathfrak{d}\mu}\right)^{|\rho_i|-1}\frac{\delta^{\max{2,|\rho_1|-1}}}{(\mu\mathfrak{d})^{|\rho_1|-1}}\e^{\mathfrak{a}}.
		\end{multline*}
		
		Any particle removed at time $\delta$ has a clustering collision during $[0,\delta]$. Therefore, $\sum_{i=1}^\gr{r}(|\rho_{i}|-1)$ is bigger than $n_{k+2}-n_k$. In addition, we have chosen $\theta$ bigger than $\delta$ so
		\begin{multline*}
		\int\ind_{\mathcal{R}^{\underline{\rho},\underline{\mathfrak{p}}}_{(s_i)_i}}(Z_{n_{k+2}})\prod_{i=1}^{\gr{r}}\Delta_{\mathfrak{p}_i}(Z_{\rho_i})\frac{e^{-\mathcal{H}_{n_{k+2}}(Z_{n_{k+2}})}}{(2\pi)^{\frac{dn_{k+2}}{2}}}\ud X_{2,n_{k+2}}\ud V_{n_{k+2}}\\
		\leq (\gr{r}-1)!\frac{C^{n_{k+2}}}{(\mu\mathfrak{d})^{n_{k+2}-1}}t^{n_k-1} \theta^{(n_{k+2}-n_k-2)_+}\delta^2\e^{\mathfrak{a}}.
		\end{multline*}
		
		We sum now on the parameters $(s_i)_i$ and $(\mathfrak{p}_i)$. Because the  sizes of the possible clusters are bounded by $\gamma$, the $|\mathfrak{P}(\rho_i)|$ are smaller than some $C_\gamma>0$ depending only on $\gamma$. The  number of collision parameters  $(s_i)_i$ is equal to $2^{n_{k+2}}$ and
		\begin{multline*}\int\big|\bar{\Phi}_{\underline{n}_{k+2}}^{k'}(Z_{n_{k+2}})\big|M^{\otimes n_{k+2}}\ud X_{2,n_{k+2}}\ud V_{n_{k+2}}\\
		\leq \frac{\|h\|(CC_\gamma 4\gamma)^{n_{k+2}}}{n_{k+2}!(\mu\mathfrak{d})^{n_{k+2}-1}}t^{n_k-1}\theta^{(n_{k+2}-n_k-2)_+} \delta^2\e^{\mathfrak{a}}\sum_{r=1}^{n_{k+2}}\sum_{\underline{\rho}\in\mathcal{P}^r_{n_{k+2}}}(\gr{r}-1)!\;.\end{multline*}
		
		The last sums can be bounded by
		\begin{align*}
		\frac{1}{n_{k+2}!}\sum_{r=1}^{n_{k+2}}\sum_{\underline{\rho}\in\mathcal{P}^r_{n_{k+2}}}(\gr{r}-1)! &= 	\frac{1}{n_{k+2}!}\sum_{\gr{r}=1}^{n_{k+2}}\sum_{\substack{k_1+\cdots+k_\gr{r}=n_{k+2}\\k_i\geq 1}} \frac{n_{k+2}!}{k_1!\cdots k_\gr{r}!}\frac{(\gr{r}-1)!}{\gr{r}!}\\
		&\leq\sum_{\gr{r}=1}^{n_{k+2}}\sum_{\substack{k_1+\cdots+k_\gr{r}=n_{k+2}\\k_i\geq 1}} \frac{1}{k_1!\cdots k_\gr{r}!}\leq e^{n_{k+2}}
		\end{align*}
		
		This ends the proof of the first inequality.
	\end{proof}
	
	\begin{proof}[Proof of \eqref{Estimation reco loc 2}.] As the $\bar{\Phi}_{\underline{n}_{k+2}}^{k'}$ are symmetric, it is sufficient to study 
	\begin{equation}\label{eq:PhiPhi}
		\big|\bar{\Phi}_{\underline{n}_{k+2}}^{k'}(Z_{[n_{k+2}]})\bar{\Phi}_{\underline{n}_{k+2}}^{k'}(Z_{[n_{k+2}+1-m,2n_{k+2}-m]})\big|.
	\end{equation} The bound \eqref{borne 1 sur phi l} leads to
	\begin{multline*}
	\eqref{eq:PhiPhi}\leq\frac{\|h\|^2}{(n_{k+2}!)^2} \sum_{\substack{(q_1,\omega_1\omega_2)\\(q_1'\omega'_1\omega'_2)}}\sum_{\varpi\subset[n_{k+2}]}\sum_{\substack{(s_i)_i\\(s'_i)_i}}\sum_{\substack{(\lambda_1,\cdots,\lambda_{\gr{l}})\\(\lambda'_1,\cdots,\lambda'_{\gr{l}})}}\ind_{\mathcal{R}^{0,t_s-\delta}_{\{q_1\},(s_i)_i}}(\bar{\ds{Z}}_{\omega_1}(\delta))\ind_{\mathcal{R}^{0,t_s-\delta}_{\{q_1\},(s'_i)_i}}(\bar{\ds{Z'}}_{\omega'_2}(\delta)) \\[-0pt]
	\times\Bigg(\sum_{T\in\mathcal{T}([\gr{l}])}\prod_{(i,j)\in E(T)} \ind_{{\lambda}_i\so {\lambda}_j} \DDelta_{|\lambda_1|}^{[n_{k+1}]}(Z_{\lambda_1})\prod_{i=2}^{\gr{l}}\DDelta_{|\lambda_i|}(Z_{\lambda_i}) \Bigg)\\
	\times\Bigg(\sum_{T'\in\mathcal{T}([\gr{l}'])}\prod_{(i,j)\in E(T')} \ind_{{\lambda'}_i\so {\lambda'}_j} \DDelta_{|\lambda'_1|}^{[n_{k+1}]}(Z_{\lambda'_1})\prod_{i=2}^{\gr{l}}\DDelta_{|\lambda'_i|}(Z_{\lambda'_i}) \Bigg)\ind_{\mathfrak{n}(k'\delta)=n_{k}}\ind_{\mathcal{O}_\varpi}.
	\end{multline*}
	where we have denoted $\mathcal{T}(\sigma)$ the set of connected and simply connected graphs with vertices $\sigma$, a finite set. The sets $\mathcal{O}_\varpi$ have been defined in Definition \ref{def: O_pi}. In addition, we sum on all parameters such that 
	\begin{itemize}
		\item $q_1\in[n_{k+2}]$ and $q'_1\in[n_{k+2}+1-m,2{n_{k+2}}-m]$,
		\item $\omega_1\sqcup\omega_2 = [n_{k+2}]$, $\omega'_1\sqcup\omega'_2 = [n_{k+2}+1-m,2{n_{k+2}}-m]$, $q_1\in\omega_1$, $q_1\in\omega_1$ and $|\omega_1|=|\omega'_1|= n_{k+1}$,
		\item $\lambda_1\supset \omega_1$, $\lambda'_1\supset \omega'_1$, $(\lambda_2,\cdots,\lambda_{\gr{l}})$ an unordered partition of $[n_{k+2}]\setminus\omega_1$ and $(\lambda'_2,\cdots,\lambda'_{\gr{l}})$ an unordered partition of $[n_{k+2}+1-m,2{n_{k+2}}-m]\setminus\omega'_1$.
	\end{itemize}

	The pseudotrajectory $\bar{Z}(\tau)$ (respectively $\bar{Z}'(\tau)$) begins with coordinates $Z_{[n_{k+2}]}$ (respectively with coordinates $Z_{[n_{k+2}+1-m,2n_{k+2}-m]}$) with parameters $(q_1,\omega_1,\omega_2,(\lambda_1,\cdots,\lambda_{\gr{l}}))$ (respectively $(q'_1,\omega'_1,\omega'_2,(\lambda'_1,\cdots,$ $\lambda'_{\gr{l}'}))$) in the same way than in the proof of \eqref{Estimation reco loc 1}.

	Note that the right hand-side of \eqref{eq:PhiPhi} is invariant under translation. Thus, one can fix $x_1 = 0$ and integrate with respect to the other variables.
	
	For a position $Z_{2n_{k+2}-m}$, we consider $\underline{\rho}:=(\rho_1,\cdots,\rho_{\gr{r}})$ the possible clusters. As in the previous section, with a little loss of symmetry, one can suppose that $\varpi_1\subset\rho_1$. We can then construct the parameters $\mathfrak{p}_1:=(\underline{\omega}^1,\underline{\omega}'{^1},\underline{\lambda}^1,\underline{\lambda}'{^1},\varpi)$,  $\left(\mathfrak{p}_i\right)_{i\geq 2}:=\left((\underline{\omega}^i,\underline{\omega}'{^i},\underline{\lambda}^i,\underline{\lambda}'{^i})\right)_{i\geq 2}$:
	\begin{itemize}
		\item $\underline{\omega}^i:=(\omega^i_1,\omega^i_2)$ is a partition of $\rho_i\cap[n_{k+2}]$ defined by $\omega^i_j:=\omega_j\cap \rho_i$,
		\item $\underline{\omega}'{^i}:=({\omega'}^i_1,{\omega'}^i_3)$ is a partition of $\rho_i\cap[n_{k+2}+1-m,2n_{k+2}+m]$ defined by ${\omega_j'}^i:=\omega'_j\cap \rho_i$,
		\item $\underline{\lambda}^{i}:=\{\lambda^i_1:=\lambda_1\cap\rho_i\}\cup\{\lambda_j \text{~for~} j\geq 2 \text{~with~}\lambda_j\subset \rho_i\}$ a partition of $[n_{k+2}]\cap\rho_i$,
		\item $\underline{\lambda'}^{i}:=\{{\lambda'_1}^i:=\lambda'_1\cap\rho_i\}\cup\{\lambda'_j \text{~for~} j\geq 2 \text{~with~}\lambda'_j\subset \rho_i\}$ a partition of $[n_{k+2}+1-m,2n_{k+2}+m]\cap\rho_i$.
	\end{itemize}

	We denote now $\mathfrak{P}(\rho_i)$ the new set of possible parameters $\mathfrak{p}_i$ (this will not create a conflict of notations with the previous section). Because each cluster $\rho_i$ is of size at most $\gamma$, $|\mathfrak{P}(\rho_i)|$ is bounded by some constant $C_\gamma$ depending only on $\gamma$.  Defining 
	\[\Delta_{\mathfrak{p}_1}(Z_{\rho_1}):=\ind_{\!\!\!\!\tiny\begin{array}{l}Z_{\rho_1}\text{\,form\,a}\\ \text{possible\,cluster}\end{array}}\ind_{\mathcal{O}_\varpi},~~~\forall i\geq 2,~\Delta_{\mathfrak{p}_i}(Z_{\rho_i}):=\ind_{\!\!\!\!\tiny\begin{array}{l}Z_{\rho_i}\text{\,form\,a}\\ \text{possible\,cluster}\end{array}}\]
	\[\mathcal{R}^{\underline{\rho},\underline{\mathfrak{p}}}_{(s_i)_i,(s'_i)_i} :=  \Big\{Z_{2n_{k+2}-m}\in\mathcal{D}^{\underline{\rho}}_\e,\,\bar{\ds{Z}}_{\omega_1\cup\omega_2\cup\omega_3}(\delta)\in\mathcal{R}^{0,k\theta+(k-1)\delta}_{\omega_1,(s_i)_i},\bar{\ds{Z}}'_{n_{k+2}}(\delta)\in\mathcal{R}^{0,k\theta+(k-1)\delta}_{\omega'_1,(s'_i)_i}\Big\},\]
	we have as in the inequality \eqref{eq:borne utilisant les cluster possible}
	\begin{equation*}
	\eqref{eq:PhiPhi}
	\leq \frac{\gamma^{2(\gamma-2)n_{k+2}}\|h\|_0^2}{(n_{k+2}!)^2}\sum_{\gr{r}=1}^{2n_{k+2}-m}\sum_{\underline{\rho}\in\mathcal{P}^r_{p}}\sum_{\substack{(s_i)_i,(s'_i)_i\\[1pt] \underline{\mathfrak{p}}\in\underset{i}{\prod}\mathfrak{P}(\rho_i)}}~\ind_{\mathcal{R}^{\underline{\rho},\underline{\mathfrak{p}}}_{(s_i)_i,(s'_i)_i} }(Z_{2n_{k+2}-m})\prod_{i=1}^{\gr{r}}\Delta_{\mathfrak{p}_i}(Z_{\rho_i}).
	\end{equation*}
	Note that, for at least one $i$, $\underline{\varpi}^i$ is not empty. We are now constructing a clustering tree in order to estimate $\mathcal{R}^{\underline{\rho},\underline{\mathfrak{p}}}_{(s_i)_i,(s'_i)_i} $.
	
	Consider the collision graph associated with the first pseudotrajectory $\mathcal{G}^{[0,t-t_s]}_{\omega_1\cup\omega_2}$  and  the graph associated with the second one $\mathcal{G}^{[0,t-t_s]}_{\omega'_1\cup\omega'_2}$. Merge them and identify vertices in the same cluster $\rho_i$.  Keeping only the first clustering collisions, we obtain the oriented tree $T^>:=(\nu_i,\bar{\nu}_i)_{1\leq i\leq \gr{r}-1}$. Note that these clustering collisions can happen in the first or second pseudotrajectories.
	
	As in the proof of \eqref{Estimation reco loc 1} we have to bound the number of collisions of $T^>$ in the time interval $[0,2\tau]$. There are at most $(n_{k}-1+n_{k+2}-m)$ collisions during $[(k'+1)\delta,t-t_s]$ ($n_k-1$ for the first pseudotrajectory, and we have to connect $n_{k+2}-m$ particles in the second). Thus, there are at least $(\gr{r}-(n_k-1+n_{k+2}-m))_+$ clustering collisions in $[\delta,(k'+1)\delta]\subset[0,2\tau]$.
	
	We explain quickly how to estimate the $i$-th collision. As in the previous paragraph, we construct the modified tree parameters $(\nu_{(i)},\bar{\nu}_{(i)})$ and the change of variables 
	\[\forall i\in\{1,\cdots,\gr{r}-1\},~\hat{x}_i:=x_{\min\nu_{(i)}}-x_{\min\bar{\nu}_{(i)}},~\tilde{X}_i:=\tr_{x_{\min\rho_i}}X_{\rho_i},\]
	\[X_{2,2n_{k+2}-m}\mapsto(\hat{x}_1\cdots,\hat{x}_{\gr{r}-1},\tilde{X}_1,\cdots,\tilde{X}_r),\]
	and we integrate on the $(\hat{x}_i)$. 
	
	The clustering set $B_i$ is defined as follows: fix $t_{i+1}$ the time of the $(i+1)$-th clustering collision, and the relative positions $\hat{x}_{i-1},\cdots,\hat{x}_{1}$. We define the $i$-th clustering set
	\[B_i := \bigcup_{\substack{q\in \bigcup_{j\in\nu_{(i)}}\rho_j\\\overline{q}\in \bigcup_{\bar{\jmath}\in\bar{\nu}_{(i)}}\rho_{\bar{\jmath}}} }\Big(B_{i}^{q,\bar{q}}\cup B_{i}'^{q,\bar{q}}\Big)\]
	with
	\[B_{i}^{q,\bar{q}}:=\Big\{\hat{x}_i ~ \Big\vert~\exists \tau_i\in[0,t_{i+1}\wedge T_i],~|\bar{\ds{x}}_{\bar{q}}(\tau_i)-\bar{\ds{x}}_{{q}}(\tau_i)|=\e\Big\},\]
	where $T_i:=2\theta$ for the $(\gr{r}-n_k)_+$ first collisions, and $t$ else. The set $B_{i}'^{q,q'}$ is defined in the same way for the other pseudotrajectory. We can apply the estimates of the previous paragraph:
	\[\int \ind_{B_i}\ud\hat{x}_i\leq\frac{2C}{\mu\mathfrak{d}}\sum_{\substack{\nu_i\in\nu_{(i)}\\\bar{\nu}_i\in\bar{\nu}_{(i)}}}\Big(|\rho_{\nu_i}|+\mathcal{H}_{|\rho_{\nu_i}|}(Z_{\rho_{\nu_i}})\Big)\Big(|\rho_{\bar{\nu}_i}|+\mathcal{H}_{|\rho_{\bar{\nu}_i}|}(Z_{\rho_{\bar{\nu}_i}})\Big)\int_{0}^{\tau_{i+1}\wedge T_i}\ud \tau_i.\]
	
	In this way, we end up with the same situation as in the estimate of \eqref{Estimation reco loc 1}, and we can apply the same strategy:
	\[\begin{split}
	\int	|\bar{\Phi}_{\underline{n}_{k+2}}^{k'}(Z_{k+2})\bar{\Phi}_{\underline{n}_{k+2}}^{k'}(&Z_{n_{k+2}+1-m,2n_{k+2}-m})\big|\frac{e^{-\mathcal{H}_{2_{k+2}-m}(Z_{2n_{k+2}-m})}}{(2\pi)^{2n_{k+2}-m}}\ud X_{2,2n_{k+2}-m}\ud V_{2n_{k+2}-m}\\
	&\qquad\leq \frac{(2n_{k+2}-m)!\|h\|^2}{(n_{k+2}!)^2(\mu\mathfrak{d})^{2n_{k+2}-m-1}}C^{n_{k+2}}\delta^2\e^{\mathfrak{a}} \tau^{(n_{k+2}-n_k-2)_+}t^{n_k-1+n_{k+2}-m}\\
	&\qquad\leq \frac{\mu^{m-1}}{n_{k+2}^m}\left(\frac{\|h\|}{(\mu\mathfrak{d})^{n_{k+2}-1}} \tilde{C}^{n_{k+2}}\right)^2\delta^2\e^{\mathfrak{a}} \theta^{(n_{k+2}-n_k-2)_+}t^{n_k-1+n_{k+2}}
	\end{split}\]
	which concludes the proof.
	\end{proof}

	\appendix
	
	\section{The linearized Boltzmann operator without cut-off}\label{sec:The Linearized Boltzmann operator without cut-off}
	In this section, we fix the dimension $d=3$. We construct the linearized Boltzmann operator associated with the power law $1/r^s$, $s>1$ and we explain where the scaling $\mathfrak{d}_{s,\alpha} = \alpha^{2/s}$ comes from.
	
	We begin with a change of variables in the definition of the Boltzmann operator $\mathcal{L}_\alpha$. For $(v,v_*,\nu)$, we define \begin{equation}\label{eq:impact parameter}
	\vec{\rho} := \tfrac{\nu\wedge(v-v_*)}{|v-v_*|}\in{\rm span}(v-v_*)^\perp
	\end{equation}
	the \emph{impact parameters}, with the Jacobian
	\[((v-v_*)\cdot\nu)_+\ud\nu\to |v-v_*|\ud\vec{\rho}.\]
	This allows us to redefine the post-collisional velocities $(v',v_*)$ for an interaction potential $\mathcal{U}$
	\begin{equation}
	\left\{\begin{aligne}{c}
	(v',v_*'):=\lim_{t\to\infty}(v_a(t),v_b(t))\\
	\frac{\ud}{\ud t}(x_a,x_b)=(v_a,v_b),~\frac{\ud}{\ud t}(v_a,v_b)=\alpha(-\nabla \mathcal{U}(x_b-x_a),\nabla \mathcal{U}(x_b-x_a))\\[5pt]
	\lim_{t\to-\infty}(v_a(t),v_b(t))=:(v,v_*),~(v_a-v_b)\times(x_a-x_b)=|v-v_*|\vec{\rho}.
	\end{aligne}\right.
	\end{equation}
	With this definition, the scattering map can easily be defined for a not compactly supported decreasing potential.
	
	For $\lambda>0$, we make the change of coordinate 
	\[(t,x_a,x_b,v_a,v_b)\mapsto(\tilde{t},\tilde{x}_a,\tilde{x}_b,v_a,v_b):=(\lambda t,\lambda x_a,\lambda x_b,v_a,v_b)\]
	In the new coordinates, the equations of motion become 
	\[\tfrac{\ud}{\ud \tilde{t}}(\tilde{x}_a,\tilde{x}_b)=(\tilde{v}_a,\tilde{v}_b),~\tfrac{\ud}{\ud \tilde{t}}(\tilde{v}_a,\tilde{v}_b)=\alpha\left(-\nabla\mathcal{U}\left(\tfrac{\tilde{x}_a-\tilde{x}_b}{\lambda}\right),\nabla\mathcal{U}\left(\tfrac{\tilde{x}_a-\tilde{x}_b}{\lambda}\right)\right)\]
	Hence, the post-collisional parameters associated with $(v,v_*,\vec{\rho})$ and potential $\mathcal{U}$ are the same as the ones associated with  $(v,v_*,\lambda\vec{\rho})$ and potential $\mathcal{U}(\cdot/\lambda)$.
	
	Performing the change of variable $\vec{\rho}\to\alpha^{-1/s}\vec{\rho}$ in the collisional operator gives
	\[\alpha^{2/s}\mathcal{L}_{\alpha\mathcal{V}} g =\mathcal{L}_{\mathcal{U}_\alpha} g,~{\rm where}~\mathcal{U}_\alpha(r) := \alpha\mathcal{V}(r\alpha^{1/s})=\frac{f\left(r\alpha^{\frac{1}{s}}\right)}{r^s}.\]
	This new potential converges when $\alpha\to 0$ to $\mathcal{U}^s(r):=1/r^s$. It is natural to guess the convergence of the operators
	\[\frac{1}{\alpha^{2/s}}\mathcal{L}_\alpha\to \mathcal{L}_{\mathcal{U}^s}.\]
	
	Section 3 of \cite{LW} is dedicated to the proof of this claim.
	
	\section{Geometrical estimates}
	
	\label{sec: Geometric estimates}

	\subsection{Estimation the length scattering time}
		\begin{lemma}\label{lemma:borne temps collision}
		Let $\mathcal{V}$ an interaction potential that is radial and supported in a ball of diameter $\e$. 
		
		We consider two particles, $1$ and $2$, with initial coordinates \[(x_a(0),v_a(0))=(0,v_1),~(x_b(0),v_b(0))=(\e\nu,v_2),~\nu\in\mathbb{S}^{d-1},~(v_1-v_2)\cdot\nu<0,\]  following the Hamiltonian dynamics linked to 
		\[H := \tfrac{|v_a|^2+|v_b|^2}{2}+\mathcal{V}(\tfrac{x_a-x_b}{\e}).\]
		
		Then, the scattering time is bound by 
		\begin{equation}\label{eq:borne temps collision}
		\blue{[\tau]:=\inf \{\tau>0~|x_a(\tau)-x_b(\tau)|>\e\} \lesssim  \frac{\e|v_1-v_2|}{|\nu\times (v_1-v_2)|^2}.}
		\end{equation}
	\end{lemma}
	\begin{proof}
		The motion equation are written as
		\begin{equation*}
		\begin{split}
		\tfrac{d}{dt}\left(x_a+x_b\right) = \left({v}_a+v_b\right) 	&\quad\tfrac{\ud}{\ud t}\left(v_a+v_b\right)=0\\
		\tfrac{\ud}{\ud t}\left(x_a-x_b\right) = \left(v_a-v_b\right) &\quad 	\tfrac{\ud}{\ud t}\left(v_a-v_b\right)=-\tfrac{2}\e\nabla\mathcal{V}\left(\tfrac{x_a-x_b}{\e}\right).
		\end{split}
		\end{equation*}
		
		Hence, $[\tau]$ does not depend on $v_1+v_2$.
		
		We use the impact parameter $\rho:=\tfrac{|\e \nu\times (v_1-v_2)|}{|v_1-v_2|}$ defined in \eqref{eq:impact parameter}. The time $[\tau]$ is given by (see chapter 8 of \cite{GST})
		\begin{align}
		[\tau] = \frac{2}{|v_1-v_2|}\int_{r_{\min}}^{\e/2} \frac{\ud r}{\sqrt{1-\tfrac{\rho^2}{r^2}-2\frac{\mathcal{V}(r/\e)}{|v_1-v_2|^2} }},
		\end{align}
		with $r_{\min}$ defined by 
		\[1-\tfrac{\rho^2}{r_{\min}^2}-2\frac{\mathcal{V}(r_{\min}/\e)}{|v_1-v_2|^2}=0.\]
		
		We begin by performing the change of variables
		\begin{equation}\label{eq:change of variable}
		u^2:=\frac{\rho^2}{r^2}+2\frac{\mathcal{V}(r/\e)}{|v_1-v_2|^2}
		\end{equation}
		which implies
		\begin{align}		
		[\tau] = \frac{2}{|v_1-v_2|}\int_{\rho/\e}^1 \frac{u\ud u}{\sqrt{1-u^2}}\frac{r}{\frac{\rho^2}{r^2}- \frac{r \mathcal{V}'(r/\e)}{\e|v_1-v_2|^2}}
		\end{align}
		
		\blue{Using that $\mathcal{V}'$ is non positive and that $u\geq\tfrac{\rho}{r}$, $r\leq \e$
		\begin{equation*}		
		[\tau] \leq \frac{2}{|v_1-v_2|}\int_{\rho/\e}^1 \frac{\ud u}{\sqrt{1-u^2}}\frac{u\e^3}{\rho^2} \lesssim \frac{\e^3}{|v_1-v_2|\rho^2}\lesssim  \frac{\e|v_1-v_2|}{|\nu\times (v_1-v_2)|^2}.
		\end{equation*}}
	\end{proof}

	\subsection{Proof of Proposition \ref{prop:estimation avec reco}}\label{subsec:estimation reco}
	
	The goal of this section is proving the following estimations \eqref{eq:Borne sur les chevauchements} and \eqref{eq: etimation pseudo avec reco}:
	
	\begin{lemma}\label{lem:estimation des reco}
	Let $\mathcal{V}$ be an interaction potential that is radial, decreasing, and supported in a ball of radius $\e$. 
	
	There exists a $C>0$ independent of $\mathcal{V}$ such that for $t=n_k\theta$
	\begin{multline}
	\sum_{T=(q_i,\bar{q}_i,s_i)_{i\leq n_k-1}}\int_{\mathfrak{T}_{n_{k}}\times\mathfrak{G}^0_{T}}\left(1-\ind_{\mathfrak{G}^0_{T}}\right)\Lambda_T(V_{n_K},\nu_{[n_k-1]})\ud{\nu}_{[n_k-1]}\ud{\tau_i} M^{\otimes n_{k}}dV_{n_{k}}\\
	+\sum_{\bar{T}}\int_{\mathfrak{T}_{\underline{n}_{k+2}}\times\mathfrak{G}^{>,t-t_s,\bar{T}}_{\{q\},\omega,(s_i)_i}}\prod_{i=1}^{n_{k+2}-1} |(\dr{v}^\e_{q_i}(\tau_i)-\dr{v}^\e_{q'_i}(\tau_i))\cdot\nu_i| \ud{\nu_i}\ud{\tau_i} M^{\otimes n_{k+2}}dV_{n_{k+2}}\\
	\leq C^{n_{K}}(n_{k})^{n_{k}}\theta^{(n_{k}-n_{k-1}-2)_+}t^{n_k}\e^{1/5}.
	\end{multline}
	\end{lemma}
	
	The two estimations can be performed in the same way, and we will only treat the first one. The proof of this lemma is an adaptation of the proof of Lemma 8 of \cite{PSS}. The estimation is not optimal. For example, the factor $\e^{1/4}$ can be replaced in the hard spheres setting by $\e|\log\e|^r$ for some constant $r>1$ (see for example \cite{BGSS1}). However, optimal estimates use the upper bound of the collision kernel of hard spheres. Such bounds are verified for more general potential and certainly not in the limit $\alpha\to0$. The proof of \cite{PSS} (which is adapted from it) is more robust.

	\begin{proof}
		
		We need to avoid 
		\begin{itemize}
			\item an \emph{overlap}: there exists a time $\tau\in(0,t)\cap\delta\mathbb{Z}$ and two particles $q$ and $q'$ such that
			\[\left|\ds{x}_q(\tau)-\ds{x}_{{q'}}(\tau)\right|\leq \e,\]
			\item a \emph{recollision}: there exists a time $\tau\in[0,t]$ and two particles $q$ and $q'$ such that $\tau\notin\{\tau_1,\cdots,\tau_{n_k-1}\}$ and \[\left|\ds{x}_q(\tau)-\ds{x}_{{q'}}(\tau)\right| =\e~{\rm and}~ \left(\ds{x}_q(\tau)-\ds{x}_{{q'}}(\tau)\right)\cdot\left(\ds{v}_q(\tau)-\ds{v}_{{q'}}(\tau)\right)<0.\]
			\end{itemize}
		
		\paragraph{We begin with the estimation of overlap, which is easier.} As the $i$-th collision between particles $(q_i,\bar{q_i})$ can last only a time $\blue{\tfrac{\e\left|\ds{v}_{q_i}(\tau_i^-)-\ds{v}_{\bar{q}_i}(\tau_i^-)\right|}{\left|\left(\ds{v}_{q_i}(\tau_i^-)-\ds{v}_{\bar{q}_i}(\tau_i^-)\times\nu_i\right)\right|^2}}$ (see \eqref{eq:borne temps collision}), there can be an overlap only if there is some $\tau\in\delta\mathbb{Z} \cap[0,t]$ such that $\tau_i$ is in the interval
		\[I_\e(\tau,V_{n_K},\nu_{[n_{k}-1]}):=\left[\tau-\blue{\tfrac{\e\left|\ds{v}_{q_i}(\tau_i^-)-\ds{v}_{\bar{q}_i}(\tau_i^-)\right|}{\left|\left(\ds{v}_{q_i}(\tau_i^-)-\ds{v}_{\bar{q}_i}(\tau_i^-)\times\nu_i\right)\right|^2}},\tau\right]\]
		Hence, the set of parameters leading to an overlap is smaller than
		
		\begin{multline*}
		\sum_{\substack{\tau\in\delta\mathbb{Z} \cap[0,t]\\1\leq i \leq n_{k}-1}}\sum_T\int_{\mathfrak{T}_{n_{k}}\times\mathfrak{G}^0_{T}}\ind_{I_\e(\tau,V_{n_K},\nu_{[n_{k}-1]})}(\tau_i)\Lambda_T(V_{n_K},\nu_{[n_k-1]})\ud{\nu_{[n_k-1]}}\ud{\tau_{[n_k-1]}} M^{\otimes n_{k}}dV_{n_{k}}\\
		\leq  \tfrac{C^{n_{k}}t^{1+n_{k-1}}\theta^{(n_{k}-n_{k-1}-1)_+}}{\delta n_{k}^{n_{k}}} \sum_{\substack{1\leq i \leq n_{k}-1\\[3pt]T}}\int_{\mathfrak{G}^0_{T}} t\wedge\blue{\tfrac{\e\left|\ds{v}_{q_i}(\tau_i^-)-\ds{v}_{\bar{q}_i}(\tau_i^-)\right|}{\left|\left(\ds{v}_{q_i}(\tau_i^-)-\ds{v}_{\bar{q}_i}(\tau_i^-)\right)\times\nu_i\right|^2}}\Lambda_T(V_{n_K},\nu_{[n_k-1]})\ud{\nu_{[n_k-1]}} M^{\otimes n_{k}}dV_{n_{k}}
		\end{multline*}
		We  can apply the same estimates as in Lemma \ref{Borne sur la taille des parrametres d'arbre} and for the terms \eqref{eq:R4}. We conclude that the set of overlap has a measure smaller than \[\frac{\e^{1/2} C^{n_{k}}t^{n_{k-1}+1}\theta^{(n_{k}-n_{k-1}-1)_+}}{\delta n_{k}^{n_{k}}}.\]
		
		\paragraph{We treat the recollision now.} We denote $\mathcal{G}$ the collision graph, and we define $\mathcal{G}^{[0,\tau)}$ the subgraph of $\mathcal{G}$ with edges 
		\[\left\{(q,\bar{q})_{\tau',\sigma}\in E(G),~\tau'<\tau\right\}.\]
		
		If the first recollision involves particles $q$ and $q'$ at time $\tau_{\rm rec}$, we consider $\omega\subset[n_{k+2}]$ the connected components of $\{q,q'\}$ in the collision graph $\mathcal{G}^{[0,\tau_{\rm rec})}$. Before the time $\tau_{\rm rec}$, the pseudotrajectory $\ds{Z}_\omega^\e(\tau)$ and its formal limit $\ds{Z}_\omega(\tau)$ are close up to a translation, and using Lemma \ref{Borne entre la trajectoire limite et la trajectoire d'enskog}, there exists a $y_0\in\mathbb{T}$ such that 
		\[\forall \tau\in[0,\tau_{\rm rec}],~|\ds{X}^0_\omega(\tau)-\tr_{y_0}\ds{X}^\e_\omega(\tau)|\leq 2n_K\mathbb{V}\sum_{i=1}^{n_{k+2}-1}\blue{1\wedge\tfrac{\e\left|\ds{v}_{q_i}(\tau_i^-)-\ds{v}_{\bar{q}_i}(\tau_i^-)\right|}{\left|\left(\ds{v}_{q_i}(\tau_i^-)-\ds{v}_{\bar{q}_i}(\tau_i^-)\right)\times\nu_i\right|^2}}.\]
		Hence, there is a recollision if at time $\tau_{\rm rec}\in[0,t]\setminus\bigcup_i\{\tau_i\}$,
		\begin{equation}\label{eq:overlap def alternative}
		\exists~q,q'\in[n_k-1]~{\rm such~that}~
		|\ds{x}^0_q(\tau_{\rm rec})-\ds{x}^0_{q'}(\tau_{\rm rec})|\leq \e+\sum_{i=1}^{n_{k+2}-1}\blue{1\wedge\e\tfrac{2n_K\mathbb{V}\left|\ds{v}_{q_i}(\tau_i^-)-\ds{v}_{\bar{q}_i}(\tau_i^-)\right|}{\left|\left(\ds{v}_{q_i}(\tau_i^-)-\ds{v}_{\bar{q}_i}(\tau_i^-)\right)\times\nu_i\right|^2}}.
		\end{equation}
		We can study only the limiting flow and defining a contact as "there exists a time $\tau_{\rm rec}$ such that \eqref{eq:overlap def alternative} is verified": we have 
		\[\mathfrak{T}_{\underline{n}_{k}}\times\mathfrak{G}^{0}_{T}\setminus\mathfrak{T}_{\underline{n}_{k}}\times\mathfrak{G}^{\e}_{T}\subset \mathfrak{T}_{\underline{n}_{k}}\times\mathfrak{G}^{0}_{T}\cap\{\text{at~least~one~contact}\}.\]
		
		The first step is to forbid the collisions that last too long. We define $\mathcal{E}_1\subset\mathfrak{T}_{\underline{n}_{k}}\times\mathfrak{G}^{0}_{T}$ as (for some $c_1,c_2\in(0,1)$)
		\begin{multline*}
			\mathcal{E}_1:= \bigg\{\forall i\leq n_{k}-1, \left|\ds{v}_{q_{i}}(\tau^-_{i})-\ds{v}_{q_{i}}(\tau^-_{i})\right| \min\left\{1,(\tau_i-\tau_{i-1}),(\tau_{i+1}-\tau_i)\right\}\geq \frac{\e^{c_1}}{n_{k}^2\mathbb{V}}\\
			\blue{\tfrac{\left|\left(\ds{v}_{q_i}(\tau_i^-)-\ds{v}_{\bar{q}_i}(\tau_i^-)\right)\times\nu_i\right|^2}{\left|\ds{v}_{q_i}(\tau_i^-)-\ds{v}_{\bar{q}_i}(\tau_i^-)\right|}\min\left\{1,(\tau_i-\tau_{i-1}),(\tau_{i+1}-\tau_i)\right\}\geq\frac{\e^{c_2}}{n_{k}^2\mathbb{V}}}\bigg\}.
		\end{multline*}
		In $\mathcal{E}_1$ there is a contact if there exists a time $\tau_{\rm rec}$ such that 
		\begin{equation}|\ds{x}^0_q(\tau_{\rm rec})-\ds{x}^0_{q'}(\tau_{\rm rec})|\leq 3\e^{1-c_2},\end{equation}
		and 
		\blue{\begin{multline*}1-\ind_{\mathcal{E}_1} \leq \sum_{i= 1}^{n_{k}-1} \ind_{\left|\ds{v}_{q_{i}}(\tau^-_{i})-\ds{v}_{q_{i}}(\tau^-_{i})\right|\leq \frac{\e^{c_1}}{n_{k}^2\mathbb{V}}}+\ind_{\frac{\left|\left(\ds{v}_{q_i}(\tau_i^-)-\ds{v}_{\bar{q}_i}(\tau_i^-)\right)\times\nu_i\right|^2}{\left|\ds{v}_{q_i}(\tau_i^-)-\ds{v}_{\bar{q}_i}(\tau_i^-)\right|}\leq\frac{\e^{c_2}}{n_{k}^2\mathbb{V}}}\\
		+\ind_{|\tau_i-\tau_{i-1}|\leq  \max\left( \frac{\e^{c_1}}{n_{k}^2\mathbb{V}\left|\ds{v}_{q_{i}}(\tau^-_{i})-\ds{v}_{q_{i}}(\tau^-_{i})\right|},\frac{\e^{c_2}\left|\ds{v}_{q_i}(\tau_i^-)-\ds{v}_{\bar{q}_i}(\tau_i^-)\right|}{n_{k}^2\mathbb{V}\left|\left(\ds{v}_{q_i}(\tau_i^-)-\ds{v}_{\bar{q}_i}(\tau_i^-)\right)\times\nu_i\right|^2}\right)}\\
		+\ind_{|\tau_{i+1}-\tau_{i}|\leq  \max\left( \frac{\e^{c_1}}{n_{k}^2\mathbb{V}\left|\ds{v}_{q_{i}}(\tau^-_{i})-\ds{v}_{q_{i}}(\tau^-_{i})\right|},\frac{\e^{c_2}\left|\ds{v}_{q_i}(\tau_i^-)-\ds{v}_{\bar{q}_i}(\tau_i^-)\right|}{n_{k}^2\mathbb{V}\left|\left(\ds{v}_{q_i}(\tau_i^-)-\ds{v}_{\bar{q}_i}(\tau_i^-)\right)\times\nu_i\right|^2}\right)}. \end{multline*}}
		
		Now, we place ourselves in $\mathcal{E}_1$ and we fix a collision tree $T$. The first contact happens at time $\tau_{\rm rec}$ between particles $q_{\rm rec}$ and $q'_{\rm rec}$. There exists a collision $i_0$, two disjoint sequences of collisions $(i_j)_{j\leq p}$ and $(i'_j)_{j\leq p'}$ and two sequences of particles $(a_j)_{j\leq p}$, $(a'_j)_{j\leq p'}$ such that
		\begin{itemize}
			\item $\forall j\geq 1$, $i_0<i_j$ and $i_0<i'_j$,
			\item $\forall j\geq 1$,  $a_j\in\{q_{i_j},q'_{i_j}\}\cap\{q_{i_{j-1}},q'_{i_{j-1}}\}$ and $a'_j\in\{q_{i'_j},q'_{i'_j}\}\cap\{q_{i'_{j-1}},q'_{i'_{j-1}}\}$,
			\item if for $j<j'$, $a_j = a_{j'}$. Then for any $i\in[i_j,i_{j'}]$ such that $a_j\in \{q_i,q'_i\}$, we have $j\in\{i_j,i_{j+1},\cdots,i_{j'}\}$, and similarly for the sequences $(i'_j)_{j\leq p'}$, $(a'_j)_{j\leq p'}$,
			\item $a_0= q_{i_0}$, $a'_0= q'_{i_0}$ and $\{a_p,a'_{p'}\}=\{q_{\rm rec},q'_{\rm rec}\}$.
		\end{itemize}
		
		\blue{There is a contact if 
		\begin{multline}\label{eq:equation overlap 1}
		\min_{\substack{s\in[-t,t]{R}\\y_0\in \mathbb{Z}^d}}\left|y_0+\sum_{j=1}^p \ds{v}_{a_j}(]\tau_{i_{j-1}},\tau_{i_{j}}[)(\tau_{i_{j}}-\tau_{i_{j-1}})+\ds{v}_{a_p}(]\tau_{i_p},\tau_{\rm rec}[)(s-\tau_{i_p})\right.\\ \left.-\sum_{j=1}^{p'} \ds{v}_{a'_j}(]\tau_{i'_{j-1}},\tau_{i'_{j}}[)(\tau_{i'_{j}}-\tau_{i'_{j-1}})+\ds{v}_{a'_{p'}}(]\tau_{i'_{p'}},\tau_{\rm rec}[)(s-\tau_{i'_{p'}})\right|\leq \e^{1-c_2}.
		\end{multline}
		In addition, if the $(\tau_i)_i$ and $(\tau'_i)_i$ verify $\tau_{i_0}<\tau_{i_1}<\cdots<\tau_{i_p}<\tau_{\rm rec}$ and $\tau_{i_0}<\tau_{i'_1}<\cdots<\tau_{i'_p}<\tau_{\rm rec}$, there is a contact if 
		\begin{multline}\label{eq:equation overlap 3}
			\min_{\substack{s\in[\tau_{i_0+1},t]{R}\\y_0\in \mathbb{Z}^d}}\left|y_0+\sum_{j=1}^p \ds{v}_{a_j}(]\tau_{i_{j-1}},\tau_{i_{j}}[)(\tau_{i_{j}}-\tau_{i_{j-1}})+\ds{v}_{a_p}(]\tau_{i_p},\tau_{\rm rec}[)(s-\tau_{i_p})\right.\\ \left.-\sum_{j=1}^{p'} \ds{v}_{a'_j}(]\tau_{i'_{j-1}},\tau_{i'_{j}}[)(\tau_{i'_{j}}-\tau_{i'_{j-1}})+\ds{v}_{a'_{p'}}(]\tau_{i'_{p'}},\tau_{\rm rec}[)(s-\tau_{i'_{p'}})\right|\leq \e^{1-c_2}.
		\end{multline}

		\begin{remark}
			Note that we can perform the previous construction if $\tau_{\rm rec}>\tau_{i_0}$. If we study the overlap for particles of size $\e$, it is always the case. But for punctual particles, the pathology can happen before the first collision.  Our proof can be adapted by taking $i_0$ such that $i_0>i_j$ and $i_0>i')_j$, making the change of variable $V_{n_{k}}\mapsto\ds{V}_{n_{k}}(\tau_{i_0}^-)$ and looking at everything backwards.
		\end{remark}
		
		Denoting 
		\begin{align*}
		\Delta \ds{v}_j :=\ds{v}_{a_{j+1}}(]\tau_{i_{j}},\tau_{i_{j+1}}[)- \ds{v}_{a_j}(]\tau_{i_{j-1}},\tau_{i_{j}}[),&~\Delta \ds{v}'_j :=\ds{v}_{a'_{j+1}}(]\tau_{i'_{j}},\tau_{i'_{j+1}}[)- \ds{v}_{a'_j}(]\tau_{i'_{j-1}},\tau_{i'_{j}}[),\\
		~w_0:=\ds{v}_{a_0}(\tau_{i_0}^+)-\ds{v}_{a'_0}(\tau_{i_0}^+),~
		\text{and}&~w_{\rm f}=\ds{v}_{a_p}(]\tau_{i_p},\tau_{\rm rec}[)-\ds{v}_{a'_{p'}}(]\tau_{i_{p'}'},\tau_{\rm rec}[),
		\end{align*}
		the equation \eqref{eq:equation overlap 1} can be written as
		\begin{equation}\label{eq:equation overlap 2}
		\min_{\substack{s\in[-t,t]\\y_0\in \mathbb{Z}^d}}\left|y_0-\tau_{i_0}w_0 +\sum_{j=1}^p \Delta\ds{v}_{j}\tau_{i_{j}}-\sum_{j=1}^{p'} \Delta\ds{v}_{j}'\tau_{i'_{j}}+sw_{\rm f}\right|	\leq \e^{1-c_2}
		\end{equation}
		
		In the following, we denote $\hat{w}_{\rm f}:=\tfrac{{w}_{\rm f}}{|{w}_{\rm f}|}$. We take the cross product with $\hat{w}_{\rm f}$ in \eqref{eq:equation overlap 2}:
		\begin{equation}\label{eq:equation overlap 4}
		\min_{\substack{y_0\in \mathbb{Z}^d}}\left|(y_0-\tau_{i_0}w_0)\times \hat{w}_{\rm f} +\sum_{j=1}^p \Delta\ds{v}_{j}\times \hat{w}_{\rm f}\tau_{i_{j}}-\sum_{j=1}^{p'} \Delta\ds{v}'_{j}\times \hat{w}_{\rm f}\tau_{i'_{j}}\right|	\leq \e^{1-c_2}.
		\end{equation}

		We fix $y_0$ such that the minimum is reached. We distinguish three cases:
		
		\begin{enumerate}
			\item If there exists a $\ell_1\in\mathbb{N}$ such that $|\Delta\ds{v}_{\ell_1}\times \hat{w}_{\rm f}|\geq\tfrac{\e^{c_3}}{n_{k}}$ (for some $c_3\in(c_1,1)$), $\tau_{i_{\ell_1}}$ has to be in an interval of length $\e^{1-c_2-c_3}$. 
		
			\item Else, if $y_0$ is non zero, using that 
			\[w_0 =-\sum_{j=1}^p \Delta\ds{v}_{j}+\sum_{j=1}^{p'} \Delta\ds{v}_{j}'-w_{\rm f}\]
			we have 
			\begin{equation*}|w_0\times \hat{w}_{\rm f}| \leq \sum_{j=1}^p \left|\Delta\ds{v}_{j}\times \hat{w}_{\rm f}\right|+\sum_{j=1}^{p'} \left|\Delta\ds{v}_{j}\times \hat{w}_{\rm f}\right|\leq \e^{c_3}.
			\end{equation*}
			Hence $\left|y_0\times w_f\right|$ has to be smaller than $2t\mathbb{V}\e^{c_3}+\e^{1-c_2}$. Combining with \eqref{eq:equation overlap 4}, we deduce that there exists a collision $j$, a vector $y_0\in\mathbb{Z}^d$ with $|y_0|\leq n_Kt\mathbb{V}$, and a couple of particles $(q,q')$ such that $w_{\rm f} = \left(\ds{v}_{q}(\tau_j^-)-\ds{v}_{q'}(\tau_j^-)\right)$, and then
			\[\left|\left(\ds{v}_{q}(\tau_j^-)-\ds{v}_{q'}(\tau_j^-)\right)\times y_0\right|\leq 2t\mathbb{V}\e^{c_3}+\e^{1-c_2}.\]
			
			\item We treat now the case $y_0= 0$. We begin to show that there exists some $\ell'$ such that $|\Delta\ds{v}_{\ell'}|$ or $|\Delta\ds{v}'_{\ell'}|$ is bigger than $\tfrac{\e^{c_1}}{3n_{k}t}$ (in the following, we suppose that it is $|\Delta\ds{v}_{\ell'}|$).
			\begin{proof}
			We proceed by contradiction. As for all $j\leq p$, we are in one of the three following cases,
			\begin{enumerate}
			 	\item $i_{j-1},i_{j+1}$ are smaller than $i_j$, and
			 	\[\left|\Delta\ds{v}_j\right| = \left|\ds{v}_{q_{i_j}}(\tau_{i_j}^-)-\ds{v}_{q'_{i_j}}(\tau_{i_j}^-)\right|,\]
			 	\item $i_{j-1},i_{j+1}$ are bigger than $i_j$, and
			 	\[\left|\Delta\ds{v}_j\right| = \left|\ds{v}_{q_{i_j}}(\tau_{i_j}^+)-\ds{v}_{q'_{i_j}}(\tau_{i_j}^+)\right|,\]
			 	\item $i_{j-1}<i_j<i_{j+1}$ or $i_{j-1}<i_j<i_{j+1}$.
			 \end{enumerate}
			 As the difference of velocities is conserved by the scattering map and we are in $\mathcal{E}_1$, the cases (a) and (b) are impossible. Using that $i_0<i_1$, we deduce that $\tau_{i_0}<\tau_{i_1}<\cdots <\tau_{i_p}<\tau_{\rm rec}$. In the same way, $\tau_{i'_0}<\tau_{i'_1}<\cdots <\tau_{i'_{p'}}<\tau_{\rm rec}$.
			 
			 Using triangular inequality,
			 \[|w_f-w_0| = \left|\sum_{j=1}^p \Delta\ds{v}_{j}-\sum_{j=1}^{p'} \Delta\ds{v}'_{j}\right|\leq \frac{\e^{c_1}}{3t}.\]
			 Considering the $s\in[\tau_{i_0+1},t]$ such that the minimum of  \eqref{eq:equation overlap 4} is reached, we have 
			 \begin{align*}
			 	|(\tau_{i_0+1}-\tau_{i_0})w_0|&\leq|(\tau_{i_0}-s)w_0|\\
			 	&\leq \e^{1-c_2}+\left(\sum_{j=1}^p \left|\Delta\ds{v}_{j}\right|+\sum_{j=1}^{p'}\left| \Delta\ds{v}'_{j}\right|\right)t+s|w_f-w_0|\\
			 	&\leq \e^{1-c_2}+\frac{2\e^{c_1}}{3},
			 \end{align*}
			 which is impossible since we are in $\mathcal{E}_1$.
		\end{proof}
		As $|\Delta\ds{v}_{\ell'}\times \hat{w}_{\rm f}|\leq\tfrac{\e^{c_3}}{n_{k}}$, (else it would have been treated in Point (1))
		\[\left|\frac{\Delta\ds{v}_{\ell_0}}{|\Delta\ds{v}_{\ell_0}|}\times \hat{w}_{\rm f}\right|\leq\frac{t\e^{c_3}}{n_{k}|\Delta\ds{v}_{\ell_0}|}\leq t\e^{c_3-c_1} ~{\rm and}~\left|w_0\times\hat{w}_f\right|\leq \e^{c_3}.\]
		Finally, we obtain
		\begin{equation}
		\left|\frac{\Delta\ds{v}_{\ell_0}}{|\Delta\ds{v}_{\ell_0}|}\times {w}_{0}\right|\leq \e^{c_3}+t\e^{c_3-c_1}.
		\end{equation}
		\end{enumerate}}
		
		\begin{definition}
			For $\omega\in\mathbb{R}^d$ and $\nu\in\mathbb{S}^{d-1}$, we denote (using $\xi_\alpha$ defined in Definition \ref{def: le scatering, sa vie son oeuvre})
			\[(\tfrac{w'}{2},-\tfrac{w'}{2},\nu')=\xi_\alpha((w/2,-w/2,\nu)).\]
			We define
			\begin{equation}
				\zeta_1(w,\nu):=w,~\zeta_2(w,\nu):=w',~\zeta_3(w,\nu):=\frac{w-w'}{2}~\text{and}~\zeta_4(w,\nu):=\frac{w+w'}{2}.
			\end{equation}
		\end{definition}
		Note that $\Delta\ds{v}_{\ell_0}$ is equal one of the $(\pm\zeta_i(\ds{v}_{q_{\ell_0}}(\tau_{\ell_0}^-)-\ds{v}_{q'_{\ell_0}}(\tau_{\ell_0}^-),\nu_{\ell_0})_{i\leq 4}$.
		
		We conclude that
		\begin{multline}
		\int_{\mathfrak{T}_{n_{k}}}\ud{\tau_{n_k}}\ind_{\rm contact}
		\leq \frac{C^{n_{k}}t^{n_{k-1}}\theta^{(n_{k}-n_{k-1}-1)_+}}{n_{k}^{n_{k}}} \Bigg(\e^{1-c_2-c_3}+ \\
		+\sum_i\ind_{\left|\ds{v}_{q_{i}}(\tau^-_{i})-\ds{v}_{q_{i}}(\tau^-_{i})\right|\leq \e^{c_1}}\tfrac{\e^{c_1}}{\left|\ds{v}_{q_{i}}(\tau^-_{i})-\ds{v}_{q_{i}}(\tau^-_{i})\right|} +\ind_{\frac{\left|\left(\ds{v}_{q_i}(\tau_i^-)-\ds{v}_{\bar{q}_i}(\tau_i^-)\right)\times\nu_i\right|^2}{\left|\ds{v}_{q_i}(\tau_i^-)-\ds{v}_{\bar{q}_i}(\tau_i^-)\right|}\leq\e^{c_2}}\\
		+\sum_{\substack{y_0\in\mathbb{Z}^d\\|y_0|\leq tn_k\mathbb{V}}}\sum_{\substack{i\leq n_k\\ (q,q')}}\ind_{\left|\left(\ds{v}_{q}(\tau_j^-)-\ds{v}_{q'}(\tau_j^-)\right)\times y_0\right|\leq 2t\mathbb{V}\e^{1-c_2}}\\
		+\sum_{i,j}\sum_{\ell= 1}^4\ind_{\left|\frac{\zeta_\ell(\ds{v}_{q_{j}}(\tau_{j}^-)-\ds{v}_{q'_{j}}(\tau_{j}^-),\nu_j)}{|\zeta_\ell(\ds{v}_{q_{j}}(\tau_{j}^-)-\ds{v}_{q'_{j}}(\tau_{j}^-),\nu_j)|}\times (\ds{v}_{q_{i}}(\tau_{i}^+)-\ds{v}_{q'_{i}}(\tau_{i}^+))\right|\leq t\e^{c_3-c_1}}\Bigg).
		\end{multline}
	
		We have to integrate now with respect to $(\nu_{[n_{k}-1]},V_{n_{k}})$. As in the proof of \eqref{eq: borne R1}, we use the applications $\Xi^i_T$ defined by \eqref{eq:def des Xi}.
		
		We recall that 
		\[\Xi_T^{i+1}\Xi_T^{i}\cdots \Xi_T^1(\nu_{[n_{k}-1]},V_{n_{k}}) = (\tilde{\nu}_{[n_{k}-1]},\tilde{V}_{n_k}=\ds{V}_{n_{k}}(\tau_i^{+}))\]
		and that the Jacobian of this application is $1$. We can integrate with respect to $\ds{v}_{q}(\tau^-_{i})-\ds{v}_{q}(\tau^-_{i})$. 
		
		We treat only the last singularity; the other ones can be estimated in the same way. If we remove the edges $(q_j,q'_j)$ from $T$, either $q_j$ or $q_j'$ is not in the connected component of $\{q_i,q'_i\}$. Without loss of generality, we suppose that it is $q_j$. We denote $\omega$ the connected components of $q_j$ in $T\setminus\{(q_j,q'_j,s_j)\}$. Before the collision $j$, the particles of $\omega$ are independent of the other ones, and as before, we can construct an application of Jacobian $1$
		\[\bar{\Xi}:(\nu_{[n_{k}-1]},V_{n_{k}})\mapsto(\bar{\nu},\ds{V}_{\omega^c})(\tau_i^+),\ds{V}_{\omega}(\tau_j^-)).\]
		In addition,
		\[\begin{split}
		|V_{n_{k}}|^2&=\frac{|\ds{V}_{n_{k}}(\tau_i^+)|^2}{2}+\frac{|\ds{V}_{n_{k}}(\tau_j^-)|^2}{2}\\
		&\geq\frac{|\ds{V}_{\omega^c}(\tau_i^+)|^2}{2}+\frac{|\ds{v}_{q_j}(\tau_j^-)|^2+|\ds{v}_{q_j}(\tau_j^-)|^2}{2}+\frac{|\ds{V}_{\omega\setminus \{q_j\}}(\tau_j^-)|^2}{2}\\
		&\geq\frac{|\ds{V}_{\omega^c}(\tau_i^+)|^2}{2}+\frac{|\ds{v}_{q_j}(\tau_j^-)-\ds{v}_{q_j}(\tau_j^-)|^2}{4}+\frac{|\ds{V}_{\omega\setminus \{q_j\}}(\tau_j^-)|^2}{2}, 
		\end{split} \]
		and denoting $w:=\ds{v}_{q_j}(\tau_j^-)-\ds{v}_{q_j}(\tau_j^-)$, we can integrate with respect to the velocities
		\begin{align*}
		\sum_T&\int_{\mathfrak{G}^0_{T}}\ind_{\left|\frac{\zeta_\ell(\ds{v}_{q_{j}}(\tau_{j}^-)-\ds{v}_{q'_{j}}(\tau_{j}^-),\nu_j)}{|\zeta_\ell(\ds{v}_{q_{j}}(\tau_{j}^-)-\ds{v}_{q'_{j}}(\tau_{j}^-),\nu_j)|}\times (\ds{v}_{q_{i}}(\tau_{i}^+)-\ds{v}_{q'_{i}}(\tau_{i}^+))\right|\leq\e^{1/4}}\Lambda_T(V_{n_K},\nu_{[n_K-1]}) M^{\otimes n_{k}}\ud{\nu_{[n_{k}-1]}}\ud V_{n_{k}}\\
		&\leq\int\ind_{\left|\frac{\zeta_\ell(w,\nu_j)}{|\zeta_\ell(w,\nu_j)|}\times ({v}_{q_{i}}-{v}_{q'_{i}})\right|\leq\e^{1/4}}\frac{(C(1+|V_{n_k}|^2))^{n_k}e^{-\frac{|V_{[n_{k}]\setminus\{q_j\}}|^2}{4}-\frac{|w|^2}{8}}}{(2\pi)^{\frac{n_{k}d}{2}}}\ud{\nu_{[n_{k}-1]}}\ud \tilde{V}_{[n_{k}]}\ud{w}\\
		&\leq n_{k}^{2n_k}C^{n_{k}}\e^{c_3-c_1}.
		\end{align*}
		We have first integrate with respect to $(v_{q_i},c_{q_i'})$ and then the $(w,\nu_j)$. Finally, we obtain
		\begin{multline}
		\frac{1}{(\mu\mathfrak{d})^{n_{k}-1}}\sum_T\int_{\mathfrak{T}_{n_{k}}\times\mathfrak{G}^0_{T}}\ind_{\rm reco}\prod_{i=1}^{n_{k}-1} |(\ds{v}_{q_i}(\tau_i)-\ds{v}_{q'_i}(\tau_i))\cdot\nu_i| \ud{\nu_i}\ud{\tau_i} M^{\otimes n_{k}}dV_{n_{k}}\\
		\leq \frac{(n_{k}C)^{n_{k}}t^{n_{k-1}}\theta^{(n_{k}-n_{k-1}-1)_+}}{(\mu\mathfrak{d})^{n_{k}-1}}\left(\e^\frac{c_2}{2}+\e^{c_3-c_1}+\e^{c_1}+\e^{1-c_2}+\e^{1-c_2-c_3}\right).
		\end{multline}
		We choose $c_1 = \tfrac{1}{5}$ and $c_2 =c_3= \tfrac{2}{5}$. This concludes the proof.
	\end{proof}
	
	\subsection{Proof of \eqref{eq:estimation a la con}}
	\begin{lemma}\label{lem:preuve du lemme a la con}
		For $r\leq 2\gamma$,
		\begin{equation}
		\int\ind_{\substack{Z_{r}\text{\,form\,a~~~}\\ {\rm possible\,cluster}}}M^{\otimes r}\ud Z_{r}\leq C_\gamma \mu^{-r+1}\delta^{r-1}.
		\end{equation}
	\end{lemma}
	\begin{proof}
		First choose a family $\omega_1,\cdots,\omega_p$ of subset covering $[r]$, and $(\underline{\lambda}_i)_{i\leq p} = (\lambda_i^1,\cdots,\lambda_i^{l_i})_{i\leq p}$ a family of partitions of $\omega_i$. As $n$ is bounded, there are a finite number of possible $((\omega_i)_i,(\underline{\lambda}_i)_i)$. We construct the graph $\mathcal{G}$ as the merge of the collision graph of $\ds{Z}(\tau,Z_{\omega_i},\underline{\lambda}_i)$ on $[0,\delta]$, and we extract $T$ the clustering tree (there are less than $(2\gamma)^{2\gamma}$ possible clustering trees). We can then adapt the proof of \eqref{Estimation sans reco 2} (where we treated only two pseudotrajectories), and we obtain the expected result.
	\end{proof}

	\subsection{Proof of \eqref{estimation recollision}}\label{subsec: derniere appencicite}
	We recall that $\mathcal{O}_r\subset\mathbb{D}^r $ is the set
	\begin{multline*}\mathcal{O}_r:=\Big\{Z_r\in \mathbb{D}^r,~\exists (\lambda_1,\cdots,\lambda_{\gr{l}}),~\text{the~collision~graph~of~}\ds{Z}_r(\cdot,Z_z,(\lambda_1,\cdots,\lambda_{\gr{l}}))\text{~on~}[0,\delta]~\text{is~}\\
	\text{connected~and~the~pseudotrajectory~has~a~collision~or~a~multiple~interaction}\Big\}.\end{multline*}

	\begin{prop}
		There exists a positive constant $C_r$  depending only on the dimension and the number of particles $r$ such that
		\begin{equation}
		\int_{\mathbb{T}^{r-1}\times B_r(\mathbb{V})} \ind_{\mathcal{O}_r}(Z_r)\frac{e^{-\mathcal{H}_r}}{(2\pi)^{\frac{rd}{2}}} \ud X_{2,r}\ud V_r\leq \frac{C_r}{(\mu\mathfrak{d})^{r-1}}\delta^{2}\theta^{r-2}\e^\mathfrak{1/12},
		\end{equation}
		where $B_r(\mathbb{V})$ is the ball of radius $\mathbb{V}$ in dimension $rd$ (we use that $\delta=\e^{1/12}$).
	\end{prop}
	
	\begin{proof}
		For $r=2$, we have 
	\[\int_{\mathbb{D}^2} \ind_{\mathcal{O}_2}(Z_2)\frac{e^{-\mathcal{H}_2}}{(2\pi)^d} \ud x_2\ud V_2 = C \e^d \leq \frac{C\delta^2\e^{1/12}}{\mu\mathfrak{d}}\]
	as $\delta^2\e^{1/12} = \e^{3/12}\leq \e$.

	Fix parameters $(\lambda_1,\cdots\lambda_{\gr{l}})$ and denote $\mathcal{R}_{(\lambda_1,\cdots\lambda_{\gr{l}})}\subset\mathbb{D}^r$ the set of initial configuration such that the pseudotrajectory has a connected collision graph and a pathological collision.
	
	As we suppose that the collision graph is connected, we can construct a clustering tree $T:=(q_i,\bar{q}_i)_{i\leq r-1}$ as in the proof of Proposition \ref{prop:estimation sans reco}. We define $\tau_{\rm path}$ the time of the first pathological collision. The corresponding collision can either create a loop in the collision graph or be a clustering multiple \red{interaction}\blue{encounter}.
	
	The first case can be treated as a recollision, which is already done in the proof of Proposition \ref{prop:estimation sans reco} and in the preceding section, we have
	\begin{equation*}
	\int_{\mathbb{T}^{r-1}\times B_r(\mathbb{V})} \ind_{\mathcal{R}_{(\lambda_1,\cdots\lambda_{\gr{l}})}^{\rm reco}}\frac{e^{-\mathcal{H}_r}}{(2\pi)^{\frac{rd}{2}}} \ud X_{2,r}\ud V_r\leq \frac{C_r\delta^{r-2}}{(\mu\mathfrak{d})^{r-1}}\e^{1/5}\leq \frac{C_r}{(\mu\mathfrak{d})^{r-1}}\delta^{2}\theta^{r-2}\e^\mathfrak{1/12}.
	\end{equation*}
	
	In the second case, there are two clustering collisions $j<\tilde\jmath$ such that $\{q_j,\bar{q}_j\}\cap\{q_{\tilde\jmath},\bar{q}_{\tilde\jmath}\}$ and
	\[\forall \tau\in(\tau_j,\tau_{\tilde\jmath}),~|\ds{x}_{q_j}(\tau)-\ds{x}_{\bar{q}_j}(\tau)|<\e.\]
	Two particles $(q_j,q'_j)$ can stay at a distance shorter than $\e$ on a time interval shorter than \[\blue{1\wedge\frac{\e|v_{q_j}(\tau_j)-v_{\bar{q}_j}(\tau_j)|}{|(v_{q_j}(\tau_j)-v_{\bar{q}_j}(\tau_j))\wedge(x_{q_j}(\tau_j)-x_{\bar{q}_j}(\tau_j))|^2}}\] (whose integral is bounded by $\e^\frac12$). Hence, using the same strategy than in the proof of Proposition \ref{prop:estimation sans reco}:
	\begin{align*}
	\int_{\mathbb{T}^{r-1}\times B_r(\mathbb{V})} \ind_{\mathcal{R}_{(\lambda_1,\cdots\lambda_{\gr{l}})}^{\rm mult}}\frac{e^{-\mathcal{H}_r}}{(2\pi)^{\frac{rd}{2}}}\ud X_{2,r}\ud V_r\leq \frac{C_r\delta^{r-1}}{(\mu\mathfrak{d})^{r-2}}\e^{1/2}\leq\frac{C_r}{(\mu\mathfrak{d})^{r-1}}\delta^{2}\theta^{r-2}\e^\mathfrak{1/12}.
	\end{align*}
	
	Summing on all the possible $(\lambda_1,\cdots\lambda_{\gr{l}})$, we obtain the expected result.
	\end{proof}

	\begin{proof}[Proof of \eqref{estimation recollision}]
		We have now to prove Estimation \eqref{estimation recollision}:
	\[\int_{{\mathbb{T}}^{r-1}\times(\mathbb{R}^d)^{r}} \ind_{\mathcal{O}_\varpi}\ind_{\!\!\!\!\tiny\begin{array}{l}Z_{r}\text{\,form\,a}\\\text{possible\,cluster}\end{array}}\tfrac{e^{-\frac{1}{2}\mathcal{H}_{r}(Z_{r})}}{(2\pi)^{dr/2}}\ud{X}_{2,r}\ud V_{r}\leq\frac{C_r}{(\mu\mathfrak{d})^{r-1}}\delta^{2}\theta^{r-2}\e^\mathfrak{1/12},\]
	Without loss of generality, we suppose that $1\in\varpi$.
	
	Fix the  family $\omega_1,\cdots,\omega_p$ of subset covering $[n]$,
	and $(\underline{\lambda}_i)_{i\leq p} = (\lambda_i^1,\cdots,\lambda_i^{l_i})_{i\leq p}$ a family of partition of $\omega_i$ such that the union of the collision graph associated to parameters $((\lambda_i^j)_j)_i $ is connected.
	
	We begin by fix $Z_\varpi$ and sum the clustering of the particle in $[n]\setminus\varpi$
	\[\int_{{\mathbb{T}}^{n-1}\times(\mathbb{R}^d)^{n}} \ind_{\mathcal{O}_\varpi}\ind_{\!\!\!\!\tiny\begin{array}{l}Z_{r}\text{\,form\,a}\\\text{possible\,cluster}\end{array}}\frac{e^{-\frac{1}{2}\mathcal{H}_{n}(Z_{n})}}{(2\pi)^{dn/2}}\ud{X}_{\varpi}\ud V_{\varpi}
	\leq\ind_{\mathcal{O}_\varpi}\frac{e^{-\frac{1}{2}\mathcal{H}_{|\varpi|}(Z_{\varpi})}}{(2\pi)^{d|\varpi|/2}} \frac{C_n\delta^{n-|\varpi|}}{(\mu\mathfrak{d})^{|\varpi|}}.\]
	Then, we integrate with respect to $\ud{X}_{\varpi\setminus\{1\}}\ud V_\varpi$.
	\end{proof}
	
	{\large\textbf{Acknowledgment}}: The author thanks Sergio Simonella and Raphael Winter for their numerous suggestions as well as for the misprint corrections which helped to improve the paper. He also thanks Laure Saint-Raymond for stimulating and fruitful discussions. The author acknowledges financial support from the European Union (ERC, PASTIS, Grant Agreement n$^\circ$101075879).\footnote{{Views and opinions expressed are however those of the author only and do not necessarily reflect those of the European Union or the European Research Council Executive Agency. Neither the European Union nor the granting authority can be held responsible for them.}}
	
	\bibliographystyle{alpha}
	
\end{document}